\documentclass[12pt]{article}
\usepackage{fullpage, amsmath, amssymb, amsfonts, amsthm, stmaryrd, url, hyperref, combelow, mathtools}
\usepackage[all]{xy}
\usepackage[pdftex]{color}
\newtheorem{theorem}{Theorem}[subsection]
\newtheorem{lemma}[theorem]{Lemma}
\newtheorem{cor}[theorem]{Corollary}
\newtheorem{conj}[theorem]{Conjecture}
\newtheorem{prop}[theorem]{Proposition}
\theoremstyle{definition}
\newtheorem{defn}[theorem]{Definition}
\newtheorem{hypothesis}[theorem]{Hypothesis}
\newtheorem{example}[theorem]{Example}
\newtheorem{remark}[theorem]{Remark}
\newtheorem{convention}[theorem]{Convention}

\numberwithin{equation}{theorem}

\newcommand{\CC}{\mathbb{C}}

\newcommand{\FF}{\mathbb{F}}
\newcommand{\Fp}{\mathbb{F}_p}
\newcommand{\GG}{\mathbb{G}}
\newcommand{\LL}{\mathbb{L}}

\newcommand{\PP}{\mathbb{P}}
\newcommand{\Qp}{\mathbb{Q}_p}
\newcommand{\QQ}{\mathbb{Q}}
\newcommand{\RR}{\mathbb{R}}
\newcommand{\Zp}{\mathbb{Z}_p}
\newcommand{\ZZ}{\mathbb{Z}}
\newcommand{\bA}{\mathbf{A}}
\newcommand{\bB}{\mathbf{B}}
\newcommand{\bC}{\mathbf{C}}

\newcommand{\be}{\mathbf{e}}
\newcommand{\bv}{\mathbf{v}}
\newcommand{\bw}{\mathbf{w}}

\newcommand{\calB}{\mathcal{B}}
\newcommand{\calC}{\mathcal{C}}
\newcommand{\calE}{\mathcal{E}}
\newcommand{\calF}{\mathcal{F}}
\newcommand{\calG}{\mathcal{G}}
\newcommand{\calH}{\mathcal{H}}
\newcommand{\calI}{\mathcal{I}}
\newcommand{\calJ}{\mathcal{J}}
\newcommand{\calL}{\mathcal{L}}
\newcommand{\calM}{\mathcal{M}}
\newcommand{\calO}{\mathcal{O}}
\newcommand{\calP}{\mathcal{P}}
\newcommand{\calR}{\mathcal{R}}
\newcommand{\gothm}{\mathfrak{m}}
\newcommand{\gotho}{\mathfrak{o}}
\newcommand{\gothp}{\mathfrak{p}}

\newcommand{\gothU}{\mathfrak{U}}
\newcommand{\gothV}{\mathfrak{V}}
\newcommand{\dual}{\vee}

\DeclareMathOperator{\an}{an}
\DeclareMathOperator{\ann}{ann}
\DeclareMathOperator{\Aut}{Aut}
\DeclareMathOperator{\bd}{bd}
\DeclareMathOperator{\Cont}{Cont}
\DeclareMathOperator{\cont}{cont}
\DeclareMathOperator{\coker}{coker}
\DeclareMathOperator{\CPhi}{\mathcal{C}\Phi}
\DeclareMathOperator{\Disc}{Disc}
\DeclareMathOperator{\dR}{dR}

\DeclareMathOperator{\et}{\acute{e}t}
\DeclareMathOperator{\Ext}{Ext}
\DeclareMathOperator{\faith}{v}
\DeclareMathOperator{\fet}{f\acute{e}t}
\DeclareMathOperator{\FEt}{\mathbf{F\acute{E}t}}
\DeclareMathOperator{\FFC}{FF}
\DeclareMathOperator{\Fitt}{Fitt}
\DeclareMathOperator{\Frac}{Frac}
\DeclareMathOperator{\Gal}{Gal}
\DeclareMathOperator{\GL}{GL}
\DeclareMathOperator{\Gr}{Gr}
\DeclareMathOperator{\Hom}{Hom}
\DeclareMathOperator{\HT}{HT}
\DeclareMathOperator{\image}{image}
\DeclareMathOperator{\Ind}{Ind}
\DeclareMathOperator{\inte}{int}
\DeclareMathOperator{\loc}{loc}
\DeclareMathOperator{\Maxspec}{Maxspec}
\DeclareMathOperator{\perf}{perf}

\DeclareMathOperator{\Pic}{Pic}
\DeclareMathOperator{\proet}{pro\acute{e}t}
\DeclareMathOperator{\Proj}{Proj}
\DeclareMathOperator{\rank}{rank}
\DeclareMathOperator{\red}{red}
\DeclareMathOperator{\Res}{Res}

\DeclareMathOperator{\Spa}{Spa}
\DeclareMathOperator{\Spec}{Spec}
\DeclareMathOperator{\spect}{sp}
\DeclareMathOperator{\Spra}{Spra}
\DeclareMathOperator{\Trace}{Trace}
\DeclareMathOperator{\Tor}{Tor}

\title{Relative $p$-adic Hodge theory, II: Imperfect period rings}
\author{Kiran S. Kedlaya and Ruochuan Liu}
\date{October 21, 2019}
\begin{document}

\maketitle

\begin{abstract}
In a previous paper, we constructed a category of $(\varphi, \Gamma)$-modules associated to any adic space over $\Qp$ with the property that the \'etale $(\varphi, \Gamma)$-modules correspond to \'etale $\Qp$-local systems; these involve sheaves of period rings for Scholze's pro-\'etale topology. 
In this paper, we first extend Kiehl's theory of coherent sheaves on rigid analytic spaces to a theory of \emph{pseudocoherent sheaves} on adic spaces, then construct a corresponding theory of pseudocoherent $(\varphi, \Gamma)$-modules. 
We then relate these objects to a more explicit construction in case the space comes equipped with a suitable infinite \'etale cover; in this case, one can decomplete the period sheaves and establish an analogue of the theorem of Cherbonnier-Colmez on the overconvergence of $p$-adic Galois representations.
As an application, we show that relative $(\varphi, \Gamma)$-modules in our sense coincide with the relative $(\varphi, \Gamma)$-modules constructed by Andreatta and Brinon in the geometric setting where the latter can be constructed.
As another application, we establish that the category of pseudocoherent $(\varphi,\Gamma)$-modules on an arbitrary rigid analytic space over a mixed-characteristic nonarchimedean field is abelian, satisfies the ascending chain condition, and is stable under various natural derived functors (including $\Hom$, $\otimes$, and pullback). Applications to the \'etale cohomology of pro-\'etale local systems will be given in a subsequent paper.
\end{abstract}

\tableofcontents

\setcounter{section}{-1}

\section{Introduction}

This paper is a continuation of \cite{part1}, in which foundations were developed for a form of \emph{relative $p$-adic Hodge theory} suitable for describing and analyzing $p$-adic \'etale local systems on $p$-adic analytic spaces. In particular, to each adic space over $\QQ_p$ is associated a category of \emph{relative $(\varphi, \Gamma)$-modules} which describes the \'etale $\QQ_p$-local systems on that space.

This paper serves two primary purposes. One of these is to consolidate the foundations laid in \cite{part1}, and in particular to enlarge the category of $(\varphi, \Gamma)$-modules associated to an adic space to a larger category of \emph{pseudocoherent
$(\varphi, \Gamma)$-modules} with better homological properties. The other is to give a 
more refined description of $(\varphi, \Gamma)$-modules, and by extension \'etale $\QQ_p$-local systems, on certain spaces admitting suitable infinite \'etale covering towers.

In the remainder of this introduction, we
recall some of the constructions in ordinary $p$-adic Hodge theory which serve as the starting points for our work, 
indicate the extent to which these are and are not generalized in \cite{part1}, then indicate what further work is carried out in this paper.
(We refer back to the introduction to \cite{part1} for additional background, including
discussion of and contrast between two different possible relativizations of $p$-adic Hodge theory.)

\subsection{Synopsis of \texorpdfstring{$p$}{p}-adic Hodge theory}
\label{subsec:synopsis}

We first give a summary of the main aspects of $p$-adic Hodge theory which are generalized in \cite{part1}.
Let $K$ be a complete discretely valued field of characteristic $0$ whose residue field $k$ is perfect of characteristic $p$. 
Let $K_\infty$ be the field obtained from $K$ by adjoining the $p$-power roots of unity within some algebraic closure. The Galois group
$\Gamma = \Gal(K_\infty/K)$ is then an open subgroup of $\ZZ_p^\times$.
By applying the \emph{field of norms} construction of Fontaine-Wintenberger
\cite{fontaine-wintenberger} to $K_\infty$, one produces a new field $F$ isomorphic to a the completed perfect closure of a formal power series field over a finite extension of $k$, with the property that the Galois groups of $K_\infty$ and $F$ are isomorphic. More precisely, this isomorphism is induced  by an explicit equivalence of the categories of finite \'etale algebras over $K_\infty$ and $F$.

The discrete representations of the
absolute Galois group $G_F$ of $F$ on finite dimensional $\Fp$-vector
spaces form a category equivalent to the category of
\emph{$\varphi$-modules} over $F$, i.e., finite-dimensional $F$-vector spaces
equipped with isomorphisms with their $\varphi$-pullbacks.
This amounts to
a nonabelian
generalization of the Artin-Schreier description of $(\ZZ/p\ZZ)$-extensions
of fields of characteristic $p$ \cite[Expos\'e~XXII, Proposition~1.1]{sga7-2}.
Similarly, the continuous representations of $G_F$ on finite free $\Zp$-modules
correspond to $\varphi$-modules over the ring $W(F)$ of Witt vectors over $F$.

One may similarly describe representations of $G_K$ by considering their restrictions to $G_{K_\infty} \cong G_F$ plus descent data for the action of 
$\Gamma_K = \Gal(K_\infty/K)$. In this way, one arrives at the construction of \emph{$(\varphi, \Gamma)$-modules} over $W(F)$. A further nuance is that in order to execute some important constructions in $p$-adic Hodge theory (e.g., the construction of Fontaine's functor $D_{\mathrm{crys}}$), one must descend $(\varphi, \Gamma)$-modules over $W(F)$ to a subring of \emph{overconvergent Witt vectors} (characterized by growth conditions on the Witt vector coefficients) and then extend scalars to a certain Fr\'echet completion called the \emph{extended Robba ring}. The $(\varphi, \Gamma)$-modules over the extended Robba ring resulting from this process are precisely those which are semistable of degree 0; the whole category of $(\varphi, \Gamma)$-modules over the Robba ring can also be described in terms of vector bundles on a certain one-dimensional scheme described by Fargues and Fontaine \cite{fargues-fontaine}.

\subsection{Relativization via perfect period sheaves}

We next describe how the previous constructions are generalized in \cite{part1}. One begins by considering certain Banach algebras over $\Qp$ called \emph{perfectoid algebras}, of which the completion of $K_\infty$ used above is an example. For such algebras, one has a \emph{perfectoid correspondence} converting these algebras into perfect Banach algebras over $\Fp$; the field of norms equivalence is then a special case of the compatibility of the perfectoid correspondence with finite \'etale extensions. 
One new feature of the perfectoid correspondence, not apparent in the field of norms constructions, is compatibility with rational localizations; this enables the globalization of the construction which we discuss below.
(The same correspondence was described independently by Scholze in \cite{scholze1}, from which we have adopted the term \emph{perfectoid}.)

For $R$ a perfectoid algebra corresponding to the perfect ring $S$ in characteristic $p$, the categories of \'etale $\Zp$-local systems over $\Spec(R)$ and $\Spec(S)$ are equivalent to each other by the perfectoid correspondence, and (by the result of Katz mentioned above) to the category of $\varphi$-modules over $W(S)$. One may again replace $W(S)$ by its subring of overconvergent Witt vectors to get an equivalent category of $\varphi$-modules, then extend the base ring to an extended Robba ring made out of Witt vectors.

Now consider either a strictly $p$-adic analytic space in the sense of Berkovich or more generally an adic space over $\Qp$ in the sense of Huber. To such a space one may associated a \emph{pro-\'etale topology} in the manner of Scholze \cite{scholze2}; this topology admits a neighborhood basis consisting of spaces associated to perfectoid Banach algebras. 
For this topology, such constructions as the Witt vectors, the overconvergent Witt vectors, and the extended Robba ring define sheaves of rings equipped with bijective Frobenius actions, which we refer to collectively as \emph{perfect period sheaves}.
One can then consider $\varphi$-modules over the various period sheaves;
the $\varphi$-modules over the sheaf of Witt vectors (or equivalently overconvergent Witt vectors) then correspond to \'etale $\Zp$-local systems on the original space, while the $\varphi$-modules over the sheaf of extended Robba rings which are pointwise semistable (an open condition) correspond to \'etale $\Qp$-local systems. There is also an analogue of the Fargues-Fontaine construction describing arbitrary $\varphi$-modules over the sheaf of extended Robba rings in terms of vector bundles on a certain analytic space,
although this construction remains somewhat mysterious: the resulting space fibers over the original space in the category of topological spaces, but not in the category of adic spaces.

\subsection{\texorpdfstring{$(\varphi, \Gamma)$}{(phi, Gamma)}-modules revisited}

Let us now return to the traditional setting of $p$-adic Hodge theory
and see to what extent the preceding constructions do and do not reproduce the original theory of $(\varphi, \Gamma)$-modules. This expands upon the discussion given in \cite[\S 9.5]{part1}. 

Take notation again as in \S\ref{subsec:synopsis}.
The Fontaine-Wintenberger construction produces not just the perfect field $F$,
but also a complete discretely ramified subfield $L$ of $F$. One also obtains a subring $\gotho_{\calE}$ of $W(F)$ stable under both $\varphi$ and $\Gamma$ which is a Cohen ring for $L$; the elements of $\gotho_{\calE}$ can be viewed as formal Laurent series in a suitable variable. Fontaine's original $(\varphi, \Gamma)$-modules are defined not over $W(F)$ but over this subring $\gotho_{\calE}$, but one shows without much trouble that the resulting categories are equivalent to each other (and hence to the category of representations of $G_K$ on finite free $\Zp$-modules).

One can further intersect $\gotho_{\calE}$ with the ring of overconvergent Witt vectors; the resulting subring $\gotho_{\calE^\dagger}$ can be characterized as the subring of $\gotho_{\calE}$ consisting of Laurent series which are not merely formal, but actually converge in some open annulus within the open unit disc.
This ring is no longer complete for the $p$-adic topology, but it is still henselian. It turns out that $(\varphi, \Gamma)$-modules over the overconvergent subring of $W(F)$ descend uniquely to $\gotho_{\calE^\dagger}$, but this is far from straightforward because the corresponding statement for $\varphi$-modules is false. Instead, the original proof of this result by Cherbonnier and Colmez \cite{cherbonnier-colmez} uses the action of $\Gamma$ in a manner closely related to the Sen-Tate decompletion of Galois cohomology of local fields
\cite{sen-ht}. One can streamline this argument  by establishing (as in \cite{kedlaya-new-phigamma}) that $\Gamma$-modules over the overconvergent subring of $W(F)$ descend uniquely to $\gotho_{\calE^\dagger}$.

The Fr\'echet completion of $\gotho_{\calE^\dagger}$ inside the extended Robba ring gives the \emph{Robba ring} $\calR$, the rings of germs of analytic functions on open annuli with outer radius 1 (over a certain $p$-adic field determined by $K$). Given the Cherbonnnier-Colmez theorem, one obtains a functor from continuous representations of $G_K$ on finite-dimensional $\Qp$-vector spaces to $(\varphi, \Gamma)$-modules over $\calR$; it was shown by Berger
\cite{berger-cst} that the standard functors of Fontaine on Galois representations (e.g., the functors of de Rham, crystalline, and semistable periods) can be computed on the Robba ring side.

The obvious difference between the results we have just described and the corresponding results from \cite{part1} is that we are now considering rings on which the action of Frobenius is injective but not surjective. For many purposes, this is not such a serious distinction; for instance, the interpretation of the Fontaine functors in terms of the Robba ring can be adapted easily to the extended Robba ring. In situations where the perfect period rings suffice, the setup of \cite{part1} provides a highly functorial adaptation of $p$-adic Hodge theory to the relative setting.

However, there are several examples of arguments for which the imperfect period rings matter; here is a representative (but not exhaustive) sample.
\begin{itemize}
\item
Berger's original proof in \cite{berger-cst} of Fontaine's conjecture that de Rham representations are potentially semistable involves a reduction to a monodromy theorem for $p$-adic differential equations over the Robba ring (see \cite{andre-monodromy, kedlaya-monodromy, mebkhout-monodromy}). Passing to the extended Robba ring destroys the action of the Lie algebra of $\Gamma$, precluding an argument of this form.

\item
On the imperfect period rings, one may define an operator $\psi$ as the reduced trace of Frobenius; this extends to $(\varphi, \Gamma)$-modules. As originally observed by Fontaine, the kernel and cokernel of this operator can be interpreted in terms of \emph{Iwasawa cohomology}.
This has important consequences for $p$-adic $L$-functions via the construction of certain explicit reciprocity laws, by work of Perrin-Riou originally, and more recently Pottharst and Nakamura \cite{pottharst, nakamura}.
One can also use $\psi$ to help establish the basic properties of cohomology for arithmetic families of $(\varphi, \Gamma)$-modules; see \cite{kpx}.

\item
The $p$-adic Langlands correspondence for $\GL_2(\Qp)$ constructed by Colmez \cite{colmez-langlands} makes essential use of imperfect period rings, again because the action of $\Gamma$ on these rings is better-behaved from an analytic point of view.

\end{itemize}

\subsection{Imperfect relative period rings}

We now arrive at one of the primary purposes of the present paper: to provide a framework for extending the construction of imperfect period rings to the relative setting. Note that while we discuss it first, it is sited in the second half of the paper; the first half of the paper is devoted to more foundational issues so as to amplify the reach of these constructions (see later in this introduction).

By analogy with the fact that the formation of Cohen rings is functorial for perfect fields but not for arbitrary fields of characteristic $p$, one cannot hope to emulate the level of functoriality of \cite{part1} by defining imperfect period sheaves on the pro-\'etale topology of an arbitrary adic space. Instead, one is forced to consider special towers of finite \'etale covers; in order to transfer information obtained from such towers, one must revert back to the perfect period sheaves.

In a general axiomatic framework (\S\ref{sec:axiomatic}), we construct imperfect period rings and establish an analogue of the Cherbonnier-Colmez theorem descending $(\varphi, \Gamma)$-modules. This bears a passing resemblance to the Tate-Sen formalism used by Berger and Colmez \cite{berger-colmez} to give a new proof of Cherbonnier-Colmez compatible with arithmetic families; however, the Tate-Sen formalism seems to be rather difficult to work with in our context of relative $p$-adic Hodge theory. Namely, it seems to require a deep understanding of higher ramification theory, perhaps in the form developed by Abbes and Saito \cite{abbes-saito1, abbes-saito2}, in order to feed into the machinery of relative $(\varphi, \Gamma)$-modules proposed by Scholl \cite{scholl}. It is entirely possible that such an understanding could be obtained, but we have opted instead to work around this issue by following the approach to Cherbonnier-Colmez given in \cite{kedlaya-new-phigamma}. The cost of this choice is that it is not entirely straightforward to check that a given tower of covers fits into our framework.

By analogy with Sen's theorem on ramification in $p$-adic Lie extensions \cite{sen-lie},
we expect that this formalism can be applied to a wide class of $p$-adic Lie towers.
In \S\ref{sec:towers fields}, we consider various cases of towers over fields.
We establish that the formalism applies to the standard cyclotomic tower, thus recovering the approach to the Cherbonnier-Colmez theorem discussed in \cite{kedlaya-new-phigamma};
we also identify some other towers for which we conjecture that the formalism can be applied, such as the Kummer towers considered by Caruso \cite{caruso}.

In \S\ref{sec:toric}, we consider towers over more general affinoid algebras. At this stage, we restrict attention to towers derived from the standard cover of a torus via the multiplication-by-$p$ maps; these include
the towers for which the analogue of Cherbonniez-Colmez was established by Andreatta and Brinon \cite{andreatta-brinon}. At the level of generality we work, one obtains such covers locally on any smooth analytic space over a $p$-adic field or a perfectoid field;
one also has a logarithmic variant suitable for dealing with normal crossings divisors on smooth spaces. (Since the logarithmic construction involves compactifying a torus, it can in principle be carried out in the presence of toroidal singularities; however, we omit this level of generality here.) These constructions should make it possible to extend
numerous results based on \cite{andreatta-brinon}; as a sample, we record a comparison of cohomology used by Colmez and Nizio\l\  \cite{colmez-niziol} in their work on syntomic cohomology and semistable comparison isomorphisms, for which \cite{andreatta-brinon} is not sufficient.

Although we only consider toric towers here, we expect that the formalism we have introduced will apply to other interesting classes of towers. One intriguing special case is the class of Lubin-Tate towers associated to the universal deformations of formal groups in characteristic $p$. This includes the case of the cyclotomic extension of a $p$-adic field, but the more general case requires some new ideas, especially the interplay between formal groups and Witt vectors as exploited in the recent paper of Cais and Davis \cite{cais-davis}. (The recent thesis of Carter \cite{carter}, which establishes Scholl's criterion for these towers, is also likely to play a pivotal role.)
Just as the usual imperfect period rings appear in Colmez's interpretation of the local Langlands correspondence for $\GL_2(\Qp)$, one may expect the Lubin-Tate towers to play an analogous role for $\GL_n(\Qp)$. We plan to discuss these points in another subsequent paper.

\subsection{Pseudocoherent modules}
\label{subsec:introduction pseudocoherent}

The other primary purpose of this paper is to enrich the foundations of the theory of relative $(\varphi, \Gamma)$-modules with the aim of making calculations in homological algebra. This in turn requires enriching the foundations of the general theory of adic spaces to obtain something resembling a theory of coherent sheaves.
Under suitable noetherian hypothesis, one can emulate Kiehl's theory of coherent sheaves on rigid analytic spaces fairly directly, although surprisingly we were unable to find a prior reference for this (see \S\ref{subsec:coherence noetherian} for discussion).
However, since we need to work with perfectoid spaces, we are forced to develop a theory that works in the absence of noetherian hypotheses.

Our model for this is the theory of \emph{pseudocoherent modules} over rings,
i.e., modules admitting infinite projective resolutions by finite projective modules.
This class of modules was studied extensively in SGA6 \cite[Chapter 2]{sga6}
and later used extensively in the work of Thomason and Trobaugh \cite[\S 2]{thomason-trobaugh}. They do not form an abelian category in general, but they do have the ``two out of three'' property: any module that sits in a short exact sequence with two pseudocoherent modules is itself pseudocoherent. (See \S\ref{subsec:pseudocoherent}
for more basic discussion.) Using this property, in \S\ref{sec:foundations adic}
we construct a class of \emph{pseudocoherent sheaves} on an arbitrary adic space, which coincides with the coherent sheaves in the noetherian case and enjoys analogous properties in general (e.g., Tate acyclicity and Kiehl glueing). This is somewhat more subtle than either of the cases of schemes or locally noetherian adic spaces,
because it is unclear whether rational localizations give rise to flat ring homomorphisms;
instead, we establish some weaker flatness properties sufficient for dealing with pseudocoherent sheaves.

In \S\ref{sec:perfectoid supplemental}, we limit attention to perfectoid space and establish corresponding results with the analytic and \'etale topologies replaced by the pro-\'etale topology, and even some more exotic topologies such as the \emph{v-topology} recently introduced by Scholze \cite{scholze-berkeley2} by analogy with the h-topology for schemes. (We also make some more cosmetic refinements to the foundations of perfectoid spaces described in \cite{part1}; for instance, we allow Fontaine's mixed-characteristic perfectoid spaces as in \cite{fontaine-bourbaki}, and we allow \emph{reified adic spaces} as in \cite{kedlaya-reified}.)
In \S\ref{sec:perfect period sheaves}, we apply the preceding discussion to the perfect period sheaves, so as to obtain a theory of \emph{pseudocoherent $\varphi$-modules} on any adic space enlarging the category of $\varphi$-modules (retronymically called
\emph{projective $\varphi$-modules} when it is necessary to emphasize the distinction).
For example, in the context of classical $p$-adic Hodge theory, the pseudocoherent $\varphi$-modules coincide with the \emph{generalized $(\varphi, \Gamma)$-modules}
considered in \cite{liu-herr} and elsewhere.
We expect that the theory can be further extended from adic spaces to certain analytic stacks, such as the \emph{diamonds} of \cite{scholze-berkeley2} (see also \cite{scholze-etale}); however, we postpone this discussion to a subsequent paper.

Note that the treatment of imperfect period rings, which was discussed earlier in this introducton, appears in the body of the paper only after the preceding topics have been treated. This juxtaposition enables us to carry out the imperfect constructions uniformly,
including pseudocoherent $(\varphi, \Gamma)$-modules on a (mostly) equal footing
with their projective counterparts; this is crucial for applications. (One reason is that we do not know how to construct projective resolutions of pseudocoherent $(\varphi, \Gamma)$-modules by projective $(\varphi, \Gamma)$-modules; this makes certain reductions from pseudocoherent to projective less than straightforward.)

\subsection{Applications and coming attractions}

Our intended application of the results of this paper is to the cohomology of 
relative $(\varphi, \Gamma)$-modules, and by extension pro-\'etale local systems,
on rigid analytic spaces. The use in this study of imperfect period rings derived from toric towers echoes the use of such towers in 
the work of Faltings on the $p$-adic comparison isomorphism
\cite{faltings-purity1} and in Scholze's more recent construction of  the comparison isomorphism \cite{scholze2}. 

We will give some results along these lines in a subsequent paper. In preparation for this,
in this paper (\S\ref{sec:applications}) we establish some structural properties of the category of pseudocoherent $(\varphi, \Gamma)$-modules over a rigid analytic space (over an arbitrary base field). For one, we verify that a space is determined by its pro-\'etale topology exactly up to \emph{seminormalization} in the sense of Traverso, Swan, et al.
(see \S\ref{subsec:seminormal rings} and \S\ref{subsec:seminormality adic}); this incorporates the Ax-Sen-Tate theorem on Galois invariants of completed algebraic closures of analytic fields \cite{ax}. (While it would be quite natural to give a new proof of the latter theorem in the language of perfectoid fields, we do not do so here.)

We then verify that the category of pseudocoherent modules over the completed structure sheaf $\widehat{\calO}$ on the pro-\'etale site of a rigid analytic space has various convenient properties analogous to the category of modules over a noetherian ring
(despite the fact that its construction involves nonnoetherian rings);
for instance, it is an abelian category, satisfies the ascending chain condition, and is stable under the internal $\Hom$, tensor product, and pullback functors (and their associated derived functors). The key to the analysis is to first show that the ideals of the complete structure sheaf all arise from the base space; one can then define \emph{Fitting ideals} associated to general pseudocoherent sheaves, which behave as if we are working over a noetherian ring. While Fitting ideals constitute a fairly coarse invariant of a module over a ring, they do retain a great deal of information in codimension 1; we may thus use them to make inductive arguments to control pseudocoherence, using a glueing argument for pseudocoherent modules in the spirit of the Beauville-Laszlo theorem \cite{beauville-laszlo}.

We finally show (\S\ref{subsec:final applications}) that the category of pseudocoherent $(\varphi, \Gamma)$-modules over a rigid analytic space has similar properties to the ones described for pseudocoherent $\widehat{\calO}$-modules. The key point is to show that the failure of such objects to be projective is in a sense controlled by $\widehat{\calO}$-modules, so the previous results may be applied. Over a $p$-adic field, this is true in a quite literal sense; over a general field, some extra arguments are required to show
that the failure of projectivity can be isolated into two separate contributions: one coming from
$\widehat{\calO}$-modules, which can be controlled using the preceding discussion; and another coming from
replacing one perfectoid field with another field with the same tilt. The second contribution turns out to be essentially negligible; this is a manifestation of the fact that the perfectoid correspondence does not descend to a meaningful correspondence between rigid analytic spaces over corresponding perfectoid fields.
By contrast, if one considers somewhat less rigid information, then at least for smooth spaces some sort of correspondence does exist; for instance, there is a relationship at the level of motives described by Vezzani \cite{vezzani}.

\subsection*{Acknowledgments}
Thanks to Fabrizio Andreatta, Bhargav Bhatt, Olivier Brinon, Pierre Colmez, Brian Conrad, Laurent Fargues, Jean-Marc Fontaine, 
David Hansen,
S\'andor Kov\'acs, 
Wies\l awa Nizio\l, Arthur Ogus, Sam Payne,
David Rydh, Peter Scholze,
Karl Schwede,
Kazuma Shimomoto,
Jean-Pierre Serre,
David Speyer,
Michael Temkin, Jared Weinstein, Liang Xiao, Sarah Zerbes,
and Xinwen Zhu for helpful discussions.
Kedlaya was supported by NSF grants DMS-0545904 (CAREER), DMS-1101343, DMS-1501214, DMS-1802161;
DARPA grant HR0011-09-1-0048; MIT (NEC Fund,
Cecil and Ida Green Career Development Professorship); UC San Diego
(Stefan E. Warschawski Professorship); 
IAS (visiting professorship 2018--2019;
and a Guggenheim Fellowship during fall 2015.
Liu was partially supported by IAS under NSF grant DMS-0635607 and NSFC-11571017.
In addition, both authors thank MSRI for its hospitality during fall 2014, as supported by NSF grant DMS-0932078, and Kedlaya thanks ICERM for its hospitality during fall 2015.

\section{Algebraic preliminaries}

As in \cite{part1}, we begin with some preliminary constructions and arguments not specific to adic or perfectoid spaces.

\subsection{Pseudocoherent and fpd modules}
\label{subsec:pseudocoherent}

In order to conduct homological algebra calculations over nonnoetherian rings, we will need to be careful about finite generation properties. We will consider two closely related concepts: the concept of a \emph{pseudocoherent module} from
\cite[D\'efinition~2.1]{sga6} (see also \cite[\S 2.2]{thomason-trobaugh}
and \cite[\S 1.3]{abbes}),
and the more familiar concept of a \emph{module of finite projective dimension}.
(Beware that the term \emph{pseudocoherent module} carries a different meaning in \cite[Exercise I.2.11]{bourbaki-ac}.)

\begin{defn}
Let $R$ be a (commutative) ring. 
For $m \in \ZZ \cup \{\infty\}$, an $R$-module $M$ is \emph{$m$-pseudocoherent} if there exists a projective resolution 
\[
\cdots \to P_{1} \to P_0 \to M \to 0
\]
with $P_i$ finitely generated for all $i \leq m$. 
In particular, $M$ is 0-pseudocoherent (resp.\ $1$-pseudocoherent) if and only if it is finitely generated (resp.\ finitely presented). We write \emph{pseudocoherent} as shorthand for \emph{$\infty$-pseudocoherent}.

We say that $M$ is \emph{$m$-fpd} if there exists a finite projective resolution for which $P_i = 0$ for all $i>m$.
An equivalent condition is that $\Ext^i_R(M,N) = 0$ for all $R$-modules $N$ and all $i \geq m+1$, or even just for $i=m+1$ \cite[Tag~065R]{stacks-project}.
We say $M$ \emph{has finite projective dimension}, or for short is \emph{fpd}, if it is $m$-fpd for some $m$; the minimum value of $m$ is called the \emph{projective dimension} of $M$.
By convention, we take the projective dimension of the zero module to be $-\infty$.
\end{defn}

\begin{remark}
An $R$-module $M$ is fpd if and only if the singleton complex $0 \to M \to 0$ is \emph{perfect}, i.e., isomorphic in the derived category to a bounded complex of finite projective $R$-modules. It is thus tempting to refer to fpd modules as \emph{perfect} modules, as in
\cite[Tag 0656]{stacks-project},
but due to other uses of the word \emph{perfect} in our work we have resisted this temptation.
\end{remark}

\begin{remark} \label{R:pseudocoherent tensor product}
Given two projective resolutions
\[
\cdots \to P_{1} \to P_0 \to M \to 0, \qquad
\cdots \to P'_{1} \to P'_0 \to M' \to 0,
\]
we may construct a projective resolution of $M \otimes_R M'$ by taking the totalization of the double complex $P_i \otimes_R P'_j$. This observation has the following consequences.
\begin{itemize}
\item The tensor product of two $m$-pseudocoherent modules is again $m$-pseudocoherent
(compare \cite[Corollaire~2.16.1(b)]{sga6}).
\item The tensor product of an $m$-fpd module with an $n$-fpd module is $(m+n)$-fpd.
\end{itemize}
\end{remark}

\begin{remark} \label{R:no internal Homs}
The category of pseudocoherent (resp.\ fpd) modules over a ring $R$ is not guaranteed to admit internal Homs.
\end{remark}

\begin{lemma} \label{L:pseudocoherent 2 of 3}
Let 
\[
0 \to M_1 \to M \to M_2 \to 0 
\]
be a short exact sequence of $R$-modules and let $m$ be a nonnegative integer.
\begin{enumerate}
\item[(a)]
If $M_1$ and $M_2$ are $m$-pseudocoherent, then so is $M$.
\item[(b)]
If $M_1$ is $(m-1)$-pseudocoherent and $M$ is $m$-pseudocoherent, then $M_2$ is $m$-pseudocoherent.
\item[(c)]
If $M$ is $m$-pseudocoherent and $M_2$ is $(m+1)$-pseudocoherent, then $M_1$ is $m$-pseudocoherent.
\item[(d)]
If $M_1$ and $M_2$ are $m$-fpd, then so is $M$.
\item[(e)]
If $M_1$ is $m$-fpd and $M$ is $(m+1)$-fpd, then $M_2$ is $(m+1)$-fpd.
\item[(f)]
If $M$ is $m$-fpd and $M_2$ is $(m+1)$-fpd, then $M_1$ is $m$-fpd.
\end{enumerate}
In particular, if any two of $M, M_1, M_2$ are pseudocoherent (resp.\ fpd), then so is the third; that is, the pseudocoherent (resp.\ fpd) $R$-modules form an exact tensor category.
\end{lemma}
\begin{proof}
Straightforward, or see \cite[Proposition~2.5(b)]{sga6} or \cite[Tag~064R]{stacks-project} and \cite[Tag~065S]{stacks-project}.
\end{proof}

\begin{remark} \label{R:transfer pseudocoherence to quotient}
Let $R$ be a ring and suppose $f \in R$ is not a zero-divisor. Then $R/fR$ admits the projective resolution
\[
\cdots \to 0 \to R \stackrel{\times f}{\to} R \to 0
\]
and hence is a 1-fpd $R$-module. 

Now let $M$ be an $m$-pseudocoherent (resp.\ $m$-fpd) $R$-module. 
By Remark~\ref{R:pseudocoherent tensor product}, $M/fM = M \otimes_R (R/fR)$ is $m$-pseudocoherent (resp.\ $(m+1)$-fpd);
by Lemma~\ref{L:pseudocoherent 2 of 3}, $fM = \ker(M \to M/fM)$ is $(m-1)$-pseudocoherent (resp.\ $m$-fpd), and $M[f] = \{\bv \in M: f\bv = 0\} = \ker(\times f: M \to fM)$ is $(m-2)$-pseudocoherent (resp.\ $(m+1)$-fpd).

With the same notation, put $T = \bigcup_{n=1}^\infty M[f^n]$. If $M$ is $m$-pseudocoherent, then $T$ is not guaranteed to be $(m-2)$-pseudocoherent, but it is at least \emph{ind-$(m-2)$-pseudocoherent}.
\end{remark}

\begin{lemma} \label{L:stable sequence}
Let $R$ be a ring and choose $f \in R$.
Let $M_0 \to M_1 \to \cdots$ be a sequence of surjective morphisms of pseudocoherent $R$-modules such that for each $i$, the quotient of $M_i$ by its $f$-power-torsion submodule $T_i$ is pseudocoherent. If the sequences $M_{0,f} \to M_{1,f} \to \cdots$
and $M_0/fM_0 \to M_1/fM_1 \to \cdots$ both stabilize,
then so does the original sequence.
\end{lemma}
\begin{proof}
By omitting initial terms, we may assume that the sequences $M_{0,f} \to M_{1,f} \to \cdots$
and $M_0/fM_0 \to M_1/fM_1 \to \cdots$ are both constant; in this case, we prove that the original sequene is also constant.
By Lemma~~\ref{L:pseudocoherent 2 of 3}, $T_i$ is finitely generated, and hence annihilated by $f^m$ for some positive integer $m$. Since $M_{i,f} \cong M_{i+1,f}$,
we have  $\ker(M_i \to M_{i+1}) \subseteq T_i$;
on the other hand, $M_i/fM_i \cong M_{i+1}/fM_{i+1}$, so any element of $\ker(M_i \to M_{i+1})$ is infinitely $f$-divisible. It follows that the kernel vanishes, proving the claim.
\end{proof}

\begin{remark} \label{R:noetherian pseudoflat}
By Lemma~\ref{L:pseudocoherent 2 of 3}, the following conditions are equivalent.
\begin{enumerate}
\item[(a)]
The ring $R$ is coherent, i.e., any finitely generated ideal is finitely presented. For instance, this holds if $R$ is noetherian.
\item[(b)]
Any finitely presented $R$-module is pseudocoherent.
\item[(c)]
The pseudocoherent $R$-modules form an abelian category.
\end{enumerate}
On the other hand, even when $R$ is noetherian, 
the only case where every pseudocoherent $R$-module $M$ is guaranteed to be fpd is when $R$ is regular of finite dimension $m$, in which case
$M$ is guaranteed to be $m$-fpd
 \cite[Corollary~19.6]{eisenbud}.
\end{remark}

\begin{remark} \label{R:coherent conditions}
There are numerous ways to characterize coherent rings; see for instance \cite[Theorem~2.3.2]{graz}.
A consequence of one of these characterizations is that if $\{R_i\}$ is a directed system of coherent rings with flat transition maps, then $\varinjlim_i R_i$ is again coherent \cite[Theorem~2.3.3]{graz}. For example, a polynomial ring in possibly infinitely many variables over a field is coherent, as is the perfect closure of such a polynomial ring
when the base field is of characteristic $p$.
\end{remark}

\begin{remark}
If the ring $R$ is separated and complete with respect to a non-zero divisor $f$ and $R/fR$ is noetherian, then so is $R$. Namely, the associated graded ring $\Gr R = \bigoplus_{n=0}^\infty f^n R /f^{n+1} R$ is isomorphic to $R = (R/fR) \llbracket T \rrbracket$ (via the map taking the class of $f^n R$ in $f^n R /f^{n+1} R$ to $T^n$), which is a completion of the noetherian ring $(R/fR)[T]$ and is thus noetherian \cite[Theorem~7.1]{eisenbud}. 
Given an ideal $I \subseteq R$, let $J \subseteq \Gr R$ be the ideal of leading terms, i.e., the ideal generated by the class in $f^n R/f^{n+1} R$ of each element of $I \cap f^n R$ for each $n \geq 0$. One then checks easily that any set of elements of $I$ lifting a finite set of generators of $J$ constitutes a finite generating set for $I$.

By contrast, there is no corresponding statement for coherent rings. This is because
there exists a coherent ring $R$ such that $R \llbracket T \rrbracket$ is not coherent (see \cite[\S 8.1]{graz}).
\end{remark}

\begin{remark} \label{R:Bezout torsion}
Let $R$ be a Pr\"ufer domain (an integral domain in which every finitely generated ideal is projective). Then $R$ is coherent by Remark~\ref{R:coherent conditions}, so by Remark~\ref{R:noetherian pseudoflat} every finitely presented $R$-module is pseudocoherent. In fact, one can make some more precise statements.
\begin{enumerate}
\item[(a)]
Any finitely generated torsion-free $R$-module is projective;
see for example \cite{kaplansky-prufer}.
\item[(b)]
Let $M$ be a finitely generated (resp.\ finitely presented) $R$-module. Then the torsion submodule $T$ of $M$ is finitely generated (resp.\ finitely presented) and the quotient $M/T$ is finite projective (by (a)).
\item[(c)]
Any finitely presented $R$-module is 1-fpd (by (b)).
\end{enumerate}
\end{remark}

\begin{remark} \label{R:Fitting pseudocoherent}
We do not know if the Fitting ideals of a pseudocoherent module 
(as discussed in \cite[Definition~1.1.2]{part1})
are themselves pseudocoherent. One complication is that the formation of Fitting ideals is compatible with extension of scalars, but not with restriction of scalars along a quotient: if $I$ is an ideal of a ring $R$ and $M$ is a finitely generated $R$-module annihilated by $I$, the Fitting ideals of $M$ need not contain $I$. In particular, they are not determined by their images in $R/I$, which are the Fitting ideals of $M$ as an $R/I$-module.
\end{remark}

We next consider the effect of base extension on the pseudocoherent and fpd conditions.
\begin{remark} \label{R:need flatness}
For $R \to S$ a ring homomorphism, it is not true in general that the base extension of a pseudocoherent (resp.\ fpd) $R$-module is a pseudocoherent (resp.\ fpd) $S$-module.
However, the following statements do hold.
\begin{itemize}
\item
For $m \leq 1$, the base extension of an $m$-pseudocoherent $R$-module is an $m$-pseudocoherent $S$-module.
However, this fails in general for $m \geq 2$.
\item
If $S$ is flat, 
then base extension takes $m$-pseudocoherent (resp.\ $m$-fpd) $R$-modules to $m$-pseudocoherent (resp.\ $m$-fpd) $S$-modules,
and the base extension functor on pseudocoherent $R$-modules is exact.
\item
If $S$ is coherent, then base extension takes pseudocoherent $R$-modules to pseudocoherent $S$-modules by Remark~\ref{R:noetherian pseudoflat}. 
\item
If $S$ is noetherian and regular, then base extension takes fpd $R$-modules to fpd $S$-modules by Remark~\ref{R:noetherian pseudoflat}.
\item
If $S$ is pseudocoherent as an $R$-module, then base extension takes pseudocoherent $R$-modules to pseudocoherent $S$-modules
\cite[Tag~064Z]{stacks-project},
and an $S$-module is pseudocoherent if and only if it is pseudocoherent as an $R$-module.
Additionally, if $M$ is a pseudocoherent $R$-module, then
$\Tor_i^R(M, S)$ is a pseudocoherent $S$-module for all $i \geq 0$.
\item
If $S$ is fpd as an $R$-module, then base extension need not take fpd $R$-modules to fpd $S$-modules. More precisely, if $M$ is an fpd $R$-module, then $M \otimes_R S$ is fpd as an $R$-module by Remark~\ref{R:pseudocoherent tensor product}, but need not be an fpd $S$-module even if $S$ is noetherian (see Example~\ref{exa:not fpd}).
\end{itemize}
\end{remark}

\begin{example} \label{exa:not fpd}
Let $K$ be a field and put $R = K[x,y]$ and $S = R/(xy)$. By 
Remark~\ref{R:noetherian pseudoflat}, every finitely generated $R$-module is fpd.
In particular, $S$ and the ideal $I = (x,y)$ of $S$ are fpd as $R$-modules, but the latter is fpd not as an $S$-module: from the projective resolution
\[
\cdots\, S^2 \stackrel{(a,b) \to (by, ax)}{\longrightarrow} S^2 \stackrel{(a,b) \to (by, ax)}{\longrightarrow} S^2 \stackrel{(a,b) \mapsto ax+by}{\longrightarrow} I \to 0
\]
we see that $\Tor_i^S(I, S/I) = (S/I)^2 \neq 0$ for all $i>0$.
\end{example}

\begin{remark} \label{R:pseudoflat}
Let $R$ be a ring and let $S$ be an $R$-module; then $S$ is flat over $R$ if and only if 
$\Tor_1^R(P, S) = 0$ for every 1-pseudocoherent $R$-module $P$. For $m>1$, to ensure that base extension from $R$ to $S$ defines an exact functor from $m$-pseudocoherent $R$-modules to $m$-pseudocoherent $S$-modules, as in Remark~\ref{R:need flatness}, it would be sufficient to require that $\Tor_1^R(P,S) = 0$ for every $m$-pseudocoherent $R$-module $P$. However, this would not imply the vanishing of $\Tor_i^R(P,S)$ for $i>1$ for such $P$.
\end{remark}

\subsection{Modules over open mapping rings}

We record some general definitions and statements about topological rings and modules, particularly those for which the Banach open mapping theorem holds.

\begin{defn}
A topological abelian group is \emph{metrizable} if it is Hausdorff and admits a countable fundamental system of open neighborhoods of 0. It is easily verified that this is equivalent to the condition that the topology is induced by some (not necessarily nonarchimedean) absolute value.
\end{defn}

\begin{defn} \label{D:natural topology}
Let $R$ be a topological ring.
By the \emph{natural topology} on a finitely generated $R$-module $M$, we will mean the quotient topology induced by some surjection $R^n \to M$ of $R$-modules for some nonnegative integer $n$ (for the product topology on $R^n$). By emulating the proof of \cite[Lemma~2.2.6]{part1}, one sees that this topology does not depend on the choice of the surjection $R^n \to M$. 
Beware that this topology is not guaranteed to be Hausdorff; however, if it is Hausdorff, then it is metrizable as long as the topology on $R$ is metrizable.
\end{defn}

\begin{remark} \label{R:finitely generated strict}
Let $R$ be a complete metrizable topological ring. Let $M$ be a finitely generated $R$-module equipped with its natural topology. For any exact sequence $0 \to K \to F \to M \to 0$ of $R$-modules with $F$ finite free, the map $F/K \to M$ is a homeomorphism by the definition of the natural topology.
If $M$ is Hausdorff, then $K$ is closed,
so $M$ is also complete; in general, $M$ surjects onto its completion.
\end{remark}

\begin{remark}
For $R,M$ as in Remark~\ref{R:finitely generated strict}, we are unaware of any direct relationship between $M$ being complete for the natural topology and the Fitting ideals of $M$ being closed in $R$, outside of the trivial fact that finite projective $R$-modules are always complete for the natural topology.
\end{remark}

\begin{defn} \label{D:power series module}
For $M$ an abelian group, let $M \llbracket T \rrbracket$ 
(resp.\ $M \llbracket T,T^{-1} \rrbracket$) be the group of formal sums 
$\sum_{n=0}^\infty m_n T^n$ (resp.\ $\sum_{n \in \ZZ} m_n T^n$) with $m_n \in M$.
Let $M[T]$ (resp.\ $M[T,T^{-1}]$) be the subgroup of $M \llbracket T \rrbracket$
(resp.\ $M \llbracket T,T^{-1} \rrbracket$) consisting of formal sums with all but finitely many coefficients equal to 0.

Now suppose $M$ is a topological abelian group. 
Let $M\{T\}$ (resp.\ $M\{T,T^{-1}\}$) be the subgroup of $M \llbracket T \rrbracket$
(resp.\ $M \llbracket T,T^{-1} \rrbracket$) consisting of formal sums 
$\sum_{n=0}^\infty m_n T^n$ (resp.\ $\sum_{n \in \ZZ} m_n T^n$) such that the sequence
$\{m_n\}_{n=0}^\infty$ (resp.\ the sequences $\{m_n\}_{n=0}^\infty$ and
$\{m_{-n}\}_{n=0}^\infty$) are null sequences in $M$. If $M$ is complete and metrizable, then
$M\{T\}$ (resp.\ $M\{T,T^{-1}\}$) may be identified with the completion
of $M[T]$ (resp.\ $M[T,T^{-1}]$)
with respect to the topology of uniform convergence.
\end{defn}

\begin{lemma} \label{L:complete descent}
Let $R \to S$ be a morphism of complete metrizable topological rings which splits in the category of complete topological $R$-modules. Then $R \to S$ is an effective descent morphism for finite projective modules.
\end{lemma}
\begin{proof}
The splitting condition ensures that $R \to S$ is universally injective: for every complete $R$-module $M$, the map $M \to M \widehat{\otimes}_R S$ is injective. We may thus apply the general descent theorem of Joyal--Tierney \cite{joyal-tierney}, or even just emulate the corresponding argument for ordinary modules over a ring \cite[Tag~08WE]{stacks-project}.
\end{proof}

\begin{defn} \label{D:open mapping}
An \emph{open mapping ring} is a topological ring $R$ with the following additional properties.
\begin{enumerate}
\item[(a)]
The ring $R$ admits a null sequence composed of units.
\item[(b)]
The topology on $R$ is metrizable.
\item[(c)]
The topology on $R$ is complete (by which we will always mean Hausdorff and complete). By (b), to check completeness we need only check the convergence of Cauchy sequences, rather than more general nets.
\end{enumerate}
For example, any Banach ring (under our running conventions; see Convention~\ref{conv:Banach}) is an open mapping ring.
\end{defn}

\begin{remark} \label{R:series multiplication}
In Definition~\ref{D:power series module}, if $M$ is a topological module over an open mapping ring $R$, then for any $f \in R$, multiplication by $1-fT$ is bijective on $M \llbracket T \rrbracket$ and hence injective on $M\{T\}$.

By contrast, it is less obvious that $T-f$ is injective on $M\{T\}$. However, this does hold if $M$ is $f$-torsion-free, or more generally if $M[f^n] = M[f^{n+1}]$ for some $n$. 
\end{remark}

For open mapping rings, we have the following version of the open mapping theorem.
\begin{theorem}[Open mapping theorem] \label{T:open mapping}
Let $R$ be an open mapping ring.
Let $M$ be a complete metrizable topological $R$-module.
Let $N$ be a topological $R$-module.
Let $f: M \to N$ be a homomorphism of $R$-modules.
\begin{enumerate}
\item[(a)]
If $M$ is finitely generated as an $R$-module, then $f$ is continuous.
\item[(b)]
If $N$ is complete and metrizable and $f$ is continuous and surjective, then $f$ is strict.
\end{enumerate}
\end{theorem}
\begin{proof}
See \cite[Theorem~1.6, Theorem 1.17]{henkel}.
\end{proof}

\begin{remark}
Theorem~\ref{T:open mapping} suffices for our purposes because by convention, we require every Banach ring to contain a topologically nilpotent unit. However, Theorem~\ref{T:open mapping} would remain true (with some small modifications to the proof) if we only assumed that the topologically nilpotent elements of $R$ generate the unit ideal; this would correspond to considering Banach rings which are \emph{free of trivial spectrum} as in 
\cite[Remark~2.3.9]{part1} (or equivalently, considering adic Banach rings whose adic spectra are \emph{analytic} in the sense of Huber). We will not pursue this point further here, but see \cite{kedlaya-aws} for some discussion.
\end{remark}

For Banach rings, we have the following corollary of the open mapping theorem.
\begin{cor} \label{C:closure finitely generated}
Let $R$ be a Banach ring. Let $M$ be a finitely generated $R$-module which is complete for its natural topology. Let $N$ be an $R$-submodule of $M$ with completion $\widehat{N}$. If $\widehat{N}$ is finitely generated, then $N = \widehat{N}$ is closed in $M$ and complete for the natural topology.
\end{cor}
\begin{proof}
By Theorem~\ref{T:open mapping}, the subspace topology on $\widehat{N}$ coincides with the natural topology; we may thus reduce to the case $M = \widehat{N}$. Choose module generators $\bv_1,\dots,\bv_n$ of $M$; we can then find $\bw_1,\dots,\bw_n \in N$ such that $\bw_j = \sum_i A_{ij} \bv_i$ for some matrix $A$ over $R$ such that
$A-1$ is topologically nilpotent (here we use the fact that $R$ is a Banach ring rather than a general open mapping ring). But then the matrix $A$ is invertible, so $M \subseteq N$ and the claim follows.
\end{proof}

This has the following consequence which we have used previously \cite[Remark~2.2.11]{part1}.
\begin{cor} \label{C:noetherian complete}
Let $R$ be a Banach ring whose underlying ring is noetherian. Then every finitely generated $R$-module is complete for its natural topology.
\end{cor}
\begin{proof}
Let $P$ be a finitely generated $R$-module. Form an exact sequence
\[
0 \to M \to N \to P \to 0
\]
of $R$-modules in which $N$ is finite free (and hence complete for its natural topology). Since $R$ is noetherian, both $M$ and its completion for the natural topology of $N$ are finitely generated $R$-modules. We may thus apply Corollary~\ref{C:closure finitely generated} to deduce that $M$ is complete within $N$. Conseqently, $P$ is complete for the quotient topology induced from $N$, which by definition is its natural topology.
\end{proof}

\begin{defn}
Let $R$ be an open mapping ring. An $R$-module $M$ is \emph{strictly $m$-pseudocoherent} (resp.\ \emph{strictly $m$-fpd}) if it is $m$-pseudocoherent (resp.\ $m$-fpd) and complete for the natural topology. 
The completeness is a genuine condition even for $m = \infty$: for example, one can find examples where $R$ is a Banach ring, $f \in R$ is not a zero-divisor, and the principal ideal $fR$ is not closed in $R$ (see \cite[Proposition~3.14]{mihara} for an explicit construction). In this case, $R/fR$ is 1-fpd but not complete for the natural topology.
\end{defn}

\subsection{Analytic group actions}
\label{subsec:group actions}

It will be helpful later to distinguish between \emph{continuous} and \emph{analytic} actions of a $p$-adic Lie group on a topological module, to introduce a comparison of cohomology due to Lazard \cite{lazard-groups}, and to introduce a vanishing criterion for cohomology groups taken from \cite{kedlaya-hs}. In particular, this condition will be needed
to check some of the conditions on towers of adic spaces considered in
\S\ref{sec:axiomatic}. (It will also be useful for the case of Lubin-Tate towers, to be treated in a subsequent paper.)

\begin{hypothesis}
Throughout \S\ref{subsec:group actions},
let $\Gamma$ be a $p$-adic Lie group, or more generally a \emph{profinite $p$-analytic group} in the sense of \cite[III.3.2.2]{lazard-groups}.
\end{hypothesis}

\begin{remark}
If the group $\Gamma$ acts compatibly on a ring $R$ and an $R$-module $M$, we have the ``Leibniz rule'' identity
\begin{equation} \label{eq:Leibniz rule}
(\gamma-1)(r \bv)
= (\gamma-1)(r) \bv + \gamma(r)
(\gamma-1)(\bv) \qquad (r \in R, \bv \in M).
\end{equation}
We will use this identity frequently in order to reduce statements about the action of $\Gamma$ on $M$ to statements about the actions on $R$ and on module generators of $M$.
\end{remark}

\begin{defn}
Let $A$ be the completion of the group ring $\Zp[\Gamma]$ with respect to the $p$-augmentation ideal $\ker(\Zp[\Gamma] \to \Fp)$. Put $I = \ker(A \to \Fp)$;
we view $A$ as a filtered ring using the $I$-adic filtration.
We also define the associated valuation:
for $x \in A$, let $w(A;x)$ be the supremum of those nonnegative integers $i$ for which $x \in I^i$.
\end{defn}

\begin{defn} \label{D:analytic group action}
An \emph{analytic $\Gamma$-module} is a left $A$-module $M$ complete with respect to a valuation $w(M; \bullet)$ for which there exist $a>0, c \in \RR$ such that
\[
w(M; xy) \geq a w(A; x) + w(M; y) + c \qquad (x \in A, y \in M).
\]
Equivalently, there exist 
an open subgroup $\Gamma_0$ of $\Gamma$ and a constant $c > 0$ such that
\begin{equation} \label{eq:analytic condition}
w(M; (\gamma-1)y) \geq w(M; y) + c \qquad (\gamma \in \Gamma_0, y \in M);
\end{equation}
it then follows formally that for any $c>0$, there exists an open subgroup $\Gamma_0$ of $\Gamma$ for which
\eqref{eq:analytic condition} holds.
\end{defn}

\begin{example} \label{exa:lazard finite}
Let $M$ be a free $\Zp$-module of finite rank on which $\Gamma$ acts continuously. Then $M$ is an analytic $A$-module for the valuation defined by any basis; see \cite[Proposition~V.2.3.6.1]{lazard-groups}.
\end{example}

\begin{prop} \label{P:extend analyticity along affinoid}
Let $R \to S$ be an affinoid homomorphism of Banach rings
equipped with continuous $\Gamma$-actions.
If the action of $\Gamma$ on $R$ is analytic for the function $w(R; \bullet) = -\log \left| \bullet \right|$, then the same holds for $S$.
\end{prop}
\begin{proof}
By definition, there exists a strict surjection $R\{T_1,\dots,T_n\}\to S$ for some $n$.
By hypothesis, for any $c>0$,  there exists an open subgroup $\Gamma_0$ of $\Gamma$ for which \eqref{eq:analytic condition} holds for all $y \in R$. Since the action of $\Gamma$ on $S$ is continuous, we can also choose $\Gamma_0$ so that \eqref{eq:analytic condition} holds when $y$ is equal to the image of any of $T_1,\dots,T_n$ in $S$.
For $c$ suitably large, we may use \eqref{eq:Leibniz rule} and the strictness of $R\{T_1,\dots,T_n\}\to S$ to deduce \eqref{eq:analytic condition}  (for a possibly different but still positive value of $c$) for all $y \in S$.
\end{proof}

\begin{defn}
Let $M$ be an analytic $\Gamma$-module. A cochain $\Gamma^i \to M$ is \emph{analytic} if
for every homeomorphism between an open subspace $U$ of $\Gamma^i$ and an open subspace $V$ of $\Zp^n$ for some nonnegative integer $n$, the induced function $V \to M$ is locally analytic (i.e., locally represented by a convergent power series expansion).
By the proof of \cite[Proposition~V.2.3.6.3]{lazard-groups}, the analytic cochains form a subcomplex $C^{\bullet}_{\an}(\Gamma,M)$ of $C^{\bullet}_{\cont}(\Gamma,M)$; we thus obtain \emph{analytic cohomology} groups $H^i_{\an}(\Gamma,M)$ and natural homomorphisms $H^i_{\an}(\Gamma,M) \to H^i_{\cont}(\Gamma,M)$.
\end{defn}

\begin{theorem}[Lazard] \label{T:lazard analytic}
If $M$ is an analytic $\Gamma$-module, then the inclusion $C^{\bullet}_{\an}(\Gamma,M) \to C^{\bullet}_{\cont}(\Gamma,M)$ is a quasi-isomorphism. That is, the continuous cohomology of $M$ can be computed using analytic cochains.
\end{theorem}
\begin{proof}
In the context of Example~\ref{exa:lazard finite}, this is the statement of 
\cite[Th\'eor\`eme~V.2.3.10]{lazard-groups}. However, the proof of this statement only uses the stronger hypothesis in the proof of \cite[Proposition~V.2.3.6.1]{lazard-groups}, which we have built into the definition of an analytic $\Gamma$-module. The remainder of the proof of \cite[Th\'eor\`eme~V.2.3.10]{lazard-groups} thus carries over unchanged.
\end{proof}

\begin{remark} \label{R:finite Gamma subcomplex}
In considering Theorem~\ref{T:lazard analytic}, it may help to consider the first the case of 1-cocycles: every 1-cocycle is cohomologous to a crossed homomorphism, which is analytic because of how it is determined by its action on topological generators. More generally, given a system of elements of $\Gamma$ which form a basis of its Lie algebra via the logarithm map, Lazard constructs a split quasi-isomorphic inclusion $C^{\bullet}_{\mathrm{qm}}(\Gamma,M) \subset C^{\bullet}_{\cont}(\Gamma,M)$ of complexes in which $C^i_{\mathrm{qm}}(\Gamma,M)$ is a finite direct sum of copies of $M$. For example, in the case where $\gamma \in \Gamma$ is a pro-$p$ topological generator, 
then $C^{\bullet}_{\mathrm{qm}}(\Gamma,M)$ is the complex
\[
0 \to M \stackrel{\gamma-1}{\to} M \to 0.
\]
In general, $C^i_{\mathrm{qm}}(\Gamma,M)$ consists of the homomorphisms into $M$ from
the corresponding term of 
Lazard's \emph{quasi-minimal complex} associated to the chosen system of elements of $\Gamma$;
see \cite[\S 2.2.1]{lazard-groups} for the construction.
\end{remark}

We now collect some criteria for vanishing of cohomology groups,
starting with a form of the Hochschild-Serre spectral sequence.

\begin{lemma} \label{L:hochschild-serre}
For any closed normal subgroup $\Gamma'$ of $\Gamma$ and any topological $\Gamma$-module $M$, there is a spectral sequence
\[
E_2^{p,q} = H^p_{\cont}(\Gamma/\Gamma', H^q_{\cont}(\Gamma', 
M))
\Longrightarrow H^{p+q}_{\cont}(\Gamma, M).
\]
\end{lemma}
\begin{proof}
The explicit construction of the spectral sequence for finite groups given in \cite[\S 2]{hochschild-serre} carries over without change.
For further discussion, see \cite[\S V.3.2]{lazard-groups}.
\end{proof}

\begin{cor} \label{C:hochschild-serre}
Let $\Gamma'$ be a closed subgroup of $\Gamma$ which is \emph{subnormal}, i.e., there exists a sequence $\Gamma' = \Gamma_0 \subseteq \cdots \subseteq \Gamma_n = \Gamma$
of closed inclusions in which $\Gamma_i$ is normal in $\Gamma_{i+1}$ for $i=0,\dots,n-1$. Then for any topological $\Gamma$-module $M$, if $H^i_{\cont}(\Gamma',M) = 0$ for all $i \geq 0$,
then also $H^i_{\cont}(\Gamma,M) = 0$ for all $i \geq 0$.
\end{cor}

\begin{remark} \label{R:not subnormal}
If $\Gamma$ is a finite $p$-group, then $\Gamma'$ is automatically subnormal in $\Gamma$ and so Corollary~\ref{C:hochschild-serre} applies. By contrast, the conclusion of Corollary~\ref{C:hochschild-serre} does not hold for arbitrary $\Gamma'$ even in the category of finite groups; see \cite{kedlaya-hs} for some simple examples.
\end{remark} 

In spite of Remark~\ref{R:not subnormal},
in the case of analytic group actions, we have the following result.
\begin{theorem} \label{T:kill analytic cohomology}
Let $M$ be a Banach space over an analytic field of characteristic $p$.
Equip $M$ with a continuous action of $\Gamma$
which is analytic for the function $w(M; \bullet) = -\log \left| \bullet \right|$.
Let $\Gamma'$ be a pro-$p$ procyclic subgroup of $\Gamma$. 
If $H^i_{\cont}(\Gamma', M) = 0$ for all $i \geq 0$, then $H^i_{\cont}(\Gamma, M) = 0$ for all $i \geq 0$.
\end{theorem}
\begin{proof}
See \cite{kedlaya-hs}.
\end{proof}

\subsection{Seminormal rings and universal homeomorphisms}
\label{subsec:seminormal rings}
 
In order to state the Ax-Sen-Tate theorem for rigid analytic spaces, we
need an intermediate condition between reduced and normal rings. 
The condition we use was introduced by Swan \cite{swan}, as a modification of an earlier definition of Traverso \cite{traverso} which had been further studied by Greco--Traverso \cite{greco-traverso}.

\begin{defn} \label{D:seminormal}
For $R$ a ring, consider the map
\begin{equation} \label{eq:seminormal}
R \to \{(y,z) \in R \times R: y^3 = z^2\}, \qquad x \mapsto (x^2, x^3).
\end{equation}
If \eqref{eq:seminormal} is injective, then $(0,0)$ has the unique preimage $0$ and so $R$ cannot have any nonzero elements which square to $0$. Conversely, if $w \neq x \in R$ have the same image, then
\[
(w-x)^3 = 4(w^3-x^3) + 3(x-w)(w^2-x^2) = 0
\]
and so $w-x \in R$ is nonzero and nilpotent. In other words, \eqref{eq:seminormal} is injective if and only if $R$ is reduced.

We say that the ring $R$ is \emph{seminormal} if the map \eqref{eq:seminormal} is bijective. This property is Zariski local, i.e., it passes to localizations and may be checked on a Zariski covering (see also Lemma~\ref{L:seminormal maximal}); consequently, we may unambiguously define seminormality as a property also of schemes.
By \cite[Theorem~4.1]{swan}, the forgetful functor from seminormal schemes to arbitrary schemes has a right adjoint, called \emph{seminormalization}; the seminormalization is integral over the original scheme, so the normalization factors through it.
\end{defn}

\begin{remark} \label{R:perfect seminormal}
Any normal reduced ring is seminormal.
Also, if $R$ is perfect of characteristic $p$, then $R$ is seminormal: any $(y,z)$ in the image of \eqref{eq:seminormal} is the image of
\[
x = (y^{1/p})^a (z^{1/p})^b, \qquad (a,b) = \begin{cases} (1,0) & p=2 \\
(\frac{p-3}{2}, 1) & p>2. \end{cases}
\]
\end{remark}

\begin{lemma} \label{L:seminormal maximal}
Let $R$ be a ring. If $R_\gothm$ is seminormal for every maximal ideal $\gothm$ of $R$, then $R$ is seminormal.
\end{lemma}
\begin{proof}
See \cite[Proposition~3.7]{swan}.
\end{proof}

\begin{defn}
A morphism $Y \to X$ of schemes is \emph{subintegral} if it induces an isomorphism of seminormalizations. The justification for this terminology will be given by
Lemma~\ref{L:universal homeomorphism1} below.
A morphism of rings $R \to S$ is \emph{subintegral} if the morphism
$\Spec(S) \to \Spec(R)$ is subintegral.

A morphism $Y \to X$ of schemes is a \emph{universal homeomorphism} if for any scheme $Z$, the induced map $\left|Y \times_X Z\right| \to \left|Z\right|$ of topological spaces is a homeomorphism.
For example, for any scheme $Y$ with seminormalization $Y'$, the adjunction map $Y' \to Y$ is a universal homeomorphism (see Lemma \ref{L:universal homeomorphism1} below).
\end{defn}

\begin{lemma} \label{L:universal homeomorphism1}
Let $f: Y \to X$ be a morphism of schemes. Then $f$ is subintegral (resp.\ is a universal homeomorphism) if and only it satisfies the following conditions.
\begin{enumerate}
\item[(a)]
The map $f$ is integral.
\item[(b)]
The map $\left| f \right|: \left| Y \right| \to \left| X \right|$ is a bijection.
\item[(c)]
For each $y \in Y$ mapping to $x \in X$, $\kappa(y)$ is isomorphic to (resp.\ is a purely inseparable extension) of $\kappa(x)$.
\end{enumerate}
\end{lemma}
\begin{proof}
For the first case, see \cite[Theorem~2.5]{swan}.
For the second case, note that by \cite[Tag~04DF]{stacks-project}, $f$ is a universal homeomorphism if and only if it is
integral, universally injective, and surjective. Moreover, if $K \to L$ is a integral morphism of fields, then $\Spec(L) \to \Spec(K)$ is universally injective if and only if $L/K$ is purely inseparable. This yields the claim.
\end{proof}
\begin{cor} \label{C:subintegral base extension}
The subintegrality property of a morphism of schemes
is stable under arbitrary base extension.
\end{cor}
\begin{proof}
Of the conditions of Lemma~\ref{L:universal homeomorphism1}, (a) and (c) are manifestly stable under base extension. As for (b), note that if $f: Y \to X$ induces an isomorphism of seminormalizations, then it is also a universal homeomorphism; since that property is by design stable under base extension, (b) remains true under any base extension.
\end{proof}

\begin{cor}
Let $f: Y \to X$ be a morphism of schemes over $\QQ$. Apply the seminormalization functor to obtain another morphism $f': Y' \to X'$. Then $f$ is a universal homeomorphism if and only if $f'$ is an isomorphism.
\end{cor}

\begin{cor}
If $f: Y \to X$ is a universal homeomorphism of schemes, then the induced map $(Y \times_X Y)_{\red} \to Y_{\red}$ is an isomorphism.
\end{cor}
\begin{proof}
By Lemma~\ref{L:universal homeomorphism1}, $f$ is integral, hence affine, hence separated.
Hence the diagonal $\Delta: Y \to Y \times_X Y$ is a closed immersion; since
$Y \times_X Y \to Y$ is a universal homeomorphism, so then is $\Delta$.
The map $\Delta_{\red}: Y_{\red} \to (Y \times_X Y)_{\red}$ is a closed immersion of reduced schemes which is a bijection on points, hence an isomorphism. It follows that 
$(Y \times_X Y)_{\red} \to Y_{\red}$ is an isomorphism.
\end{proof}

\begin{remark} \label{R:seminormal etale}
For $f: Y \to X$ a universal homeomorphism of schemes, it can be shown that $f_{\et}: Y_{\et} \to X_{\et}$ is also a homeomorphism. See for example \cite{conrad-mathoverflow}.
\end{remark}

\begin{remark}
Lemma~\ref{L:universal homeomorphism1} implies that the definition of seminormality
given above agrees with the one used in \cite{greco-traverso, traverso}, modulo the running hypothesis imposed therein: only noetherian rings are considered, and a seminormal ring is required to have finite normalization. 
Due to the presence of these noetherian hypotheses, we will be extremely careful about citing results from \cite{greco-traverso, traverso}.

One may formally promote some results by eliminating noetherian hypotheses by the usual approach: since seminormality is manifestly preserved by formation of direct limits, 
we may study it by writing a general ring as a direct limit of finitely generated $\ZZ$-subalgebras. See Lemma~\ref{L:seminormal finite approximation} below.
\end{remark}

\begin{lemma} \label{L:seminormal finite approximation}
Let $k$ be an excellent ring.
Any seminormal $k$-algebra is the (filtered) direct limit of its subalgebras which are
finitely generated over $k$ and seminormal. In particular, such subalgebras are noetherian and even excellent.
\end{lemma}
\begin{proof}
Let $R$ be a seminormal $k$-algebra. For each finite subset $S$ of $R$,
let $R_{S,0}$ be the subring of $R$ generated by $S$; this is a finitely
generated $k$-algebra which is reduced but not necessarily seminormal. Let $R_{S,1}$ be the seminormalization of $R_{S,0}$ in $R$, i.e., the smallest seminormal subring of $R$ containing $R_{S,0}$.
Since $R_{S,0}$ is reduced and excellent, the normalization $R_{S,2}$ of $R_S$ is a finite 
$R_S$-algebra and the induced map $R_{S,1} \to R_{S,2}$ is injective. It follows that $R_{S,1}$ is also a finite $R_S$-algebra, and hence is finitely generated over $k$.
This proves the claim.
\end{proof}

\begin{defn}
Let $R \subset S$ be an inclusion of rings.
There then exists a maximal $R$-subalgebra of $S$ which is subintegral over $R$
\cite[Lemma~2.2]{swan}; it is called the \emph{seminormalization} of $R$ in $S$ and denoted $\prescript{S}{+}R$.
\end{defn}

\begin{defn}
For $R \subset S$ an inclusion of rings, the annihilator $\ann_R(S/R)$ of $S/R$ as an $R$-module is called the \emph{conductor} of $R$ in $S$. Note that while it is defined as an ideal of $R$, it is also an ideal of $S$; in fact, it is the largest such ideal.
\end{defn}

\begin{lemma} \label{L:seminormal descent}
Let $R \to S$ be a faithfully flat morphism of rings. If $S$ is seminormal, then $R$ is seminormal.
\end{lemma}
\begin{proof}
Since $S$ is reduced, so then is $R$.
Choose $y,z \in R$ with $y^3 = z^2$ and let $R'$
be the reduced quotient of $R[x]/(x^2-y,x^3-z)$;
then $R \to R'$ is injective and $\Spec(R') \to \Spec(R)$ is subintegral.
Put $S' = R' \otimes_R S$; by Corollary~\ref{C:subintegral base extension},
$\Spec(S') \to \Spec(S)$ is also subintegral.

Let $I$ be the conductor of $R$ in $R'$.
By faithful flatness, $IS=IS'$ is the conductor of $S$ in $S'$.
Since $S$ is seminormal, by \cite[Lemma~1.3]{traverso}, $IS$ is a radical ideal in $S'$ and hence also in $S$. In particular, $IS$ is the intersection of prime ideals in $S'$ so $I = IS' \cap R' = IS \cap R'$ is the intersection of prime ideals in $R'$; that is, $I$ is a radical ideal in both $R$ and $R'$.

However, it is obvious that $y,z \in I$ because $x^ny,x^nz \in R$ for all $n\geq1$.
Consequently, the reduced ring $R'/I$ contains $R/I$ but is a quotient
of $(R/I)[T]/(T^2, T^3)$, whose reduced quotient is again $R/I$.
We deduce that $R' = R$; since this was true for arbitrary $y,z$, it follows that $R$ is seminormal.
\end{proof}

\begin{remark}
In the case where $R$ and $S$ are noetherian and $S$ has finite normalization,
Lemma~\ref{L:seminormal descent} is asserted in \cite[Corollary~1.7]{greco-traverso}. However, this ultimately depends on \cite[Lemma~1.7, Corollary 1.8]{traverso},
whose proofs we were unable to follow. An alternate approach is suggested in
\cite[Remark, p.\ 215]{swan}, but we were unable to follow this argument either
(it was unclear to us why the application of \cite[Theorem~2.8]{swan} has its stated consequence).
\end{remark}

\begin{lemma} \label{L:seminormal etale2}
Let $X$ be a scheme and let $f: Y \to X$ be a finite \'etale surjective morphism.
Then $X$ is seminormal if and only if $Y$ is.
\end{lemma}
\begin{proof}
We may assume $X = \Spec(A)$ is affine.
The ``if'' direction follows from Lemma~\ref{L:seminormal descent}, so we focus on the ``only if'' direction.

Suppose that $X$ is seminormal.
Let $Y' \to Y$ be the seminormalization of $Y$; by Remark~\ref{R:seminormal etale}, the finite \'etale equivalence relation on $Y$ coming from $f$ transfers to a finite \'etale equivalence relation on $Y'$. By \cite[Tag~07S5]{stacks-project}, we obtain a finite \'etale morphism $f': Y' \to X'$ and a morphism $X' \to X$ such that the diagram
\[
\xymatrix{
Y' \ar[r] \ar^{f'}[d] & Y \ar^{f}[d] \\
X' \ar[r] & X
}
\]
is cartesian. By Lemma~\ref{L:universal homeomorphism1}, $X' \to X$ defines an isomorphism of seminormalizations, and hence is itself an isomorphism because $X$ is seminormal.
Hence $Y' \to Y$ is an isomorphism, so $Y$ is seminormal.
\end{proof}

\begin{lemma} \label{L:seminormal field extension}
Let $X$ be a scheme over a field $k$. Let $\ell/k$ be a field extension
such that $\ell \otimes_k \ell$ is a reduced ring. (If $\ell/k$ is an algebraic extension, this means precisely that $\ell$ is separable.)
Then $X$ is seminormal if and only if $X \times_{\Spec(k)} \Spec(\ell)$ is seminormal.
\end{lemma}
\begin{proof}
We may assume from the outset that $X = \Spec(R)$ is affine.
Again, the ``if'' direction follows from Lemma~\ref{L:seminormal descent}, so we focus on the ``only if'' direction.

Suppose that $R$ is seminormal. To check that $R \otimes_k \ell$ is seminormal, we
use Lemma~\ref{L:seminormal finite approximation} and transfinite induction
to reduce to the case where
$R$ is a finitely generated $k$-algebra and $\ell$ is either a finite separable extension of $k$ or a purely transcendental extension. In the former case,
we may apply Lemma~\ref{L:seminormal etale2} to deduce that $R \otimes_k \ell$ is seminormal; in the latter case, it is enough to check that $R \otimes_k k[x] = R[x]$ is seminormal, for which we may use \cite[Proposition~5.2(a)]{greco-traverso}.
\end{proof}

\begin{remark}
Using Lemma~\ref{L:seminormal finite approximation}, one can extend the results of  \cite[\S 5]{greco-traverso} to establish the stability of seminormality under various other types of base extensions.
\end{remark}

\begin{lemma} \label{L:excellent seminormal}
Let $R$ be a noetherian local ring with completion $\widehat{R}$.
\begin{enumerate}
\item[(a)]
If $\widehat{R}$ is seminormal, then $R$ is seminormal.
\item[(b)]
If $R$ is excellent and seminormal, then $\widehat{R}$ is seminormal.
\end{enumerate}
\end{lemma}
\begin{proof}
See \cite[Corollary~1.8, Corollary~5.3]{greco-traverso}.
\end{proof}

\begin{remark}
A theorem of Swan \cite[Theorem~1]{swan} implies that a reduced ring $R$ is seminormal if and only if the map $\Pic(R) \to \Pic(R[X])$ is an isomorphism. For a simple proof
of this theorem, see \cite{coquand}.
\end{remark}

\begin{defn} \label{D:h-topology}
Following Voevodsky \cite{voevodsky}, we define the \emph{$h$-topology} on the category of schemes as the Grothendieck topology generated by finite families $\{p_i: U_i \to X\}$ of morphisms of finite type for which $\sqcup p_i: \sqcup U_i \to X$ is a universal topological epimorphism. For example, any morphism of finite type which induces an isomorphism of seminormalizations is a covering for the $h$-topology.
\end{defn}

\begin{prop} \label{P:h-topology seminormal}
An excellent $\QQ$-scheme $X$ is seminormal if and only if the pushforward of the structure sheaf from the $h$-topology on $X$ to the Zariski topology is isomorphic to the structure sheaf on $X$.
\end{prop}
\begin{proof}
Let $X_h$ be the $h$-topology and let $\nu: X_h \to X$ be the canonical morphism.
Note first that $X$ is reduced if and only if $\calO_X \to \nu_* \nu^* \calO_X$ is injective;
we may thus assume hereafter that this is the case.

Suppose first that $X$ is not seminormal, and let $f: Y \to X$ be the seminormalization.
By Lemma~\ref{L:universal homeomorphism1}, $f$ is affine, so $f_* \calO_Y \neq \calO_X$;
and $f$ is a cover for the $h$-topology, so $\calO_X \to \nu_* \nu^* \calO_X$ is not surjective.

Conversely, suppose that $X$ is seminormal. The fact that $\calO_X \to \nu_* \nu^* \calO_X$ is surjective follows from \cite[Proposition~3.2.10]{voevodsky} as explained in
\cite[Proposition~4.5]{huber-jorder}. (The latter result includes a stronger running hypothesis, namely that $X$ is separated of finite type over a field of characteristic $0$, but the reduction to \cite[Proposition~3.2.10]{voevodsky} is formal and that result
applies to any excellent $\QQ$-scheme.)
\end{proof}

\begin{cor} \label{C:h-topology seminormal}
Let $X$ be a seminormal excellent $\QQ$-scheme, let $X_h$ be its $h$-topology, and let $\nu: X_h \to X$ be the canonical morphism.
Then for any locally finite free $\calO_X$-module $\calF$, the adjunction morphism
$\calF \to \nu_* \nu^* \calF$ is an isomorphism.
\end{cor}

\begin{remark}
It is reasonable to expect Proposition~\ref{P:h-topology seminormal} to remain true without the restriction that $X$ is a $\QQ$-scheme. For example,
one key step in the argument would be to show that for $f: Y \to X$ a projective birational morphism between regular excellent schemes, we have $R^i f_* \calO_Y = 0$ for all $i>0$, and indeed this has been recently established by 
Chatzistamatiou and R\"ulling \cite[Theorem~1.1]{chatzistamatiou-rulling}.
\end{remark}

\subsection{Pseudocoherence and Beauville-Laszlo glueing}
\label{subsec:BL glueing}

The Beauville-Laszlo glueing theorem 
(discussed previously in this series in \cite[Proposition~1.3.6]{part1})
asserts that if $R$ is a ring and $f \in R$ is not a zero-divisor, then
$f$-torsion-free $R$-modules can be recovered from their base extensions to the localization $R_f$ and the $f$-adic completion $\widehat{R}$; moreover, one can detect whether a module is finitely generated, finitely presented, or finite projective from the base extensions.
Here, we do some related commutative algebra which will help establish pseudocoherence of various modules.

\begin{hypothesis} \label{H:BL glueing}
Throughout \S\ref{subsec:BL glueing}, let $R$ be a ring and suppose that $f \in R$ is not a zero-divisor. 
Let $\widehat{R}$ denote the $f$-adic completion of $R$.
For $M$ an $R$-module, write $M[f^n]$ for the $f^n$-torsion submodule and $M[f^\infty]$ for $\bigcup_{n=1}^\infty M[f^n]$.
\end{hypothesis}

\begin{remark} \label{R:torsion bound}
Let $M$ be an $f$-power-torsion $R$-module, so that 
$M = M[f^\infty]$.
Suppose that $M/fM$ is finitely generated. Absent any further hypothesis, it can be the case that $M$ is not finitely generated; e.g., if
$M = R_f/R$ then $M/fM = 0$.
This will necessitate conditions on $f$-torsion in some of our subsequent results.
However, if we assume also that $M$ is killed by some power of $f$, then $M$ is finitely generated.
\end{remark}

\begin{remark} \label{R:no divisible map}
A related (albeit trivial) observation is that if $R$ is $f$-adically separated (that is, $\bigcap_{n=1}^\infty f^n R = 0$), then for any finite projective $R$-module $P$, we have
$\Hom_R(P_f, R) = 0$.
\end{remark}

\begin{remark} \label{R:subfinitely generated}
Let $\calC_R$ be the category of $R$-modules admitting injections into finitely generated $R$-modules. We gather some observations about $\calC_R$, by way of comparison with the category of finitely generated $R$-modules.
\begin{itemize}
\item
It is obvious that $\calC_R$ is closed under formation of subobjects. It is also closed under quotients: if $M \to N$ is injective, $N$ is finitely generated, and $P$ is a submodule of $M$, then $M/P \to N/P$ is injective and $N/P$ is finitely generated.
\item
The category $\calC_R$ is not in general closed under extensions.
However, if $0 \to M_1 \to M \to M_2 \to 0$ is an exact sequence of $R$-modules
such that $M_1 \in \calC_R$ and $M_2$ is itself finitely generated, then $M \in \calC_R$.
To wit, by pulling back the exact sequence
we may reduce to the case where $M_2$ is finite free; then any injection of $M_1$ into a finitely generated $R$-module $N_1$ extends to a map $M \to N_1$, and combining this map with the projection $M \to M_2$ yields an injective morphism
$M \to N_1 \oplus M_2$.
\item
If $M$ is an $R/fR$-module and $M \in \calC_R$, it does not follow that $M \in \calC_{R/fR}$.
\end{itemize}
\end{remark}

\begin{lemma} \label{L:splitting to finite}
Let $M$ be an $R$-module such that $M/fM$ is finitely generated.
Suppose that there exists a homomorphism $M \to N$ of $R$-modules such that
$N$ is finitely generated and $M_f \to N_f$ is a split inclusion. Then 
$M/M[f^\infty]$ is finitely generated; in particular, if $M[f^\infty]$ is finitely generated, then so is $M$. (This remains true even if $f$ is a zero-divisor in $R$.)
\end{lemma}
\begin{proof}
We may reduce at once to the case where $M[f^\infty] = N[f^\infty] = 0$; in this case, $M \to M_f$ is itself injective, and we may identify $M$ with its image in $M_f$.
Choose a generating set $\be_1,\dots,\be_n$ of $N$ and 
a splitting $s: N_f \to M_f$ of the inclusion $M_f \to N_f$. We can then find a nonnegative integer $k$ such that $s(f^k \be_1),\dots,s(f^k \be_n) \in M$.
Choose $\be'_1,\dots,\be'_m \in M$ which generate $M/fM$; they then also generate $M/f^k M$.

Given $\bv \in M$, in $N$ we have $\bv = \sum_{i=1}^n r_i \be_i$ for some $r_i \in R$.
Multiplying by $f^k$ and then applying $s$, we obtain $f^k \bv = \sum_{i=1}^n r_i s(f^k \be_i)$. This means that in $M$, the $R$-span of the $s(f^k \be_i)$ contains $f^k M$, so
\[
\be'_1,\dots,\be'_m, s(f^k \be_1), \dots, s(f^k \be_n)
\]
form a finite generating set of $M$.
\end{proof}

In order to apply Lemma~\ref{L:splitting to finite}, the following observations will be helpful. These may be thought of as providing an analogue of the construction of Fitting ideals \cite[Definition~1.1.2]{part1}, but instead of quantifying obstructions to projectivity of finitely generated modules, we are quantifying obstructions to finite generation of modules which are contained in finitely generated modules.

\begin{lemma} \label{L:splitting ideal}
Let $\iota: M \to N$ be a homomorphism of $R$-modules.
\begin{enumerate}
\item[(a)]
Let $I$ be the set of $g \in R$ for which $M_g \to N_g$ is a split inclusion.
Then $I$ is a radical ideal of $R$.
\item[(b)]
Suppose that $R$ is reduced, $\iota$ is injective, and $N$ is finitely generated.
Then $I$ is not contained in any minimal prime ideal of $R$.
\end{enumerate}
\end{lemma}
\begin{proof}
To prove (a), note that clearly $I$ is closed under multiplication by $R$ and extraction of roots, so we need only check closure under addition. That is, we must check that if $g = g_1 + g_2$ with $g_1, g_2 \in I$, then $g \in I$; this formally reduces to the case $g=1$,
in which case $\Spec(R)$ is covered by the open subspaces $\Spec(R_{g_1})$ and $\Spec(R_{g_2})$.
Choose splittings of the inclusions $M_{g_1} \to N_{g_1}, M_{g_2} \to N_{g_2}$; their base-extensions to $R_{g_1 g_2}$ then differ by an element of 
\[
\Hom_{R_{g_1 g_2}}((N/M)_{g_1 g_2}, M_{g_1 g_2}) =
\Hom_R(N/M, M)_{g_1 g_2}.
\]
Since the quasicoherent sheaf associated to $\Hom_R(N/M, M)$ on the affine scheme
$\Spec(R)$ is acyclic, we may modify the splittings so that they glue to a splitting of $M \to N$.

To prove (b), let $\gothp$ be a minimal prime ideal of $R$; since $R$ is reduced, $R_\gothp$ is a field and $N_\gothp$ is a finite-dimensional $R_\gothp$-vector space. Choose a basis of $N_\gothp$; its elements can all be realized in $N_g$ for some $g \in R \setminus \gothp$. By Nakayama's lemma, by replacing $g$ with a suitable multiple, we can ensure that the chosen basis is also a basis of $N_g$ over $R_g$. (This can also be arranged using Fitting ideals.) Since $M_\gothp$ injects into $N_\gothp$, it is also a finite-dimensional $R_\gothp$-vector space and the injection $M_\gothp \to N_\gothp$ admits a splitting; by possibly replacing $g$ again, we can ensure that this splitting is induced by a morphism $N_g \to M_g$. By possibly replacing $g$ once more, we can ensure that the composition $N_g \to M_g \to N_g$ is a projector, at which point $N_g \to M_g$ is a splitting of the inclusion $M_g \to N_g$; hence $g \in I \setminus \gothp$. This yields (b).
\end{proof}

\begin{cor} \label{C:splitting ideal}
Let $M$ be an $R$-module.
\begin{enumerate}
\item[(a)]
Let $I$ be the set of $g \in R$ for which there exists a homomorphism $M \to N$ with $N$ finitely generated and $M_g \to N_g$ a split inclusion.
Then $I$ is a radical ideal of $R$.
\item[(b)]
If $R$ is reduced and $M$ injects into some finitely generated $R$-module, then $I$ is not contained in any minimal prime ideal of $R$.
\end{enumerate}
\end{cor}
\begin{proof}
To check (a), suppose that $g_1 + g_2 = g$ and $g_1, g_2 \in I$.
For $i=1,2$, choose a homomorphism $M \to N_i$ with $N_i$ finitely generated 
and $M_{g_i} \to (N_i)_{g_i}$ a split inclusion. Then $M \to N_1 \oplus N_2$ is also a homomorphism with $N_1 \oplus N_2$ finitely generated; we may split the inclusion $M_{g_i} \to (N_1 \oplus N_2)_{g_i}$ by first projecting onto $(N_i)_{g_i}$. By Lemma~\ref{L:splitting ideal}(a), we see that $M_g \to (N_1 \oplus N_2)_g$ is also a split inclusion. This yields (a), given which (b)
is immediate from Lemma~\ref{L:splitting ideal}(b).
\end{proof}

We now return to pseudocoherent modules and give a glueing statement.
\begin{lemma} \label{L:pseudocoherent glueing}
Let $M$ be an $R$-module satisfying one of the following conditions.
(As per Hypothesis~\ref{H:BL glueing}, $f$ is not a zero-divisor in $R$.)
\begin{enumerate}
\item[(a)]
The module $M$ is finitely generated over $R$, the module $M_f$ is finite projective over $R_f$, and the modules $M/fM$ and $M[f]$ are pseudocoherent over $R/fR$.
\item[(b)]
The module $M/fM$ is pseudocoherent over $R$, and the module $M$ injects
into a finite free $R$-module $F$ in such a way that $M_f$ is a direct summand of $F_f$.
\end{enumerate}
Then $M$ is pseudocoherent.
\end{lemma}
\begin{proof}
We first verify that (b) implies (a). It is obvious that $M_f$ is finite projective over $R_f$ and that $M[f] = 0$ is pseudocoherent over $R/fR$, so we need only check that $M$ is finitely generated over $R$; this holds by Lemma~\ref{L:splitting to finite}.

We may now assume that (a) holds.
Since $f$ is not a zero-divisor, we may tensor the exact sequence
\[
0 \to R \stackrel{\times f}{\to} R \to R/fR \to 0
\]
with $M$ to deduce that $\Tor^R_1(M, R/fR) = M[f]$.
By (a), we may construct an exact sequence $0 \to N \to F' \to M \to 0$ of $R$-modules with $F'$ finite free; by tensoring over $R$ with $R/fR$, we obtain an exact sequence
\[
0 \to \Tor^R_1(M, R/fR) = M[f] \to N/fN \to F'/fF' \to M/fM \to 0.
\]
Using Lemma~\ref{L:pseudocoherent 2 of 3}, we see that $N/fN$ is pseudocoherent over $R/fR$. On the other hand, since $R_f$ is flat over $R$, we have an exact sequence
\[
0 \to N_f \to F_f \to M_f \to 0
\]
which splits because $M_f$ is projective, so $N_f$ is a direct summand of $F_f$.
Hence $N$ satisfies (b), so we may iterate the argument to see that $M$ is pseudocoherent.
\end{proof}

\begin{remark} \label{R:no pseudoflat completion}
It is reasonable to ask whether  Lemma~\ref{L:pseudocoherent glueing}
remains true if condition (a) is weakened to assert that $M_f$ is pseudocoherent, rather than finite projective.
One difficulty is that if $R$ is not noetherian, then $\widehat{R}$ can fail to be a flat $R$-module, making it less than clear whether $M \otimes_R \widehat{R}_f$ is pseudocoherent over $\widehat{R}_f$.
For example, in \cite{beauville-laszlo} one finds an example in which $\varphi \in R$ is not a zero-divisor (so $R/\varphi R$ is fpd of global dimension 1) and
$\Tor_1^R(R/\varphi R, \widehat{R}) \neq 0$; a similar discussion follows
\cite[Theorem~10.17]{atiyah-macdonald}.
See \cite[Tag~0BNU]{stacks-project} for additional examples,
and \cite[Tag~0BNI]{stacks-project} for more discussion around the Beauville-Laszlo theorem.
\end{remark}

One can partially answer Remark~\ref{R:no pseudoflat completion} by restricting the torsion of the modules in question.

\begin{lemma} \label{L:tor completion}
Let $M$ be a finitely generated (resp.\ finitely presented and $f$-torsion-free) $R$-module. Then the map $M \otimes_R \widehat{R} \to \widehat{M} = \varprojlim_n M/f^n M$ is surjective (resp.\ bijective).
\end{lemma}
\begin{proof}
If $M$ is finitely generated, we may choose an exact sequence $0 \to N \to F \to M \to 0$ with $F$ finite free; we then have a commutative diagram
\[
\xymatrix{
 & N \otimes_R \widehat{R} \ar[r] \ar[d] & F \otimes_R \widehat{R} \ar[r] \ar[d] & M \otimes_R \widehat{R} \ar[r] \ar[d] & 0 \\
0 \ar@{-->}[r] & \widehat{N} \ar[r] & \widehat{F} \ar[r] & \widehat{M} \ar[r] & 0
}
\]
with exact rows (excluding the dashed arrow for the moment). The middle vertical arrow is surjective; hence the right vertical arrow is surjective. 

If $M$ is also finitely presented, then the same logic applied to $N$ shows that the left vertical arrow is surjective; if $M$ is also $f$-torsion-free, then we may add the dashed arrow to the diagram while preserving exactness. Hence the right vertical arrow is bijective.
\end{proof}
\begin{lemma} \label{L:pseudocoherent tor vanishing}
Let $M$ be a pseudocoherent $R$-module with finitely generated $f$-torsion. Then  $\Tor_1^R(M, \widehat{R}) = 0$, and $M\otimes_R\widehat{R}$ has finitely generated $f$-torsion.
\end{lemma}
\begin{proof}
Suppose first that $M$ is $f$-torsion-free.
Choose an exact sequence $0 \to N \to F \to M \to 0$ with $F$ finite free. By Lemma~\ref{L:pseudocoherent 2 of 3}, $N$ is again pseudocoherent and $f$-torsion-free; we thus have a commutative diagram
\[
\xymatrix{
0 \ar@{-->}[r] & N \otimes_R \widehat{R} \ar[r] \ar[d] & F \otimes_R \widehat{R} \ar[r] \ar[d] & M \otimes_R \widehat{R} \ar[r] \ar[d] & 0 \\
0 \ar[r] & \widehat{N} \ar[r] & \widehat{F} \ar[r] & \widehat{M} \ar[r] & 0
}
\]
in which the vertical arrows are isomorphisms by Lemma~\ref{L:tor completion} while the bottom row is exact because $M$ is $f$-torsion-free; this yields that $M\otimes_R\widehat{R}\cong\widehat{M}$ is $f$-torsion-free. We may thus add the dashed arrow to the top row while preserving exactness, yielding that $\Tor_1^R(M, \widehat{R}) = 0$.

Suppose next that $M = R/f^n R$ for some positive integer $n$. Since $f$ is not a zero-divisor in $R$, we have an exact sequence 
\[
0 \to R \stackrel{\times f^n}{\to} R \to M \to 0.
\]
It is straightforward to see that  $f$ is not a zero-divisor in $\widehat{R}$, thus this sequence remains exact upon tensoring over $R$ with $\widehat{R}$. Hence $\Tor_1^R(M, \widehat{R}) = 0$ and $M \otimes_R \widehat{R} \cong R/f^n R$ has finitely generated $f$-torsion.

Suppose next that $M$ is $f$-torsion. Choose a positive integer $n$ such that $f^n M = 0$. 
Choose an exact sequence $0 \to N \to F \to M \to 0$ of $R/f^n R$-modules with $F$ finite free; 
this sequence remains exact upon tensoring over $R$ with $\widehat{R}$ (this being the same as tensoring over
$R/f^n R$ with $\widehat{R}/f^n \widehat{R}$, which is isomorphic to $R/f^n R$). Since $\Tor_1^R(F, \widehat{R}) = 0$ by the previous paragraph, we deduce that $\Tor_1^R(M, \widehat{R}) = 0$ and $M \otimes_R \widehat{R} \cong M$ has finitely generated $f$-torsion.

Consider finally the general case. Let $T$ be the $f$-power-torsion submodule of $M$; by 
Remark~\ref{R:transfer pseudocoherence to quotient}, $T$ and $M/T$ are pseudocoherent. We may thus deduce from above that $\Tor_1^R(T, \widehat{R}) = \Tor_1^R(M/T, \widehat{R}) = 0$, $T\otimes_R\widehat{R}$ has finitely generated $f$-torsion and $(M/T)\otimes_R\widehat{R}$ is $f$-torsion-free, concluding the lemma. 
\end{proof}

\begin{lemma} \label{L:glueable module}
Let $M$ be an $R$-module with finitely generated $f$-torsion. Then $M \to M \otimes_R (R_f \oplus \widehat{R})$ is injective.
\end{lemma}
\begin{proof}
If $M$ is $f$-torsion-free, this is obvious because $M \to M \otimes_R R_f$ is injective. If $M$ is finitely generated and $f$-torsion, this is also clear because for any positive integer $n$ for which $f^n M = 0$, we have
\[
M \otimes_R \widehat{R} = M \otimes_{R/f^n R} \widehat{R}/f^n \widehat{R}
= M \otimes_{R/f^n R} R/f^n R = M.
\]
The general case follows at once.
\end{proof}

\begin{prop} \label{P:Beauville-Laszlo with limited torsion}
Base extension defines an equivalence of exact tensor categories between pseudocoherent $R$-modules with finitely generated $f$-torsion and descent data for pseudocoherent modules with finitely generated $f$-torsion on the diagram $\widehat{R} \to \widehat{R}_f \leftarrow R_f$.
\end{prop}
\begin{proof}
By Lemma~\ref{L:pseudocoherent tor vanishing}, the base extension functor is exact; it thus takes pseudocoherent modules with finite $f$-torsion to descent data for pseudocoherent modules with finite $f$-torsion. 
By Lemma~\ref{L:glueable module}, every module of the source category is a \emph{glueable module}
in the sense of the generalized Beauville-Laszlo theorem \cite[Tag~0BNI]{stacks-project}; hence the base extension functor is fully faithful.

To prove essential surjectivity, it suffices to check that if we apply the generalized Beauville-Laszlo theorem to a descent datum, then the resulting module $M$ is pseudocoherent with finitely generated $f$-torsion. We may apply \cite[Tag~0BNN]{stacks-project} to see that $M$ is finitely generated; by covering $M$ with a finite free module and iterating the construction (using exactness from the previous paragraph), we may deduce that $M$ is pseudocoherent. We also know from Beauville-Laszlo that the sequence
\[
0 \to M \to M \otimes_R (R_f \oplus \widehat{R}) \to M \otimes_R \widehat{R}_f \to 0
\]
is exact; hence $M$ and $M \otimes_R \widehat{R}$ have the same $f$-power torsion,
and so $M$ has finitely generated $f$-torsion.
\end{proof}

\begin{remark}
In the Beauville-Laszlo theorem, it is possible to slightly weaken the hypothesis that $f$ is not a zero-divisor in $R$; see \cite[Tag~0BNI]{stacks-project}. We will not attempt 
to emulate this generalization here.
\end{remark}

\section{Supplemental foundations of adic spaces}
\label{sec:foundations adic}

We continue with some complements to \cite{part1} concerning foundations of the theory of adic spaces. The main goal is to extend the theory of vector bundles on adic spaces
to a more comprehensive theory of \emph{pseudocoherent sheaves}. See \cite[Lecture~1]{kedlaya-aws} for an alternate presentation.

\setcounter{theorem}{0}
\begin{convention} \label{conv:Banach}
We retain all conventions about Banach rings, adic Banach rings, and (pre)adic spaces from \cite{part1}. In particular, all Banach rings we consider contain a topologically nilpotent unit, and hence are open mapping rings.
We define the \emph{pro-\'etale topology} associated to a preadic space as in
\cite[\S 9.1]{part1}.
\end{convention}

\begin{remark} \label{R:reduce to countable}
We make explicit here a formal reduction argument which appears several times without detailed discussion in \cite{part1}. 
When considering statements in algebraic geometry, one often considers only finitely many elements at a time; for such statements, one may typically check the statement over finitely generated $\ZZ$-algebras (which are noetherian and even excellent) and then view an arbitrary ring as a direct limit of such rings. Analogously, when considering algebraic statements about modules over Banach rings, one generally only considers at most countably many elements at a time. For example, when considering objects in the pro-\'etale site over an adic space $X$, it is generally sufficient to make calculations on diagrams containing only countably many morphisms; by taking fibred products, we may then further reduce to linear towers, i.e., to sequences of morphisms $\cdots \to Y_1 \to Y_0 = X$.
\end{remark}

\subsection{An aside on reified adic spaces}
\label{subsec:reified}

One can make a parallel development of relative $p$-adic Hodge theory in the language of \emph{reified adic spaces}, as described in \cite{kedlaya-reified}. This theory is the analogue of Huber's theory of adic spaces where the base rings are equipped with a distinguished norm, rather than a norm topology; this provides better compatibility with Berkovich's theory of general analytic (rather than strictly analytic) spaces, especially in the case where one works over a trivially valued base field, while still providing the flexibility to deal with the nonnoetherian rings arising in the theory of perfectoid spaces (as indicated in \cite[\S 11]{kedlaya-reified}).

Much of the following discussion admits a 
parallel treatment in this framework, but we prefer not to weigh down the exposition with a full second track. Instead, we provide a few basic definitions here and continue the discussion through a thread of brief remarks
(Remarks~\ref{R:really strongly noetherian}, \ref{R:really strongly noetheran sheafy},
\ref{R:reified pseudoflat}, \ref{R:reified pseudocoherent}).

\begin{defn} \label{D:associated graded}
Let $A$ be a Banach ring. Following Temkin \cite{temkin-local2},
we define the associated graded ring
\[
\Gr A = \bigoplus_{r>0} \Gr^r A, \qquad \Gr^r A = A^{\circ, r}/A^{\circ \circ, r}
\]
with $A^{\circ,r}$ (resp. $A^{\circ \circ, r}$) being the set of $a \in A$
for which the sequence $\{r^{-n} \left| a^n \right|\}_{n=1}^\infty$ is bounded (resp.\ null). This construction depends on the norm up to equivalence (not just up to its norm topology).

A \emph{graded adic Banach ring} is a pair $(A,A^{\Gr})$ in which $A$ is a Banach ring
and $A^{\Gr} = \bigoplus_{r>0} A^{\Gr,r}$ is an integrally closed graded subring of $\Gr A$. Given such a pair,
for $r>0$, let $A^{+,r}$ be the inverse image of $A^{\Gr,r}$ in $A^{\circ, r}$;
by taking $A^+ = A^{+,1}$, we get an associated adic Banach ring $(A,A^+)$.
\end{defn}

\begin{defn}
For $A$ a Banach ring, a \emph{reified semivaluation} on $A$ is a semivaluation 
$v: A \to \Gamma_v$ whose value group $\Gamma_v$ is equipped with a distinguished subgroup identified with $\RR_{>0}$. Two reified semivaluations $v_1, v_2$ on $A$ are declared to be equivalent if for all $a \in A, r \geq 0$, 
\[
v_1(a) \leq r \Longleftrightarrow v_2(a) \leq r,
\qquad
v_1(a) \geq r \Longleftrightarrow v_2(a) \geq r.
\]
A reified semivaluation $v$ on $A$ is \emph{commensurable} if for all $\gamma \in \Gamma_v$, we have $\gamma \geq r$ for some $r>0$.
\end{defn}

\begin{defn}
Let $(A,A^{\Gr})$ be a graded adic Banach ring. The \emph{reified adic spectrum}
$\Spra(A,A^{\Gr})$ consists of equivalence classes of commensurable reified semivaluations $v$ on
$A$ such that $v(a) \leq r$ for all $r>0$, $a \in A^{+,r}$.
A \emph{rational subspace} of $\Spra(A,A^{\Gr})$ is a subspace of the form
\begin{equation} \label{eq:general rational subspace}
\{v \in \Spra(A,A^{\Gr}): v(f_1) \leq q_1 v(g), \dots, v(f_n) \leq q_n v(g) \}
\end{equation}
where $f_1,\dots,f_n,g \in A$ generate the unit ideal and $q_1,\dots,q_n$ are nonnegative real numbers. For the topology on $\Spra(A,A^{\Gr})$ generated by rational subspaces,
$\Spra(A,A^{\Gr})$ is a spectral space and every rational subspace is quasicompact
\cite[Theorem~6.3]{kedlaya-reified}.
\end{defn}

\begin{defn}
Let $(A,A^{\Gr})$ be a graded adic Banach ring. 
For $U$ a rational subspace of $\Spra(A,A^{\Gr})$ as in \eqref{eq:general rational subspace}, put
\[
B = A\{T_1/q_1,\dots,T_n/q_n\}/(gT_1 - f_1, \dots, gT_n - f_n)
\]
(where $A\{T_1/q_1,\dots,T_n/q_n\}$ is a weighted Tate algebra as in \cite[Definition~2.2.15]{part1})
and let $B^{\Gr}$ be the integral closure of the image of $A[T_1,\dots,T_n]$ 
in $\Gr B$ (placing $T_i$ in degree $q_i$).
Then $(A,A^{\Gr}) \to (B,B^{\Gr})$ is initial among morphisms
$(A,A^{\Gr}) \to (C,C^{\Gr})$ of graded adic Banach rings for which
$\Spra(C,C^{\Gr})$ maps into $U$; moreover, the map $\Spra(B,B^{\Gr}) \to U$ is a homeomorphism \cite[Lemma~7.4]{kedlaya-reified}.
In particular, $(B,B^{\Gr})$ is functorially determined by $U$, so we may define a \emph{structure presheaf} $\calO$ on $\Spra(A,A^{\Gr})$ evaluating to $B$ on $U$.
(We also obtain a presheaf $\calO^{\Gr}$ evaluating to $B^{\Gr}$ on $U$.)

We may formally glue together reified adic spectra to define \emph{reified preadic spaces} as in \cite[Definition~8.2.3]{part1}. We say that $(A,A^{\Gr})$ is \emph{sheafy} if the structure presheaf on $\Spra(A,A^{\Gr})$ is a sheaf; as in \cite[Definition~7.18]{kedlaya-reified}, we may glue together the reified adic spectra of sheafy graded adic Banach rings to obtain certain locally ringed spaces (with extra structure) called \emph{reified adic spaces}.
\end{defn}

\begin{remark} \label{R:reified}
Note that by requiring the first member of a graded adic Banach ring to be a Banach ring, we are conventionally requiring the presence of a topologically nilpotent unit.
In the theory of reified adic spaces, the role of topologically nilpotent units
is sometimes transferred to the real normalizations; notably, the reified analogue of Tate's reduction process \cite[Lemma~7.10]{kedlaya-reified} does not require a topologically nilpotent unit. However, there are several contexts where such units remain necessary, including the open mapping theorem (Theorem~\ref{T:open mapping}); see Remark~\ref{R:reified pseudoflat} for another example.
\end{remark}

\subsection{Flatness of rational localizations}

We begin by assembling some statements about the flatness of rational localizations (and \'etale morphisms) of strongly noetherian adic Banach rings. Since flatness and pseudoflatness are the same thing for morphisms out of a noetherian ring, these results will be subsumed by \S\ref{subsec:pseudoflat rational}; however, we prefer to state these known results first.

\begin{lemma} \label{L:flat plus faithful}
Let $A \to B$ be a flat ring homomorphism. Then the following conditions are equivalent.
\begin{enumerate}
\item[(a)]
The morphism $A \to B$ is faithfully flat; that is,
for any homomorphism $M \to N$ of $A$-modules, if $M \otimes_A B \to N \otimes_A B$ is injective then $M \to N$ is injective.
\item[(b)]
The map $\Spec(B) \to \Spec(A)$ is surjective.
\item[(c)]
The image of $\Spec(B) \to \Spec(A)$ contains $\Maxspec(A)$.
\end{enumerate}
\end{lemma}
\begin{proof}
See \cite[Tag 00HQ]{stacks-project}.
\end{proof}

\begin{lemma} \label{L:faithful to flat}
Let $A \to B, B \to C$ be ring homomorphisms such that the composition $A \to B \to C$ is flat and $B \to C$ is faithfully flat. Then $A \to B$ is flat.
\end{lemma}
\begin{proof}
See \cite[Tag 0584]{stacks-project}.
\end{proof}

\begin{defn}
Let $(A,A^+)$ be an adic Banach ring (not necessarily sheafy).
By an \emph{\'etale covering} of $(A,A^+)$, we will mean a finite family of morphisms
 $\{(A,A^+) \to (B_i, B_i^+)\}_i$ which form a covering of
$\widetilde{\Spa}(A,A^+)$ for the \'etale topology (e.g., a rational covering).
\end{defn}

\begin{lemma} \label{L:covering faithful}
Let $(A,A^+)$ be an adic Banach ring.
For any \'etale covering $\{(A,A^+) \to (B_i, B_i^+)\}_i$,
the image of $\Spec(\bigoplus_i B_i) \to \Spec(A)$ contains $\Maxspec(A)$.
\end{lemma}
\begin{proof}
We first check the case of a rational covering.
For each $\gothm \in \Maxspec(A)$, apply \cite[Corollary~2.3.5]{part1} to construct some $\alpha \in \calM(A)$ for which $\gothm = \ker(A \to \calH(\alpha))$. For some $i$, $\alpha$ extends to $\beta \in \calM(B_i)$, and $\ker(\beta)$ is a prime ideal of $B_i$ lifting $\gothm$. Then we may apply the natural map $\calM(B_i)\to\Spa(B_i, B_i^+) $ (see \cite[Definition~2.4.6]{part1}) to conclude. 

To verify the 
general case, we check the criteria of \cite[Proposition~8.2.21]{part1}:
criterion (a) and (b) are both formal, criterion (c) is checked above, and criterion (d) is trivially true.
\end{proof}

\begin{defn}
A Banach ring $A$ is \emph{strongly noetherian} if for every nonnegative integer $n$,
the ring $A\{T_1,\dots,T_n\}$ is noetherian.
An adic Banach ring $(A,A^+)$ is strongly noetherian if $A$ is strongly noetherian;
in this case, $(A,A^+)$ is sheafy (by \cite[Proposition~2.4.16]{part1}) and every rational localization of $(A,A^+)$ is again strongly noetherian 
(by \cite[Lemma~2.4.13]{part1}).
\end{defn}

\begin{lemma} \label{L:rational localization is flat}
Let $(A,A^+) \to (B,B^+)$ be an \'etale morphism of strongly noetherian adic Banach rings. Then the homomorphism $A \to B$ is flat.
\end{lemma}
\begin{proof}
Apply \cite[Lemma~1.7.6]{huber}.
\end{proof}

\begin{cor} \label{C:covering faithfully flat}
Let $(A,A^+)$ be a strongly noetherian adic Banach ring.
For any \'etale covering $\{(A,A^+) \to (B_i, B_i^+)\}_i$,
the morphism $A \to \bigoplus_i B_i$ is faithfully flat.
\end{cor}
\begin{proof}
Combine Lemma~\ref{L:flat plus faithful}, Lemma~\ref{L:covering faithful},
and Lemma~\ref{L:rational localization is flat}.
\end{proof}

\begin{remark}
We do not know whether Lemma~\ref{L:rational localization is flat} remains true if one drops the strongly noetherian hypothesis. We will establish some weaker statements in this direction later (\S\ref{subsec:pseudoflat rational}).
\end{remark}

\begin{remark}
Since every finite module over a noetherian Banach ring is complete
(Corollary~\ref{C:noetherian complete}), Lemma~\ref{L:rational localization is flat} can be interpreted as saying that base extension (via the completed tensor product) along a rational localization defines an exact functor on categories of finite Banach modules. This does not remain true for the full categories of Banach modules, in part because the completed tensor product involves composing a left-exact functor with a right-exact functor. See Example~\ref{exa:no topological flatness}.
\end{remark}

The following example is based on a discussion with Peter Scholze. 
\begin{example} \label{exa:no topological flatness}
Put 
\begin{align*}
(A,A^+) &= (\Qp\{T\}, \Zp\{T\}), \\
(B,B^+) &= (\Qp\{T/p^{-1}\}, \Zp\{T/p^{-1}\}), \\
(C,C^+) &= (\Qp\{(T-1)/p^{-1}\}, \Zp\{(T-1)/p^{-1}\}).
\end{align*}
Then $(A,A^+) \to (B,B^+), (A,A^+) \to (C,C^+)$ are rational localizations corresponding to disjoint subsets of $\Spa(A,A^+)$, so $B \widehat{\otimes}_A C = 0$.
That is, the functor $\bullet \widehat{\otimes}_A B$ on the category of Banach modules over $A$ is not left-exact, because it maps the injection $A \to C$ to the map
$B \to 0$. In fact, it is not right-exact either: if $0 \to M \to N \to P \to 0$
is a short exact sequence of Banach modules over $A$, then the sequence
\[
M \widehat{\otimes}_A B \to N \widehat{\otimes}_A B \to P \widehat{\otimes}_A B \to 0
\]
is exact at $P \widehat{\otimes}_A B$ by Theorem~\ref{T:open mapping}, but not necessarily at $N \widehat{\otimes}_A B$. This can be seen using an example of the form
\[
A \stackrel{\times f}{\to} A \to A/fA \to 0
\]
such as in Example~\ref{exa:not stably pseudocoherent}.
 
\end{example}

\begin{remark} \label{R:really strongly noetherian}
For reified adic spaces, the analogues of Lemma~\ref{L:rational localization is flat}
and Corollary~\ref{C:covering faithfully flat} are given in \cite[Corollary~8.11]{kedlaya-reified}. However, this requires a slightly stronger noetherian condition: one must start with a Banach ring $A$ which is \emph{really strongly noetherian}, meaning that all weighted Tate algebras $A\{T_1/q_1,\dots,T_n/q_n\}$ are noetherian (not just when $q_1 = \cdots = q_n = 1$).
\end{remark}

\subsection{Coherence in the noetherian setting}
\label{subsec:coherence noetherian}

We continue by discussing coherent sheaves on adic spectra of strongly noetherian Banach rings, starting with an analogue of Tate's acyclicity theorem for coherent sheaves.
This will be subsumed by our results on pseudocoherent sheaves (\S\ref{subsec:pseudocoherent sheaves}), but may be of independent interest for readers only interested in the noetherian case.

\begin{theorem} \label{T:strongly noetherian Tate}
Let $(A,A^+)$ be a strongly noetherian adic Banach ring.
Let $M$ be a finitely generated $A$-module and let $\tilde{M} = M \otimes_A \calO_X$ be the associated sheaf on $X = \Spa(A,A^+)$.
Then for every finitely generated $A$-module $M$, we have
\[
H^i(X, \tilde{M}) = \begin{cases} M & i=0 \\ 0 & i>0.
\end{cases}
\]
\end{theorem}
\begin{proof}
Apply \cite[Theorem~2.5]{huber2}.
\end{proof}

In the strongly noetherian case, one can also establish an analogue of Kiehl's glueing theorem for coherent sheaves, which for some reason does not appear in the prior literature. We thus give a proof which is a fairly straightforward consequence of some glueing arguments introduced in \cite{part1}. However, for the benefit of the reader who may have arrived here via an external reference and may not otherwise be familiar with \cite{part1}, we spell out the argument in somewhat greater detail than might otherwise be warranted.

\begin{lemma} \label{L:strongly noetherian finitely generated}
Let $(A,A^+)$ be a strongly noetherian adic Banach ring.
Then for any coherent sheaf $\calF$ on $X = \Spa(A,A^+)$, the $A$-module $M$ of global sections of $\calF$ is finitely generated and the natural map $\tilde{M} \to \calF$ is surjective.
\end{lemma}
\begin{proof}
By definition, $X$ admits a rational covering $\{U_i\}_{i \in I}$ such that the restriction of $\calF$ to each $U_i$ is the cokernel of a morphism $\calG_{i,1} \to \calG_{i,2}$ of finite free $\calO$-modules. Put $G_{i,j} = \calG_{i,j}(U_i)$ for $j=1,2$ and put
$N_i = \ker(G_{i,1} \to G_{i,2})$.
By Lemma~\ref{L:rational localization is flat}, the sequence
\[
0 \to \tilde{N}_i \to \calG_{i,1} \to \calG_{i,2} \to \left. \calF \right|_{U_i} \to 0
\]
is exact; consequently, by Theorem~\ref{T:strongly noetherian Tate}, we may identify
$\left.\calF \right|_{U_i}$ with $\tilde{M}_i$ for $M_i = \calF(U_i)$.

Let us say that a rational covering $\{U_i\}_{i \in I}$ has property $\calP$ if for any coherent sheaf $\calF$ on $X$
whose restriction to each $U_i$ is the cokernel of a morphism of finite free $\calO$-modules, the $A$-module $M$ of global sections of $\calF$ is finitely generated and the natural map $\tilde{M} \to \calF$ is surjective.
It suffices to check that every covering has property $\calP$ by verifying the three axioms of Tate's reduction process \cite[Proposition~2.4.20]{part1}.
\begin{enumerate}
\item[(a)]
The property $\calP$ can be checked after passage to a more refined covering: this is obvious.
\item[(b)]
The property $\calP$ holds for a composition of two coverings if it holds for each covering individually: this is also obvious.
\item[(c)]
The property $\calP$ holds for any simple Laurent covering: to verify this, choose any $f \in A$,
let $U_1, U_2, U_{12}$ be the rational subsets
\[
\{v \in X: v(f) \leq 1\},
\{v \in X: v(f) \geq 1\},
\{v \in X: v(f) = 1\}
\]
of $X$ and let
\[
(A, A^+) \to (A_1, A_1^+),
(A, A^+) \to (A_2, A_2^+),
(A, A^+) \to (A_{12}, A_{12}^+)
\]
be the corresponding rational localizations.
By Theorem~\ref{T:strongly noetherian Tate}, the sequence
\begin{equation} \label{eq:exact sequence at ring level}
0 \to A \to A_1 \oplus A_2 \to A_{12} \to 0,
\end{equation}
in which the map to $A_{12}$ is the difference between the canonical homomorphisms, is a strict exact sequence of $A$-modules. In addition, $A_2$ has dense image in $A_{12}$ (since the map $A[f^{-1}] \to A_{12}$ has dense image)
and the map $\calM(A_1 \oplus A_2) \to \calM(A)$ is surjective; the rings $A, A_1, A_2, A_{12}$ therefore form a \emph{glueing square} in the sense of \cite[Definition~2.7.3]{part1}. This allows us to follow an axiomatization of the proof  of Kiehl's original theorem presented in \cite[\S 2.7]{part1}; this approach in turn depends on some glueing arguments for coherent sheaves on bare rings described in \cite[\S 1.3]{part1}, which also recover the Beauville-Laszlo theorem on glueing of vector bundles \cite[Remark~2.7.9]{part1}.

By Theorem~\ref{T:strongly noetherian Tate}, for $i=1,2$,
restriction of sections defines a map $\psi_i: M_i \to M_{12}$ for which the induced map
$M_i \otimes_{A_i} A_{12} \to M_{12}$ is an isomorphism.
Note that $M$ is the kernel of the map $\psi_1 - \psi_2: M_1 \oplus M_2 \to M_{12}$.
By a short explicit calculation (see \cite[Lemma~2.7.4]{part1}),
one sees that the natural maps $M \otimes_A A_i \to M_i$
are surjective. We may thus find a finitely generated $A$-submodule $M_0$ of $M$
such that $M \otimes_A A_i \to M_i$ is surjective for $i=1,2$.
For any $s \in M$, we may use Nakayama's lemma to show that $M_0 \to M_0 + As$ is surjective (see \cite[Lemma~2.3.12]{part1} for details); consequently, $M = M_0$ is itself a finitely generated $A$-module and $\tilde{M} \to \calF$ is surjective.
\end{enumerate}
As noted above, 
\cite[Proposition~2.4.20]{part1} now yields the desired result.
\end{proof}

\begin{theorem} \label{T:stably noetherian Kiehl}
Let $(A,A^+)$ be a strongly noetherian adic Banach ring.
Then the categories of coherent sheaves on $X = \Spa(A,A^+)$ and finitely generated $A$-modules are equivalent via the global sections functor.
\end{theorem}
\begin{proof}
Note first that Lemma~\ref{L:strongly noetherian finitely generated} implies that the global sections functor on coherent sheaves does indeed take values in finitely generated $A$-modules. Note next that by Theorem~\ref{T:strongly noetherian Tate},
composing the functor $M \mapsto \tilde{M}$ from finitely generated $A$-modules to coherent sheaves on $X$ with the global sections functor yields an equivalence.
It thus suffices to check that the composition in the opposite direction is an
equivalence, by showing that for any coherent sheaf $\calF$ with $\calF(X) = M$, the natural map
$\tilde{M} \to \calF$ is an isomorphism.
However, this map is clearly injective, and it is surjective by
Lemma~\ref{L:strongly noetherian finitely generated}, so the proof is complete.
\end{proof}

\begin{remark} \label{R:really strongly noetheran sheafy}
In the setting of reified adic spaces, 
if one starts with a really strongly noetherian Banach ring, the resulting space is again sheafy \cite[Theorem~8.15]{kedlaya-reified},
the structure presheaf is again an acyclic sheaf on rational subspaces
\cite[Theorem~7.14]{kedlaya-reified},
and coherent sheaves correspond to finite modules \cite[Theorem~8.16]{kedlaya-reified}.
\end{remark}

\subsection{Pseudoflatness for rational localizations}
\label{subsec:pseudoflat rational}

It is not known in general whether rational localizations of arbitrary Banach rings are flat.
To make up for this, we introduce a weaker notion.

\begin{defn}
Let $(A,A^+)$ be an adic Banach ring. 
For $m \geq 0$, a topological $A$-module $S$ is \emph{stably $m$-pseudocoherent} 
(resp.\ \emph{stably $m$-fpd}) if 
$S$ is $m$-pseudocoherent (resp.\ $m$-fpd) and for every rational localization $(A,A^+) \to (B,B^+)$, the $B$-module $S \otimes_A B$ is complete for the natural topology. As usual, we omit $m$ when it equals $\infty$.
In case $A$ is strongly noetherian, Corollary~\ref{C:noetherian complete} implies that any finite $A$-module is
stably pseudocoherent, but not necessarily fpd.
\end{defn}

\begin{example} \label{exa:not stably pseudocoherent}
It is possible to produce examples in which $(A,A^+)$ is an adic Banach ring, $(A,A^+) \to (B,B^+)$ is a rational localization, $f \in A$ is not a zero-divisor of either $A$ or $B$,
and $f$ generates a closed ideal of $A$ but not of $B$. In this case, $A/fA$ is strictly pseudocoherent but not stably pseudocoherent.

The following explicit construction was suggested by Scholze.
Choose an integer $N \geq 3$ coprime to $p$. Let $U_0 = \Spa(A_0,A_0^+)$ be a strict neighborhood of the ordinary locus in the modular curve $X(N)_K$ over a perfectoid field $K$.
Let $X(N p^\infty)_K$ be the modular curve of infinite $p$-level and tame level $N$,
as a perfectoid space \cite[Theorem~3.1.2]{scholze-torsion}. 
Let $\psi: X(N p^\infty)_K \to \PP^{1,\an}_K$ be the Hodge-Tate period map, let $D \subseteq \PP^{1,\an}_K$ be the closed unit disc, and let $U$ be the intersection in $X(N p^\infty)_K$ 
of the inverse images of $U_0$ and $D$. 
Then $U = \Spa(A,A^+)$ is an affinoid perfectoid space and the map $\psi$ corresponds to an element $f \in A$. The inverse image in $U$ of the ordinary locus $V_0$ of $X(N)_K$ is a rational subspace
$V = \Spa(B,B^+)$.

Since $V_0$ is one-dimensional, its Shilov boundary coincides with the topological boundary 
of its maximal Hausdorff quotient (see \cite[\S 3.4]{bpr}); 
consequently, the Shilov boundary of a finite cover of $V_0$ is contained in the inverse image of the Shilov boundary of $V_0$. It follows that 
the Shilov boundary $Z$ of $U$ is contained in the inverse image of the Shilov boundary of $U_0$. 
Since $\alpha(f) \neq 0$ for each $\alpha \in Z$, by compactness these values are bounded away from $0$. This implies that $fA$ is a closed ideal of $A$ (compare \cite[Theorem~1.7]{escassut}).

By contrast, $f$ maps $V$ surjectively onto the pro-discrete subset $\ZZ_p$ of $D$.
Again, the Shilov boundary $Z'$ of $V$ is the inverse image of the Shilov boundary of $V_0$, but now there do exist points $\alpha \in Z'$ for which $\alpha(f) = 0$. By \cite[Theorem~1.4, Theorem~1.7]{escassut}, 
it follows that either $f$ is a zero-divisor in $B$ or $fB$ is not a closed ideal of $B$.
To rule out the former, suppose that $g \in B$ satisfies $fg=0$ but $g \neq 0$;
then $g$ must vanish identically on the complement of $\psi^{-1}(0)$. 
Let $V_n = \Spa(B_n, B_n^+)$ be the inverse image of $V$ in $X(Np^n)_K$;
we can then write $g$ as the limit of some sequence $\{g_n\}$ with $g_n \in B_n$.
Let $c>0$ be the spectral norm of $g$; for $n$ sufficiently large, $g-g_n$ has spectral norm at most $c/2$,
and so the restriction of $g_n$ to $\psi^{-1}(0)$ has spectral norm $c$.
This implies the same for the restriction of $g_n$ to $\psi^{-1}(p^n \lambda)$ for each $\lambda \in \ZZ_p$
due to the existence of a commutative diagram
\[
\xymatrix{
V \ar[r] \ar[d] & \ZZ_p \ar[d] \\
V_n \ar[r] & \ZZ/p^n \ZZ.
}
\]
(That is, the image in $V_n$ of a point in $V$ determines the Hodge-Tate period map modulo $p^n$.)
However, this contradicts the fact that for $n$ sufficiently large, $g_n - g_{n+1}$ has spectral norm at most $c/2$.
From this contradiction, we deduce that $f$ is not a zero-divisor in $B$, and so $fB$ is not a closed ideal of $B$.

As an aside, we point out that it would be of great interest to construct a similar example in which $f$ is indeed a zero-divisor in $B$, as this would give an example of a rational localization which is not flat.
\end{example}

One important source of stably pseudocoherent modules is the following observation.
\begin{remark} \label{R:pseudocoherent from sheafy}
Let $(A,A^+)$ be a sheafy adic Banach ring, let $I$ be a strictly pseudocoherent ideal of $A$,
and let $(A/I)^+$ be the integral closure of the image of $A^+$ in $A/I$.
If $(A/I, (A/I)^+)$ is again sheafy, then for any rational localization $(A, A^+) \to (B,B^+)$,
the map $A/I \to B/IB$ is again a rational localization; consequently, $A/I$ (and hence $I$) is stably pseudocoherent \cite[Theorem~1.2.7]{kedlaya-aws}.
In fact, the converse is also true: if $A/I$ is stably pseudocoherent,
then $(A/I, (A/I)^+)$ is again sheafy. As we will not need this result, we defer to \cite[Theorem~1.4.20]{kedlaya-aws} for its proof.
\end{remark}

\begin{defn} \label{D:topologically flat}
Let $(A,A^+)$ be an adic Banach ring. 
For $m \geq 0$, a topological $A$-module $B$ is \emph{$m$-pseudoflat} if 
for every stably $m$-pseudocoherent $A$-module $P$, we have $\Tor_1^A(P, B) = 0$; as in Remark~\ref{R:pseudoflat}, this does not imply the vanishing of $\Tor_i^A(P,B)$ for $i>0$. We say that $B$ is \emph{pseudoflat} if it is $\infty$-pseudoflat; in particular, if $B$ is $m$-pseudoflat for some finite $m$, then $B$ is pseudoflat.
\end{defn}

\begin{remark} \label{R:topologically flat reduction}
With notation as in Definition~\ref{D:topologically flat}, the following conditions are equivalent.
\begin{enumerate}
\item[(a)]
For any strictly $m$-pseudocoherent $A$-module $M$, $\Tor^A_1(M,B) = 0$. This implies that $B$ is $m$-pseudoflat.)
\item[(b)]
For every short exact sequence
\[
0 \to M \to N \to P \to 0
\]
in which $P$ is strictly $m$-pseudocoherent, the map $M \otimes_A B \to N \otimes_A B$ is injective.
\item[(c)]
The conclusion of (b) holds in all cases where $N$ is finite projective.
By Lemma~\ref{L:pseudocoherent 2 of 3} and Corollary~\ref{C:closure finitely generated}, in such cases $M$ is $(m-1)$-pseudocoherent.
\end{enumerate}
However, if $(A,A^+) \to (B,B^+)$ is a rational localization and the above conditions hold, 
it is not immediately clear that for any stably $m$-pseudocoherent $A$-module $P$, the $B$-module $P \otimes_A B$ is stably $m$-pseudocoherent: it is clear that $(P \otimes_A B) \otimes_B C \cong P \otimes_A C$ is complete for every rational localization $(B,B^+) \to (C,C^+)$, but not that $P \otimes_A B$ is $m$-pseudocoherent. 
(For a short exact sequence as in (b), the sequence remains exact but we do not know that $M \otimes_A B$ is $(m-1)$-pseudocoherent.)
However, this does hold if $m=2$ (because $1$-pseudocoherence is preserved by arbitrary base extension) or $m=\infty$.
In particular, for these values of $m$, the composition of an $m$-pseudoflat rational localization followed by any $m$-pseudoflat morphism is again $m$-pseudoflat.
\end{remark}

A ready supply of pseudoflat modules is given by the following construction.
\begin{defn} \label{D:pro-projective}
Let $(A,A^+)$ be an adic Banach ring. A complete metrizable
topological $A$-module $B$ is \emph{pro-projective}
if there exists a filtered family of continuous $A$-linear projectors $\pi_1, \pi_2, \ldots: B \to B$ whose images are finite projective $A$-modules such that for each $\bv \in B$, the sequence $\pi_1(\bv), \pi_2(\bv), \dots$ converges to $\bv$; this definition is stable under arbitrary base change (of Banach modules).
For example, $A\{T\}$ and $A\{T,T^{-1}\}$ are pro-projective.
\end{defn}

\begin{lemma} \label{L:pro-projective topologically flat}
Let $(A,A^+)$ be an adic Banach ring. 
Let $B$ be a pro-projective $A$-module.
For any strictly $2$-pseudocoherent $A$-module $M$,
$M \otimes_A B$ is complete for the natural topology and $\Tor^A_1(M,B) = 0$.
\end{lemma}
\begin{proof}
Let 
\[
0 \to N \to A^n \to M \to 0
\]
be an exact sequence of $A$-modules.
Consider the commutative diagram
\[
\xymatrix{
& N \otimes_A B \ar[r] \ar[d] & A^n \otimes_A B \ar[r] \ar[d] & M \otimes_A B \ar[r] \ar[d] & 0 \\
0 \ar[r] & N \widehat{\otimes}_A B \ar[r] &A^n  \widehat{\otimes}_A B \ar[r] & M  \widehat{\otimes}_A B \ar[r] & 0 
}
\]
with the completions in the second row being taken for the natural topology.
In this diagram,
the first row is obviously exact; the second row is exact because $B$ is pro-projective; the middle vertical arrow is bijective; and the left vertical arrow is surjective because $N$ is finitely generated
(see Remark~\ref{R:finitely generated strict}). It follows that the right vertical arrow is an isomorphism.
This shows that $M \otimes_A B$ is complete for the natural topology.

Since $M$ is 2-pseudocoherent, $N$ is finitely presented, so we may repeat the argument (using a different diagram) to see that the left vertical arrow is an isomorphism. This shows that $\Tor_1^A(M, B) = 0$, completing the proof.
\end{proof}

\begin{cor} \label{C:pro-projective topologically flat} 
Any pro-projective $A$-module is $2$-pseudoflat.
\end{cor}
\begin{proof}
This follows from Lemma~\ref{L:pro-projective topologically flat}.
\end{proof}

\begin{cor} \label{C:module power series}
Let $(A,A^+)$ be an adic Banach ring. Let $M$ be a finitely generated (resp.\ finitely presented) $A$-module which is complete with respect to its natural topology. Then 
the natural morphisms
\begin{gather}
\label{eq:module power series1}
M \otimes_A A\{T\} \to M\{T\},  \\
\label{eq:module power series2}
M \otimes_A A\{T,T^{-1}\} \to M\{T,T^{-1}\}
\end{gather}
are surjective (resp.\ bijective).
\end{cor}
\begin{proof}
This is immediate from (the proof of) Lemma~\ref{L:pro-projective topologically flat}.
\end{proof}

We now establish pseudoflatness for rational localizations of sheafy
adic Banach rings. This subsumes the cases of strongly noetherian rings,
by \cite[Theorem~2.4.5]{part1}, and of stably uniform rings, by
\cite[Theorem~2.8.10]{part1}.

\begin{lemma} \label{L:weak flatness0}
Let $(A,A^+)$ be a sheafy adic Banach ring. 
Choose $f \in A$, define the simple Laurent covering
\[
U_1 = \{v \in \Spa(A,A^+): v(f) \leq 1\}, \qquad
U_2 = \{v \in \Spa(A,A^+): v(f) \geq 1\},
\]
and let $(A, A^+) \to (B_1, B_1^+)$, $(A, A^+) \to (B_2, B_2^+)$,
$(A, A^+) \to (B_{12}, B_{12}^+)$
 be the rational localizations corresponding to $U_1, U_2, U_1 \cap U_2$.
Then the maps
\[
A\{T\} \stackrel{\times T-f}{\longrightarrow} A\{T\},
\quad
A\{T\} \stackrel{\times 1-fT}{\longrightarrow} A\{T\},
\quad
A\{T^{\pm}\} \stackrel{\times T-f}{\longrightarrow} A\{T^{\pm}\},
\quad
A\{T^{\pm}\} \stackrel{\times 1-fT}{\longrightarrow} A\{T^{\pm}\}
\]
are strict inclusions.
\end{lemma}
\begin{proof}
Since $A\{T\} \to A\{T^{\pm}\}$ is a strict inclusion, it suffices to check the assertions involving $A\{T^{\pm}\}$. We treat only the case of $1-fT$, the case $T-f$ being similar in light of the identity $T-f = T(1-fT^{-1})$.

By the sheafy hypothesis (plus \cite[Theorem~2.7.4]{part1} and the open mapping theorem), the map 
\[
A\{T^{\pm}\} \to B_1\{T^{\pm}\} \oplus B_2\{T^{\pm}\}
\]
is a strict inclusion. It thus suffices to check that $A\{T^{\pm}\} \stackrel{\times 1-fT}{\to} A\{T^{\pm}\}$ is a strict inclusion for $A =B_1$ and $A=B_2$. We treat only the case $A = B_1$,
the case $A = B_2$ being similar in light of the identity $1-fT = fT(1-f^{-1}T^{-1})$.

Let $R$ be the set of formal sums $\sum_{n \in \ZZ} c_n T^n$ with $c_n \in B_1$ such that
$c_n$ remains bounded as $n \to \infty$ and converges to 0 as $n \to -\infty$; this is a Banach ring for the Gauss norm and contains $B_1\{T^{\pm}\}$ as a closed subring. In $R$, $1-fT$ admits the multiplicative inverse $1 + fT + f^2T^2 + \cdots$; multiplication by this element defines a homeomorphism from $(1-fT)B_1\{T^{\pm}\}$ to $B_1\{T^{\pm}\}$.
\end{proof}

\begin{remark} \label{R:exact sequence from flatness}
From the proof of \cite[Theorem~2.8.10]{part1}, one reads off a converse to Lemma~\ref{L:weak flatness0}: if the conclusion holds as written, then the sequence
\[
0 \to A \to B_1 \oplus B_2 \to B_{12} \to 0
\]
is exact. If the same holds with $A$ replaced by any rational localization, then
it follows from \cite[Proposition~2.4.20]{part1} that $(A,A^+)$ is sheafy.
This gives a criterion for sheafiness that can be thought of as giving rise to the negative examples of
\cite{buzzard-verberkmoes} and \cite{mihara}.
\end{remark}

\begin{lemma} \label{L:weak flatness1}
With notation as in Lemma~\ref{L:weak flatness0},
the morphism $A \to B_2$ is $2$-pseudoflat.
\end{lemma}
\begin{proof}
Set notation as in Remark~\ref{R:topologically flat reduction}(c) with $m=2$. By Remark~\ref{R:series multiplication}, Corollary~\ref{C:module power series}, and Lemma~\ref{L:weak flatness0}, we have a commutative diagram
\begin{equation} \label{eq:weak flatness}
\xymatrix{
 &  & 0 \ar[d] & 0 \ar[d] & \\
0 \ar[r] & M \otimes_A A\{T\} \ar[r] \ar^{\times (1-fT)}[d]&  N \otimes_A A\{T\} \ar[r] \ar^{\times (1-fT)}[d] & P \otimes_A A\{T\} \ar[r] \ar^{\times (1-fT)}[d] & 0 \\
0 \ar[r] & M \otimes_A A\{T\} \ar[r] \ar[d] & N \otimes_A A\{T\} \ar[r] \ar[d] & P \otimes_A A\{T\} \ar[r] \ar[d] & 0 \\
0 \ar@{-->}[r] & M \otimes_A B_2 \ar[r] \ar[d]&  N \otimes_A B_2 \ar[r] \ar[d] & P \otimes_A B_2 \ar[r] \ar[d] & 0 \\
& 0 & 0 & 0 &
}
\end{equation}
in which excluding the dashed arrow, all rows and columns are exact. However, this implies exactness including the dashed arrow by the snake lemma.
\end{proof}

In the setting of Remark~\ref{R:series multiplication}, it is not clear to us how to make an analogous argument with $1-fT$ replaced by $T-f$ when $f$ is not a unit in $A$. However,
the severity of this issue is somewhat diminished by the following argument.

\begin{lemma} \label{L:weak flatness2}
With notation as in Lemma~\ref{L:weak flatness0}, the morphisms
$A \to B_1$, $A \to B_{12}$ are $2$-pseudoflat.
\end{lemma}
\begin{proof}
By Lemma~\ref{L:weak flatness1}, the morphism $A \to B_2$
is $2$-pseudoflat. By Lemma~\ref{L:weak flatness1} applied with $f$ replaced by its inverse in $B_2$,
$B_2 \to B_{12}$ is $2$-pseudoflat. Using Remark~\ref{R:topologically flat reduction}, it follows that $A \to B_2 \to B_{12}$ is $2$-pseudoflat.

Set notation as in Remark~\ref{R:topologically flat reduction}(c) with $m=2$ and suppose that $P$ is stably $2$-pseudocoherent over $A$.
By the previous paragraph, we obtain a commutative diagram
\begin{equation} \label{eq:weak flatness2}
\xymatrix{
 &  & 0 \ar[d] & 0 \ar[d] & \\
0 \ar[r] & M \ar[r]\ar[d]  &  N \ar[r] \ar[d] & P \ar[r] \ar[d] & 0 \\
0 \ar@{-->}[r] & M \otimes_A (B_1 \oplus B_2) \ar[r] \ar[d] & N \otimes_A (B_1 \oplus B_2) \ar[r] \ar[d] & P \otimes_A (B_1 \oplus B_2) \ar[r] \ar[d] & 0 \\
0 \ar[r] & M \otimes_A B_{12} \ar[r] \ar[d]&  N \otimes_A B_{12} \ar[r] \ar[d] & P \otimes_A B_{12} \ar[r] \ar[d] & 0 \\
& 0 & 0 & 0 &
}
\end{equation}
in which excluding the dashed arrow, the rows and columns are exact. By diagram chasing, we obtain exactness of the dashed arrow; this shows that $A \to B_1$ is $2$-pseudoflat.
 \end{proof}
 
\begin{remark}
In \eqref{eq:weak flatness}, the tensor products are already complete 
by Corollary~\ref{C:module power series}. However, this does not mean that one gets an analogous diagram for an arbitrary strict exact sequence $0 \to M \to N \to P \to 0$ of Banach modules, because of the poor behavior of the completed tensor product.
Indeed, even for $A$ noetherian, the existence of such a diagram in general would ultimately lead to a contradiction against Example~\ref{exa:no topological flatness}.
\end{remark}

\begin{theorem} \label{T:weak flatness}
Let $(A,A^+)$ be a sheafy adic Banach ring.
For any rational localization $(A,A^+) \to (B,B^+)$, $A \to B$ is $2$-pseudoflat.
(By Remark~\ref{R:topologically flat reduction}, this implies that if $M$ is a stably $m$-pseudocoherent $A$-module for some $m \in\{2,\infty\}$, then $M \otimes_A B$ is a stably $m$-pseudocoherent $B$-module.)
\end{theorem}
\begin{proof}

We apply \cite[Proposition~2.4.24]{part1}:
criterion (a) follows from Remark~\ref{R:topologically flat reduction}, while criterion (b) follows from 
Lemma~\ref{L:weak flatness1} and Lemma~\ref{L:weak flatness2}.

\end{proof}

\begin{remark}\label{R:reified pseudoflat}
The previous arguments carry over directly to reified adic Banach rings. For example, Lemma~\ref{L:weak flatness0} must be stated for a cover of the form
\[
U_1 = \{v \in \Spra(A,A^+): v(f) \leq q\}, \qquad
U_2 = \{v \in \Spra(A,A^+): v(f) \geq q\}
\]
for some $q>0$, but the proof does not change in any essential way.
\end{remark}

\subsection{Pseudocoherent sheaves}
\label{subsec:pseudocoherent sheaves}

Using pseudoflatness, we may develop analogues of the theorems of Tate and Kiehl 
for pseudocoherent sheaves. One may further follow \cite{sga6} to define a concept of \emph{pseudocoherent complexes}, but this is far more subtle than working with pseudocoherent complexes of modules over a ring (as discussed in detail in \cite[\S 2]{thomason-trobaugh}) due to the limited nature of pseudoflatness; we will not pursue it here.

\begin{theorem} \label{T:pseudocoherent acyclicity}
Let $(A,A^+)$ be a sheafy adic Banach ring.
For $M$ a stably pseudocoherent $A$-module,
let $\tilde{M}$ be the presheaf for which $\tilde{M}(U)$ is the inverse limit of $M \otimes_A B$ over all rational subspaces $\Spa(B,B^+)$ of $\Spa(A,A^+)$ contained in $U$.
Then for any rational subdomain $U = \Spa(B,B^+)$ of $X$ and any rational covering $\gothV$ of $U$, we have
\[
H^i(U, \tilde{M}) = \check{H}^i(U,\tilde{M}; \gothV) = \begin{cases} M \otimes_A B & i=0 \\ 0 & i>0. \end{cases}
\]
In particular, $\tilde{M}$ is a sheaf.
\end{theorem}
\begin{proof}
By Theorem~\ref{T:weak flatness}, $M \otimes_A B$ is a stably pseudocoherent $B$-module.
Consequently, by \cite[Proposition~2.4.20, Proposition~2.4.21]{part1},
this reduces immediately to Lemma~\ref{L:weak flatness2}.
\end{proof}

\begin{cor} \label{C:locally fpd}
Let $(A,A^+)$ be a sheafy adic Banach ring.
Let $M$ be a stably pseudocoherent $A$-module.
Suppose that there exist a nonnegative integer $m$ and a 
rational covering $\gothV$ of $\Spa(A,A^+)$ such that
$M \otimes_A \calO(V)$ is an $m$-fpd $\calO(V)$-module for each $V \in \gothV$. Then $M$ is $m$-fpd.
\end{cor}
\begin{proof}
We first treat the case $m=0$. If $M \otimes_A \calO(V)$ is projective (i.e., $0$-fpd) for each $V \in \gothV$, then $\tilde{M}$ is a locally finite free sheaf of $\calO$-modules on $\Spa(A,A^+)$. By \cite[Theorem~2.7.7]{part1}, $H^0(\Spa(A,A^+), \tilde{M})$ is a finite projective $A$-module; since this module coincides with $M$ by 
Theorem~\ref{T:pseudocoherent acyclicity}, it follows that $M$ is projective.

We next treat the case $m>0$.
Choose a projective resolution $\cdots \to P_1 \to P_0 \to M \to 0$ of $M$ by finite projective $A$-modules. For $i \geq 0$, put $M_i = \image(P_{i+1} \to P_i)$,
so that $M_0 = M$ and $M_i = \ker(P_i \to P_{i-1})$ for $i>0$.
By Lemma~\ref{L:pseudocoherent 2 of 3} and induction on $i$,
each $M_i$ is a pseudocoherent $A$-module.

Choose $V \in \gothV$ and let $(A,A^+) \to (B,B^+)$ be the corresponding rational localization. 
By Theorem~\ref{T:weak flatness}, the sequence 
\[
\cdots \to P_1 \otimes_A B \to P_0 \otimes_A B \to M \otimes_A B \to 0
\]
is again a projective resolution of $M \otimes_A B$,
and the map $M_i \otimes_A B \to \ker(P_i \otimes_A B \to P_{i-1} \otimes_A B)$ is an isomorphism for all $i>0$.
Since $M \otimes_A B$ is $m$-fpd, we may apply \cite[Tag~00O5]{stacks-project} to see that $M_i \otimes_A B$ is projective for all $i \geq m$. By the $m=0$ case, we see that
$M_i$ is projective for all $i \geq m$, and hence $M$ is $m$-fpd.
\end{proof}

\begin{defn}
Let $X$ be an adic space. We say that a sheaf $\calF$ of $\calO_X$-modules is \emph{pseudocoherent} (resp.\ \emph{fpd}) if it is locally the sheaf associated to a stably pseudocoherent (resp.\ stably fpd) module. By Theorem~\ref{T:weak flatness}, this remains true upon restriction from $X$ to an open subspace.
\end{defn}

\begin{lemma} \label{L:refined Kiehl}
With notation as in Lemma~\ref{L:weak flatness0},
the morphism $A \to B_1 \oplus B_2$ is an effective descent morphism for stably pseudocoherent modules over Banach rings and for stably fpd modules over Banach rings.
\end{lemma}
\begin{proof}
By Corollary~\ref{C:locally fpd}, it suffices to check the case of pseudocoherent modules.
Put $X = \Spa(A,A^+)$ and let $\gothV$ be the cover $\{U_1, U_2\}$ of $X$ as in 
Lemma~\ref{L:weak flatness0}.
Given a descent datum for stably pseudocoherent modules, we may associate to it a pseudocoherent sheaf $\calF$ on $X$.
By the proof of \cite[Lemma~2.7.4]{part1}, the sheaf $\calF$ is finitely generated
and $\check{H}^i(X, \calF; \gothV) = 0$ for all $i>0$.
By Theorem~\ref{T:pseudocoherent acyclicity}, $H^i(U_1, \calF) = H^i(U_2, \calF) = 0$
for $i>0$, so (by \cite[Tag 01EW]{stacks-project} again)
we see that $H^i(X, \calF) = 0$ for all $i>0$.

Since $\calF$ is finitely generated, we may construct a surjection $\calE \to \calF$ for $\calE$ a finite free $\calO$-module. By Theorem~\ref{T:weak flatness}, the kernel $\calG$ of $\calE \to \calF$ is again the pseudocoherent sheaf associated to a descent datum, so again $H^1(X, \calG) = 0$; consequently,
\begin{equation} \label{eq:sequence of psc sheaves}
0 \to \calG(X) \to \calE(X) \to \calF(X) \to 0
\end{equation}
is exact. By induction on $m$ (uniformly in $\calF$) and Lemma~\ref{L:pseudocoherent 2 of 3}, we see that $\calF(X)$ is $m$-pseudocoherent for all $m$; it is moreover complete by virtue of being closed in $\calF(U_1) \oplus \calF(U_2)$.

In the commutative diagram
\[
\xymatrix{
0 \ar[r] & \calG(X) \otimes_A B_i \ar[r] \ar[d] & \calE(X) \otimes_A B_i \ar[r] \ar[d] & \calF(X) \otimes_A B_i \ar[r] \ar[d] & 0 \\
0 \ar[r] & \calG(U_i) \ar[r] & \calE(U_i) \ar[r] & \calF(U_i) \ar[r] & 0
}
\]
the first row is exact by Theorem~\ref{T:weak flatness}, while the second row is exact by
Theorem~\ref{T:pseudocoherent acyclicity}.
The middle vertical arrow is bijective by Theorem~\ref{T:pseudocoherent acyclicity} (for the trivial module), while the left and right vertical arrows are surjective by  \cite[Lemma~2.7.4]{part1} again. By the five lemma, the right vertical arrow is also injective; it follows that $\calF \cong \tilde{M}$ for $M = \calF(X)$, so in particular $M$ realizes the original descent datum. 

To complete the proof, it remains to check that $M$ is stably pseudocoherent. To check this, let $(A,A^+) \to (C,C^+)$ be a rational localization and put $Y = \Spa(C,C^+)$; by the previous argument, $\calF(Y)$ is a strictly pseudocoherent $C$-module, so it is enough to check that the natural map $M \otimes_A C \to \calF(Y)$ is an isomorphism. To this end, note that in the diagram
\[
\xymatrix{
 & \calG(X) \otimes_A C \ar[r] \ar[d] & \calE(X) \otimes_A C \ar[r]\ar[d] & \calF(X) \otimes_A C \ar[r] \ar[d] & 0 \\
0 \ar[r]& \calG(Y) \ar[r] & \calE(Y) \ar[r] & \calF(Y) \ar[r] & 0
}
\]
the top row is obtained from the exact sequence \eqref{eq:sequence of psc sheaves} by tensoring with $C$,
while the bottom row is exact by applying \eqref{eq:sequence of psc sheaves} with $X$ replaced by $Y$.
Since the middle column is an isomorphism, the right column is surjective; repeating the argument with a different diagram shows that the left column is also surjective, so the right column is an isomorphism.
\end{proof}

\begin{theorem} \label{T:refined Kiehl}
For $(A,A^+)$ a sheafy adic Banach ring, the global sections functor defines an exact equivalence of categories between pseudocoherent (resp.\ fpd) sheaves of $\calO$-modules 
on $\Spa(A,A^+)$ and stably pseudocoherent (resp.\ stably fpd) $A$-modules.
\end{theorem}
\begin{proof}
We apply \cite[Proposition~2.4.20]{part1} to the property that for 
$\{(B,B^+) \to (C_i, C_i^+)\}$ a rational covering of a rational subdomain of $\Spa(A,A^+)$, the map $B \to \bigoplus_i C_i$ is an effective descent morphism
for stably pseudocoherent (resp.\ stably fpd) modules over Banach rings and the descent functor is exact.
To check the criteria of \cite[Proposition~2.4.20]{part1},
note that 
criterion (a) follows from Theorem~\ref{T:pseudocoherent acyclicity},
criterion (b) is formal, and criterion (c) follows from Lemma~\ref{L:refined Kiehl}.(Alternatively, one may deduce the equivalence in this manner and then deduce exactness from Theorem~\ref{T:pseudocoherent acyclicity}).
\end{proof}

\begin{cor} \label{C:base extension of psc}
For $(A,A^+) \to (B,B^+)$ a morphism  of sheafy adic Banach rings inducing an open immersion $\Spa(B,B^+) \to \Spa(A,A^+)$ (e.g., a rational localization), base extension defines an exact functor from
stably pseudocoherent (resp. stably fpd) $A$-modules to stably pseudocoherent (resp.\ stably fpd)
$B$-modules. In particular, $A \to B$ is pseudo\-flat.
\end{cor}
\begin{proof}
Let $0 \to M \to N \to P \to 0$ be an exact sequence of $A$-modules in which $N$ is finite projective and $P$ is stably pseudocoherent; by Lemma~\ref{L:pseudocoherent 2 of 3} and Theorem~\ref{T:weak flatness}, $M$ is stably pseudocoherent. By Theorem~\ref{T:refined Kiehl}, the sequence $0 \to \tilde{M} \to \tilde{N} \to \tilde{P} \to 0$
of associated sheaves is exact. Put $Y = \Spa(B,B^+)$ and consider the diagram
\[
\xymatrix{
 & M \otimes_A B \ar[r] \ar[d] & N \otimes_A B \ar[r] \ar[d] & P \otimes_A B \ar[r] \ar[d] & 0 \\
0 \ar[r] & H^0(Y, \tilde{M}) \ar[r] & H^0(Y, \tilde{N}) \ar[r] & H^0(Y, \tilde{P}) \ar[r] & 0;
}
\]
the bottom row is exact by Theorem~\ref{T:refined Kiehl}.
The middle vertical arrow is exact, so the right vertical arrow is surjective. By forming a similar exact sequence with $M$ on the right, we see that the left vertical arrow is also surjective; hence the right vertical arrow is injective. The upshot is that for any stably pseudocoherent $A$-module $P$, $P \otimes_A B \cong H^0(Y, \tilde{P})$ is a stably pseudocoherent $B$-module (by Theorem~\ref{T:refined Kiehl} again) and the functor $P \mapsto P \otimes_A B$ is exact. This implies all the claims.

\end{proof}

\begin{cor}
For $(A,A^+)$ a sheafy adic Banach ring, let 
$\calF^\bullet$ be a (cohomologically indexed) complex of sheaves of $\calO$-modules on $\Spa(A,A^+)$. Then the following conditions are equivalent.
\begin{enumerate}
\item[(a)]
The complex $\calF^\bullet$ has pseudocoherent (resp. fpd) cohomology sheaves 
which vanish in sufficiently large (resp.\ sufficiently small and sufficiently large) degrees.
\item[(b)]
The complex $\calF^\bullet$ is quasi-isomorphic to the complex associated to a bounded above (resp.\ bounded) complex of finite projective $A$-modules whose cohomology groups are stably pseudocoherent (resp.\ stably fpd).
\end{enumerate}
\end{cor}
\begin{proof}
It is clear that (b) implies (a); it thus suffices to check that (a) implies (b).
We build a bounded above complex $\calG^\bullet$ of finite free $\calO$-modules 
and a morphism $\calG^\bullet \to \calF^\bullet$ inducing isomorphisms $h^i(\calG^\bullet) \to h^i(\calF^\bullet)$ for all $i$, by the following inductive procedure: we ensure that for each $i$, the map from the truncated complex $0 \to \calG^i \to \calG^{i+1} \to \cdots$ to $\calF^\bullet$ induces a surjection $h^i(\calG^\bullet) \to h^i(\calF^\bullet)$ and isomorphisms $h^j(\calG^\bullet) \to h^j(\calF^\bullet)$ for $j>i$. As a base case, note that for $i$ sufficiently large we may take $\calG^\bullet$ to be the zero complex.

To carry out the induction from $i$ to $i-1$, it suffices to construct a commuting square
\[
\xymatrix{
\calG^{i-1} \ar[r] \ar[d] & \calG^i \ar[d] \\
\calF^{i-1} \ar[r] & \calF^i
}
\]
with $\calG^{i-1}$ finite free such that $h^i(\calG^\bullet) \to h^i(\calF^\bullet)$ is an isomorphism and $\ker(\calG^{i-1} \to \calG^i) \to h^{i-1}(\calF^\bullet)$ is surjective.

Note that by Lemma~\ref{L:pseudocoherent 2 of 3} and Theorem~\ref{T:weak flatness},
the kernel of a surjective morphism of pseudocoherent sheaves is again pseudocoherent; by applying this argument repeatedly to the exact sequences
\begin{gather*}
0 \to \ker(\calG^j \to \calG^{j+1}) \to \calG^j \to \image(\calG^j \to \calG^{j+1}) \to 0 \\
0 \to \image(\calG^{j-1} \to \calG^{j}) \to
\ker(\calG^{j} \to\calG^{j+1}) \to h^{j}(\calG^{\bullet}) \to 0
\end{gather*}
starting from the first index $j>i$ for which $\calG^{j+1} = 0$ and working down, we
deduce that the sheaves
\[
\ker(\calG^i \to \calG^{i+1}), \qquad
\calH^i = \ker(\ker(\calG^i \to \calG^{i+1}) \to h^i(\calF^\bullet))
\]
are pseudocoherent.
In light of Theorem~\ref{T:refined Kiehl}, we may thus ensure that $h^i(\calG^\bullet) \to h^i(\calF^\bullet)$ is an isomorphism
by choosing generators of $\calG^{i-1}$, mapping these to generators of $\calH^i$, then
choosing suitable images in $\calF^{i-1}$ to make the square commute.
By adding additional generators mapping to zero in $\calG^{i-1}$, we may ensure further that $\ker(\calG^{i-1} \to \calG^i) \to h^{i-1}(\calF^\bullet)$ is surjective.
This completes the induction.

By Corollary~\ref{C:base extension of psc}, the cohomology groups of $\calG^\bullet$ are stably pseudocoherent (resp.\ stably fpd) $A$-modules.
To complete the proof, it suffices to observe that in the fpd case, 
one can take $\calG^i = 0$ for $i$ sufficiently small. For this, apply
\cite[Tag~0658]{stacks-project}.
\end{proof}

We next extend the previous discussion to the \'etale site.

\begin{hypothesis}
For the remainder of \S\ref{subsec:pseudocoherent sheaves}, let $(A,A^+)$ be a sheafy adic Banach ring, put $X= \Spa(A,A^+)$, and assume that the \'etale site $X_{\et}$ (in the sense of \cite[Definition~8.2.19]{part1}) admits a stable  basis $\calB$ of adic affinoid spaces containing $X$. 
\end{hypothesis}

\begin{defn}
Choose $\Spa(B,B^+) \in \calB$. 
For $m \geq 0$, a topological $B$-module $S$ is \emph{\'etale-stably $m$-pseudocoherent} 
(resp.\ \emph{\'etale-stably $m$-fpd}) (with respect to $\calB$) if 
$S$ is $m$-pseudocoherent (resp.\ $m$-fpd) and for every morphism
$(B,B^+) \to (C,C^+)$ with $\Spa(C,C^+) \in \calB$, the $C$-module $S \otimes_B C$ is complete for the natural topology. Note that this implies that $S$ is stably $m$-pseudocoherent. (We do not know whether the converse holds.)
In case $A$ is strongly noetherian, Corollary~\ref{C:noetherian complete} implies that any finite $C$-module is
\'etable-stably pseudocoherent.

We say that $S$ is \emph{$m$-\'etale-pseudoflat} over $B$ (with respect to $\calB$) if for every \'etale-stably $m$-pseudocoherent $B$-module $P$ (with respect to $\calB$), we have $\Tor^B_1(P,S) = 0$. (Note that this terminology is not meant to imply that $S$ is either \'etale over $B$, in the case where it is a $B$-algebra, or $m$-pseudoflat over $B$.)
In particular, if $S = C$ for some $\Spa(C,C^+) \in \calB$ and $m \in \{2,\infty\}$, this implies that for $M$ an \'etale-stably $m$-pseudocoherent $B$-module, $M \otimes_B C$ is an \'etale-stably $m$-pseudocoherent $C$-module.
\end{defn}

\begin{lemma} \label{L:etale pseudoflat}
There exists a stable basis $\calB'$ of $X_{\et}$ contained in $\calB$ and containing $X$ such that every morphism between objects of $\calB'$ is $2$-\'etale-pseudoflat with respect to either $\calB$ or $\calB'$.
\end{lemma}
\begin{proof}
We take $\calB'$ to consist of those objects $\Spa(B,B^+) \in \calB$ for which
$(A,A^+) \to (B,B^+)$ is a composition of rational localizations and finite \'etale morphisms;
such a morphism $(A,A^+) \to (B,B^+)$ is $2$-\'etale-pseudoflat (with respect to either $\calB$ or $\calB'$) by Theorem~\ref{T:weak flatness}.
(In fact each factor in the composition is even $2$-pseudoflat, but due to the presence of the finite \'etale morphisms we cannot extend this conclusion to the composition.)
To prove that a general morphism $(B,B^+) \to (C,C^+)$ in $\calB'$ is $2$-\'etale-pseudoflat, it will suffice to show
that this morphism is itself a composition of rational localizations and finite \'etale morphisms. To this end,
we write
$(A,A^+) \to (B,B^+)$, $(A,A^+) \to (C,C^+)$ as compositions
\begin{gather*}
(A,A^+) = (B_0,B_0^+) \to \cdots \to (B_m, B_m^+) = (B,B^+), \\
(A,A^+) = (C_0,C_0^+) \to \cdots \to (C_n, C_n^+) = (C,C^+)
\end{gather*}
in which each morphism is a rational localization or a faithfully finite \'etale morphism
(noting that a general finite \'etale morphism may be factored as a rational localization followed by a faithfully finite \'etale morphism),
then induct on $m$ with $m=0$ as the base case. For the induction step, 
it suffices to check that the composition 
$(B_1,B_1^+) \to (B,B^+) \to (C,C^+)$ admits a factorization of the desired form, as then we may apply the induction hypothesis to the morphism $(B_1,B_1^+) \to (B,B^+)$, $(B_1,B_1^+) \to (C,C^+)$ to conclude.
The map
\[
(B_1, B_1^+) \to (B_1, B_1^+) \widehat{\otimes}_{(A,A^+)} (C_n, C_n^+)
\]
admits a factorization of the desired form as
\[
(B_1, B_1^+) \to (B_1, B_1^+) \widehat{\otimes}_{(A,A^+)} (C_1, C_1^+) \to 
\cdots \to (B_1, B_1^+) \widehat{\otimes}_{(A,A^+)} (C_n, C_n^+).
\] 
If $(A,A^+) \to (B_1, B_1^+)$ is a rational localization, then $(B_1, B_1^+) \widehat{\otimes}_{(A,A^+)} (C_n, C_n^+). = (C,C^+)$ and the induction step is done. 
If $(A,A^+) \to (B_1,B_1^+)$ is a faithfully finite \'etale morphism, we instead factor $(B_1,B_1^+) \to (C,C^+)$
as $(B_1,B_1^+) \to (B_1, B_1^+) \widehat{\otimes}_{(A,A^+)} (C_n, C_n^+) $, which has the desired form,
followed by
\[
(B_1, B_1^+) \widehat{\otimes}_{(A,A^+)} (C_n, C_n^+) 
\to (B,B^+) \widehat{\otimes}_{(A,A^+)} (C_n, C_n^+) \to (C,C^+) \widehat{\otimes}_{(A,A^+)} (C_n, C_n^+) \to (C,C^+)
\]
where the last map is multiplication. This composition is a morphism in $\calB$ between objects which are finite \'etale over $(C,C^+)$, and so is itself finite \'etale (but not necessarily faithful).

It remains to verify that $\calB'$ is indeed a basis (and hence a stable basis) of $X_{\et}$; that is,
any covering of $\Spa(A,A^+)$ by objects of $\calB$ can be refined to a covering by objects of $\calB'$.
This follows by the last paragraph of the proof of \cite[Proposition~8.2.20]{part1}.
\end{proof}

\begin{theorem} \label{T:pseudocoherent etale acyclicity}
Let $\calB'$ be a stable basis of $X_{\et}$ such that every morphism between objects of $\calB'$ is \'etale-pseudoflat with respect to $\calB'$ (e.g., the basis given in Lemma~\ref{L:etale pseudoflat}; note that Theorem~\ref{T:weak flatness etale} will allow taking $\calB'=\calB$).
Let $M$ be an \'etale-stably pseudocoherent $A$-module with respect to $\calB'$.
Let $\tilde{M} = M \otimes_A \calO$ be the associated presheaf on $X_{\et}$.
Then for any $U = \Spa(B,B^+) \in \calB'$ and any \'etale covering $\gothV$ of $U$ by elements of $\calB'$, we have
\[
H^i(U, \tilde{M}) = \check{H}^i(U,\tilde{M}; \gothV) = \begin{cases} M \otimes_A B & i=0 \\ 0 & i>0. \end{cases}
\]
In particular, $\tilde{M}$ extends to a sheaf which does not change sections on $\calB'$.
\end{theorem}
\begin{proof}
We check the criteria of \cite[Proposition~8.2.21]{part1}:
criterion (a) is immediate from Theorem~\ref{T:pseudocoherent acyclicity},
while criterion (b) follows from faithfully flat descent for modules
\cite[Tag~023M]{stacks-project}.
\end{proof}

\begin{defn} \label{D:psc etale sheaf}
Suppose that every morphism between objects of $\calB$ is \'etale-pseudoflat with respect to $\calB$ (this condition will be made redundant by Theorem~\ref{T:weak flatness etale}).
We say that a sheaf of $\calO_{X_{\et}}$-modules is \emph{pseudocoherent} (resp.\ \emph{fpd})
(with respect to $\calB$)
if it is locally the sheaf associated to an \'etale-stably pseudocoherent (resp.\ \'etale-stably fpd) module (with respect to $\calB$).
\end{defn}

\begin{lemma} \label{L:refined Kiehl etale}
Let $\calB'$ be a stable basis of $X_{\et}$ such that every morphism between objects of $\calB'$ is \'etale-pseudoflat with respect to $\calB'$.
With notation as in Lemma~\ref{L:weak flatness0},
the morphism $A \to B_1 \oplus B_2$ is an effective descent morphism for \'etale-stably pseudocoherent modules with respect to $\calB'$.
\end{lemma}
\begin{proof}
We emulate the proof of Lemma~\ref{L:refined Kiehl}.
Put $X = \Spa(A,A^+)$ and let $\gothV$ be the cover $\{U_1, U_2\}$ of $X_{\et}$ as in 
Lemma~\ref{L:weak flatness0}.
Given a descent datum for \'etale-stably pseudocoherent modules, we may associate to it a pseudocoherent sheaf $\calF$ on $X_{\et}$. Choose a surjection $\calE \to \calF$ for $\calE$ a finite free $\calO$-module and let $\calG$ be the kernel; by our choice of $\calB'$, $\calG$ is again the pseudocoherent sheaf associated to a descent datum.
By Theorem~\ref{T:pseudocoherent etale acyclicity}, for $j=1,2$,
$H^0(U_{j,\et},\calF) = H^0(U_j, \calF)$ and $H^i(U_j, \calF) = 0$ for all $i>0$.
Since $\gothV$ is already a covering in the analytic topology, we deduce that
$\calF$ is acyclic on $X_{\et}$ and $H^0(X, \calF) = H^0(X_{\et}, \calF)$.
In particular, by Lemma~\ref{L:refined Kiehl}, 
\begin{equation} \label{eq:sequence of psc sheaves etale}
0 \to \calG(X) \to \calE(X) \to \calF(X) \to 0
\end{equation}
is an exact sequence of stably pseudocoherent $A$-modules; it thus remains to check that $\calF(X)$ is \'etale-stably pseudocoherent. To check this, let $(A,A^+) \to (C,C^+)$ be any morphism in $\calB'$ and put $Y = \Spa(C,C^+)$; by the previous argument, $\calF(Y)$ is a stably pseudocoherent $C$-module, so it is enough to check that the natural map $\calF(X)\otimes_A C \to \calF(Y)$ is an isomorphism. To this end, note that in the diagram
\[
\xymatrix{
 & \calG(X) \otimes_A C \ar[r] \ar[d] & \calE(X) \otimes_A C \ar[r]\ar[d] & \calF(X) \otimes_A C \ar[r] \ar[d] & 0 \\
0 \ar[r]& \calG(Y) \ar[r] & \calE(Y) \ar[r] & \calF(Y) \ar[r] & 0
}
\]
the top row is obtained from the exact sequence \eqref{eq:sequence of psc sheaves etale} by tensoring with $C$,
while the bottom row is exact by applying \eqref{eq:sequence of psc sheaves etale} with $X$ replaced by $Y$.
Since the middle column is an isomorphism, the right column is surjective; repeating the argument with a different diagram shows that the left column is also surjective, so the right column is an isomorphism.
\end{proof}

\begin{theorem} \label{T:refined Kiehl etale}
Let $\calB'$ be a stable basis of $X_{\et}$ containing $X$ such that every morphism between objects of $\calB'$ is \'etale-pseudoflat with respect to $\calB'$, e.g., the basis given in Lemma~\ref{L:etale pseudoflat}.
Then the global sections functor defines an exact equivalence of categories between pseudocoherent (resp.\ fpd) sheaves of $\calO$-modules 
on $X_{\et}$ with respect to $\calB'$ and \'etale-stably pseudocoherent (resp.\ fpd) $A$-modules with respect to $\calB'$.
\end{theorem}
\begin{proof}

On account of \cite[Theorem~8.2.22]{part1}, we need only treat the pseudocoherent case.
As in Lemma~\ref{L:etale pseudoflat},
we apply \cite[Proposition~8.2.20]{part1} to the property that for 
$\{(B,B^+) \to (C_i, C_i^+)\}$ a covering in $\calB'$,
the map $B \to \bigoplus_i C_i$ is an effective descent morphism
for \'etale-stably pseudocoherent (resp.\ \'etale-stably fpd) modules with respect to $\calB'$
and the descent functor is exact.
Of the conditions of \cite[Proposition~8.2.20]{part1},
(a) follows from Theorem~\ref{T:pseudocoherent etale acyclicity},
(b) is formal, (c) follows from Lemma~\ref{L:refined Kiehl etale},
and (d) follows from faithfully flat descent for modules
\cite[Tag~03OD]{stacks-project}.

\end{proof}

\begin{theorem} \label{T:weak flatness etale}
Let $(A, A^+) \to (B,B^+)$ be an \'etale morphism of sheafy adic Banach rings
and put $Y = \Spa(B,B^+)$. 
Let $\calB'$ be the stable basis of $X_{\et}$ given by Lemma~\ref{L:etale pseudoflat}.
\begin{enumerate}
\item[(a)]
For any \'etale-stably pseudocoherent $A$-module $P$ with respect to $\calB$ or $\calB'$,
the natural map $P \otimes_A B \to H^0(Y_{\et}, \tilde{P})$, where $\tilde{P}$ is the sheaf associated to $P$ as in Theorem \ref{T:pseudocoherent etale acyclicity}, is an isomorphism,
and $P \otimes_A B$ is complete for its natural topology.
\item[(b)]
The morphism $A \to B$ is \'etale-pseudoflat with respect to $\calB$ or $\calB'$.
\end{enumerate}
\end{theorem}
\begin{proof}
Let 
\[
0 \to M \to N \to P \to 0
\]
be an exact sequence of $A$-modules in which $N$ is finite projective and $P$ is \'etale-stably pseudocoherent with respect to $\calB'$.
By Lemma~\ref{L:pseudocoherent 2 of 3}, $M$ is a pseudocoherent $A$-module. Moreover, for any $\Spa(C, C^+)\in\calB'$, we have the exact sequence 
\[
0 \to M\otimes_AC \to N\otimes_AC \to P\otimes_AC \to 0,
\] 
yielding that $M\otimes_AC$ is complete for its natural topology. This implies that $M$ is \'etale-stably pseudocoherent with respect to $\calB'$. Let $\tilde{M}$ be the associated sheaf in $X_{\et}$ as in Theorem \ref{T:pseudocoherent etale acyclicity}.

Let $\calB''$ be the stable basis of $Y_{\et}$ given by Lemma~\ref{L:etale pseudoflat}; then for any $U \in \calB'$ such that $X \to U$ factors through $Y$, we have $\calB'|_U = \calB''|_U$
(see the proof of Lemma~\ref{L:etale pseudoflat}). Consequently, $\tilde{M}$ restricts to a pseudocoherent sheaf on $Y_{\et}$ with respect to $\calB''$.
By Theorem~\ref{T:pseudocoherent etale acyclicity} and Theorem~\ref{T:refined Kiehl etale}, we have another commutative diagram
\[
\xymatrix{
 & M \otimes_A B \ar[r] \ar[d] & N \otimes_A B \ar[r] \ar[d] & P \otimes_A B \ar[r] \ar[d] & 0 \\
0 \ar[r] & H^0(Y_{\et}, \tilde{M}) \ar[r] & H^0(Y_{\et}, \tilde{N}) \ar[r] & H^0(Y_{\et}, \tilde{P}) \ar[r] & 0
}
\]
with exact rows in which the middle vertical arrow is an isomorphism; consequently, the right vertical arrow is surjective. Since $M$ is also \'etale-stably pseudocoherent with respect to $\calB'$, we can repeat the argument using a different diagram to show that the left vertical arrow is surjective; this in turn implies that the right vertical arrow is injective. Thus the right vertical arrow (as well as the left vertical arrow) is an isomorphism. This yields that $A\to B$ is \'etale-pseudoflat with respect to $\calB'$. This proves (b). (Note that being \'etale-pseudoflat with respect to $\calB'$ implies being \'etale-pseudoflat with respect to $\calB$.) By Theorem \ref{T:refined Kiehl etale}, $P\otimes_AB\cong H^0(Y_{\et}, \tilde{P})$ is \'etale-stably pseudocoherent with respect to $\calB''$; in particular, it is complete for its natural topology.
This proves (a) for $\calB'$, and hence for $\calB$.
\end{proof}

\begin{cor} \label{C:etale stably psc by basis}
The definitions of an \'etale-stably pseudocoherent (resp.\ fpd) $A$-module, an \'etale-pseudoflat morphism,
and a pseudocoherent (resp.\ fpd) sheaf of $\calO$-modules on $X_{\et}$ do not depend on the choice of $\calB$.
\end{cor}
\begin{proof}
Let $\calB'$ be the stable basis of $X_{\et}$ given by Lemma~\ref{L:etale pseudoflat}.
By Theorem~\ref{T:weak flatness etale}(a), the definition of an \'etale-stably pseudocoherent (resp.\ fpd) $A$-module is the same with respect to $\calB$ or $\calB'$; this implies the same for the definition of an \'etale-pseudoflat morphism.
By Theorem~\ref{T:weak flatness etale}(b), we can use Definition~\ref{D:psc etale sheaf} to define a pseudocoherent (resp.\ fpd) sheaf on $X_{\et}$ with respect to $\calB$; by the previous discussion plus Theorem~\ref{T:refined Kiehl etale},
this definition is also the same with respect to $\calB$ or $\calB'$.
\end{proof}

\begin{remark} \label{R:reified pseudocoherent}
In light of Remark~\ref{R:reified pseudoflat}, 
Theorem~\ref{T:pseudocoherent acyclicity}, Theorem~\ref{T:refined Kiehl},
Theorem~\ref{T:pseudocoherent etale acyclicity}, and Theorem~\ref{T:refined Kiehl etale}
may be adapted without incident to reified adic spaces.
\end{remark}

\subsection{Quasi-Stein spaces}
\label{subsec:quasi-Stein}

In preparation for working with period sheaves, we extend the preceding results on pseudocoherent sheaves from adic affinoid spaces to certain spaces analogous to Kiehl's quasi-Stein spaces \cite{kiehl-cartan}. This amounts to introducing a (commutative) analogue of the
\emph{Fr\'echet-Stein algebras} of Schneider-Teitelbaum in the absence of noetherian hypotheses. (Some of these ideas already appear in \cite{part1}; see for example
\cite[Lemma 5.3.2]{part1}.)

\begin{remark}
Throughout \S\ref{subsec:quasi-Stein},
we work in the language of adic spaces, but everything carries over without change to the context of reified adic spaces.
\end{remark}

\begin{defn} \label{D:quasi-Stein}
A \emph{quasi-Stein space} is an adic space $X$ such that $\calO(X)$ contains a topologically nilpotent unit and $X$ can be written as the union of an ascending sequence $U_0 \subseteq U_1 \subseteq \cdots$ of adic affinoid subspaces such that for each $i \geq 0$, the map
$\calO(U_{i+1}) \to \calO(U_i)$ has dense image. Note that by quasicompactness, every adic affinoid subspace of $X$ is contained in some $U_i$; moreover,
$\calO(X) = \varprojlim_i \calO(U_i)$ is an open mapping ring.
\end{defn}

We follow the proof of \cite[Satz~2.4]{kiehl-cartan} to obtain an analogous result.

\begin{lemma} \label{L:Kiehl calc}
With notation as in Definition~\ref{D:quasi-Stein}, 
put $A = \calO(X)$ and $A_i = \calO(U_i)$, 
let $M_i$ be a Banach module over $A_i$, and let
$\varphi_i: M_{i+1} \widehat{\otimes}_{A_{i+1}} A_i \to M_i$ be a bounded surjective morphism.
Use the induced maps $M_{i+1} \to M_i$ to view the $M_i$ as an inverse system.
\begin{enumerate}
\item[(a)]
The module $M = \varprojlim_i M_i$ has dense image in each $M_i$. In particular, 
the induced map $M \widehat{\otimes}_{A} A_i \to M_i$ is surjective 
\item[(b)]
We have $R^1 \varprojlim_i M_i = 0$.
\end{enumerate}
\end{lemma}
\begin{proof}
For $i=0,1,\dots$ in succession, choose a Banach norm $\left| \bullet \right|_i$ on $M_i$ in such a way that $\left| \varphi_i(\bv_{i+1}) \right|_i \leq \frac{1}{2} \left| \bv_{i+1} \right|_{i+1}$ for all $\bv_{i+1} \in M_{i+1}$.
Pick $\bv_i \in M_i$ and $\epsilon > 0$. For $j=1,2,\dots$, we may choose
$\bv_{i+j} \in M_{i+j}$ such that 
$\left| \bv_{i+j} - \varphi_{i+j}(\bv_{i+j+1}) \right|_{i+j} \leq \epsilon$.
For each $j$, the images of the $\bv_{i+j+k}$ in $M_{i+j}$ converge to an element
$\bw_{i+j}$ satisfying $\varphi_{i+j}(\bw_{i+j+1}) = \bw_{i+j}$
and $\left| \bv_i - \bw_i \right| \leq \epsilon$;
this proves (a).

Let $N$ be the product $M_1 \times M_2 \times \cdots$.
The maps $M_{i+1} \to M_i$ together define a map $\varphi: N \to N$;
the group $R^1 \varprojlim_i M_i$ may be identified with the cokernel of $1-\varphi$ on $N$. To show that this vanishes, let $N'$ be the completed direct sum of the $M_i$ for the supremum norm; then $1-\varphi$ restricts to an endomorphism of $N'$, whose cokernel surjects onto the cokernel of $1-\varphi$ on $N$ because $M_{i+1}\to M_{i}$ has dense image. However, $N'$ is complete and $\left| \varphi(\bv) \right| \leq \frac{1}{2} \left| \bv \right|$ for all $\bv \in N'$, so $1-\varphi$ admits an inverse on $N'$;
this proves (b). (Compare \cite[Remarques~0.13.2.4]{ega3-1}.)
\end{proof}

\begin{cor} \label{C:quasi-Stein tensor}
With notation as in Lemma~\ref{L:Kiehl calc},
put $M = \varprojlim_i M_i$.
Suppose that for all $i$, $M_i$ is a stably pseudocoherent $A_i$-module and
$\varphi_i$ is an isomorphism. Then the map $M \otimes_A A_i \to M_i$ is an isomorphism.
\end{cor}
\begin{proof}
By Lemma~\ref{L:Kiehl calc}, there exist finitely many elements of $M$ which generate $M_i$ over $A_i$. That is, there exists a morphism $F \to M$ of $A$-modules with $F$ finite free such that the induced map $F \otimes_A A_i \to M_i$ is surjective.
For $j \geq i$, put $F_j = F \otimes_A A_j$ and $N_j = \ker(F_j \to M_j)$. 
By Corollary~\ref{C:base extension of psc},
the induced maps $N_j \otimes_{A_j} A_i \to N_i$ are isomorphisms.
Put $N = \varprojlim_i N_i$; we may apply Lemma~\ref{L:Kiehl calc} again to deduce that 
$N \otimes_A A_i \to N_i$ is surjective and $R^1 \varprojlim_i N_i = 0$. Since $F \cong \varinjlim_i F_i$, we deduce that $F \to M$ is surjective. To summarize, in the diagram
\[
\xymatrix{
N \otimes_A A_i \ar[r] \ar[d] & F \otimes_A A_i \ar[r] \ar[d] & M \otimes_A A_i \ar[r] \ar[d] & 0\\
N_i \ar[r] & F_i \ar[r] & M_i \ar[r] & 0,
}
\]
the rows are exact, the first vertical arrow is surjective, and the second vertical arrow is an isomorphism. By the five lemma, the third vertical arrow is bijective, as desired.
\end{proof}

\begin{theorem} \label{T:Kiehl quasi-Stein}
Let $X$ be a quasi-Stein space. Let $\calF$ be a pseudocoherent sheaf of $\calO$-modules on $X$.
\begin{enumerate}
\item[(a)]
For any adic affinoid subspace $U$ of $X$, $\calF(X) \to \calF(U)$ has dense image.
\item[(b)]
For each $x \in X$, the image of $\calF(X)$ generates the $\calO_x$-module $\calF_x$.
\item[(c)]
The groups $H^j(X, \calF)$ vanish for all $j>0$.
\end{enumerate}
\end{theorem}
\begin{proof}
Set notation as in Definition~\ref{D:quasi-Stein}
and put $A_i = \calO(U_i)$, $M_i = \calF(U_i)$.
By Theorem~\ref{T:refined Kiehl}, 
the map $M_{i+1} \otimes_{A_{i+1}} A_i \to M_i$ is an isomorphism;
we may thus apply Lemma~\ref{L:Kiehl calc}(a) to deduce (a) and (b) (using that since $U$ is quasicompact, it is contained in some $U_i$).

Let $\gothU$ be the open covering of $X$ by the $U_i$.
By Theorem~\ref{T:pseudocoherent acyclicity},
$\calF$ is acyclic on each $U_i$. By \cite[Tag 01EW]{stacks-project} again, for $j>0$ we have 
\begin{align*}
\check{H}^j(X, \calF; \gothU) &= H^j(X, \calF) \\
\check{H}^{j}(U_i, \calF; \{U_0,\dots,U_i\}) &= H^j(U_i, \calF) = 0.
\end{align*}
In particular, for $j>1$ the restriction maps
\[
\check{H}^{j-1}(U_{i+1}, \calF; \{U_0,\dots,U_{i+1}\}) \to \check{H}^{j-1}(U_i, \calF; \{U_0,\dots,U_i\})
\]
are surjective, which implies (see \cite[Hilfssatz~2.6]{kiehl-cartan})
that
\[
\varprojlim_i \check{H}^{j}(U_i, \calF; \{U_0,\dots,U_i\}) \to \check{H}^j(X, \calF; \gothU)
\]
is bijective. This yields the vanishing of $H^j(X, \calF)$ for $j>1$; we obtain vanishing for $j=1$ by noting that $\check{H}^1(X, \calF; \gothU) = R^1 \varprojlim_i M_i$
and applying Lemma~\ref{L:Kiehl calc}(b).
This proves (c).
\end{proof}

\begin{cor} \label{C:quasi-Stein global sections}
For $X$ a quasi-Stein space,
the global sections functor from pseudocoherent sheaves of $\calO$-modules on $X$ to $\calO(X)$-modules is exact.
\end{cor}

\begin{cor} \label{C:local generators}
Let $X$ be a quasi-Stein space, let $\calF$ be a pseudocoherent sheaf of $\calO$-modules on $X$, and choose $\bv_1,\dots,\bv_n \in \calF(X)$. Then the following conditions are equivalent.
\begin{enumerate}
\item[(a)]
The elements $\bv_1,\dots,\bv_n$ generate $\calF(X)$ as a module over $\calO(X)$.
\item[(b)]
For each adic affinoid subspace $U$ of $X$, $\bv_1,\dots,\bv_n$ generate $\calF(U)$ as a module over $\calO(U)$. 
\end{enumerate}
\end{cor}

\begin{proof}
If (a) holds, then Theorem~\ref{T:Kiehl quasi-Stein} implies that for each $U$, $\bv_1,\dots,\bv_n$ generate a dense submodule of $\calF(U)$; since $\calF(U)$ is finitely generated, Corollary~\ref{C:closure finitely generated}
implies that (b) holds. Conversely, suppose that (b) holds.
Let $\calG$ be the free sheaf of $\calO$-modules on the generators $\bv_1,\dots,\bv_n$;
we then have a natural surjective morphism $\calG \to \calF$ of pseudocoherent sheaves.
By Corollary~\ref{C:quasi-Stein global sections}, $\calG(X) \to \calF(X)$ is surjective;
this proves (a).
\end{proof}

\begin{cor} \label{C:quasi-Stein finitely generated}
Let $X$ be a quasi-Stein space. Let $\calF$ be a vector bundle on $X$.
If the $\calO(X)$-module $\calF(X)$ is finitely generated, then it is projective.
(The finite generation condition is needed because of Remark~\ref{R:uniformly pseudocoherent} below.)
\end{cor}
\begin{proof}
By hypothesis, there exist a finite free sheaf $\calG$ of $\calO$-modules and a morphism
$\calG \to \calF$ such that $\calG(X) \to \calF(X)$ is surjective.
By Corollary~\ref{C:local generators}, $\calG \to \calF$ is a surjective morphism of sheaves;
by Lemma~\ref{L:pseudocoherent 2 of 3}, the kernel $\calH$ of this morphism is pseudocoherent.
Since $\calF$ is a vector bundle, for each adic affinoid subspace $U$ of $X$, 
the surjection $\calG(U) \to \calF(U)$ of $\calO(U)$-modules splits; these splittings form 
a class in $H^1(X, \calF^\dual \otimes \calH)$, which vanishes by Theorem~\ref{T:Kiehl quasi-Stein}(c).
Consequently, $\calG(X) \to \calF(X)$ splits in the category of $\calO(X)$-modules, proving the claim.
\end{proof}

\begin{cor} \label{C:quasi-Stein flat}
Let $X$ be a locally noetherian quasi-Stein space. Then for any adic affinoid subspace $U$ of $X$, the morphism $\calO(X) \to \calO(U)$ is flat.
\end{cor}
\begin{proof}
Choose $U_i$ as in Definition~\ref{D:quasi-Stein}; by Lemma~\ref{L:rational localization is flat}, it suffices to check the case where $U = U_i$ for some $i$.
Put $A = \calO(X)$, $A_i = \calO(U_i)$, and let $I$ be a finitely generated ideal of $A$.
Then $IA_i$ is a finitely generated ideal of $A_i$; by Lemma~\ref{L:rational localization is flat} again, $IA_i \otimes_{A_i} A_j \cong IA_j$. 
Let $I'$ be the inverse limit of the $IA_i$; then $I' \subseteq A$ by the left exactness of inverse limits, and evidently $I \subseteq I'$. By Corollary~\ref{C:local generators},
we must in fact have $I = I'$; hence by Corollary~\ref{C:quasi-Stein tensor}, the map $I \otimes_A A_i \to IA_i$ is an isomorphism. This implies that $A \to A_i$ is flat.
\end{proof}

\begin{remark}
In general, we would expect an analogue of
Corollary~\ref{C:quasi-Stein flat} based on Theorem~\ref{T:weak flatness}, but we have not been able to derive a general result. Instead, we will treat individual cases as they arise.
\end{remark}

\begin{cor} \label{C:quasi-Stein Prufer}
Let $X$ be a connected quasi-Stein space 
which is locally noetherian and 
regular of dimension $1$. Then $\calO(X)$ is a Pr\"ufer domain
(see Remark~\ref{R:Bezout torsion}).
\end{cor}
\begin{proof}

We must check that any finitely generated ideal $I$ of $\calO(X)$ is projective.
By Corollary~\ref{C:quasi-Stein flat}, for any adic affinoid subspace $U$ of $X$, 
the map $I \otimes_{\calO(X)} \calO(U) \to I \calO(U)$ is an isomorphism.
Consequently, there is an exact sequence
\[
0 \to \calI \to \calO \to \calO/\calI \to 0
\]
of pseudocoherent sheaves on $X$ such that for each $U$, we have
$\calI(U) = I \calO(U)$. 
By our hypotheses on $X$, $\calO(U)$ is a Dedekind domain, so $\calI(U)$ is a finite projective $\calO(U)$-module.
By Corollary~\ref{C:quasi-Stein global sections}, 
\[
0 \to \calI(X) \to \calO(X) \to \calO(X)/\calI(X) \to 0
\]
is exact. By this plus Corollary~\ref{C:local generators}, we have $\calI(X) = I$.
By Corollary~\ref{C:quasi-Stein finitely generated}, we deduce that $I$ is projective.

\end{proof}

Note that in Corollary~\ref{C:quasi-Stein global sections}, the global sections of a pseudocoherent sheaf on $X$ do not necessarily form a pseudocoherent $\calO(X)$-module, due to possible failure of finite generation
(hence the finite generation condition in Corollary~\ref{C:quasi-Stein flat}). We address this issue next.

\begin{remark} \label{R:uniformly pseudocoherent}
For $X$ a quasi-Stein space, consider the following conditions on a pseudocoherent sheaf $\calF$ on $X$.
\begin{enumerate}
\item[(a)]
The module $\calF(X)$ is finitely generated over $\calO(X)$.
\item[(b)]
There exists an integer $n$ such that for every adic affinoid subspace $U$ of $X$, 
the module $\calF(U)$ is generated over $\calO(U)$ by at most $n$ elements. Equivalently,
with notation as in Definition~\ref{D:quasi-Stein},
for all $i$, the module $\calF(U_i)$ is generated over $\calO(U_i)$ by at most $n$ elements.
\item[(c)]
There exists an integer $n$ such that for $x \in X$,
the module $\calF_x$ is generated over $\calO_x$ by at most $n$ elements.
Equivalently, for every adic affinoid subspace $U$ of $X$ (or even just the ones in some covering), the Fitting ideal $\Fitt_n(\calF(U)) \subseteq \calO(U)$ is trivial.
\end{enumerate}
It is easily seen that (a) implies (b) and that (b) implies (c).
In order to deduce the reverse implications, we need to add some geometric hypotheses.
\end{remark}

In the locally noetherian case, we have the following result. The proof is due to Bellovin \cite[Lemma~2.1.7]{bellovin} in the case of rigid analytic spaces, but generalizes without substantial changes.
\begin{prop} \label{P:noetherian quasi-Stein}
Suppose that $X$ is a quasi-Stein space which is locally noetherian of finite dimension (i.e., it can be covered by the adic or reified adic spectra of strongly noetherian adic Banach rings of bounded Krull dimension). Then
conditions (a), (b), (c) of Remark~\ref{R:uniformly pseudocoherent} are equivalent.
\end{prop}
\begin{proof}
We induct on $m = \dim(X)$. The case $m=0$ is clear, as then $X$ is a disjoint union of points $x_1, x_2, \dots$
and we have
\[
\calO(X) = \prod_{i=1}^\infty \calO(x_i), \qquad
\calF(X) = \prod_{i=1}^\infty \calF(x_i).
\]
Suppose now that $m>0$. Choose $U_i$ as in Definition~\ref{D:quasi-Stein}. We construct a sequence of closed, reduced, zero-dimensional subspaces $Y_i \subseteq U_i$ as follows: given $Y_j$ for all $j>i$, we take $Y_i$ to be $\cup_{j<i} Y_j$ together with one point 
in $U_i \setminus \cup_{j<i} U_j$ on each irreducible component of $U_i$ not meeting
$\cup_{j<i} Y_j$ (and hence not meeting $\cup_{j<i} U_j$). Put $Y = \cup_i Y_i$ and let $i: Y \to X$ be the inclusion.
Since $\dim(Y) = 0$, 
$i^* \calF$ is finitely generated over $\calO(Y)$. By
Corollary~\ref{C:quasi-Stein global sections}, $\calF(X)$ surjects onto
$(i_* i^* \calF)(X) = \calF(Y)$, so we can find a subsheaf $\calF'$ of $\calF$ generated by finitely many global sections
such that $i^* \calF'$ surjects onto $i^* \calF$. Put $\calG = \calF'/\calF$;
then $\calG$ is a pseudocoherent sheaf arising as $j_* j^* \calG$ for the closed immersion $j: Z \to X$ defined by $\ann_{\calO_X} \calG$. Since $Y$ meets each irreducible component of each $U_i$, $Z$ is locally noetherian of dimension at most $m-1$; moreover,
the fiberwise ranks of $\calG$ are bounded because $\calG$ is a quotient of $\calF$.
We may thus apply the induction hypothesis to show that $\calG(X)$ is finitely generated. By Corollary~\ref{C:quasi-Stein global sections}, we may lift a finite generating set of $\calG(X)$ to $\calF(X)$; by combining these with the generators of $\calF'$, we obtain a finite set of global sections of $\calF$ which generates $\calF_x$ for all $x \in X$.
These sections then generate $\calF(X)$ by Corollary~\ref{C:local generators}.
\end{proof}

Without noetherian hypotheses, we do not have an argument of this form; instead, we must assume the existence of a suitable covering.
\begin{defn}
For $m$ a nonnegative integer, an \emph{$m$-uniform covering} of an adic space $X$ is a finite or countable open covering $\gothV$ of $X$ such that each $V \in \gothV$ has nonempty intersection with at most $m$ other elements of $\gothV$, and each quasicompact subspace of $X$ has nonempty intersection with at most finitely many elements of $\gothV$.
\end{defn}

\begin{example}
Let $K$ be an analytic field, choose $x \in K^\times$ with $\left| x \right|  = c < 1$,
and put $c_{-1} = 0$ and $c_j = c^{-j}$ for $j \geq 0$.
Let $X$ be the affine space over $K$ with coordinates $T_1,\dots,T_n$.
As $j_1,\dots,j_n$ run over the nonnegative integers, the subspaces
\[
\{x \in X: \left| T_i \right| \in [c_{{j_i}-1},c_{j_i}] \quad (i=1,\dots,n)\}
\]
of $X$ form an $m$-uniform covering for $m = 3^n-1$.
\end{example}

\begin{lemma} \label{L:coloring}
Let $\gothV = \{V_j\}_{j \in J}$ be an $m$-uniform covering of an adic space $X$.
Then there exists a partition of $J$ into subsets $J_0,\dots,J_m$ such that for any $l \in \{0,\dots,m\}$ and any distinct $j,k \in J_l$, the spaces $V_j$ and $V_k$ are disjoint.
\end{lemma}
\begin{proof}
After enumerating the elements of $J$, we may assign them to $J_0,\dots,J_m$ inductively using the greedy algorithm: when the turn for $k \in J$ arrives, the previously assigned indices $j$ for which $V_j \cap V_k \neq \emptyset$ can cover at most $m$ of the $m+1$ subsets, so an assignment is always possible.
(This is the usual argument in graph theory that a graph of maximum valence $m$ has chromatic number at most $m+1$; countability is not essential.)
\end{proof}

The following argument is modeled on \cite[Proposition~2.1.13]{kpx}.
\begin{prop} \label{P:Stein space uniform covering}
Let $X$ be a quasi-Stein space.
Suppose that there exists an $m$-uniform covering $\gothV = \{V_j\}_{j \in J}$ for some nonnegative integer $m$.
Then conditions (a) and (b)
of Remark~\ref{R:uniformly pseudocoherent} are equivalent; moreover, it suffices to check (b) for $U \in \gothV$.
\end{prop}
\begin{proof}
By Lemma~\ref{L:coloring}, we may partition $J$ into subsets
$J_0,\dots,J_m$ such that for any $l \in \{0,\dots,m\}$ and any
distinct elements $j,k \in J_l$, we have $V_j \cap V_k = \emptyset$.
By Corollary~\ref{C:local generators}, it suffices to produce, for each $l \in \{0,\dots,m\}$, a finite set of global sections of $\calF$ generating $\calF(V_j)$ for all $j \in J_l$. 

Fix a well-ordering on the (at most countable) set $J_l$.
Choose a sequence $U_i$ as in Definition~\ref{D:quasi-Stein}, putting $U_{-1} = \emptyset$.
Fix Banach norms $\left| \bullet \right|_i$ on $\calF(U_i)$ (with no compatibility restrictions).
For $j \in J_l$, let $i(j)$ be the largest value of $i$ such that $U_i \cap V_{j} = \emptyset$, and put 
\[
W_j = U_{i(j)} \cup \bigcup_{k \in J_l, k< j} V_{k}.
\]
Then $W_j \cap V_{j} = \emptyset$, so $\calO(W_j)$ contains an element which restricts to zero on $W_j$ and to the inverse of a topologically nilpotent unit on $V_{j}$. By approximating such an element using Theorem~\ref{T:Kiehl quasi-Stein}(a), we can find  $x_j \in \calO(X)$ which restricts to a topologically nilpotent element on $W_j$ and to the inverse of a topologically nilpotent element on $V_{j}$.

To complete the proof, we will construct elements 
$\bv_{j,1},\dots,\bv_{j,n} \in \calF(X)$ for each $j \in J$ satisfying the following conditions.
\begin{itemize}
\item
For $t=1,\dots,n$, we have $\left| \bv_{j,t} \right|_{i(j)} \leq 2^{-j}$.
This will ensure that the sequence $\{\bv_{j,t}\}_j$ converges to a limit $\bv_t \in \calF(X)$.

\item
For each $k < j$, there exists an $n \times n$ matrix $A$ over $\calO(V_k)$ such that $A-1$ has topologically nilpotent entries and $\bv_{j,u} = \sum_t A_{tu} \bv_{k,u}$ in $\calF(V_{k})$. This will ensure that $\bv_{j,1},\dots,\bv_{j,n}$ generate $\calF(V_k)$ as a module over $\calO(V_k)$
if and only if $\bv_{k,1},\dots,\bv_{k,n}$ do so.

\item
The images of $\bv_{j,1},\dots,\bv_{j,n}$ in $\calF(V_{j})$ generate this module over $\calO(V_{j})$. This plus the previous condition will ensure that $\bv_1,\dots,\bv_n$ do likewise, completing the proof.
\end{itemize} 
We achieve the desired construction as follows.
For $t=1,\dots,n$, define $\bv_{j-1,t}$ to be $\bv_{k,t}$ if $j$ has a predecessor $k \in J_l$, and 0 otherwise.
Choose generators $\bw_1,\dots,\bw_n$ of $\calF(V_j)$.
Then for any sufficiently large integer $s$, we may take
\[
\bv_{j,t} = \bv_{j-1,t} + x_j^s \bw_t \qquad (t=1,\dots,n)
\]
to fulfill the desired conditions.
\end{proof}

\section{Supplemental foundations of perfectoid spaces}
\label{sec:perfectoid supplemental}

We next give some complements to \cite{part1} concerning foundations of the theory of perfectoid spaces. 

\setcounter{theorem}{0}

\begin{convention}
Hereafter, we continue to work mainly in the language of adic spaces, but everything can be carried over to the setting of reified adic spaces. In cases where this is straightforward, we simply observe in the statements of relevant theorems that ``the reified analogue also holds,'' and leave it to the reader to make the straightforward modifications required to see this. When more needs to be said, we insert a remark to do so: see Remarks~\ref{R:Fontaine primitive reified}, \ref{R:perfectoid correspondence reified}, \ref{R:tensor product reified},
\ref{R:perfectoid compatibility reified}, \ref{R:pseudocoherent reified}.
\end{convention}

\begin{remark}
We insert here a historical remark which came to our attention after \cite{part1} was completed: the definition of a perfectoid field appears already in \cite[\S 3.1]{matignon-reversat} under the terminology \emph{hyperperfect field} (\emph{corps hyperparfait}). As far as we know, this did not have any direct influence on subsequent developments.
\end{remark}

\begin{remark} \label{R:perfect uniform}
Starting in \cite[\S 3.1]{part1}, the notion of a \emph{perfect uniform} Banach algebra over $\FF_p$ is used intensively. However, at multiple points in \cite{part1} it is argued that the open mapping theorem (Theorem~\ref{T:open mapping}) implies that \emph{any} perfect Banach algebra $R$ over $\FF_p$ is uniform.
Namely, fix a topologically nilpotent unit $u$ in $R$; since the uniform condition depends only on the underlying topology
of $R$ (see condition (d) of \cite[Definition~2.8.1]{part1}), as in \cite[Remark~2.8.18]{part1} 
we may  work with a norm $\left| \bullet \right|_1$ on $R$ for which 
$\left| u \right|_1 \left| u^{-1}\right|_1= 1$ (that is, $u$ is a \emph{uniform unit} in $R$).
Now define a second norm $\left| \bullet \right|_2$ on $R$ by the formula $\left| x \right|_2 = \left| x^{1/p} \right|_1^p$; since $R$ is evidently complete with respect to $\left| \bullet \right|_2$ and $u$ is a uniform unit of the same norm with respect to both norms, the open mapping theorem for $\FF_p((u))$ implies that the two norms are equivalent, yielding that $R$ is uniform.
\end{remark}

\subsection{Generalized Witt vectors}
\label{subsec:generalized Witt}

We begin with a variant of the Witt vectors which is useful for uniformly describing several distinct situations. This type of construction appears prominently in the work of Drinfel'd \cite[\S 1]{drinfeld-witt} and Hazewinkel \cite[(18.6.13)]{hazewinkel},
and more recently in the work of Cais--Davis \cite{cais-davis} which we use as our primary reference.

\begin{hypothesis} \label{H:towers}
Throughout \S\ref{sec:perfectoid supplemental},
fix a positive integer $h$,
let $E$ be a complete discretely valued field with residue field $\FF_{p^h}$,
let $\varpi$ be a uniformizer of $E$,
and normalize the absolute value on $E$ so that $\left| \varpi \right| = p^{-1}$.
All Banach rings we consider will be $\gotho_E$-Banach algebras.
\end{hypothesis}

\begin{defn}
Let $W_\varpi$ be the functor of generalized Witt vectors associated to $\varpi$; this is the unique functor on $\gotho_E$-algebras with the following properties.
\begin{enumerate}
\item[(a)]
The underlying functor from $\gotho_E$-algebras to sets takes $R$ to $R^{\{0,1,\dots\}}$.
\item[(b)]
The \emph{ghost map} $W_{\varpi}(\bullet) \to \bullet^{\{0,1,\dots\}}$ given by
\[
(a_n)_{n=0}^\infty \mapsto \left( \sum_{m=0}^n \varpi^m a_m^{p^{h(n-m)}} \right)_{n=0}^\infty
\]
is a natural transformation of functors.
\end{enumerate}
See \cite[Proposition~2.3]{cais-davis} for references.
The functor $W_\varpi$ also comes with an endomorphism $\varphi_\varpi$, the \emph{Frobenius endomorphism}, corresponding to the left shift on ghost components.
For $r \in R$, write $[r]$ for the Witt vector $(r,0,\dots)$; the map
$[\bullet]: R \to W_\varpi(R)$ is then multiplicative.
\end{defn}

\begin{remark} \label{R:same Witt vectors}
By analogy with \cite[Definition~3.2.1]{part1},
one may define a \emph{strict $\varpi$-ring} to be a $\varpi$-torsion-free, $\varpi$-adically complete $\gotho_E$-algebra whose quotient modulo $\varpi$ is perfect.
As in \cite[Theorem~3.2.4]{part1}, one sees that reduction modulo $\varpi$ defines an equivalence of categories from strict $\varpi$-rings to perfect $\FF_{p^h}$-algebras,
with $W_\varpi$ providing a quasi-inverse.

For $S$ a perfect $\FF_{p^h}$-algebra, the universal property of the usual Witt vector ring $W(S)$ gives rise to a canonical morphism $W(S) \to W_\varpi(S)$ inducing an isomorphism $W_\varpi(S) \cong W(S) \otimes_{\Zp} \gotho_E$;
moreover, the diagram
\[
\xymatrix{
W(S) \ar^{\varphi^h}[r] \ar[d] & W(S) \ar[d] \\
W_\varpi(S) \ar^{\varphi_\varpi}[r] & W_{\varpi}(S)
}
\]
commutes.
In particular, the underlying ring of $W_\varpi(S)$ depends functorially only on $E$ and $S$, not on $\varpi$.
However, the choice of $\varpi$ does factor into the construction of the map $\theta$ arising in the perfectoid correspondence; see Definition~\ref{D:correspondence}.
\end{remark}

\begin{remark} \label{R:homogeneity}
In order to operate at the level of generality we have just introduced, we will need to 
know that a number of statements made in \cite{part1} remain valid if one replaces ordinary Witt vectors with generalized Witt vectors. Most of these statements can be accounted for using the following priniciple: the only fine structure of Witt vector arithmetic over perfect $\FF_p$-algebras used in the proofs is \cite[Remark~3.2.3]{part1}, and this carries over. More precisely, 
we have
\[
[\overline{x}] + [\overline{y}] = [\overline{x} + \overline{y}] + 
\sum_{n=1}^\infty \varpi^n [P_n(\overline{x}, \overline{y})]
\]
for some universal polynomials  $P_n$ in $\overline{x}^{p^{-n}}, \overline{y}^{p^{-n}}$
over $\FF_{p^h}$
which are homogeneous of total degree 1 (for the grading in which $\overline{x}, \overline{y}$ each have degree 1). 

An alternative approach would be to distinguish the cases where $E$ is of mixed characteristic and of equal characteristic $p$. In the former case,
$E$ is a finite extension of $\QQ_p$, so every ring we define in terms of $E$ will be a finite projective module over the corresponding ring defined in terms of $\QQ_p$, and the statements in question will then literally reduce to their counterparts in \cite{part1}. In the latter case, one must give separate arguments, but these are generally technically simpler than in the mixed characteristic case because the Witt vector construction degenerates: for any perfect $\FF_{p^h}$-algebra $R$ we simply have $W_\varpi(R) = R \llbracket \varpi \rrbracket$. We will take this approach explicitly in \S\ref{subsec:slopes}, where we rely on statements not appearing in \cite{part1}, for which we are not prepared to assert such a sweeping claim.
\end{remark}

\subsection{Primitive elements after Fontaine}
\label{subsec:Fontaine primitive}

We first introduce a generalization of the concept of a \emph{primitive element of degree $1$} given by Fontaine \cite{fontaine-bourbaki}. This will be used in \S\ref{subsec:Fontaine perfectoid} to give a corresponding generalization of perfectoid rings and spaces.

\begin{hypothesis} \label{H:primitive elements}
Throughout \S\ref{subsec:Fontaine primitive}, let $(R,R^+)$ be a perfect uniform adic Banach algebra over $\FF_{p^h}$. Let $\alpha$ denote the spectral norm on $R$.
\end{hypothesis}

\begin{lemma}
An element of $R^+$ is topologically nilpotent if and only if it is divisible in $R^+$ by some topologically nilpotent unit of $R$.
\end{lemma}
\begin{proof}
Since $R^+ \subseteq R^\circ$, any product of an element of $R^+$ with a topologically nilpotent element is again topologically nilpotent. Conversely, suppose $\overline{x} \in R^+$ is topologically nilpotent. By convention, $R$ is required to contain a topologically nilpotent unit $\overline{u}$. By \cite[Theorem~2.3.10]{part1}, we have
$\left| \overline{x} \right|_{\spect} < 1$, so for any sufficiently large $n$ we have
\[
\left| \overline{x} \overline{u}^{-p^{-n}} \right|_{\spect} \leq 
\left| \overline{x} \right|_{\spect} \left| \overline{u}^{-1} \right|_{\spect}^{p^{-n}} < 1.
\]
For such $n$, both $\overline{u}^{p^{-n}}$ and $\overline{x} \overline{u}^{-p^{-n}}$ are topologically nilpotent and hence belong to $R^+$. This proves the claim.
\end{proof}

\begin{defn} \label{D:Fontaine primitive}
An element $z = \sum_{n=0}^\infty \varpi^n [\overline{z}_n] \in W_\varpi(R^+)$ is \emph{Fontaine primitive (of degree $1$)} if $\overline{z}_0$ is topologically nilpotent and $\overline{z}_1$ is a unit in $R^+$. 
An ideal of $W_\varpi(R^+)$ is \emph{Fontaine primitive (of degree $1$)} if it is principal and some (hence any) generator is Fontaine primitive.
\end{defn}

\begin{remark}
One can also make sense of an element of $W_\varpi(R^+)$ being primitive of degree $n$ for $n$ an integer greater than 1; however, we will not use this concept. Accordingly, we abbreviate \emph{(Fontaine) primitive of degree $1$} to \emph{(Fontaine) primitive} hereafter.
\end{remark}

\begin{remark}
Crucially, in Definition~\ref{D:Fontaine primitive} we do not require $\overline{z}_0$ to be a unit of $R$. For instance, $z=\varpi$ is Fontaine primitive.
\end{remark}

\begin{defn}
Let $W_\varpi(R^+)[[R]]$ be the subring of $W_\varpi(R)$ generated by $W_\varpi(R^+)$ and $[\overline{z}]$ for all $\overline{z} \in R$. This ring can also be written as $W_\varpi(R^+)[[\overline{u}]^{-1}]$ for any single topologically nilpotent unit $\overline{u} \in R$.
\end{defn}

We have the following result, which consolidates \cite[Lemma~5.5, Lemma~5.16]{kedlaya-witt}
and \cite[Lemma~3.3.6, Lemma~3.3.9]{part1}.
\begin{lemma} \label{L:primitive division}
Suppose $z \in W_\varpi(R^+)$ is Fontaine primitive. Let $S$ be either $W_\varpi(R^+)$ or $W_\varpi(R^+)[[R]]$.
\begin{enumerate}
\item[(a)]
Every $x \in S$ is congruent modulo $z$ to some
$y = \sum_{i=0}^\infty \varpi^n [\overline{y}_n]$ with $\alpha(\overline{y}_0) \geq \alpha(\overline{y}_n)$ for all $n>0$.
\item[(b)]
For any $\epsilon >0$ and any nonnegative integer $m$, we may further ensure that
for all $\beta \in \calM(R)$,
\begin{align*}
\beta(\overline{y}_1) &\leq \max\{\epsilon, \beta(\overline{z}_0)^{p^{-1}+\cdots+p^{-m}} \beta(\overline{y}_0)\} \\
\beta(\overline{y}_n) &\leq \max\{\epsilon, \beta(\overline{y}_0)\}.  
\end{align*}
\end{enumerate}
\end{lemma}
\begin{proof}
Write $z = [\overline{z}_0] + \varpi u$ with $u \in W_\varpi(R^+)^\times$.
We define $w_0, w_1, \ldots \in S$ as follows.
Put $w_0 = 0$. Given $w_l$, put $y_l = x - w_l z$,
write $y_l = \sum_{n=0}^\infty \varpi^n [\overline{y}_{l,n}]$,
and put $w_{l+1} = w_l + u^{-1} \sum_{n=0}^\infty \varpi^n [\overline{y}_{l,n+1}]$.
By writing
\[
y_{l+1} = [\overline{y}_{l,0}] - [\overline{z}_0] u^{-1} \sum_{n=0}^\infty \varpi^n [\overline{y}_{l,n+1}]
\]
and using the homogeneity of Witt vector arithmetic (Remark~\ref{R:homogeneity}),
one sees that 
\[
\sup_n \{\beta(\overline{y}_{l+1,n})\} \leq \max\{\beta(\overline{y}_{l,0}), \beta(\overline{z}_0) \sup_n \{\beta(\overline{y}_{l,n})\}\}
\qquad (\beta \in \calM(R)).
\]
To deduce (a), note that if $y_l$ does not have the desired property for any $l$, then
\[
\sup_n \{\alpha(\overline{y}_{l+1,n})\} \leq \alpha(\overline{z}_0) \sup_n \{\alpha(\overline{y}_{l,n})\}
\]
for all $l$, so the $w_l$ converge in $W_\varpi(R^+)$ to a limit $w$ satisfying $wz = x$.

To obtain (b), choose $N$ such that
\[
\alpha(\overline{z}_0)^N \sup_n \{\alpha(\overline{y}_{0,n})\} < \epsilon.
\]
Then for all $l \leq N$, we have
\[
\alpha(\overline{y}_{l,n}) \leq \max\{\epsilon, \alpha(\overline{y}_{l,0})\}.
\]
As in \cite[Lemma~3.3.9]{part1},
we may again use the homogeneity of Witt vector arithmetic (Remark~\ref{R:homogeneity})
to see by induction on $m$ that (b) holds with $y = y_l$ for any $l \geq N+m$.
\end{proof}

\begin{remark}\label{R:Fontaine primitive reified}
In the setting of reified adic spaces, there is no separate definition of a Fontaine primitive element; we make the definition by passing from a graded adic Banach ring $(A,A^{\Gr})$ to its associated adic Banach ring $(A,A^+)$ as in Definition~\ref{D:associated graded}.
\end{remark}

\begin{defn} \label{D:Gauss norm}
For $r>0$, as in \cite[\S 3.3]{part1} we may construct a power-multiplicative norm $\lambda(\alpha^r)$ on $W_{\varpi}(R^+)$satisfying 
\begin{equation} \label{eq:Gauss norm}
\lambda(\alpha^r)\left( \sum_{n=0}^\infty \varpi^n [\overline{x}_n] \right) = \sup_n \{p^{-n} \alpha(\overline{x}_n)^r\}.
\end{equation}
For any fixed $x$, the map $r \mapsto \lambda(\alpha^r)(x)$ is log-convex
\cite[Lemma~4.2.3]{part1}, \cite[Lemma~4.4]{kedlaya-noetherian}.
\end{defn}

\subsection{Perfectoid rings after Fontaine}
\label{subsec:Fontaine perfectoid}

In \cite{part1}, there was a hard distinction made between perfectoid algebras over $\Qp$ and perfect uniform Banach algebras in characteristic $p$. The definition of perfectoid algebras given by Fontaine \cite{fontaine-bourbaki} consolidates both cases and also includes some intermediate cases in which $p$ is nonzero but not invertible.

\begin{defn} \label{D:Fontaine perfectoid}
An adic Banach ring $(A,A^+)$ (with $A$ an $\gotho_E$-algebra, as per Hypothesis~\ref{H:towers}) is \emph{Fontaine perfectoid} if $A$ is uniform
 and contains a topologically nilpotent unit $u$ such that $u^p$ divides $p$ in $A^+$ (so in particular $p$ is topologically nilpotent) and $\overline{\varphi}: A^+/(u) \to A^+/(u^p)$ is surjective; as in \cite[Remark~3.4.9]{part1}, this implies that $\overline{\varphi}: A^+/(u') \to A^+/(u')$ is surjective whenever $u' \in A^+$ is a unit in $A$ dividing $p$ in $A^+$.
(This argument does not apply to $u' = p$ if $p$ is not itself a unit in $A$, but the claim in this case will follow from Theorem~\ref{T:Fontaine perfectoid correspondence}.)
A Banach algebra $A$ over $\gotho_E$ is Fontaine perfectoid if $(A,A^\circ)$ is Fontaine perfectoid.
\end{defn}

\begin{remark} \label{R:p-th root of pu}
For any $u$ as in Definition~\ref{D:Fontaine perfectoid},
we can choose $u' \in A^+$ with $(u')^p \equiv u \pmod{u^p A^+}$;
then $u'$ is also a topologically nilpotent unit in $A$.
\end{remark}

\begin{prop} \label{P:extend perfectoid}
Let $(A,A^+)$ be an adic Banach ring. Then $(A,A^+)$ is Fontaine perfectoid if and only if $A$ is Fontaine perfectoid.
\end{prop}
\begin{proof}
It is equivalent to check that if $(A,A^{+ \prime})$ is a second adic Banach ring with the same underlying Banach ring $A$ and $(A,A^+)$ is Fontaine perfectoid, then so is $(A,A^{+ \prime})$. 
By Remark~\ref{R:p-th root of pu},
we may choose the topologically nilpotent unit $u$ so that
$\overline{\varphi}: A^+/(u^{p^2}) \to A^+/(u^{p^3})$ is surjective
and $p/u^p$ is topologically nilpotent; in this case $u^p$ divides $p$ also in $A^{+ \prime}$. We may check that $\overline{\varphi}: A^{+ \prime}/(u) \to A^{+ \prime}/(u^p)$ is surjective by imitating the proof of \cite[Proposition~3.6.2(d)]{part1}: for $x \in A^{+\prime}$, we have $u^p x \in A^+$, so we can find
$y \in A^+$ with $y^p - u^p x \in u^{p^3} A^+ \subseteq u^{p^3-1} A^{+ \prime}$.
Hence $(y/u)^p - x \in u^{p^3-p-1} A^{\prime +}$, so $y/u \in A^{+\prime}$ (since $A^{+ \prime}$ is integrally closed) and $(y/u)^p - x \in u^p A^{+ \prime}$. This proves the claim.
\end{proof}

\begin{remark} \label{R:perfect uniform2}
If $(A,A^+)$ is an adic Banach ring of characteristic $p$, then $A$ is Fontaine perfectoid if and only if it is perfect (and hence uniform, by Remark~\ref{R:perfect uniform}). If $(A,A^+)$ is an adic Banach algebra over $\Qp$, then by \cite[Proposition~3.6.2(d,e)]{part1}, $(A,A^+)$ is Fontaine perfectoid if and only if it is perfectoid in the sense of \cite[Definition~3.6.1]{part1}.
In particular, an analytic field $K$ is Fontaine perfectoid if and only if it either is of characteristic $0$ and perfectoid, or of characteristic $p$ and perfect;
this means that it is not necessary to retrace any of \cite[\S 3.5]{part1} here.
\end{remark}

\begin{remark}
Recall that for $A$ a uniform Banach ring, there are various inequivalent norms on $A$ which define the same norm topology \cite[Remark~2.8.18]{part1}. In particular, if $A$ is a uniform Banach ring in which $p$ is a topologically nilpotent unit, there is a unique choice of norm defining the same topology under which $A$ becomes a uniform Banach algebra over $\Qp$. Note that the definition of a Fontaine perfectoid ring depends only on the underlying topology, so it is invariant under changes of norm of this sort. 
\end{remark}

The perfectoid correspondence for Fontaine perfectoid rings takes the following form.
\begin{defn} \label{D:correspondence}
Let $(A,A^+)$ be a Fontaine perfectoid adic Banach ring. Let 
$R^+(A,A^+)$ be the inverse limit of $A^+/(u^p)$ under $\overline{\varphi}$
for any $u$ as in Definition~\ref{D:Fontaine perfectoid}
(the exact choice being immaterial by \cite[Lemma~3.4.2]{part1}).
By Remark~\ref{R:p-th root of pu}, the class of $u$ in $A^+/(u^p)$ lifts to a
topologically nilpotent element $\overline{u} \in R^+(A,A^+)$.
Put $R(A,A^+) = R^+(A,A^+)[\overline{u}^{-1}]$; by Lemma~\ref{L:p-th power inverse limit}, this depends only on $A$, not on $A^+$ or the choice of $\overline{u}$.
As in \cite[Proposition~3.4]{cais-davis}, we see that the natural maps
\begin{equation} \label{eq:theta identification}
\varprojlim_{\varphi_\varpi} W_\varpi(A^+)
\to \varprojlim_{\varphi_\varpi} W_\varpi(A^+/(u^p))
\to W_\varpi(\varprojlim_{\overline{\varphi}} A^+/(u^p))
= W_\varpi(R^+(A, A^+))
\end{equation}
are isomorphisms. 
Define the map $\theta:
W_{\varpi}(R^+(A, A^+)) \to A^+$  by identifying the source with
$\varprojlim_{\varphi_\varpi} W_{\varpi}(A^+)$ using \eqref{eq:theta identification}, projecting the inverse limit onto its final term $W_{\varpi}(A^+)$, then projecting onto the first (Witt or ghost) component
$A^+$. Note that $\theta$ is surjective because $\overline{\varphi}: A^+/(u) \to A^+/(u^p)$ is surjective. Put $I(A,A^+)=\ker(\theta)$.
In the case $h=1$, $E = \QQ_p$, $\varpi = p$, we recover the map $\theta$
constructed in \cite[\S 3]{part1}; see \cite[Proposition~3.10]{cais-davis}.

Let $((R,R^+),I)$ be a pair in which $(R,R^+)$ is a perfect uniform adic Banach ring and $I$ is an ideal which is Fontaine primitive. Put $A^+((R,R^+),I) = W_\varpi(R^+)/I$
and $A((R,R^+),I) = W_\varpi(R^+)[[R]]/I$; note that the latter depends only on $R$ and $I$, not on $R^+$.
\end{defn}

We have the following analogue of \cite[Proposition~3.6.25(a)]{part1}.
\begin{lemma} \label{L:p-th power inverse limit}
For $(A,A^+)$ a Fontaine perfectoid adic Banach ring,
the maps
\[
R(A,A^+) \mapsto \varprojlim_{x \mapsto x^p} A,
\qquad R^+(A,A^+) \mapsto \varprojlim_{x \mapsto x^p} A^+
\]
taking $\overline{x}$ to $(\dots, \theta([\overline{x}^{1/p}]), \theta([\overline{x}]))$
are bijections.
\end{lemma}
\begin{proof}
Put $R = R(A,A^+)$, $R^+ = R^+(A,A^+)$,
and $R' = \varprojlim_{x \mapsto x^p} A^+$; note that there is also a natural map
$R' \to \varprojlim_{\overline{\varphi}} A^+/(u^p) = R^+$.
The fact that the composition $R^+ \to R' \to R^+$ is the identity is a consequence of the construction of $\theta$.
To see that the composition $R' \to R^+ \to R'$ is the identity, note that any sequence $(\dots, x_1, x_0)$ maps to another sequence $(\dots,y_1,y_0)$ such that $x_i - y_i \in \varpi^n A^+$ for all $i$ for $n=1$, but this formally implies the same for $n+1$, hence for all $n$. This proves that the second map is a bijection; the bijectivity of the first map follows from the equalities
\[
R = \varinjlim_n \overline{u}^{-n} R^+, \qquad
\varprojlim_{x \mapsto x^p} A = \varinjlim_n \overline{u}^{-n} R'.
\]
The proof is thus complete.
\end{proof}

\begin{theorem} \label{T:Fontaine perfectoid correspondence}
The functors
\[
(A,A^+) \mapsto ((R,R^+),I)(A,A^+), \qquad
((R,R^+),I) \mapsto (A,A^+)((R,R^+),I)
\]
define equivalences of categories between the category of Fontaine perfectoid adic Banach rings and the category of pairs $((R,R^+),I)$ in which $(R,R^+)$ is a perfect uniform adic Banach ring and $I$ is an ideal of $W_\varpi(R^+)$ which is Fontaine primitive.
\end{theorem}
\begin{proof}
As this is closely analogous to \cite[Theorem~3.6.5]{part1}, we only give a quick summary of the proof. Given $(A,A^+)$, we equip $R$ with the pullback of the spectral norm on $A$ along the composition $R(A,A^+) \to \varprojlim_{x \mapsto x^p} A \to A$. Using the fact that $\varpi$ is topologically nilpotent in $A$, we see that this formula defines a power-multiplicative norm under which $R$ is complete. It follows easily that $(R,R^+)$ is perfect uniform.
We may construct an element of $I(R,R^+)$ which is Fontaine primitive as follows: for $\overline{u}$ as in Definition~\ref{D:correspondence}, 
lift $\varpi/\theta([\overline{u}])$ to some $v \in W_\varpi(R^+)$ and take $t=-1$, so that
$\varpi t + [\overline{u}]v \in \ker(\theta)$.
Using Lemma~\ref{L:primitive division}, one checks that $\varpi t + [\overline{u}]v$ is in fact a generator of $\ker(\theta)$.

Given $((R,R^+), I)$, equip $W_\varpi(R^+)[[R]]$ with the lift of the spectral norm on $R$ via the function 
$\lambda$ of Definition~\ref{D:Gauss norm}, then equip $A^+ = W_\varpi(R^+)/I$ and $A = W_\varpi(R^+)[[R]]/I$ with the quotient norm. 
By Lemma~\ref{L:primitive division}, this is again a power-multiplicative norm under which $A$ is complete; it is then apparent that $(A,A^+)$ is Fontaine perfectoid
and that the two functors compose in either direction to natural isomorphisms.
\end{proof}

\begin{remark} \label{R:same quotient}
With notation as in Definition~\ref{D:correspondence} and Theorem~\ref{T:Fontaine perfectoid correspondence}, we have $A^+/(u) \cong R^+/(\overline{u})$.
\end{remark}

\begin{remark} \label{R:perfectoid correspondence reified}
As in \cite[Theorem~11.7]{kedlaya-reified},
the reified analogue of Theorem~\ref{T:Fontaine perfectoid correspondence} puts in correspondence graded adic Banach rings $(A,A^{\Gr})$ which are Fontaine perfectoid (meaning that $A$ is Fontaine perfectoid) with pairs $((R,R^{\Gr}), I)$ in which $R$ is a perfect uniform Banach algebra over $\FF_p$ and $I$ is an ideal of $W_\varpi(R^+)$ which is Fontaine primitive. Given $(A,A^{\Gr})$, 
let $(A,A^+)$ be the associated adic Banach ring;
we may apply Theorem~\ref{T:Fontaine perfectoid correspondence} to obtain $((R,R^+),I)$,
and then take 
\[
R^{+,r} = \varprojlim 
\left(
\cdots \stackrel{x \mapsto x^p}{\longrightarrow} A^{+,r/p} \stackrel{x \mapsto x^p}{\longrightarrow} A^{+,r} \right).
\]
Given $((R,R^{\Gr}), I)$, let $(R,R^+)$ be the associated adic Banach ring;
we may apply Theorem~\ref{T:Fontaine perfectoid correspondence} to  obtain
$(A,A^+)$, and then take
$A^{+,r}$ to be the image of $W_\varpi(R^{+,r})$ in $A$.
\end{remark}

\begin{remark} \label{R:complex almost optimal}
We make explicit a variant of \cite[Remark~3.1.6]{part1} which was already used several times in \cite{part1}. Let $C^\bullet$ be a complex in which each term is a perfect uniform Banach ring and each morphism is a $\ZZ_p$-linear combination of ring homomorphisms. If such a complex is strict exact at some position, then it is almost optimal there. Likewise, if we apply Theorem~\ref{T:Fontaine perfectoid correspondence} to each ring with respect to a compatible family of Fontaine primitive ideals, by Lemma~\ref{L:primitive division} the resulting untilting complex is again almost optimal exact.
\end{remark}

We now return to the process of extending the discussion of \cite[\S 3.6]{part1} to Fontaine perfectoid rings. We continue with an analogue of \cite[Example~3.6.6]{part1}.
\begin{example} \label{exa:polynomials}
Suppose that $(A,A^+)$ and $((R,R^+),I)$ correspond as in Theorem~\ref{T:Fontaine perfectoid correspondence}. For any index set $J$ and any function $f: J \to (0, +\infty)$, let $B,S$ be the completions of
\[
A[T_j^{p^{-\infty}}: j \in J], \qquad
R[T_j^{p^{-\infty}}: j \in J]
\]
for the weighted Gauss norms under which $T_j^{p^{-n}}$ has norm $f(j)^{p^{-n}}$.
Then $B$ is perfectoid, $S$ is perfect, and $(B,B^\circ)$ corresponds to $((S,S^\circ), IW_\varpi(S^\circ))$ via Theorem~\ref{T:Fontaine perfectoid correspondence}. 

Suppose in addition that $f(j) = 1$ for all $j$. Let $B^+, S^+$ be the completions of
\[
A^+[T_j^{p^{-\infty}}: j \in J], \qquad
R^+[T_j^{p^{-\infty}}: j \in J]
\]
in $B,S$. Then  $(B,B^+)$ corresponds to $((S,S^+), IW_\varpi(S^+))$ via Theorem~\ref{T:Fontaine perfectoid correspondence}. 

Returning to the case of general $f$, suppose that $(A,A^{\Gr})$ and $((R,R^{\Gr}), I)$ correspond as in Remark~\ref{R:perfectoid correspondence reified}. View the rings
\[
A^{\Gr}[T_j^{p^{-\infty}}: j \in J], \qquad
R^{\Gr}[T_j^{p^{-\infty}}: j \in J]
\]
as graded rings with $T_j^{p^{-n}}$ placed in degree $p^{-n} f(j)$; then let
$B^{\Gr}, S^{\Gr}$ be the images of these rings in $\Gr B, \Gr S$.
Then $(B,B^{\Gr})$ corresponds to $((S, S^{\Gr}), IW_\varpi(S^+))$ via 
Remark~\ref{R:perfectoid correspondence reified}.
\end{example}

We next have an analogue of \cite[Proposition~3.6.11]{part1}, but with an extra argument added to remedy a deficiency in the proof of that statement (see Remark~\ref{R:gap in tensor product}).

\begin{theorem} \label{T:tensor product}
Let $(A,A^+) \to (B,B^+), (A,A^+) \to (C,C^+)$ be morphisms of Fontaine perfectoid adic Banach rings. Let $((R,R^+),I)$ be the pair corresponding to $(A,A^+)$ via
Theorem~\ref{T:Fontaine perfectoid correspondence} and put
$(S,S^+) = (R,R^+)(B,B^+), (T,T^+) = (R,R^+)(C,C^+)$.
Put
\[
(D,D^+) = (B,B^+) \widehat{\otimes}_{(A,A^+)} (C,C^+),
\qquad
(U,U^+) = (S,S^+) \widehat{\otimes}_{(R,R^+)} (T,T^+).
\]
\begin{enumerate}
\item[(a)]
The rings
$(D,D^+)$ and $((U,U^+), IW_\varpi(U^+))$ correspond via Theorem~\ref{T:Fontaine perfectoid correspondence}.
\item[(b)]
The tensor product norm on $D$ induced by the spectral norms on $B$ and $C$ coincides with the spectral norm on $D$.
\end{enumerate}
\end{theorem}
\begin{proof}
Construct $(\tilde{S}, \tilde{S}^+)$, $(\tilde{T}, \tilde{T}^+)$ out of $(R,R^+)$ as in 
Example~\ref{exa:polynomials} using the index sets $S^+, T^+$; we then have canonical strict surjections $(\tilde{S},\tilde{S}^+) \to (S,S^+)$,
$(\tilde{T},\tilde{T}^+) \to (T,T^+)$. Apply Theorem~\ref{T:Fontaine perfectoid correspondence} to $((\tilde{S},\tilde{S}^+), IW_\varpi(\tilde{S}^+))$,
$((\tilde{T},\tilde{T}^+), IW_\varpi(\tilde{T}^+))$ to obtain
$(\tilde{B},\tilde{B}^+), (\tilde{C}, \tilde{C}^+)$.
Put
\[
(\tilde{D},\tilde{D}^+) = (\tilde{B}, \tilde{B}^+) \widehat{\otimes}_{(A,A^+)} (\tilde{C}, \tilde{C}^+),
\qquad
(\tilde{U}, \tilde{U}^+) = (\tilde{S}, \tilde{S}^+) \widehat{\otimes}_{(R,R^+)} (\tilde{T}, \tilde{T}^+).
\]
By Example~\ref{exa:polynomials}, 
$(\tilde{D},\tilde{D}^+)$ and $((\tilde{U}, \tilde{U}^+), IW_\varpi(\tilde{U}^+))$ correspond via Theorem~\ref{T:Fontaine perfectoid correspondence}.

Let $J,K$ be the kernels of $\tilde{S} \to S, \tilde{T} \to T$. We then have a strict exact sequence
\[
(J \widehat{\otimes}_R \tilde{T}) \oplus (\tilde{S} \widehat{\otimes}_R K) \to \tilde{U} \to U \to 0.
\]
In particular, any $\overline{r} \in \ker(\tilde{U} \to U)$ can be written as a convergent sum of simple tensors $\overline{s}_i \otimes \overline{t}_i$ with $(\overline{s}_i, \overline{t}_i) \in (J, \tilde{T}) \cup (\tilde{S}, K)$.
Using Witt vector arithmetic and the fact that $J,K$ are $\varphi^{-1}$-stable ideals, we may write $[\overline{r}] \in W_\varpi(\tilde{U})$ as a convergent sum $\sum_i \varpi^{n_i} [\overline{s}'_i] \otimes [\overline{t}'_i]$ with $(\overline{s}'_i, \overline{t}'_i) \in (J, \tilde{T}) \cup (\tilde{S}, K)$. Consequently, $\theta([\overline{r}]) \in \tilde{D}$ maps to zero in $B \widehat{\otimes}_A C$;
the map $\tilde{D} \to D$ thus factors through an injective map $A((U,U^+),IW_\varpi(U^+)) \to D$. We also have a map in the other direction coming from the universal property of the completed tensor product; by examining the effects on $\theta([\overline{s}]), \theta([\overline{t}])$ for $\overline{s} \in S, \overline{t} \in T$, we see that these maps are inverses of each other. This proves (a).

To prove (b), we check that the map $D \to A((U,U^+),IW_\varpi(U^+))$ is isometric for the tensor product norm on $D$ and the spectral norm on the target. Given $x \in A((U,U^+),IW_\varpi(U^+))$ of norm $c>0$, apply Lemma~\ref{L:primitive division} to write $x = \sum_{i=0}^\infty \varpi^i \theta([\overline{x}_i])$ with $\theta([\overline{x}_i])$ having norm at most $c$.
By \cite[Remark~3.1.6(c)]{part1}, the spectral norm on $U$ coincides with the tensor product norm; hence for any $\epsilon > 0$, we may write $\overline{x}_i \in U$ as a convergent sum of simple tensors $\overline{s}_i \otimes \overline{t}_i$ each of norm at most $(1+\epsilon)c$. Using Witt vector arithmetic, we thus obtain an expression of $x$ as a convergent sum $\sum_{i=0}^\infty \varpi^{n_i} \theta([\overline{s}'_i]) \otimes \theta([\overline{t'}_i])$ in which each simple tensor $\overline{s}'_i \otimes \overline{t}'_i$ has norm at most $(1+\epsilon)c$. This proves the claim.
\end{proof}

\begin{remark} \label{R:gap in tensor product}
The proof of \cite[Proposition~3.6.11]{part1} is incomplete: it is essentially the proof
of Theorem~\ref{T:tensor product}(b), without an argument to show that
$B \widehat{\otimes}_A C \to A(S \widehat{\otimes}_R T)$ is injective. However, Theorem~\ref{T:tensor product} implies that
\cite[Proposition~3.6.11]{part1} is correct as stated.
\end{remark}

\begin{remark} \label{R:tensor product reified}
The statement of Theorem~\ref{T:tensor product} adapts to reified adic spectra, with an analogous proof, as follows.
Suppose that $(A,A^{\Gr})$ and $((R,R^{\Gr}),I)$ correspond via
Remark~\ref{R:perfectoid correspondence reified}.
For morphisms $(A,A^{\Gr}) \to (B,B^{\Gr}), (A,A^{\Gr}) \to (C,C^{\Gr})$
corresponding to $(R,R^{\Gr}) \to (S,S^{\Gr}), (R,R^{\Gr}) \to (T,T^{\Gr})$,
put
\[
(D,D^{\Gr}) = (B,B^{\Gr}) \widehat{\otimes}_{(A,A^{\Gr})} (C,C^{\Gr}),
\qquad
(U,U^{\Gr}) = (S,S^{\Gr}) \widehat{\otimes}_{(R,R^{\Gr})} (T,T^{Gr})
\]
with tensor products defined as in \cite[Definition~6.1]{kedlaya-reified}.
Then one may show as above that
$(D,D^{\Gr})$ and $((U,U^{\Gr}), IW_\varpi(U^+))$ correspond via Remark~\ref{R:perfectoid correspondence reified}.
\end{remark}

We next have an analogue of \cite[Theorem~3.6.14(a,b)]{part1}.
\begin{theorem} \label{T:Fontaine perfectoid homeomorphism}
With notation as in Theorem~\ref{T:Fontaine perfectoid correspondence}, there is a functorial identification $\Spa(A,A^+) \cong \Spa(R,R^+)$ which matches up rational subspaces on both sides. More precisely, for $\overline{f}_1, \dots, \overline{f}_n, \overline{g} \in R$, the rational subspace
\[
\{v \in \Spa(R,R^+): v(\overline{f}_i) \leq v(\overline{g}) \quad (i=1,\dots,n)\}
\]
corresponds to
\[
\{v \in \Spa(A,A^+): v(f_i) \leq v(g) \quad (i=1,\dots,n)\}
\mbox{ for } f_i = \theta([\overline{f}_i]), g_i = \theta([\overline{g}_i]),
\]
and every rational subspace of $\Spa(A,A^+)$ can be written in this form.
\end{theorem}
\begin{proof}
We first define the map on points.
In one direction, given $v \in \Spa(A,A^+)$, we obtain a valuation $w \in \Spa(R,R^+)$
by specifying that $w(\overline{x}) \leq w(\overline{y})$ if and only if
$v(\theta([\overline{x}])) \leq v(\theta([\overline{y}]))$.
By \cite[Theorem~3.5.3]{part1}, the induced map $\calM(A) \to \calM(R)$ is a bijection,
and to check that the original map is a bijection we need only consider the case where $A$ and $R$ are analytic fields. In this case, either $A = R$ and there is nothing to check, or $A$ is an analytic field over $\QQ_p$ and we may apply \cite[Theorem~3.6.14(a)]{part1} as written.

We now have a functorial identification $\Spa(A,A^+) \cong \Spa(R,R^+)$ on the level of sets. To check that this map is a homeomorphism, it suffices to check the compatibility of rational subspaces. For $\overline{f}_1,\dots,\overline{f}_n,\overline{g} \in R$,
if we put $f_i = \theta([\overline{f}_i])$, $g = \theta([\overline{g}])$, then
$f_1,\dots,f_n,g$ generate the unit ideal in $A$ if and only if 
$\overline{f}_1,\dots,\overline{f}_n,\overline{g}$ generate the unit ideal in $R$
(by applying \cite[Corollary~2.3.6]{part1} in both rings).
This means that rational subspaces in $\Spa(R,R^+)$ identify with rational subspaces of $\Spa(A,A^+)$. Conversely, if $U$ is the rational subspace of $\Spa(A,A^+)$ defined by the parameters $f_1,\dots,f_n,g$ as in \cite[(2.4.3.2)]{part1}, then 
by \cite[Remark~2.4.7]{part1} there exists $\epsilon>0$ such that any $f'_1,\dots,f'_n,g' \in A$ satisfying $\left| f_i - f_i' \right| < \epsilon, \left| g - g' \right| < \epsilon$
generate the unit ideal and also define the rational subspace $U$.
Using Lemma~\ref{L:primitive division}, we may choose $\overline{f}_1,\dots,\overline{f}_n,\overline{g} \in R$ such that
\[
\alpha(f_i - \theta([\overline{f}_i])) \leq \max\{\alpha(\varpi) \alpha(f_i), \epsilon\}, \quad
\alpha(g - \theta([\overline{g}])) \leq \max\{\alpha(\varpi) \alpha(g), \epsilon\} \qquad
(\alpha \in \calM(A)).
\]
Then $U$ corresponds to the rational subspace of $\Spa(R,R^+)$ defined by
$\overline{f}_1,\dots,\overline{f}_n,\overline{g} \in R$, as claimed.
\end{proof}

\begin{remark} \label{T:perfectoid homeomorphism reified}
For $(A, A^{\Gr})$ and $((R, R^{\Gr}), I)$ corresponding as in Remark~\ref{R:perfectoid
correspondence reified}, there is likewise a functorial identification $\Spra(A,A^{\Gr}) \cong \Spra(R, R^{\Gr})$ which matches up rational subspaces on both sides:
for $\overline{f}_1, \dots, \overline{f}_n, \overline{g} \in R$ and $q_1,\dots,q_n > 0$,
the rational subspace
\[
\{v \in \Spra(R,R^{\Gr}): v(\overline{f}_i) \leq q_i v(\overline{g}) \quad (i=1,\dots,n)\}
\]
corresponds to
\[
\{v \in \Spra(A,A^{\Gr}): v(f_i) \leq q_i v(g) \quad (i=1,\dots,n)\}
\mbox{ for } f_i = \theta([\overline{f}_i]), g_i = \theta([\overline{g}_i]).
\]
\end{remark}

We next give an omnibus result on the stability of the perfectoid property under passage along various classes of morphisms. This includes most of the remaining results of \cite[\S 3.6]{part1}, including \cite[Theorem~3.6.14(c), Theorem~3.6.17, Proposition~3.6.19, Theorem~3.6.21]{part1}. (The original proof of \cite[Proposition~3.6.19]{part1} is incomplete because it does not establish uniformity; the argument given here addresses this point.)

\begin{theorem} \label{T:Fontaine perfectoid compatibility}
Let $\calP$ be one of the following classes of morphisms of adic Banach rings:
\begin{enumerate}
\item[(i)]
rational localizations;
\item[(ii)]
morphisms with uniform target and dense image;
\item[(iii)]
surjective morphisms with uniform target;
\item[(iv)]
finite \'etale morphisms;
\item[(v)]
morphisms $(C,C^+) \to (D,D^+)$
in which $D^+$ is the completion of $C^+$ with respect to a finitely generated ideal containing an ideal of definition of $C^+$;
\item[(vi)]
morphisms $(C,C^+) \to (D,D^+)$
in which $D^+$ is the completion of an \'etale $C^+$-algebra with respect to an ideal of definition of $C^+$;
\item[(vii)]
morphisms $(C,C^+) \to (D,D^+)$
in which $D^+$ is the completion of an algebraic localization of $C^+$ with respect to an ideal of definition of $C^+$.
\end{enumerate}
Set notation as in Theorem~\ref{T:Fontaine perfectoid correspondence}.
Then the following statements hold.
\begin{enumerate}
\item[(a)]
Let $\overline{\psi}: (R,R^+) \to (S,S^+)$ be a morphism having property $\calP$. Then $(S,S^+)$ is again perfect uniform, and the corresponding morphism $(A,A^+) \to (B,B^+)$ also has property $\calP$.
\item[(b)]
Let $\psi: (A,A^+) \to (B,B^+)$ be a morphism having property $\calP$. Then $(B,B^+)$ is again Fontaine perfectoid, and the corresponding morphism $(R,R^+) \to (S,S^+)$ also has property $\calP$.
\end{enumerate}
\end{theorem}

We will prove the individual cases of Theorem~\ref{T:Fontaine perfectoid compatibility} separately, recording corollaries as we go.

\begin{proof}[Proof of Theorem~\ref{T:Fontaine perfectoid compatibility}, case (i)]
In light of Theorem~\ref{T:Fontaine perfectoid homeomorphism} and the universal property of rational localizations, it suffices to check (a). Moreover, by
\cite[Proposition~2.4.24]{part1}, we need only check the case where
$(R,R^+) \to (S,S^+)$ is part of a simple Laurent covering. This can be carried out using the explicit calculation made in \cite[Lemma~3.6.13]{part1}, which may be summarized as follows.
Let $\overline{g} \in R$ be an element defining the simple Laurent covering $U_+, U_-$ of $\Spa(R,R^+)$.
For any sufficiently large $r>0$, the ring $\tilde{R} = R\{(T/r)^{p^{-\infty}}\}$ 
(i.e., the completion of $R[T^{1/p^n}: n \geq 0]$ for the weighted Gauss norm
under which $T^{p^{-n}}$ has norm $r^{p^{-n}}$)
admits a bounded map to $R$ taking $T^{p^{-n}}$ to $\overline{g}^{p^{-n}}$. This corresponds via Theorem~\ref{T:Fontaine perfectoid homeomorphism} to the map from $\tilde{A} = A\{(T/r)^{p^{-\infty}}\}$ to $A$ taking $T^{p^{-n}}$ to $\theta([\overline{g}^{p^{-n}}])$. As in Example~\ref{exa:polynomials}, we see that the rings
\[
\tilde{A}\{U\}/(T-U) \cong A\{T^{p^{-\infty}}\}, \qquad \tilde{A}\{U\}/(TU-1) \cong A \{T^{-p^{-\infty}}, (T/r)^{p^{-\infty}}\}
\]
are perfectoid algebras corresponding to
\[
\tilde{R}\{U\}/(T-U) \cong R\{T^{p^{-\infty}}\}, \qquad \tilde{R}\{T/r,U\}/(TU-1) \cong R\{T^{-p^{-\infty}}, (T/r)^{p^{-\infty}}\}.
\]
Take the maps from $\tilde{A}$ to these two rings and then perform the base extension back to $A$; we then obtain the rational localizations of $A$ corresponding to $U_+, U_-$.
By Theorem~\ref{T:tensor product} and Proposition~\ref{P:extend perfectoid}, the targets of these maps are Fontaine perfectoids. From the universal property of rational localizations, we obtain the desired result.
\end{proof}

\begin{cor} \label{C:Fontaine perfectoid stably uniform}
Any Fontaine perfectoid ring is stably uniform, and hence sheafy by
\cite[Theorem~2.8.10]{part1}.
\end{cor}
\begin{cor} \label{C:Fontaine perfectoid etale acyclic}
Let $(A,A^+)$ be a Fontaine perfectoid adic Banach ring and put $X = \Spa(A,A^+)$.
\begin{enumerate}
\item[(a)]
We have $H^0(X, \calO) = H^0(X_{\et}, \calO) = A$.
\item[(b)]
For $i>0$, $H^i(X, \calO) = H^i(X_{\et}, \calO) = 0$.
\item[(c)]
For $i>0$, the groups $H^i(X, \calO^+)$, $H^i(X_{\et}, \calO^+)$ are almost zero (i.e., annihilated by $A^{\circ \circ}$).
\end{enumerate}
\end{cor}
\begin{proof}
To deduce (a) and (b), combine Corollary~\ref{C:Fontaine perfectoid stably uniform}
with \cite[Theorem~2.4.23, Theorem~8.2.22]{part1}.
For (c), add Remark~\ref{R:complex almost optimal}.
\end{proof}

\begin{proof}[Proof of Theorem~\ref{T:Fontaine perfectoid compatibility}, case (ii)]
In the setting of (a), the ring $S$ is reduced and admits the 
dense perfect $\FF_p$-subalgebra $\overline{\psi}(R)$, so $S$ is also perfect.
Since the set of finite sums $\sum_{i=0}^n \varpi^i [\overline{x}_i] \in W_\varpi(S)$ has dense image in $B$, it follows that $\psi(A)$ is dense in $B$.

In the setting of (b), let $\alpha,\beta,\overline{\alpha},\overline{\beta}$ be the spectral norms on $A,B,R,S$, respectively. 
Given $x \in \psi(A) \cap B^+$, choose $w \in \psi^{-1}(x)$.
By Lemma~\ref{L:primitive division}, we can find $\overline{w} \in R$ such that
\[
\gamma(w - \theta([\overline{w}])) \leq \gamma(\varpi) \max\{\gamma(w), \beta(x)\} \qquad
(\gamma \in \calM(A)).
\]
Put $y = \psi(\theta([\overline{w}^{1/p}]))$; then $\gamma(x - y^p) \leq \gamma(\varpi) \beta(x)$ for all $\gamma \in \calM(B)$, so $\beta(x-y^p) \leq \beta(\varpi) \beta(x)$.
Since $B$ is uniform and $\psi(A)$ is dense in $B$, it follows that $B$ is perfectoid.
Now given $\overline{x} \in S$, choose $w \in A$ with $\beta(\psi(w) - \theta([\overline{x}])) \leq \beta(\varpi) \overline{\beta}(x)$, then apply Lemma~\ref{L:primitive division} again to find $\overline{w} \in R$ such that
\[
\gamma(w - \theta([\overline{w}])) \leq \gamma(\varpi) \max\{\gamma(w), \overline{\beta}(\overline{x})\} \qquad
(\gamma \in \calM(A)).
\]
Put $\overline{y} = \overline{\psi}(\overline{w})$; then $\gamma(\theta([\overline{x} - [\overline{y}])) \leq \gamma(\varpi) \overline{\beta}(\overline{x})$ for
$\gamma \in \calM(B)$, so $\overline{\beta}(\overline{x} - \overline{y}) \leq \beta(\varpi) \overline{\beta}(\overline{x})$, yielding (b).
\end{proof}

\begin{cor} \label{C:nonuniform perfectoid tensor product}
Let $A \to B, A \to C$ be morphisms of Banach rings such that $B,C$ are Fontaine perfectoid. Then the uniform completion of
$B \otimes_A C$  is Fontaine perfectoid (whether or not $A$ is).
\end{cor}
\begin{proof}
By case (ii) of Theorem~\ref{T:Fontaine perfectoid compatibility},
it suffices to produce a Fontaine perfectoid mapping to the uniform completion
of $B \otimes_{A} C$ with dense image.
To achieve this, write $C$ as a quotient of $A\{C^\circ\}$ as in Example~\ref{exa:polynomials}; then the Fontaine perfectoid ring $B\{C^\circ\}$ has the desired property.
\end{proof}

\begin{cor} \label{C:perfectoid direct limit}
Let $\{A_i\}_{i \in I}$ be a filtered direct system of Fontaine perfectoid Banach rings.
Equip $\varinjlim_i A_i$ with the infimum of the spectral norms
and let $A$ be the completion of $\varinjlim_i A_i$. Then $A$ is perfectoid.
\end{cor}
\begin{proof}
Let $B$ be the set of tuples $(a_i)_i \in \prod_{i \in I} A_i$ such that for each $\epsilon > 0$,
there exist only finitely many indices $i$ for which $\left| a_i \right| <\epsilon$;
then $B$ is a Banach ring with respect to the supremum norm, and is moreover Fontaine perfectoid.
There is a natural bounded homomorphism $B \to A$ taking $(a_i)_i$ to $\sum_i a_i$;
since $A$ is uniform and this homomorphism has dense image,
Theorem~\ref{T:Fontaine perfectoid compatibility}(ii) implies that $A$ is perfectoid.
(A more direct proof is also possible.)
\end{proof}

Before continuing, we introduce the following analogue of \cite[Lemma~3.1.6(a)]{part1},
generalizing \cite[Proposition~3.6.9(c)]{part1}.
\begin{lemma} \label{L:strict to almost optimal}
Any strict homomorphism of Fontaine perfectoid Banach rings is almost optimal.
\end{lemma}
\begin{proof}
Let $\psi: A \to B$ be such a homomorphism. Let $\alpha, \beta$ be the spectral norms on $A,B$. By hypothesis, there exists $c \geq 1$ such that every $b \in \image(\psi)$ lifts to some $a \in A$ with
$\alpha(a) \leq c \beta(b)$; it suffices to confirm that the constant $c^{1/p}$ has the same property. Given $b_l \in \image(\psi)$, lift $b_l^p$ to $a_l \in A$ with
$\alpha(a_l) \leq  \beta(b_l^p)$, then apply Lemma~\ref{L:primitive division} to find $\overline{x} \in R$ with
\[
\gamma(a_l - \theta([\overline{x}])) \leq \gamma(\varpi) \max\{\gamma(a_l), \beta(b_l)^p\} \qquad (\gamma \in \calM(A)).
\]
Put $u_l = \theta([\overline{x}^{1/p}])$, $v_l = \psi(u_l)$, and $b_{l+1} = b_l - v_l$;
note that $\alpha(u_l) \leq c^{1/p} \beta(b_l)$.
For $\gamma \in \calM(B)$, if $E$ is of characteristic $0$ and $\gamma(\varpi) > 0$, then the calculation
made in the proof of \cite[Proposition~3.6.9(c)]{part1} shows that
$\gamma(b_{l+1}) \leq \gamma(\varpi)^{1/p} \beta(b_l)$. However, this is also clear if $E$ is of characteristic $p$ or $\gamma(\varpi) = 0$, because then $\calH(\gamma)$ has unique $p$-th roots; it follows that $\beta(b_{l+1}) \leq \beta(\varpi)^{1/p} \beta(b_l)$.
If we now start with $b_0 = b$, the resulting series $\sum_{l=0}^\infty u_l$ converges to a limit $a$ satisfying $\psi(a)= b$ and $\alpha(a) \leq c^{1/p}\beta(b)$.
\end{proof}

\begin{proof}[Proof of Theorem~\ref{T:Fontaine perfectoid compatibility}, case (iii)]
In the setting of (a), $\overline{\psi}$ is almost optimal by \cite[Remark~3.1.6(a)]{part1}.
The claim thus follows from Lemma~\ref{L:primitive division}.

In the setting of (b), $B$ is perfectoid by case (ii), and 
$\psi$ is almost optimal by Lemma~\ref{L:strict to almost optimal}.
Consequently, the induced map $A^+/(u) \to B^+/(u)$ is almost surjective,
as then is the map $R^+/(\overline{u}) \to S^+/(\overline{u})$.
It follows that $\overline{\psi}$ is surjective.
\end{proof}

\begin{proof}[Proof of Theorem~\ref{T:Fontaine perfectoid compatibility}, case (iv)]
We prove (a) by following \cite[Lemma~3.6.20]{part1}.
We may assume that $S$ is of constant rank $d$ as an $R$-module.
For $S \in \FEt(R)$, $S$ is again perfect. By choosing elements $\overline{x}_1,\dots,\overline{x}_n \in S^+$ which generate $S$ as an $R$-module and then applying $\overline{\varphi}^{-1}$ as needed, we produce a morphism 
from a finite free $R^+$-module to $S^+$ whose cokernel is killed by $\overline{u}$.
This lifts to a morphism from a finite free $W_\varpi(R^+)$-module to $W_\varpi(S^+)$ whose cokernel is killed by $[\overline{u}]$. Quotienting by $I$ and inverting $u$, we see that $B$ is a finite $A$-module.
By \cite[Proposition~2.8.4, Lemma~3.5.4]{part1}, $B$ is locally free of constant rank $d$ as an $A$-module.
Since $S \in \FEt(R)$, $S$ is also a finite projective module over $S \otimes_R S$.
By Theorem~\ref{T:tensor product}, we may repeat the preceding argument to see that
$B$ is finite projective as a module over $B \otimes_A B$. It follows that $B \in \FEt(A)$.

To prove (b), it suffices to check that the functor $\FEt(R) \to \FEt(A)$ given by (a)
is an equivalence of categories. For $R$ an analytic field, this is
\cite[Theorem~3.5.6]{part1}. The general result follows as in \cite[Theorem~3.6.21]{part1} using the henselian property of adic local rings
\cite[Lemma~2.4.17]{part1} and the fact that finite \'etale morphisms of adic Banach rings
glue over rational coverings \cite[Theorem~2.6.9]{part1}.
\end{proof}

At this point, we formally recover the almost purity theorem.
\begin{cor}[Almost purity theorem] \label{C:almost purity}
Let $A$ be a Fontaine perfectoid ring and let $B$ be a faithfully finite \'etale $A$-algebra.
Then $B^\circ$ is an almost finite projective $A^\circ$-module.
In particular, the map $\Trace: B^{\circ \circ} \to A^{\circ \circ}$ is surjective.
\end{cor}
\begin{proof}
This follows from Theorem~\ref{T:Fontaine perfectoid compatibility}(iv)
as in the proof of \cite[Theorem~5.5.9]{part1}.
\end{proof}

\begin{proof}[Proof of Theorem~\ref{T:Fontaine perfectoid compatibility}, case (v)]
The effect of the completion can be characterized alternately as replacing either
$A^+/(u)$ or $R^+/(\overline{u})$ with the completion with respect to some finitely generated ideal. Since these rings are isomorphic by Remark~\ref{R:same quotient}, both (a) and (b) are clear.
\end{proof}

\begin{proof}[Proof of Theorem~\ref{T:Fontaine perfectoid compatibility}, case (vi)]
This again follows from Remark~\ref{R:same quotient} after applying the Elkik-Arabia lifting theorem \cite{elkik}, \cite{arabia} to see that every \'etale extension of $A^+/(u)$ (resp.\ $R^+/(\overline{u})$) can be lifted to $A^+$ (resp.\ $R^+$).
\end{proof}

\begin{proof}[Proof of Theorem~\ref{T:Fontaine perfectoid compatibility}, case (vii)]
This is immediate from (i).
\end{proof}

We have the following extension of \cite[Proposition~3.6.22]{part1}, but the proof requires a slightly modified argument.
\begin{theorem} \label{T:Fontaine finite etale descent}
Let $A \to B$ be a faithfully finite \'etale morphism of Banach rings in which $p$ is topologically nilpotent. If $B$ is Fontaine perfectoid, then so is $A$.
\end{theorem}
\begin{proof}
Let $S$ be the perfect uniform Banach algebra corresponding to $B$ via Theorem~\ref{T:Fontaine perfectoid correspondence}. Using finite \'etale descent for affine schemes as in the proof of \cite[Proposition~3.6.22]{part1}, we obtain a perfect uniform Banach algebra $R$ and a faithfully finite \'etale morphism $R \to S$ with the property that the faithfully finite \'etale morphisms $B \to B \otimes_A B$ and $S \to S \otimes_R S$ coincide.

By Theorem~\ref{T:Fontaine perfectoid correspondence}, the ideal $\ker(\theta): W(S^\circ) \to B^\circ$ is principal and any generator $z = \sum_{n=0}^\infty p^n [\overline{z}_n] \in W(S^\circ)$ has the property that $z_1 = (z - [\overline{z}_0])/p$ is a unit in $W(B^\circ)$. Choose such a generator and put $z' = z/z_1
= \sum_{n=0}^\infty p^n [\overline{z}'_n]$; then for $n \geq 1$, we have $\overline{z}'_n - 1 \in S^{\circ \circ}$.

By comparing the two images of $z'$ in $W((S \otimes_R S)^\circ)$, we obtain a class in
$H^1(\Spa(R,R^{\circ})_{\et}, 1 + W(\bullet ^{\circ \circ}))$. This class can be shown to vanish using
\cite[Remark~3.1.6]{part1}; see the proof of \cite[Lemma~5.4]{kedlaya-noetherian}. We thus deduce that
$\ker(\theta)$ admits a generator $z'' \in W(R^\circ)$.

The pair $(R, z'' W(R^\circ))$ now corresponds to a perfectoid algebra $A'$ in such a way that
$R \to S$ corresponds to a faithfully finite \'etale morphism $A' \to B$
and the two maps $S \to S \otimes_R S$ corresponds to $B \to B \otimes_{A'} B \cong B \otimes_A B$.
Consequently, the equalizer of the two maps $B \to B \otimes_{A'} B$ is equal to both $A$ and $A'$,
yielding an isomorphism $A \cong A'$. We conclude that $A$ is perfectoid.
\end{proof}

We also have the following related result. This is related to the treatment 
of finite group quotients of rigid spaces  (therein applied to the geometric interpretation of overconvergent modular forms) in
\cite[\S 6.4]{chojecki-hansen-johansson}.
Using this result, Hansen \cite{hansen} has shown that if $X$ is a perfectoid space
with an action of a finite group $G$, in many cases the quotient $X/G$ exists as a perfectoid space. (A weaker result, with a more complicated proof, appears in \cite[Proposition~2.5]{shen}.)
\begin{theorem} \label{T:perfectoid group quotient}
Let $A$ be a Fontaine perfectoid ring on which the finite group $G$ acts (on the right).
Then the invariant subring $A^G$ is again Fontaine perfectoid.
\end{theorem}
\begin{proof}
Put $R = R(A)$ in the notation of Theorem~\ref{T:Fontaine perfectoid correspondence}.
Note first that $R^G$ is perfect uniform and contains a topologically nilpotent unit
$\overline{u}$ (e.g.,
the norm of any topologically nilpotent unit of $R$), so it is perfectoid.
For $n$ sufficiently large, $u = \theta([\overline{u}^{p^{-n}}])$ is a topologically nilpotent unit of $A^G$ such that $u^{p}$ divides $p$ in $A^{G \circ} = A^{\circ G}$.
It thus remains to check that Frobenius is surjective on $A^{G \circ}/(u^p)$ for some such $u$.

We first treat the case where $\#G = p^k$ for some $k$.
We may choose $u$ so that $u^{p^{k+1}}$ divides $p$ in $A^{G \circ}$.
For each $x \in A^{G \circ}$, we can find $\overline{y} \in R^{\circ}$
with $x \equiv \theta([\overline{y}]) \pmod{u^{p^{k+1}} A^{\circ}}$; for $g \in G$,
we also have $x \equiv \theta([\overline{y}^g]) \pmod{u^{p^{k+1}} A^{\circ}}$.
Put 
\[
\overline{z} = \prod_{g \in G} (\overline{y}^g)^{1/p^{k+1}} \in R^{\circ}, \qquad
z = \theta([\overline{z}]);
\]
then $x^{p^k} = \prod_{g \in G} x \equiv z^{p^{k+1}} \pmod{u^{p^{k+1}} A^{G \circ}}$, so
$x \equiv z^p \pmod{u^p A^{G \circ}}$.

We next treat the general case. Let $P$ be a $p$-Sylow subgroup of $G$
and let $S$ be a set of left coset representatives of $P$ in $G$.
For each $x \in A^{G \circ}$, by the previous paragraph we can find $y \in A^{P \circ}$
with $x \equiv y^p \pmod{u^p A^{\circ}}$; we then also have $x \equiv (y^g)^p \pmod{u^p A^{\circ}}$ for each $g \in S$. Since $[G:P]$ is coprime to $p$, we may form
\[
z = \frac{1}{[G:P]} \sum_{g \in S} y^g \in A^{G \circ}
\]
and deduce that $x \equiv z^p \pmod{u^p A^{G \circ}}$.
\end{proof}

\begin{remark} \label{R:perfectoid compatibility reified}
Components (i), (ii), (iii), (iv) of Theorem~\ref{T:Fontaine perfectoid compatibility},
and additionally Corollary~\ref{C:Fontaine perfectoid stably uniform},
may be adapted to the reified case without incident.
To adapt components (v), (vi), (vii) of Theorem~\ref{T:Fontaine perfectoid compatibility},
one may pass from $(C,C^{\Gr})$ to the associated pair $(C,C^+)$,
form the morphism $(C,C^+) \to (D,D^+)$ as before,
then choose $D^{\Gr}$ so that $D^{+,r} = C^{+,r} \widehat{\otimes}_{C^+} D^+$ for all $r>0$.
\end{remark}

One construction in \cite[\S 3.6]{part1} which does not directly adapt to Fontaine perfectoid rings is \cite[Lemma~3.6.26]{part1}. The following replacement is adapted from \cite[Lemma~10.1.5]{scholze-berkeley2}.
\begin{lemma} \label{L:perfectoid cover}
Let $R$ be a Banach ring in which $p$ is topologically nilpotent.
\begin{enumerate}
\item[(a)]
There exists a filtered family $\{R \to R_i\}_{i \in I}$ of finite \'etale morphisms
with direct limit $R_\infty$ such that every injective finite \'etale morphism $R_\infty \to S$ splits.
\item[(b)]
Let $\tilde{R}_\infty$ be the separated completion of $R_\infty$ for the spectral seminorm.
Then $\tilde{R}_\infty$ is Fontaine perfectoid.
\end{enumerate}
\end{lemma}
\begin{proof}
Part (a) is obvious from Zorn's lemma. To prove (b), let $u_0 \in R_\infty$ be any topologically nilpotent unit.
Put $P_n(T) = T^{p^n} - u_0 T - u_0$; then
\begin{align*}
\Disc(P_n(T)) &= \Res(P_n(T), P_n'(T)) \\
&= \Res(T^{p^n} - u_0 T - u_0, p^n T^{p^{n-1}} - u_0) \\
&= \Res( p^n T^{p^n-1} - u_0, (1-p^n) u_0 T - p^n u_0) \\
&= u_0^{p^n-1} (p^{np^n} - (1 - p^{n})^{p^n-1} u_0).
\end{align*}
For $n$ sufficiently large, $p^{np^n} - (1 - p^{n})^{p^n-1} u_0$ is a unit, so $\tilde{R}_\infty \to \tilde{R}_\infty[T]/(P_n(T))$ is finite \'etale; by construction, this means that $P_n(T)$ has a root $u \in \tilde{R}_\infty$.
Note that $1 + u$ is a unit in $\tilde{R}_\infty$, so $u^{p^n}$ divides $u_0$; since $p$ is topologically nilpotent in $R$, by taking $n$ large enough we may ensure that $u^p$ divides $p$ and that $p/u^p$ is also topologically nilpotent.

It remains to check that $\tilde{R}_\infty^\circ/(u) \to \tilde{R}_\infty^\circ/(u^p)$ is surjective. 
For $f \in \tilde{R}_\infty^\circ$, put $Q(T) = T^p - u^p T - f$.
As above, we compute that
\begin{align*}
\Disc(Q(T)) &= \Res(T^p - u^p T - f, p T^{p-1} - u^p) \\
&= \Res(p T^{p-1} - u^p, (1-p) u^p T - p f) \\
&= p^p f^{p-1} - (1-p)^{p-1} u^{p^2},
\end{align*}
which is a unit in $\tilde{R}_\infty$ because $p/u^p$ is topologically nilpotent.
Hence $Q(T)$ has a root $g \in \tilde{R}_\infty$, which by construction satisfies $g^p \equiv f \pmod{u^p}$.
\end{proof}

\begin{defn}
For $X$ a preadic space (or a reified preadic space), an element $Y$ of $X_{\proet}$ is a \emph{perfectoid subdomain} if it admits a pro-\'etale presentation $\varprojlim_i Y_i$ satisfying the following conditions.
\begin{enumerate}
\item[(a)]
There exists an index $j \in I$ such that the maps $Y_i \to Y_j$ are faithfully finite \'etale for all $i \geq j$ and the space $Y_j$ is a preadic affinoid space.
\item[(b)]
The separated completion of $\calO(Y)$ for the spectral seminorm is a Fontaine perfectoid ring.
\end{enumerate}
By Lemma~\ref{L:perfectoid cover}, the perfectoid subdomains form a neighborhood basis of $X_{\proet}$.
\end{defn}

\begin{remark}
We provide here a table correlating the basic results on perfectoid fields, rings, and spaces in this document with its prequel \cite{part1} and with Scholze's foundational paper \cite{scholze1}.

\begin{center}
\begin{tabular}{c|c|c|c}
Topic & In \cite{scholze1} & In \cite{part1} & Ref. herein \\
\hline
Definition of a perfectoid field & 3.1 & 3.5.1 & N/A \\ 
\hline
Tilting equivalence for fields & 3.7 & 3.5.3, 3.5.6 & N/A\\
\hline
Tilting for rings &5.2 & 3.6.5 & \ref{T:Fontaine perfectoid correspondence}\\
\hline
Compatibility with rational localization & 6.3(i,ii) & 3.6.14 & \ref{T:Fontaine perfectoid homeomorphism}, \ref{T:Fontaine perfectoid compatibility} \\
\hline
Compatibility with tensor product & 6.18 &3.6.11 & \ref{T:tensor product} \\
\hline
Almost purity & 7.9 & 3.6.21, 5.5.9 & \ref{T:Fontaine perfectoid compatibility},
\ref{C:almost purity} \\
\hline
Almost acyclicity & 6.3(iv), 7.13 & 8.3.2 & \ref{C:Fontaine perfectoid etale acyclic} \\
\end{tabular}
\end{center}

\end{remark}

\subsection{Pseudocoherent and fpd sheaves on perfectoid spaces}

Based on \S\ref{subsec:Fontaine perfectoid}, we update \cite[\S 8.3]{part1}
to include perfectoid spaces in the sense of Fontaine. We also adapt some of our previous results on vector bundles to the cases of pseudocoherent and fpd sheaves.

\begin{defn}
A \emph{Fontaine perfectoid space} is an adic space covered by the adic spectra of Fontaine perfectoid adic Banach rings. By Theorem~\ref{T:Fontaine perfectoid correspondence}, to every Fontaine perfectoid space $X$ we may functorially associate a space $X^{\flat}$ of characteristic $p$ with the same underlying topological space; for perfectoid spaces, this is the usual perfectoid correspondence.
\end{defn}

\begin{theorem} \label{T:etale topoi}
For any Fontaine perfectoid space $X$, the \'etale topoi of $X$ and $X^\flat$ are functorially homeomorphic. In particular, the categories of \'etale $\Zp$-local systems, isogeny $\Zp$-local systems, and $\Qp$-local systems on $X$ and $X^{\flat}$ are naturally equivalent.
\end{theorem}
\begin{proof}
Immediate from Theorem~\ref{T:Fontaine perfectoid compatibility}.
\end{proof}

\begin{prop} \label{P:perfectoid acyclic}
Let $(A,A^+)$ be a Fontaine perfectoid adic Banach ring and put $X = \Spa(A,A^+)$.
\begin{enumerate}
\item[(a)]
We have $H^0(X_{\proet}, \widehat{\calO}) = A$.
\item[(b)]
For $i>0$, $H^i(X_{\proet}, \widehat{\calO}) = 0$.
\item[(c)]
For $i>0$, the group $H^0(X_{\proet}, \widehat{\calO}^+)$ is annihilated by $A^{\circ \circ}$.
\end{enumerate}
\end{prop}
\begin{proof}
As in the proof of \cite[Lemma~9.2.8]{part1}, we use Corollary~\ref{C:Fontaine perfectoid etale acyclic} to reduce to checking \v{C}ech-acyclicity for a tower of faithfully finite \'etale morphisms, which follows from
Remark~\ref{R:complex almost optimal}.
\end{proof}

\begin{lemma} \label{L:perfectoid tower splitting}
Let $A$ be a Fontaine perfectoid Banach ring.  Let
$A = A_0 \to A_1 \to \cdots$ be a sequence of faithfully finite \'etale morphisms.
Let $B$ be the completed direct limit of this sequence for the spectral seminorm. Then
we have an isomorphism
\begin{equation} \label{eq:refined Kiehl splitting1}
B \cong A \oplus \widehat{\bigoplus_{m=1}^\infty} A_m/A_{m-1}
\end{equation}
of Banach modules over $A$, where each quotient $A_m/A_{m-1}$ is equipped with the quotient norm induced by the spectral norm on $A_m$ and the direct sum is equipped with the supremum norm.
In particular, the morphism $A \to B$ is pro-projective (hence $2$-pseudoflat by Lemma~\ref{L:pro-projective topologically flat}), and it splits in the category of Banach modules over $A$ (by projecting onto the first factor of \eqref{eq:refined Kiehl splitting1}).
\end{lemma}
\begin{proof}
As in the proof of \cite[Lemma~9.2.8]{part1}, this is a consequence of
Remark~\ref{R:complex almost optimal}.
\end{proof}
\begin{cor} \label{C:topologically flat}
Let $A$ be a Fontaine perfectoid Banach ring. Let $B$ be a completed direct limit of some faithfully finite \'etale $A$-algebras, each equipped with the spectral norm. Then $A \to B$ is $2$-pseudoflat.
\end{cor}
\begin{proof}
Combine Remark~\ref{R:reduce to countable} with Lemma~\ref{L:perfectoid tower splitting}.
\end{proof}

At this point, by analogy with the \'etale case one might expect a definition of a \emph{pro-\'etale-stably pseudocoherent module} or a \emph{pro-\'etale-pseudoflat morphism}. These turn out to be unnecessary; the definitions in the \'etale
case will suffice.

\begin{theorem} \label{T:compare perfectoid cohomology}
Let $X = \Spa(A,A^+)$ be a Fontaine affinoid perfectoid space.
Let $\calB$ be a stable basis of perfectoid subdomains of $X_{\proet}$ such that every morphism $(B,B^+) \to (C,C^+)$ between objects of $\calB$ is \'etale-pseudoflat (see Definition~\ref{D:proetale pseudocoherent} below).
Let $M$ be an \'etale-stably pseudocoherent $A$-module.
Let $\tilde{M} = M \otimes_A \calO$ be the associated presheaf on $X_{\proet}$.
Then for any $U = \Spa(B,B^+) \in \calB$ and any pro-\'etale covering $\gothV$ of $U$ by elements of $\calB$, 
$M \otimes_A B$ is an \'etale-stably pseudocoherent $B$-module and
\[
H^i(U, \tilde{M}) = \check{H}^i(U,\tilde{M}; \gothV) = \begin{cases} M \otimes_A B & i=0 \\ 0 & i>0. \end{cases}
\]
In particular, $\tilde{M}$ extends to a sheaf which does not change sections on $\calB$.
\end{theorem}
\begin{proof}
We first prove the claim for a specific basis. Let $\calB'$ be the collection of perfectoid subdomains of $X_{\proet}$ which are towers of faithfully finite \'etale covers over a perfectoid subdomain in $X_{\et}$. By \cite[Lemma~9.2.5]{part1}, $\calB'$
is a basis of $X_{\proet}$; by Theorem~\ref{T:weak flatness etale}(b) plus Corollary~\ref{C:topologically flat}, 
every morphism between objects of $\calB'$ is \'etale-pseudoflat.
By Theorem~\ref{T:weak flatness etale}(a), base extension along any morphism between objects of $\calB'$ preserves the category of \'etale-stably pseudocoherent modules and is exact on such modules.

To prove the claim for $\calB'$,
by analogy with the proof of \cite[Proposition~2.4.21]{part1} or \cite[Proposition~8.2.21]{part1}, we may apply \cite[Proposition~9.2.6]{part1} to reduce to checking \v{C}ech-acyclicity
for a simple Laurent covering or a tower of faithfully finite \'etale covers. 
Since $A$ is sheafy by Corollary~\ref{C:Fontaine perfectoid stably uniform},
the case of a simple Laurent covering comes down to 
Theorem~\ref{T:pseudocoherent acyclicity}.
By Remark~\ref{R:reduce to countable}, the case of a tower of faithfully finite \'etale covers reduces formally to the case of a countable such tower, for which we may apply Lemma~\ref{L:perfectoid tower splitting}.

For general $U \in \calB$, let $\calB''$ be the collection of perfectoid subdomains of $U_{\proet}$ defined by analogy with $\calB'$.
To prove the claim, it will suffice to check that $M \otimes_A B$ is an \'etale-stably pseudocoherent $B$-module:
we can then apply the previous argument to $\calB''$ to deduce the assertions about $H^i(U, \tilde{M})$, and then
apply this knowledge to all elements of $\calB$ to recover the assertions about $\check{H}^i(U, \tilde{M}; \gothV)$.

Let
\[
0 \to P \to N \to M \to 0
\]
be an exact sequence of $A$-modules in which $N$ is finite projective; 
by the previous argument, we have an exact sequence
\[
0 \to \tilde{P} \to \tilde{N} \to \tilde{M} \to 0
\]
of sheaves on $X_{\proet}$, which we may then restrict to $U_{\et}$ to obtain an exact sequence of pseudocoherent sheaves. By Theorem~\ref{T:refined Kiehl etale}, taking global sections yields an exact sequence of \'etale-stably pseudocoherent $B$-modules. As in the proof of Theorem~\ref{T:refined Kiehl etale}, we may then apply the five lemma to the diagram
\[
\xymatrix{
 & P \otimes_A B \ar[r] \ar[d] & N \otimes_A B \ar[r] \ar[d] & M \otimes_A B \ar[r] \ar[d] & 0 \\
0 \ar[r] & H^0(U_{\et}, \tilde{P}) \ar[r] & H^0(U_{\et}, \tilde{N}) \ar[r] & H^0(U_{\et}, \tilde{M}) \ar[r] & 0
}
\]
of $B$-modules with exact rows (and a corresponding diagram with $M$ replaced by $P$) to deduce that we deduce that $M \otimes_A B \to H^0(U_{\et}, \tilde{M})$ is first surjective, and then injective. In particular, $M \otimes_A B$ is an \'etale-stably pseudocoherent $B$-module, as needed.
\end{proof}

\begin{defn} \label{D:proetale pseudocoherent}
Let $\calB$ be a stable basis of perfectoid subdomains of $X_{\proet}$ such that every morphism $(B,B^+) \to (C,C^+)$ between objects of $\calB$ is \'etale-pseudoflat.
Such a basis always exists: for example, the basis $\calB'$ from the proof of Theorem~\ref{T:compare perfectoid cohomology} has this property.  (By Theorem~\ref{T:weak flatness perfectoid}, it will follow that we may take $\calB$ to consist of all perfectoid subdomains.)

We say that a sheaf of $\calO_{X_{\proet}}$-modules is \emph{pseudocoherent} (resp.\ fpd) with respect to $\calB$
if it is locally (with respect to $\calB$) the sheaf associated to an \'etale-stably pseudocoherent (resp.\ \'etale-stably fpd) module as in Theorem~\ref{T:compare perfectoid cohomology}. Note that the latter definition (for a given base ring) does not depend on the choice of $\calB$, by Corollary~\ref{C:etale stably psc by basis}.
\end{defn}

\begin{theorem} \label{T:refined Kiehl proetale}
For $\calB$ as in Definition~\ref{D:proetale pseudocoherent},
the global sections functor defines an exact equivalence of categories between pseudocoherent (resp.\ fpd) sheaves of $\calO$-modules 
on $X_{\proet}$ with respect to $\calB$ and \'etale-stably pseudocoherent (resp.\ fpd) $A$-modules. In particular, the former category does not depend on $\calB$.
\end{theorem}
\begin{proof}
As in Lemma~\ref{L:etale pseudoflat},
we apply \cite[Proposition~9.2.6]{part1} to the property that for 
$\{(B,B^+) \to (C_i, C_i^+)\}$ a covering in $\calB$,
the map $B \to \bigoplus_i C_i$ is an effective descent morphism
for \'etale-stably pseudocoherent (resp.\ \'etale-stably fpd) modules over Banach rings and the descent functor is exact.
Of the conditions of \cite[Proposition~9.2.6]{part1},
(a) follows from Theorem~\ref{T:compare perfectoid cohomology},
(b) is formal, and (c) follows from Theorem~\ref{T:refined Kiehl etale}.
To check (d), we may apply Remark~\ref{R:reduce to countable} to reduce to countable towers.
In this context, descent at the level of modules, and completeness of the resulting module for the natural topology, follow from Lemma~\ref{L:perfectoid tower splitting} plus descent for pure morphisms \cite[Tag~08XA]{stacks-project}. To see that the descended module is \'etale-stably pseudocoherent,
use the five lemma as in the proof of Theorem~\ref{T:weak flatness etale} to show that the descent of modules commutes with base extension along an \'etale morphism $(B,B^+) \to (C,C^+)$.
\end{proof}

\begin{cor} \label{C:refined Kiehl perfectoid}
For $X$ a Fontaine perfectoid space, pullback of pseudocoherent (resp.\ fpd) sheaves over the structure sheaf $\calO_{X_{\et}}$ on $X_{\et}$ to the completed structure sheaf $\widehat{\calO}_X$ on $X_{\proet}$ defines an equivalence of categories.
\end{cor}
\begin{proof}
This follows by combining Theorem~\ref{T:refined Kiehl etale} and Theorem~\ref{T:refined Kiehl proetale}.
\end{proof}

\begin{theorem} \label{T:weak flatness perfectoid}
Let $(A,A^+)$ be a Fontaine perfectoid adic Banach ring and put $X = \Spa(A,A^+)$.
Let $Y = \Spa(B,B^+)$ be a perfectoid subdomain in $X_{\proet}$.
\begin{enumerate}
\item[(a)]
For any \'etale-stably pseudocoherent $A$-module $P$, the natural map
$P \otimes_A B \to H^0(Y_{\proet}, \tilde{P})$ is an isomorphism
and $P \otimes_A B$ is complete for its natural topology.
\item[(b)]
The map
$A \to B$ is \'etale-pseudoflat.
\end{enumerate}
\end{theorem}
\begin{proof}
Part (a) is contained in Theorem~\ref{T:refined Kiehl proetale}, so we focus on (b). Let
\[
0 \to M \to N \to P \to 0
\]
be an exact sequence of $A$-modules in which $N$ is finite projective and $P$ is \'etale-stably pseudocoherent. 
For $\calB$ as in Definition~\ref{D:proetale pseudocoherent},
by Lemma~\ref{L:pseudocoherent 2 of 3},
the associated sequence
\[
0 \to \tilde{M} \to \tilde{N} \to \tilde{P} \to 0
\]
of sheaves on $X_{\proet}$ is an exact sequence of pseudocoherent sheaves with respect to $\calB$. 
By Theorem~\ref{T:refined Kiehl proetale} again, we have a commutative diagram
\[
\xymatrix{
 & M \otimes_A B \ar[r] \ar[d] & N \otimes_A B \ar[r] \ar[d] & P \otimes_A B \ar[r] \ar[d] & 0 \\
0 \ar[r] & H^0(Y_{\proet}, \tilde{M}) \ar[r] & H^0(Y_{\proet}, \tilde{N}) \ar[r] & H^0(Y_{\proet}, \tilde{P}) \ar[r] & 0
}
\]
with exact rows in which the middle vertical arrow is an isomorphism.
As in the proof of Theorem~\ref{T:weak flatness etale}, we may use the five lemma (and the existence of a corresponding diagram in which $P$ is replaced by $M$) to argue that all of the vertical arrows are first surjective, then injective.
It follows that the first row extends to a full short exact sequence; this proves (b).
\end{proof}

\begin{remark}
We do not know to what extent Theorem~\ref{T:weak flatness perfectoid} remains true if we weaken the hypotheses on $(A,A^+)$.
For example, if $(A,A^+)$ is stably uniform and $\Spa(B,B^+)$ is a perfectoid subdomain constructed using Lemma~\ref{L:perfectoid cover}, we do not know whether $A \to B$ is pseudoflat.
For a result of this form for affinoid algebras, see Lemma~\ref{L:remains torsion-free}.
\end{remark}

\begin{prop} \label{P:lift affinoid}
Let $X$ be a Fontaine perfectoid space for which $X^{\flat}$ is affinoid perfectoid.
Then $X$ is Fontaine affinoid perfectoid.
\end{prop}
\begin{proof}
By hypothesis, $X^{\flat} = \Spa(R,R^+)$ for some perfect uniform adic Banach ring $(R,R^+)$,
and there exists a rational covering $\{(R,R^+) \to (R_i, R_i^+)\}_i$ in which $\Spa(R_i,R_i^+)$ corresponds to a subspace $U_i$ of $X$ of the form $\Spa(A_i,A_i^+)$ for some Fontaine perfectoid adic Banach ring $(A_i, A_i^+)$.
As in the proof of Theorem~\ref{T:Fontaine finite etale descent}, we may choose a generator $z_i = \sum_{n=0}^\infty p^n [\overline{z}_{i,n}]$
of the kernel of $\theta: W(R_i^+) \to A_i^+$ in such a way that $1 - \overline{z}_{i,n} \in A_i^{\circ \circ}$
for all $n>0$, use these generators to define a class in $H^1(X^{\flat}, 1 + W(\bullet^{\circ \circ}))$,
then deduce from the vanishing of this class that we may choose a single $z \in W(R^+)$ generating the kernel of each $\theta: W(R_i^+) \to A_i^+$.
We then have $X \cong \Spa(A,A^+)$ for $A^+ = W(R^+)/(z)$.
\end{proof}

\begin{remark}
In the case of a perfectoid space over a perfectoid field, the statement of Proposition~\ref{P:lift affinoid} is
included in \cite[Proposition 6.17]{scholze1}.
We do not know whether the hypothesis in Proposition~\ref{P:lift affinoid} can be weakened to say that $X^{\flat}$, which is \emph{a priori} a perfectoid space, is also a preadic affinoid space; namely, this is not known to imply that $X^{\flat}$ is affinoid perfectoid unless $X^{\flat}$ is known to be sheafy
\cite[Proposition 3.1.16]{part1}.
\end{remark}

\begin{remark} \label{R:pseudocoherent reified}
A \emph{reified Fontaine perfectoid space} is an adic space covered by the reified adic spectra of Fontaine perfectoid graded adic Banach rings.
For such spaces, the analogues of
Theorem~\ref{T:etale topoi}, Proposition~\ref{P:perfectoid acyclic}, 
Theorem~\ref{T:compare perfectoid cohomology},
 Corollary~\ref{C:refined Kiehl perfectoid},
Theorem~\ref{T:weak flatness perfectoid},
and Proposition~\ref{P:lift affinoid}
 all hold, with analogous proofs.
\end{remark}

\subsection{The v-topology}

We adapt to perfectoid spaces a construction for perfect schemes originating in \cite{bhatt-scholze}.  Our treatment is closer to the case of perfect uniform adic spaces, as treated in \cite[\S 17]{scholze-berkeley2}.
\begin{defn} \label{D:total localization}
Let $(A,A^+)$ be an adic Banach ring. 
Consider the set of rational coverings $\{(A,A^+) \to (B_i,B_i^+)\}_i$
as a directed set via refinement. 
Form a directed system by associating to each covering the ring $\bigoplus_i B_i$, equipped with its spectral seminorm,
and to each refinement of  $\{(A,A^+) \to (B_i,B_i^+)\}_i$ by  $\{(A,A^+) \to (C_j,C_j^+)\}_j$ the isometric map $\bigoplus_i B_i \to \bigoplus_j C_j$ obtained by adding the maps $B_i \to C_j$ for each pair $(i,j)$ for which $\Spa(C_j, C_j^+) \subseteq \Spa(B_i, B_i^+)$ within $\Spa(A,A^+)$.
We may then take the completed direct limit to obtain a
morphism $(A,A^+) \to (A_{\loc}, A_{\loc}^+)$; we call this the \emph{total localization} of $(A,A^+)$. If $X = \Spa(A,A^+)$, we write $X_{\loc}$ for $\Spa(A_{\loc}, A_{\loc}^+)$.
Note that the connected components of $X_{\loc}$ correspond to points of $X$.
We may similarly define a \emph{strong total localization} using only strong rational coverings; its connected components correspond to points of the Gelfand spectrum of $A$.
\end{defn}

\begin{lemma} \label{L:loc acyclic}
Let $(A,A^+)$ be a Fontaine perfectoid adic Banach ring. Then $(A_{\loc}, A_{\loc}^+)$ is Fontaine perfectoid and the augmented \v{C}ech complex
\[
0 \to A \to A_{\loc} \to A_{\loc} \widehat{\otimes}_A A_{\loc} \to \cdots
\]
is almost optimal exact.
\end{lemma}
\begin{proof}
The fact that $(A_{\loc}, A_{\loc}^+)$ is Fontaine perfectoid is immediate from
Theorem~\ref{T:Fontaine perfectoid compatibility}. For each rational covering 
$\{(A,A^+) \to (B_i,B_i^+)\}_i$, the corresponding sequence with $A_{\loc}$ replaced with $\bigoplus_i B_i$ is almost optimal exact: it is obtained via Theorem~\ref{T:Fontaine perfectoid correspondence} from a corresponding sequence in characteristic $p$ which is strict exact (by \cite[Theorem~8.2.22]{part1}).
We may thus take completed direct limits to obtain another almost optimal exact sequence.
\end{proof}

\begin{lemma} \label{L:truncate flat}
For $(A,A^+)$ an adic Banach ring, $(A_{\loc}, A_{\loc}^+) \to (B,B^+)$ a morphism
of adic Banach rings, and $u \in A$ a topologically nilpotent unit, the map $A_{\loc}^+/(u) \to B^+/(u)$ is flat. If $\Spa(B,B^+) \to \Spa(A_{\loc}, A_{\loc}^+)$ is surjective, then $A_{\loc}^+/(u) \to B^+/(u)$ is also faithfully flat.
\end{lemma}
\begin{proof}
We may check this locally with respect to the connected components of
$\Spec(A_{\loc})$, which by \cite[Proposition~2.6.4]{part1} correspond to the connected components of $\Spa(A_{\loc}, A_{\loc}^+)$. The stalk of $A_{\loc}^+$ at any connected component is a valuation ring (namely the valuation ring of $\Spa(A,A^+)$ at its corresponding point); since we are considering truncations modulo $u$, we are free to replace this valuation ring with its completion. We are thus reduced to the case $A^+/(u)=K^+/(u)$ where $(K, K^+)$ is an adic field. In this case, $B^+$ is torsion-free over $K^+$ and hence flat, so the same remains true after truncating modulo $u$.
\end{proof}

\begin{defn} \label{D:v-topology}
A \emph{v-covering} (or \emph{valuation covering}) of adic spaces is a surjective morphism $Y \to X$ such that
every quasicompact open subset of $X$ is contained in the image of some quasicompact open subset of $Y$. (This condition is modeled on a similar condition appearing in the definition of the fpqc topology for schemes.) 
For example, if $X$ is an adic affinoid space, then $X_{\loc} \to X$ is a v-covering.

We define the \emph{v-topology} on the category of Fontaine perfectoid spaces by declaring a family $\{Y_i \to X\}$ to be a covering if the morphism $\sqcup_i Y_i \to X$ is a v-covering. For $X$ a Fontaine perfectoid space, let $X_{\faith}$ be the resulting site.  (This topology was called the \emph{faithful topology} in the original version of \cite{scholze-berkeley2}; the terminology here matches that of the final version.)
\end{defn}

\begin{theorem} \label{T:faithful acyclic}
For any Fontaine perfectoid space $X$,
the structure presheaf $\calO$ on $X_{\faith}$ and its sub-presheaf $\calO^+$ are both sheaves. Moreover, if $X$ is a Fontaine affinoid perfectoid space, then $\calO$ is acyclic on $X_{\faith}$ while $\calO^+$ is almost acyclic on $X_{\faith}$.
(The reified analogue also holds.)
\end{theorem}
\begin{proof}
By the usual spectral sequence argument, it suffices to verify
\v{C}ech-(almost)-acyclicity for a cofinal family of coverings. In particular, it suffices to consider coverings of the form $Y \to X_{\loc} \to X$; it in turn suffices to treat the two coverings $Y \to X_{\loc}, X_{\loc} \to X$ individually. Of these,
the former is covered by Lemma~\ref{L:truncate flat} and the latter by 
Lemma~\ref{L:loc acyclic}.
\end{proof}
\begin{cor} \label{C:faithful acyclic vector bundle}
Let $X = \Spa(A,A^+)$ be a Fontaine affinoid perfectoid space.
Let $M$ be a finite projective $A$-module. Then
\[
H^i(X_{\faith}, \tilde{M}) = \begin{cases} M & i=0 \\ 0 & i>0. \end{cases}
\]
(The reified analogue also holds.)
\end{cor}

\begin{remark}
For $X$ a Fontaine perfectoid space, any perfectoid subdomain $Y$ of $X$
may be canonically identified with a Fontaine perfectoid space, and any pro-\'etale covering then induces a v-covering. In this manner, we obtain a morphism of sites
$X_{\faith} \to X_{\proet}$.
\end{remark}

\begin{theorem} \label{T:vector bundle faithful descent}
For $X$ a Fontaine perfectoid space, the categories of $\calO_X$-vector bundles for the analytic topology, $\calO_X$-vector bundles for the \'etale topology,
$\widehat{\calO}_X$-vector bundles for the pro-\'etale topology,
and $\calO_X$-vector bundles for the v-topology are all equivalent to each other. In case $X = \Spa(A,A^+)$ is Fontaine affinoid perfectoid, these categories are also equivalent to the category of finite projective $A$-modules.
(The reified analogue also holds.)
\end{theorem}
\begin{proof}
By \cite[Theorem~8.2.22, Theorem~9.2.15]{part1} and Corollary~\ref{C:faithful acyclic vector bundle}, it suffices to check that every vector bundle for the v-topology descends to the analytic topology.
By the same argument plus \cite[Theorem~8.2.22]{part1}, we may work \'etale locally on $X$. We may thus reduce to the case where $X = \Spa(A,A^+)$ and we are given a vector bundle specified with respect to a v-covering $Y = \Spa(B,B^+) \to X$.

In the case where $(A,A^+) = (K,K^+)$ is an adic perfectoid field, 
we may reduce to the case where $(B,B^+) = (L,L^+)$ is another adic perfectoid field.
Let $V$ be a finite-dimensional $L$-vector space equipped with a descent datum;
we can then find some lattice in $V$ which admits a compatible descent datum.
Applying faithfully flat descent to the residue fields of $K$ and $L$, we may further
ensure that the original descent datum can be described by an element of
$\GL_n(L^+ \widehat{\otimes}_{K^+} L^+)$ congruent to the identity modulo topological nilpotents. By repeated application of Theorem~\ref{T:faithful acyclic}, we may modify the choice of basis to move this matrix progressively closer to the identity; this proves the claim.
(An argument using Schauder bases and Lemma~\ref{L:complete descent} is also possible in this case.)

In the general case, using the previous paragraph (and working locally),
we may again reduce to considering a finite free module $M$ over $B$ admitting a basis whose two images in $M \otimes_B (B \widehat{\otimes}_A B)$ differ by an element of $\GL_n(B^+ \widehat{\otimes}_{A^+} B^+)$ congruent to the identity modulo topological nilpotents.
The claim then follows again by repeated application of Theorem~\ref{T:faithful acyclic}.
\end{proof}

\begin{cor}
Let $X$ a Fontaine perfectoid space, let $\nu: X_{\faith} \to X_{\proet}$ be the canonical  morphism, and let $\calF$ be an $\calO_X$-vector bundle on $X_{\faith}$. Then $R^i \nu_* \calF = 0$ for all $i>0$.
\end{cor}
\begin{proof}
By Theorem~\ref{T:vector bundle faithful descent}, $\calF$ arises by pullback from
a $\widehat{\calO}_X$-vector bundle on $X_{\proet}$; we may thus realize $\calF$ locally on $X$ as a direct sum of a free bundle. This reduces the question to the case $\calF = \calO_X$, which we may check in the case that $X  = \Spa(A,A^+)$. It further suffices to verify \v{C}ech-acyclicity for any v-covering $Y = \Spa(B,B^+) \to X$, for which we may apply Theorem~\ref{T:faithful acyclic}.
\end{proof}

\begin{remark}
For $X$ a Fontaine perfectoid space, one may ask to what extent the statements of
Theorem~\ref{T:faithful acyclic} and
Theorem~\ref{T:vector bundle faithful descent} characterize  Fontaine perfectoid spaces among all uniform adic spaces on which $p$ is topologically nilpotent. We have not attempted to answer this question.
\end{remark}

\begin{remark} \label{R:vector bundle approximation}
Let $X$ be a Fontaine perfectoid space. 
Let $\nu_{\faith}: X_{\faith} \to X$ be the canonical morphism.
Let $\calF \to \calG$ be an inclusion of $\calO_X$-modules with $\calF$ locally finitely presented and $\calG$ locally finite free.
Using the flattening property of blowups \cite[Tag~0815]{stacks-project},
one can construct a v-covering $f: Y \to X$ such that $\image(f^* \calF \to f^* \calG)$ is locally finite free. However, this image does not give rise to a sheaf on $X_{\faith}$, and hence does not descend to a locally free $\calO_X$-module.
\end{remark}

\subsection{Applications of Andr\'e's lemma}

During the preparation of this paper, the work of Andr\'e \cite{andre-dsc1, andre-dsc2} and Bhatt \cite{bhatt-dsc} became available. In these papers, Hochster's direct summand conjecture in commutative algebra is resolved using perfectoid spaces. While the details of this story are generally orthogonal to our present discussion, this work includes a key lemma with surprisingly strong consequences for the basic theory of perfectoid rings; this leads to some convenient simplifications in our work.

\begin{lemma} \label{L:Andre step}
Let $A$ be a Fontaine perfectoid Banach ring corresponding to $R$ via Theorem~\ref{T:Fontaine perfectoid correspondence}.
Let $\overline{u}, \overline{v}$ be topologically nilpotent units in $R$ and put
$u = \theta([\overline{u}]), v = \theta([\overline{v}])$.
Let $g \in A^\circ$ be an arbitrary element and put
\[
B = A \langle T^{p^{-\infty}} \rangle \left\langle \frac{T-g}{u} \right\rangle.
\]
Then $B$ is perfectoid and $A^\circ/(v) \to B^\circ/(v)$ is almost faithfully flat; that is, for 
any $A^\circ/(v)$-module $M$, the groups $\Tor_i^{A^\circ/(v)}(M, B^\circ/(v))$ are almost zero $A^\circ$-modules for all $i>0$.
\end{lemma}
\begin{proof}
We follow the proof of \cite[Theorem~2.3]{bhatt-dsc}.
Note that $A \langle T^{p^{-\infty}} \rangle$ and $R \langle \overline{T}^{p^{-\infty}} \rangle$ again correspond as
in Theorem~\ref{T:Fontaine perfectoid correspondence}.
By Lemma~\ref{L:primitive division} (applied as in the proof of Theorem~\ref{T:Fontaine perfectoid homeomorphism}),
we can find $\overline{f} \in R^\circ \langle \overline{T}^{p^{-\infty}} \rangle$ such that 
$\theta([\overline{f}])$ is congruent to $T-g$ modulo some power of $v$ and
\[
B = A \langle T^{p^{-\infty}} \rangle \left\langle \frac{\theta([\overline{f}])}{u} \right\rangle;
\]
then $A \to B$ corresponds to
$R \to R \langle \overline{T}^{p^{-\infty}} \rangle \left\langle \frac{\overline{f}}{\overline{u}} \right\rangle$. In particular, $B^\circ$ is almost isomorphic to
\[
A^\circ \langle T^{p^{-\infty}},U^{p^{-\infty}} \rangle/(u^{p^{-n}} U^{p^{-n}} - 
\theta([\overline{f}^{p^{-n}}]): n=0,1,\dots),
\]
and so $B^\circ/(v)$ is almost isomorphic to
\[
\varinjlim_{n \to \infty} B_n, \qquad B_n = (A^\circ/(v)) [T^{p^{-\infty}}, U^{p^{-n}}] / (u^{ p^{-n}} U^{p^{-n}} - 
\theta([\overline{f}^{p^{-n}}]));
\]
it therefore suffices to check that for each $n$, $B_n$ is faithfully flat over $A^\circ/(v)$.

This last statement, for a given $v$, formally implies the corresponding statement for any power of $v$; consequently, at this point we are free to replace $v$ by a suitable $p$-power root \emph{depending on $n$}. In particular, we may assume that $v^p$ divides $p$; $v^{p^n}$ divides $\theta([\overline{f}])-(T-g)$;
$v^{p^n}$ divides $g - \theta([\overline{h}])$ for some $\overline{h} \in R^\circ$; and $v^{p^n}$ divides $u$. In this case, we have
\[
B_n = (A^\circ/(v)) [T^{p^{-\infty}}, U^{p^{-n}}] / (
\theta([\overline{f}^{p^{-n}}])) = (A^\circ/(v)) [T^{p^{-\infty}}, U^{p^{-n}}] / (T^{p^{-n}} -  \theta([\overline{h}^{p^{-n}}]))
\]
and this is obviously free over $A^\circ/(v)$ on the basis 
\[
\{T^i U^j: i \in \ZZ[p^{-1}] \cap [0, p^{-n}), j \in p^{-n} \ZZ_{\geq 0}
\}.
\]
This proves the desired result.
\end{proof}

\begin{theorem}[Andr\'e] \label{T:Andre flat}
Let $A$ be a Fontaine perfectoid Banach ring, choose topologically nilpotent units $u,v$ in $A$,
and let $g \in A$ be an element. 
Let $B$ be 
the completion of 
\[
\varinjlim_{\ell \to \infty} A \langle T^{p^{-\infty}} \rangle \left\langle \frac{T-g}{u^\ell} \right\rangle
\]
for the infimum of the spectral norms (which does not depend on $u$). Then $B$ is perfectoid
and $A^\circ/(v) \to B^\circ/(v)$ is almost faithfully flat.
\end{theorem}
\begin{proof}
By Corollary~\ref{C:perfectoid direct limit},
$B$ is perfectoid.
Note that the conclusion is formally independent of the choice of either $u$ or $v$; we may thus take $u$
and $v$ as in Lemma~\ref{L:Andre step}, and let $B_\ell$ denote the ring defined therein with $u$ replaced by $u^\ell$. Then $B^\circ/(v)$ is almost isomorphic to $\varinjlim B_\ell^\circ/(v)$, so the claim follows at once from
Lemma~\ref{L:Andre step}.
\end{proof}

\begin{defn}
Let $A$ be a uniform Banach ring and let $I$ be an ideal of $A$. Define the \emph{uniform closure} of $I$ as the kernel in $A$ of the spectral seminorm on $A/I$ induced by the quotient seminorm; this is an ideal of $A$ containing the closure of $I$, but is not itself guaranteed to be closed. Nonetheless, the terminology is justified in the context of perfectoid rings by Theorem~\ref{T:uniform closure} below.
\end{defn}

\begin{theorem}[Bhatt] \label{T:uniform closure}
Let $A$ be a Fontaine perfectoid Banach ring, let $I$ be an ideal of $A$, and let $J$ be the uniform closure of $I$.
Then $J$ is a closed ideal and $A/J$ is a Fontaine perfectoid ring.
\end{theorem}
\begin{proof}
Let $R$ be the ring corresponding to $A$ via Theorem~\ref{T:Fontaine perfectoid correspondence}.
Suppose first that $I$ is the closure of an ideal generated by $\theta([\overline{g}^{p^{-n}}])$ for $\overline{g}$ running over some subset of $R$ and $n$ running over all nonnegative integers. Let $S$ be the quotient of $R$ by the closed ideal generated by $\overline{g}^{p^{-n}}$ for $\overline{g}$ and $n$ as above; then $S$ is a perfectoid ring of characteristic $p$.
By case (iii) of Theorem~\ref{T:Fontaine perfectoid compatibility}(a), $A/I$ is perfectoid and the surjective homomorphism $A \to A/I$ corresponds to $R \to S$.
In particular, $A/I$ is a uniform Banach ring, from which it follows that $I=J$; in particular, $J$ is closed and everything has been verified in this case.

To treat the general case, let $B$ be the separated completion of $A/J$ with respect to the spectral seminorm; then $B$ is a uniform Banach ring
which is perfectoid by case (ii) of Theorem~\ref{T:Fontaine perfectoid compatibility}(b),
and so it will suffice to check that $A \to B$ is surjective (as then we must have $B = A/J$).
For this, it suffices to check that $A^\circ/(u) \to B^\circ/(u)$ is surjective for some topologically nilpotent unit $u$ of $A$ for which $u^p$ divides $p$ in $A^\circ$; we may in turn check this after replacing $A$ with a larger perfectoid ring $A'$ for which $A^\circ/(u) \to A^{\prime \circ}/(u)$ is almost faithfully flat, and replacing $I$
with $IA'$ (as then $B$ is replaced by $B \widehat{\otimes}_{A} A'$).
Using Theorem~\ref{T:Andre flat} repeatedly, we may reduce to the case where $I$ admits a generating set as in the previous paragraph, and then argue as above.
\end{proof}

\begin{cor}
Let $\psi: (A,A^+) \to (B,B^+)$ and $\overline{\psi}: (R,R^+) \to (S,S^+)$ be a pair of morphisms corresponding
as in Theorem~\ref{T:Fontaine perfectoid correspondence}.  Then $\psi$ is injective if and only if $\overline{\psi}$ is injective.
\end{cor}
\begin{proof}
If $\ker(\overline{\psi})$ contains a nonzero element $\overline{x}$, then $\ker(\psi)$ contains the nonzero element
$\theta([\overline{x}])$. Conversely, if $I = \ker(\psi)$ is nonzero, then the uniform closure $J$ of $I$ is nonzero,
and $A \to B$ factors through $A/J$ because $B$ is uniform; hence
$A \to A/J$ corresponds to a morphism $R \to R'$ with nonzero kernel through which $\overline{\psi}$ factors. \end{proof}

\begin{lemma}
Let $A$ be a Fontaine perfectoid Banach ring. Let $I$ be an ideal of $A$ whose uniform closure $\overline{I}$ is finitely generated.
Then $I$ is generated by an idempotent element of $A$.
\end{lemma}
\begin{proof}
By Theorem~\ref{T:uniform closure}, $\overline{I}$ is closed and $A/\overline{I}$ is Fontaine perfectoid.
By case (iii) of Theorem~\ref{T:Fontaine perfectoid compatibility}(b), $A \to A/\overline{I}$ corresponds to a surjective morphism $R \to S$ of perfect rings.
Let $J$ be the kernel of the latter; then $\overline{I}$ is the closure of the ideal generated by $\theta([\overline{x}])$ for
$\overline{x} \in J$. 
By Corollary~\ref{C:closure finitely generated}, on one hand $I = \overline{I}$, 
and on the other hand $I$ is generated by
$\theta([\overline{x}_1]), \dots, \theta([\overline{x}_n])$ for some $\overline{x}_1,\dots,\overline{x}_n \in J$.

We now argue as in \cite[Lemma~3.2.12]{part1}. By construction, there exists an $n \times n$ matrix $M$ over $A$
such that $\theta([\overline{x}_j^{1/p}]) = \sum_i M_{ij} \theta([\overline{x}_i])$.
Let $N$ be the $n \times n$ matrix over $A$ such that $N_{ii} = \theta([\overline{x}_i^{(p-1)/p}])$
and $N_{ij} = 0$ for $i \neq j$. Then
\[
0 = \sum_i (1-MN)_{ij} \theta([\overline{x}_i^{1/p}]).
\] 
For each $\alpha \in \calM(A)$, $\det(1-MN)$ evaluates to 1 or 0 in $\calH(\alpha)$ depending on whether
or not $\alpha(\theta([\overline{x}_i])) = 0$ for $i=1,\dots,n$; it is thus an idempotent element of $A$.
To check the claim, it now suffices to consider the case where $\det(1-MN) = 1$, in which case $\overline{x}_i = 0$ for $i=1,\dots,n$ and so $I=0$; and the case where $\det(1-MN) = 0$, in which case $I = A$.
\end{proof}

\begin{remark}
Although we will not need it here, we remark that Theorem~\ref{T:uniform closure} can also be used to show that the quotient of any perfectoid ring by any maximal ideal is a perfectoid field. See \cite[Corollary~2.9.14]{kedlaya-aws}.
\end{remark}

\begin{example} \label{exa:perfectoid quotient}
Let $K$ be a perfectoid field and put $A = K \langle T^{\pm p^{-\infty}} \rangle$ and $I = (T-1)$.
We claim that for each positive integer $n$, the class of $\sum_{i=0}^{p^n-1} T^{ip^{-n}}$ has quotient norm 1
and spectral seminorm $p^{-n}$. To check the second point, we may first enlarge $K$ to containing a full set of $p$-power roots of unity, in which case each point of $\Spa(A,A^+)$ is a copy of $\Spa(K,K^+)$ in which $T^{p^{-n}}$ corresponds to a (not necessarily primitive) $p^n$-th root of unity. 
By \cite[Theorem~2.3.10]{part1}, the spectral seminorm can be computed by taking the supremum over these choices,
that is, by taking the maximum of $\left| \sum_{i=0}^{p^n-1} \zeta^{i} \right|$ as $\zeta$ runs over all $p^n$-th roots of unity in $K$. The sum evaluates to $p^n$ for $\zeta=1$ and $0$ for $\zeta \neq 1$, proving the claim.

By the previous calculation, $A/I$ is not uniform; hence by Theorem~\ref{T:uniform closure},
the uniform closure $J$ of $I$ must be strictly larger than $I$.
In fact, one can say more: for $K$ containing an algebraic closure of $\QQ_p$,
an old calculation of Fresnel-de Mathan
\cite[Th\'eor\`eme~1]{fresnel-demathan} shows that the ideal $(T-1)$ is not even reduced.
\end{example}

\subsection{Seminormality}
\label{subsec:seminormality adic}

\begin{defn}
An adic Banach ring $(A,A^+)$ is \emph{seminormal} if $A$ is seminormal as a ring and uniform as a Banach ring. (Since $A^+$ is integrally closed in $A$, it too is forced to be seminormal.) This condition is not stable under rational localizations in general, in particular because the uniform condition is not \cite[Example~2.8.7]{part1};
we are thus forced to define $(A,A^+)$ to be \emph{stably seminormal} if for every rational localization $(A,A^+) \to (B,B^+)$, $(B,B^+)$ is again seminormal.
\end{defn}

\begin{prop} \label{P:affinoid seminormal}
Let $(A,A^+)$ be an adic Banach ring in which $A$ is an affinoid algebra over an analytic field $K$. 
\begin{enumerate}
\item[(a)]
If $A$ is seminormal, it is also stably seminormal.
\item[(b)]
Suppose that $K$ is of characteristic $0$, and let $L$ be an analytic field containing $K$. If $A$ is seminormal, then so is $A \widehat{\otimes}_K L$.
\end{enumerate}
\end{prop}
\begin{proof}
Let $(A,A^+) \to (B,B^+)$ be a rational localization. 
Since $B$ is an excellent ring \cite[Remark~2.5.3]{part1}, by Lemma~\ref{L:excellent seminormal} we need only check seminormality of the completed local rings of $B$ at all maximal ideals. But these rings also arise as completed local rings of $A$ at maximal ideals
\cite[Proposition~7.3.3/5]{bgr}, so (a) follows by another application of
Lemma~\ref{L:excellent seminormal}.

To check (b), for any maximal ideal $\gothp$ of $A$,
by Lemma~\ref{L:excellent seminormal} again, the completed local ring $\widehat{A}_{\gothp}$  is seminormal.
By Lemma~\ref{L:seminormal field extension}, $\widehat{A}_{\gothp} \otimes_K L$ is seminormal (here we use the characteristic 0 hypothesis).
On the other hand, by \cite[Theorem~102]{matsumura},
$K\{T_1,\dots,T_n\} \otimes_K L$ is excellent, as then are its quotient $A \otimes_K L$ and its completed localization $\widehat{A}_{\gothp} \otimes_K L$. (Here we again
use the characteristic 0 hypothesis, but this could be lifted with more work.)
By Lemma~\ref{L:seminormal field extension} again,
the completion of $\widehat{A}_{\gothp} \otimes_K L$ at the ideal generated by $\gothp$
is seminormal. This coincides with the completion of $A \widehat{\otimes}_K L$
at the ideal generated by $\gothp$, so by Lemma~\ref{L:seminormal field extension} 
yet again, the localization of $(A \widehat{\otimes}_K L)$ at the ideal generated by $\gothp$ is seminormal. This proves the claim.
\end{proof}

\begin{remark}
For affinoid algebras over an analytic field, Proposition~\ref{P:affinoid seminormal} implies that the seminormality condition is local for the analytic topology
(and even the \'etale topology by Lemma~\ref{L:seminormal etale2}), so it 
extends immediately to rigid analytic spaces. The forgetful functor from seminormal rigid analytic spaces to reduced (hence uniform) rigid analytic spaces admits a right adjoint, called \emph{seminormalization}.
\end{remark}

We have the following extension of Remark~\ref{R:perfect seminormal}.
\begin{theorem} \label{T:perfectoid seminormal}
Any Fontaine perfectoid adic Banach ring $(A,A^+)$ is seminormal.
\end{theorem}
\begin{proof}
Let $(A,A^+)$ be a Fontaine perfectoid adic Banach ring with spectral norm $\alpha$.
Choose $(y,z) \in A \times A$ such that $y^3 = z^2$. There then exists a unique
$x \in \prod_{\beta \in \calM(A)} \calH(\beta)$ such that $x^2 = y, x^3 = z$.
It suffices to check that for any $\epsilon>0$, there exists $w \in A$ such that
$\beta(w-x) < \epsilon$ for all $\beta \in \calM(A)$.
In fact, by Proposition~\ref{P:perfectoid acyclic},
$\widehat{\calO}^+$ is an almost acyclic sheaf on $\Spa(A,A^+)_{\proet}$, so it suffices to construct $w$ locally around a single point $\beta \in \calM(A)$ (comparing the local constructions form a 1-cochain with small differential, and almost acyclicity implies that this cochain is close to a 1-coboundary). Given $y,z,\beta$, it will suffice to find a rational localization $(A,A^+) \to (B,B^+)$ encircling $\beta$ and an element $w \in B$ such that $\gamma(w-x) < \epsilon$ for all $\gamma \in \calM(B)$. This is trivial to accomplish: if $\beta(y) = \beta(z) = 0$, we can choose $(B,B^+)$ so that $w=0$ works; if $\beta(y), \beta(z) \neq 0$, we can choose $(B,B^+)$ so that $y$ and $z$ are invertible in $B$, and then $w = z/y = x$ works.
\end{proof}

\begin{cor} \label{C:pro-sheafy seminormal}
Let $(A,A^+)$ be an adic Banach ring in which $p$ is topologically nilpotent.
If $H^0(\Spa(A,A^+)_{\proet}, \widehat{\calO}) = A$, then $(A,A^+)$ is seminormal.
\end{cor}
\begin{proof}
The ring $H^0(\Spa(A,A^+)_{\proet}, \widehat{\calO})$ is always uniform, so its equality with $A$ forces $A$ is uniform.
Choose $(y,z) \in A \times A$ such that $y^3 = z^2$. By Theorem~\ref{T:perfectoid seminormal}, there exists $x \in H^0(\Spa(A,A^+)_{\proet}, \widehat{\calO})$ such that $x^2 = y, x^3 = z$; by hypothesis, we then have $x \in A$ with the same effect.
\end{proof}

\begin{remark}
The converse of Corollary~\ref{C:pro-sheafy seminormal} holds for classical affinoid algebras; see 
Theorem~\ref{T:sections of completed structure sheaf}.
\end{remark}

\begin{remark} \label{R:sections of completed structure sheaf}
Let $X = \Spa(A,A^+)$ be a Fontaine perfectoid adic affinoid space. By Theorem~\ref{T:perfectoid seminormal}, $Y$ is seminormal. By Corollary~\ref{C:faithful acyclic vector bundle}, for $\nu_{\faith}: Y_{\faith} \to Y$ the canonical morphism,
the map $\calO_Y \to \nu_{\faith *} \calO_{Y_{\faith}}$ is an isomorphism.
\end{remark}

We record a variant of \cite[Proposition~2.4.24]{part1}.
\begin{lemma} \label{L:upward localization2}
Let $(A,A^+)$ be an adic Banach ring. For any $v_0 \in \Spa(A,A^+)$ and any neighborhood $U$ of $v_0$ in $\Spa(A,A^+)$, there exists a sequence of inclusions
$\Spa(A,A^+) = V_0 \supseteq \cdots \supseteq V_m$ such that $v_0 \in V_m \subseteq U$ and for $i=1,\dots,m$, the inclusion $V_{i-1} \supseteq V_i$ is represented by a rational localization $B_{i-1} \to B_i$ such that
\[
V_i = \{v \in V_{i-1}: v(f_i) \geq 1\}
\]
for some $f_i \in B_{i-1}$.
\end{lemma}
\begin{proof}
Without loss of generality, we may take $U$ to be the rational subspace defined by some parameters $f_1,\dots,f_n,g$. Choose a topologically nilpotent unit $x \in A$; then for $i=1,\dots,n$, any sufficiently large integer $m_i$ and any sufficiently large (depending on $m_i$) integer $k_i$ satisfy the following conditions.
\begin{itemize}
\item
We have $0 < v_0(x^{m_i}) \leq v_0(g)$.
\item
We have $v_0(x^{-k_i}(f_i + x^{m_i})) \geq 1$. (This is easily seen by treating the cases $v_0(f_i) = 0$, $v_0(f_i) \neq 0$ separately.)
\end{itemize}
For any such integers, the sequence
\begin{align*}
V_1 &= \{v \in V_0: v(x^{-k_1}(f_1 + x^{m_1})) \geq 1\} \\
V_2 &= \{v \in V_1: v(g / (f_1 + x^{m_1})) \geq 1 \} \\
V_3 &= \{v \in V_2: v(g / x^{m_i}) \geq 1\} \\
&\vdots \\
V_{3n-2} &= \{v \in V_{3n-3}: v(x^{-k_n}(f_n + x^{m_n})) \geq 1\} \\
V_{3n-1} &= \{v \in V_{3n-2}: v(g / (f_n + x^{m_n})) \geq 1 \} \\
V_{3n} &= \{v \in V_{3n-1}: v(g/x^{m_n}) \geq 1 \}
\end{align*}
has the desired effect.
\end{proof}

\begin{remark} \label{R:reified still Tate}
In the reified setting, one is typically less dependent on the presence of topologically nilpotent units; however, we do not know how to avoid the use of a topologically nilpotent unit
in the reified analogue of Lemma~\ref{L:upward localization2} (whose formulation we leave to the reader). At best, the given approach may handle the case where the topologically nilpotent elements generate the unit ideal, i.e., where the ring is \emph{free of trivial spectrum} \cite[Remark~2.3.9]{part1}.
\end{remark}

\begin{prop} \label{P:diamantine}
Let $(A,A^+)$ be a sheafy adic Banach ring in which $p$ is topologically nilpotent.
Suppose that $H^0(\Spa(A,A^+)_{\proet}, \widehat{\calO}) = A$
and $H^i(\Spa(A,A^+)_{\proet}, \widehat{\calO})$ is a Banach module over $A$ for all $i>0$.
Then for every rational localization $(A,A^+) \to (B,B^+)$,
we have $H^0(\Spa(B,B^+)_{\proet}, \widehat{\calO}) = B$;
in particular, by Corollary~\ref{C:pro-sheafy seminormal}, $(A,A^+)$ is stably seminormal.
\end{prop}
\begin{proof}
It suffices to exhibit a basis for the topology of $\Spa(A,A^+)$
consisting of rational subspaces $\Spa(B,B^+)$, for each of which
we have $H^0(\Spa(B,B^+)_{\proet}, \widehat{\calO}) = B$
and $H^i(\Spa(B,B^+)_{\proet}, \widehat{\calO})$ is a Banach module over $B$ for all $i>0$.
By Lemma~\ref{L:upward localization2}, this can be achieved by showing that this property
is preserved under passage from $(A,A^+)$ to a rational localization $(B,B^+)$
in which $B= A\{T\}/(1-fT)$ for some $f \in A$.

Apply Lemma~\ref{L:perfectoid cover} to construct a Banach algebra $A'$ over $A$ which is Fontaine perfectoid and is the completed direct limit of a tower of faithfully finite \'etale $A$-algebras. Let $A''$ be the uniform completion of $A' \otimes_A A'$,
which by Corollary~\ref{C:nonuniform perfectoid tensor product} is again perfectoid.
Since $H^i(\Spa(A,A^+)_{\proet}, \widehat{\calO})$ is a Banach module over $A$ for all $i>0$,
the open mapping theorem (Theorem~\ref{T:open mapping}) implies that the complex
\[
0 \to A \to A' \to A'' \to \cdots
\]
is strict. This is then manifestly true again for the complex
\[
0 \to A\{T\} \to A'\{T\} \to A''\{T\} \to \cdots
\]
and also for the complex
\[
0 \to (1-fT) A\{T\} \to (1-fT) A'\{T\} \to (1-fT) A''\{T\} \to \cdots,
\]
proving the claim.
\end{proof}

See \S\ref{subsec:ax-sen-tate} and \cite{hansen-kedlaya} for related discussions.

\section{Perfect period sheaves}
\label{sec:perfect period sheaves}

We next return briefly to the theory of perfect period sheaves, and $\varphi$-modules over such sheaves, in order to incorporate the finite generation properties discussed in \S\ref{sec:foundations adic}. In the process, we also incorporate the expanded generality
provided by \S\ref{sec:perfectoid supplemental}.

\setcounter{theorem}{0}
\begin{convention} \label{C:drop dagger}
While we otherwise retain notation from \cite{part1}, we replace the notations $\tilde{\bA}^{\dagger,r}_*, \tilde{\bB}^{\dagger,r}_*$ with the lighter notations $\tilde{\bA}^r_*, \tilde{\bB}^r_*$. This creates no ambiguity because the new notations were not assigned in \cite{part1}.
\end{convention}

\begin{hypothesis} \label{H:pseudoflat Robba}
Throughout \S\ref{sec:perfect period sheaves}, 
retain Hypothesis~\ref{H:towers}; this includes the case where $E$ is of characteristic $p$ unless otherwise specified (see especially \S\ref{subsec:slopes}).
Let $(R,R^+)$ be a perfect uniform adic Banach algebra over $\FF_{p^h}$.
Let $\alpha$ denote the spectral norm on $R$.
Let $I$ be a (not necessarily closed) subinterval of $(0, +\infty)$.
Let $X$ be a preadic space over $\gotho_E$; we will sometimes take $X = \Spa(R,R^+)$, but when this occurs we will always say so explicitly.
\end{hypothesis}

\begin{remark}
We will use reified adic spectra to obtain some useful technical simplifications in the discussion of Robba rings. For relevant discussion, see Remarks~\ref{R:Robba reified}, 
\ref{R:reified Robba stably uniform},
\ref{R:phi-gamma point Robba}.
\end{remark}

\begin{remark}
All of the results that we state for a general preadic space can be generalized to certain stacks over the category of perfectoid spaces, such as the \emph{diamonds} of \cite{scholze-berkeley2}. We defer discussion of such matters to a subsequent paper.
(We use the word ``generalized'' cautiously here, as the functor from preadic spaces to diamonds is not fully faithful, even for rigid analytic spaces; see Theorem~\ref{T:sections of completed structure sheaf}.)
\end{remark}

\subsection{Robba rings and quasi-Stein spaces}

We begin by describing how to use the framework of quasi-Stein spaces to handle extended Robba rings, thus streamlining some discussion from \cite{part1}. We begin by recalling the construction of the desired rings from \cite[Definition~5.1.1]{part1}.

\begin{defn}
Put 
\[
\calE^{\inte}_R = W_\varpi(R), \qquad
\calE_R = W_\varpi(R)[\varpi^{-1}].
\]
For $r>0$, let $\tilde{\calR}^{\inte,r}_R$ be the completion of $W_{\varpi}(R^+)[[R]]$ with respect to the norm $\lambda(\alpha^r)$ defined in \eqref{eq:Gauss norm},
and put $\tilde{\calR}^{\bd,r}_R = \tilde{\calR}^{\inte,r}_R[\varpi^{-1}]$.
Write $\tilde{\calR}^I_R$ for the Fr\'echet completion of $W_{\varpi}(R^+)[[R]][\varpi^{-1}]$ with respect to the norms
$\lambda(\alpha^t)$ for all $t \in I$. For $I = [s,r]$ and $\varpi = p$, this agrees with
\cite[Definition~5.1.1]{part1}; for $I = (0, r]$ and $\varpi = p$, we get the ring $\tilde{\calR}^r_R$ of
\cite[Definition~5.1.1]{part1}; for $I = (0, +\infty)$ and $\varpi = p$, we get the ring $\tilde{\calR}^\infty_R$ of \cite[Definition~5.1.1]{part1}.

Let $\tilde{\calR}^{I,+}_R$ be the completion in $\tilde{\calR}^I_R$ of the subring generated by 
\[
\{x \in \tilde{\calR}^I_R: \lambda(\alpha^t)(x) < 1 \quad (t \in I)\}
\cup \{[\overline{y}]: \overline{y} \in R^+\}.
\]
For $I = [s,r]$ and $\varpi = p$, this agrees with \cite[Definition~5.3.10]{part1}.
\end{defn}

\begin{remark} \label{R:Robba norms}
For some purposes, it is more convenient to define the topology associated to the norm
$\lambda(\alpha^r)$ using the alternate norm $\lambda(\alpha^r)^{1/r}$.
For example, this guarantees that uniform units of $R$ lift to uniform units of $\tilde{\calR}^I_R$; otherwise, this would only hold when $I$ is a singleton closed interval. For an application of this, see
Remark~\ref{R:Robba reified}.
\end{remark}

\begin{defn} \label{D:Robba theta}
Let $(A,A^+)$ be a Fontaine perfectoid adic Banach ring
which corresponds to $(R,R^+)$ via Theorem~\ref{T:Fontaine perfectoid correspondence}.
Define the map $\theta: W_{\varpi}(R^+)\to A^+$
by identifying $W_{\varpi}(R^+)$ with $\varprojlim_{\varphi_\varpi} W_\varpi(A^+)$ as per \eqref{eq:theta identification},
then projecting onto $W_\varpi(A^+)$, then projecting onto the first component (or equivalently the first ghost component).
Since $(A,A^+)$ is perfectoid, 
by Theorem~\ref{T:Fontaine perfectoid correspondence}
the kernel of $\theta$ is generated by an element $z$ which is Fontaine primitive. We may extend $\theta$ to a map $\theta: \tilde{\calR}^I_R \to A$ for any interval $I$ containing $1$.
The kernel of the extended map is again generated by $z$, so $A$
is strictly pseudocoherent as a module over $\tilde{\calR}^{I}_R$;
since $(A,A^+)$ is sheafy by Corollary~\ref{C:Fontaine perfectoid stably uniform},
Remark~\ref{R:pseudocoherent from sheafy} implies that $A$ is also stably pseudocoherent as a module over $\tilde{\calR}^{I}_R$.
Similarly, Theorem~\ref{T:Fontaine perfectoid compatibility} implies that $\Spa(A,A^+)_{\et}$ admits a basis consisting of affinoid perfectoid spaces, so  by the same reasoning, $A$ is also \'etale-stably pseudocoherent
as a module over $\tilde{\calR}^{I}_R$.
By Remark~\ref{R:need flatness}, a module over $A$ is stably (resp.\ \'etale-stably) pseudocoherent if and only if it is stably (resp.\ \'etale-stably) pseudocoherent as a module over $\tilde{\calR}^{I}_R$.
\end{defn}

\begin{remark} \label{R:Robba quasi-Stein1}
By writing $I$ as a union of an increasing sequence $\{I_j\}$ of closed subintervals and writing $\tilde{\calR}^I_R$ as the inverse limit of the $\tilde{\calR}^{I_j}_R$, we see that $\tilde{\calR}^I_R$ is an open mapping ring.
\end{remark}

\begin{remark}
Using Remark~\ref{R:homogeneity},
the ``reality check'' calculations made in \cite[\S 5.2]{part1}
all carry over to general $E$ without incident. See also \S\ref{subsec:reality checks}, where some of these calculations are repeated in even greater generality.
\end{remark}

\begin{remark} \label{R:Robba reified}
Choose an extension of $R$ to a reified adic Banach ring $(R,R^{\Gr})$.
For $u>0$, let $\tilde{\calR}^{I,+,u}_R$ be the completion in $\tilde{\calR}^I_R$ of the subgroup generated by 
\[
\{x \in \tilde{\calR}^I_R: \lambda(\alpha^t)(x)^{1/t} < u \,\, (t \in I)\}
\cup \{[\overline{y}]: \overline{y} \in R^{+,u}\}.
\]
In case $I$ is a closed interval, $(\tilde{\calR}^I_R, \tilde{\calR}^{I,\Gr}_R)$ is a reified adic Banach ring for the norm $\sup\{\lambda(\alpha^t)^{1/t}: t \in I\}$.

For any inclusion $I \subseteq J \subseteq (0, +\infty)$ of closed intervals, the map $(\tilde{\calR}^J_R, \tilde{\calR}^{J,\Gr}_R) \to (\tilde{\calR}^I_R, 
\tilde{\calR}^{I,\Gr}_R)$ of reified adic Banach rings is a rational localization:
if $I= [s,r]$, then the corresponding rational subspace is
\[
\{v \in \Spra(\tilde{\calR}^J_R, \tilde{\calR}^{J,\Gr}_R):
p^{-1/s} \leq v(\varpi) \leq p^{-1/r}
\}.
\]
Since these spaces are all stably uniform (see Corollary~\ref{C:Robba stably uniform}
and Remark~\ref{R:reified Robba stably uniform} below), for any $I$ we may view $\tilde{\calR}^I_R$ as the ring of global sections of a quasi-Stein reified space.
\end{remark}

\begin{remark} \label{R:Robba quasi-Stein}
By contrast, for an inclusion $I \subseteq J \subseteq (0, +\infty)$ of closed intervals, the map $(\tilde{\calR}^J_R, \tilde{\calR}^{J,+}_R) \to (\tilde{\calR}^I_R, 
\tilde{\calR}^{I,+}_R)$ of adic Banach rings is not necessarily a rational localization. The best one can say, following \cite[Lemma~4.9]{kedlaya-noetherian}, is that if 
there exists a uniform unit $\overline{z} \in R^\times$ with $\overline{\alpha}(\overline{z}) = c \in (0,1)$, then for any $\rho>0$ we have strict isomorphisms
\begin{align*}
\tilde{\calR}^I_R\{T/\rho\}/(T-[\overline{z}]) \cong \tilde{\calR}^{I'}_R,
&\qquad I' = I \cap [\log_c \rho, +\infty), \\
\tilde{\calR}^I_R\{T/\rho\}/(T-[\overline{z}^{-1}]) \cong \tilde{\calR}^{I''}_R,
&\qquad I'' = I \cap (0, -\log_c \rho],
\end{align*}
(provided that we interpret $\tilde{\calR}^I_R$ as the zero ring when $I$ is an empty interval). 
Consequently, if an alternative to using reified spaces is desired, one can study more general inclusions of intervals by adjoining suitable uniform units to $R$ and then carrying out a descent construction using Lemma~\ref{L:perfect polynomial splitting} below.
\end{remark}

\begin{remark} \label{R:Robba reified2}
Retain notation as in Remark~\ref{R:Robba reified}. For $r,u>0$, let $\tilde{\calR}^{\inte,r,+,u}_R$ be the completion in $\tilde{\calR}^{\inte,r}_R$ of the subgroup generated by 
\[
\{x \in \tilde{\calR}^{\inte,r}_R: \lambda(\alpha^r)(x)^{1/r} < u\}
\cup \{[\overline{y}]: \overline{y} \in R^{+,u}\}
\]
For any $s \in (0,r]$, the map $(\tilde{\calR}^{\inte,r}_R, \tilde{\calR}^{\inte,r,\Gr}_R) \to (\tilde{\calR}^{\inte,s}_R, 
\tilde{\calR}^{\inte,r,\Gr}_R)$ of reified adic Banach rings is the rational localization corresponding to the rational subspace
\[
\{v \in \Spra(\tilde{\calR}^{\inte,r}_R, \tilde{\calR}^{\inte,r,\Gr}_R):
v(\varpi) \leq p^{-1/s}
\};
\]
also, 
the map $(\tilde{\calR}^{\inte,r}_R, \tilde{\calR}^{\inte,r,\Gr}_R) \to (\tilde{\calR}^{[s,r]}_R, 
\tilde{\calR}^{[s,r],\Gr}_R)$ of reified adic Banach rings is the rational localization corresponding to the rational subspace
\[
\{v \in \Spra(\tilde{\calR}^{\inte,r}_R, \tilde{\calR}^{\inte,r,\Gr}_R):
v(\varpi) \geq p^{-1/s}
\}.
\]
This last construction can be used to derive information about the rings $\tilde{\calR}^I_R$ from the rings $\tilde{\calR}^{\inte,r}_R$, as in Theorem~\ref{T:curve noetherian}.
\end{remark}

\begin{defn}
For $\rho>0$, write
\begin{gather*}
R\{(T/\rho)^{p^{-\infty}}\},
R\{(T/\rho)^{\pm p^{-\infty}}\},
R\{(\rho_1/T, T/\rho_2)^{p^{-\infty}}\}
\end{gather*}
for the respective completions of 
\[
R\{T/\rho\}[T^{p^{-\infty}}],
R\{\rho/T, T/\rho\}[T^{p^{-\infty}}],
R\{\rho_1/T, T/\rho_2\}[T^{p^{-\infty}}].
\]
\end{defn}

\begin{lemma} \label{L:perfect polynomial splitting}
For $S = R\{(T/\rho)^{p^{-\infty}}\}$
(resp.\ $S = R\{(T/\rho)^{\pm p^{-\infty}}\}$), $\tilde{\calR}^I_S$ may be canonically identified with the Fr\'echet completion of the ring
$\tilde{\calR}^I_R[T^{p^{-\infty}}]$ (resp.\ $\tilde{\calR}^I_R[T^{\pm p^{-\infty}}]$)
with respect to the seminorms
\[
\sum_{n \in \ZZ[p^{-1}]} x_n T^n \mapsto
\max\{\rho^{tn} \lambda(\alpha^t)(x_n)\} \qquad (t \in I).
\]
In particular, the morphism $\tilde{\calR}^I_R \to \tilde{\calR}^I_S$ splits in the category of topological modules over $\tilde{\calR}^I_R$.
\end{lemma}
\begin{proof}
As in \cite[Example~3.6.6]{part1}.
\end{proof}

We next introduce a construction following \cite[Theorem~5.3.9]{part1}.
\begin{defn} \label{D:Robba fractional}
Let $\gotho_{E_\infty}$ (resp.\ $E_\infty$) denote the completion of $\gotho_E[\varpi^{p^{-\infty}}]$ (resp.\ $E[\varpi^{p^{-\infty}}]$). Note that the set
$\{\varpi^n: n \in \ZZ[p^{-1}] \cap [0,1)\}$ forms a topological basis for this module over $\gotho_E$ (resp.\ $E$); in particular, the inclusion
$\gotho_E \to \gotho_{E_\infty}$ (resp.\ $E \to E_\infty$) splits in the category of topological modules over $\gotho_E$ (resp.\ of Banach modules over $E$).

Each element $x$ of $\calE^{\inte}_R \widehat{\otimes}_{\gotho_E} \gotho_{E_\infty}$
can be written uniquely as a sum $\sum_{n \in \ZZ[p^{-1}]_{\geq 0}} \varpi^n [\overline{x}_n]$ with $\overline{x}_n \in R$. For $r>0$, we may characterize $\tilde{\calR}^{\inte,r}_R \widehat{\otimes}_{\gotho_E} \gotho_{E_\infty}$ as the subset consisting of those $x$ for which 
$p^{-n} \alpha(\overline{x}_n)$ converges to 0 in the following sense: for each $\epsilon > 0$,
there are only finitely many indices $n$ for which $p^{-n} \alpha(\overline{x_n}) > \epsilon$. The formula \eqref{eq:Gauss norm} again defines a power-multiplicative norm $\lambda(\alpha^r)$ on $\tilde{\calR}^{\inte,r}_R \widehat{\otimes}_{\gotho_E} \gotho_{E_\infty}$ with respect to which this ring is complete.
\end{defn}

\begin{lemma} \label{L:integral Robba perfectoid}
The Banach ring $\tilde{\calR}^{\inte,r}_R \widehat{\otimes}_{\gotho_E} \gotho_{E_\infty}$ is Fontaine perfectoid, and corresponds via
Theorem~\ref{T:Fontaine perfectoid correspondence} to
$R\{(\overline{\varpi}/p^{-1/r})^{p^{-\infty}}\}$.
\end{lemma}
\begin{proof}
Beware that Definition~\ref{D:Fontaine perfectoid} refers to a topologically nilpotent unit therein called $\varpi$, but the element of $\tilde{\calR}^{\inte,r}_R \widehat{\otimes}_{\gotho_E} \gotho_{E_\infty}$ by that name is not a unit.
Instead, we choose a topologically nilpotent unit
$\overline{z} \in R$ and consider $z = [\overline{z}^{p^{-m}}]$ for some suitable nonnegative integer $m$; this is a topologically nilpotent unit of 
$\tilde{\calR}^{\inte,r}_R \widehat{\otimes}_{\gotho_E} \gotho_{E_\infty}$ with the desired properties.
\end{proof}

\begin{prop} \label{P:Robba perfectoid}
For any $I = [s,r]$, the Banach ring $\tilde{\calR}^{I}_R \widehat{\otimes}_{E} E_\infty$ is Fontaine perfectoid, and corresponds via
Theorem~\ref{T:Fontaine perfectoid correspondence} to
$R\{(p^{-1/s}/\overline{\varpi}, \overline{\varpi}/p^{-1/r})^{p^{-\infty}}\}$.
\end{prop}
\begin{proof}
This follows from Lemma~\ref{L:integral Robba perfectoid} by making a rational
localization (Remark~\ref{R:Robba reified2}), which preserves the perfectoid property
(Theorem~\ref{T:Fontaine perfectoid compatibility}(i)).
\end{proof}
\begin{cor} \label{C:Robba stably uniform}
The adic Banach rings $(\tilde{\calR}^{\inte,r}_R, \tilde{\calR}^{\inte,r,+}_R)$ (for any $r>0$)
and $(\tilde{\calR}^I_R, \tilde{\calR}^{I,+}_R)$ (for any $I$)
are stably uniform, and hence sheafy (by \cite[Theorem~2.8.10]{part1}). 
\end{cor}
\begin{proof}
This follows from Lemma~\ref{L:integral Robba perfectoid} and Proposition~\ref{P:Robba perfectoid} using the splitting from Definition~\ref{D:Robba fractional}.
\end{proof}

\begin{remark}  \label{R:reified Robba stably uniform}
The reified analogue of Corollary~\ref{C:Robba stably uniform} used in 
Remark~\ref{R:Robba reified} asserts that the reified adic Banach rings $(\tilde{\calR}^{\inte,r}_R, \tilde{\calR}^{\inte,r,\Gr}_R)$ (for any $r>0$)
and $(\tilde{\calR}^I_R, \tilde{\calR}^{I,\Gr}_R)$ (for any $I$)
are really stably uniform and hence sheafy (see Remark~\ref{R:reified pseudoflat}).
\end{remark}

\begin{lemma} \label{L:Robba v-covering}
Let $(R,R^{\Gr}) \to (S,S^{\Gr})$ be a morphism of perfect uniform reified adic Banach algebras over $\FF_{p^h}$ such that
$\Spra(S,S^+) \to \Spra(R,R^{\Gr})$ is surjective. Then for any closed interval $I = [s,r]$, the map
$\Spra(\tilde{\calR}^I_S, \tilde{\calR}^{I,\Gr}_S) \to \Spra(\tilde{\calR}^I_R, \tilde{\calR}^{I,\Gr}_R)$
is surjective.
\end{lemma}
\begin{proof}
The map 
\[
R\{(p^{-1/s}/\overline{\varpi}, \overline{\varpi}/p^{-1/r})^{p^{-\infty}}\}
\to S\{(p^{-1/s}/\overline{\varpi}, \overline{\varpi}/p^{-1/r})^{p^{-\infty}}\}
\]
also corresponds to a surjective morphism of reified spectra. By Proposition~\ref{P:Robba perfectoid},
this means that the map
\[
\tilde{\calR}^{I}_R \widehat{\otimes}_{E} E_\infty
\to \tilde{\calR}^{I}_S \widehat{\otimes}_{E} E_\infty
\]
also corresponds to a surjective morphism of reified spectra. On the other hand, it is clear that
\[
\tilde{\calR}^{I}_R \to \tilde{\calR}^{I}_R \widehat{\otimes}_{E} E_\infty,
\]
being a completion of a tower of faithfully finite \'etale morphisms, also corresponds to a surjective morphism of reified spectra. This proves the claim.
\end{proof}

\subsection{Pseudoflatness for Robba rings}
\label{subsec:pseudoflat Robba}

We next derive some pseudoflatness assertions for morphisms between Robba rings, by analogizing the arguments used for rational localizations. In light of Remark~\ref{R:Robba reified} and Remark~\ref{R:Robba quasi-Stein}, we will make use of reified adic spectra to abbreviate the discussion.

\begin{prop} \label{P:weak flatness perfect Robba}
Suppose that $I$ is closed.
For any closed subinterval $I'$ of $I$, the morphism
$\tilde{\calR}^I_R \to \tilde{\calR}^{I'}_R$ is $2$-pseudoflat.
\end{prop}
\begin{proof}
This follows immediately from the reified analogue of Lemma~\ref{L:weak flatness1}
using Remark~\ref{R:Robba reified}, or from Lemma~\ref{L:weak flatness1} itself using Remark~\ref{R:Robba quasi-Stein} plus Lemma~\ref{L:perfect polynomial splitting}.
\end{proof}

\begin{lemma} \label{L:Robba lift properties}
Suppose that $I= [s,r]$ is closed.
Let $(R, R^+) \to (S,S^+)$ be a morphism of one of the following types.
\begin{enumerate}
\item[(i)]
A rational localization.
\item[(ii)]
A faithfully finite \'etale morphism.
\item[(iii)]
A finite \'etale morphism.
\item[(iv)]
An \'etale morphism.
\item[(v)]
A morphism representing a perfectoid subdomain of $\Spa(R,R^+)$.
\end{enumerate}
Then $(\tilde{\calR}^I_R, \tilde{\calR}^{I,+}_R) \to (\tilde{\calR}^I_S, \tilde{\calR}^{I,+}_S)$ is of the same form. (In case (v), this should be interpreted after tensoring over $E$ with $E_\infty$.)
\end{lemma}
\begin{proof}
By the splitting from   Definition~\ref{D:Robba fractional},
it suffices to check the claim after tensoring over $E$ with $E_\infty$. In this
case, by Proposition~\ref{P:Robba perfectoid}, we are considering a morphism of perfectoid spaces, so we may use Theorem~\ref{T:Fontaine perfectoid compatibility}
to retrieve the claim from the corresponding (and obvious) statement about the morphism
\begin{gather*}
(R\{(p^{-1/s}/\overline{\varpi}, \overline{\varpi}/p^{-1/r})^{p^{-\infty}}\},
R\{(p^{-1/s}/\overline{\varpi}, \overline{\varpi}/p^{-1/r})^{p^{-\infty}}\}^{+}) \\
\to
(S\{(p^{-1/s}/\overline{\varpi}, \overline{\varpi}/p^{-1/r})^{p^{-\infty}}\},
S\{(p^{-1/s}/\overline{\varpi}, \overline{\varpi}/p^{-1/r})^{p^{-\infty}}\}^{+})
\end{gather*}
in characteristic $p$.
\end{proof}

\begin{prop} \label{P:Robba weak flatness}
Suppose that $I= [s,r]$ is closed.
Put $X = \Spa(R,R^+)$. For any perfectoid subdomain $\Spa(S,S^+)$ of $X_{\proet}$,
the morphism $\tilde{\calR}^{I}_R \to \tilde{\calR}^{I}_S$ is \'etale-pseudoflat, and also pseudoflat if
$\Spa(S,S^+) \to X$ is an open immersion.
(The reified analogue also holds.)
\end{prop}
\begin{proof}
We may invoke
Theorem~\ref{T:weak flatness},
Theorem~\ref{T:weak flatness perfectoid}, and
Lemma~\ref{L:Robba lift properties} to conclude.
\end{proof}

\begin{remark} \label{R:perfectoid tower splitting Robba}
Suppose that $I$ is closed and let $R = R_0 \to R_1 \to \cdots$ be a tower of faithfully finite \'etale morphisms with completed direct limit $S$. Using 
Lemma~\ref{L:perfectoid tower splitting}
and the splitting of Definition~\ref{D:Robba fractional},
we may see that as a Banach module over $\tilde{\calR}^I_R$, $\tilde{\calR}^I_S$ 
splits as a direct sum of $\tilde{\calR}^I_R$ plus a pro-projective Banach module.
\end{remark}

\subsection{Modules over Robba rings}
\label{subsec:modules period sheaves}

We next extend the theory of pseudocoherent and fpd modules to perfect period sheaves. 

\begin{theorem} \label{T:perfect Robba proetale cohomology}
Put $X = \Spa(R,R^+)$.
Let $M$ be an \'etale-stably pseudocoherent module over $\tilde{\calR}^{\inte,r}$ (resp.\
$\tilde{\calR}^{I}_R$) for some closed interval $I$.
Let $\tilde{M}$ be the associated sheaf of modules over $\tilde{\calR}^{\inte,r}$ (resp.\ $\tilde{\calR}^I$) on $X_{\proet}$. Then
\[
H^i(X, \tilde{M}) = H^i(X_{\et}, \tilde{M}) = H^i(X_{\proet}, \tilde{M}) = \begin{cases} M & (i=0) \\ 0 & (i>0).
\end{cases}
\]
If $M$ is projective, the corresponding equality for $H^i(X_{\faith}, \tilde{M})$ also holds.
(The reified analogue also holds.)
\end{theorem}
\begin{proof}
Using Lemma~\ref{L:integral Robba perfectoid} (resp.\ 
Proposition~\ref{P:Robba perfectoid})
and the splitting from Definition~\ref{D:Robba fractional}, we reduce to corresponding statements for perfectoid spaces:
Theorem~\ref{T:pseudocoherent acyclicity},
Theorem~\ref{T:pseudocoherent etale acyclicity},
Theorem~\ref{T:compare perfectoid cohomology},
Corollary~\ref{C:faithful acyclic vector bundle}.
\end{proof}

\begin{defn}
For $X$ a perfect uniform adic space over $\FF_{p^h}$, 
for $r>0$ (resp.\ for $I$ a closed interval),
we define a \emph{pseudocoherent/fpd sheaf} of $\tilde{\calR}^{\inte,r}$-modules (resp.\ $\tilde{\calR}^{I}$-modules)
on $X_{\proet}$ to be a sheaf of modules which locally on $X_{\proet}$
is associated to an \'etale-stably pseudocoherent/fpd module. More precisely, there should exist a covering of
$X_{\proet}$ by perfectoid subdomains $Y$ on each of which the given sheaf restricts to the sheaf associated 
to an \'etale-stably pseudocoherent/fpd module over $\tilde{\calR}^{\inte,r}_{\overline{\calO}(Y)}$ (resp.\ $\tilde{\calR}^I_{\overline{\calO}(Y)}$). Note that 
this topological condition involves base changes on $\Spa(\calR^{\inte,r}_{\overline{\calO}(Y)}, \calR^{\inte,r,+}_{\overline{\calO}(Y)})$ (resp.\ $\Spa(\calR^I_{\overline{\calO}(Y)}, \calR^{I,+}_{\overline{\calO}(Y)})$) which need not be induced from $Y$.
\end{defn}

\begin{theorem} \label{T:perfect Robba Kiehl}
Suppose that $r>0$ and the interval $I$ is closed, and put $X = \Spa(R,R^+)$. 
\begin{enumerate}
\item[(a)]

The global sections functor induces equivalences of categories between
pseudocoherent sheaves of $\tilde{\calR}^{\inte,r}$-modules (resp.\ $\tilde{\calR}^I$-modules) on $X_{\proet}$ and \'etale-stably pseudocoherent $\tilde{\calR}^{\inte,r}_R$-modules (resp.\ $\tilde{\calR}^I_R$-modules).
(The reified analogue also holds.)

\item[(b)]

The global sections functor induces equivalences of categories between
fpd sheaves of $\tilde{\calR}^{\inte,r}$-modules (resp.\ $\tilde{\calR}^I$-modules) on $X_{\proet}$ and \'etale-stably fpd $\tilde{\calR}^{\inte,r}_R$-modules (resp.\ $\tilde{\calR}^I_R$-modules).
(The reified analogue also holds.)

\item[(c)]
For any covering of $I$ by finitely many closed subintervals $J_1,\dots,J_n$, the morphism
$\tilde{\calR}^I_R \to \bigoplus_i \tilde{\calR}^{J_i}_R$ is an effective descent morphism for stably or \'etale-stably pseudocoherent (resp.\ fpd) modules over Banach rings.
\item[(d)]
For any $s \in (0,r]$, the morphism
$\tilde{\calR}^{\inte,r}_R \to \tilde{\calR}^{\inte,s}_R \oplus \tilde{\calR}^{[s,r]}_R$ is an effective descent morphism for stably or \'etale-stably pseudocoherent (resp.\ fpd) modules over Banach rings.
\end{enumerate}
\end{theorem}

\begin{proof}
To prove (a), using Lemma~\ref{L:integral Robba perfectoid}, Proposition~\ref{P:Robba perfectoid}, and the splitting from Definition~\ref{D:Robba fractional}, we may reduce the claim to statements about perfectoid spaces: 
Theorem~\ref{T:refined Kiehl},
Theorem~\ref{T:refined Kiehl etale},
Corollary~\ref{C:refined Kiehl perfectoid}.
The same applies to (b) in view of Lemma~\ref{L:complete descent}.
To prove (c) and (d),
since $(\tilde{\calR}^{\inte,r}_R, \tilde{\calR}^{\inte,r,+}_R)$ and 
$(\tilde{\calR}^I_R, \tilde{\calR}^{I,+}_R)$ are stably uniform by
Corollary~\ref{C:Robba stably uniform},
 we may directly apply Theorem~\ref{T:refined Kiehl}
 and Theorem~\ref{T:refined Kiehl etale}.
\end{proof}

\subsection{\texorpdfstring{$\varphi$}{Phi}-modules over period sheaves}
\label{subsec:perfect phi-modules}

We next extend the theory of $\varphi$-modules introduced in \cite[\S 6.1]{part1} to include pseudocoherent objects.

\begin{hypothesis}
For the remainder of \S\ref{sec:perfect period sheaves},
fix a positive integer $a$ divisible by $h$ and let $E_a$ denote the unramified extension of $E$ of degree $a$.
\end{hypothesis}

\begin{defn}
As in \cite[\S 9.3]{part1} (modulo Convention~\ref{C:drop dagger})
we define ring sheaves on $X_{\proet}$ via the following formulas, in which $Y$ denotes an arbitrary perfectoid subdomain:
\begin{gather*}
\tilde{\bA}_X(Y) = W_\varpi(\overline{\calO}(Y)),  \,
\tilde{\bA}^+_X(Y) = W_\varpi(\overline{\calO}^+(Y)), \,
\tilde{\bA}^r_X(Y) = \tilde{\calR}^{\inte,r}_{\overline{\calO}(Y)}, \,
\tilde{\bA}^\dagger_X(Y) = \tilde{\calR}^{\inte}_{\overline{\calO}(Y)}, \\
\tilde{\bB}_X(Y) = W_\varpi(\overline{\calO}(Y))[\varpi^{-1}], \,
\tilde{\bB}^r_X(Y) = \tilde{\calR}^{\bd,r}_{\overline{\calO}(Y)}, \,
\tilde{\bB}^\dagger_X(Y) = \tilde{\calR}^{\bd}_{\overline{\calO}(Y)}, \\
\tilde{\bC}^+_X(Y) = \tilde{\calR}^+_{\overline{\calO}(Y)}, \,
\tilde{\bC}^r_X(Y) = \tilde{\calR}^r_{\overline{\calO}(Y)}, \,
\tilde{\bC}^I_X(Y) = \tilde{\calR}^I_{\overline{\calO}(Y)},  \,
\tilde{\bC}_X(Y) = \tilde{\calR}_{\overline{\calO}(Y)}.
\end{gather*}
\end{defn}

\begin{defn}
Denote by $\varphi$ the endomorphism of any of the following perfect period rings or sheaves induced by the action of $\varphi_\varpi^a$ on $W_\pi(R^+)$:
\begin{equation} \label{eq:phi rings}
\tilde{\calE}^{\inte}, \tilde{\calE}, \tilde{\calR}^{\inte}, \tilde{\calR}^{\bd}, \tilde{\calR}^\infty, \tilde{\calR},
\tilde{\bA}, \tilde{\bA}^+, \tilde{\bA}^\dagger, \tilde{\bB}, \tilde{\bB}^\dagger,
\tilde{\bC}^+, \tilde{\bC}^\infty, \tilde{\bC}. 
\end{equation}
Note that these maps are all invertible.

For $r>0$, denote also by $\varphi$ the following induced morphisms between distinct rings or sheaves:
\begin{gather}
\tilde{\calR}^{\inte,r} \to \tilde{\calR}^{\inte,rp^{-ah}},
\tilde{\calR}^{\bd,r} \to \tilde{\calR}^{\bd,rp^{-ah}},
\tilde{\calR}^{r} \to \tilde{\calR}^{rp^{-ah}}, \notag \\
\tilde{\bA}^r \to \tilde{\bA}^{rp^{-ah}},
\tilde{\bB}^r \to \tilde{\bB}^{rp^{-ah}},
\tilde{\bC}^r \to \tilde{\bC}^{rp^{-ah}}.
\label{eq:phi rings2}
\end{gather}
In these cases, the target ring contains the source ring, so 
we obtain a well-defined action of $\varphi^{-1}$ on the following rings and sheaves:
\begin{gather}
\tilde{\calR}^{\inte,r},
\tilde{\calR}^{\bd,r},
\tilde{\calR}^{r},
\tilde{\bA}^r,
\tilde{\bB}^r,
\tilde{\bC}^r.
\label{eq:phi rings3}
\end{gather}

For $0 < s \leq r$, denote also by $\varphi$ the following induced morphisms between distinct rings or sheaves:
\begin{gather}
\tilde{\calR}^{[s,r]} \to \tilde{\calR}^{[sp^{-ah}, rp^{-ah}]},
\tilde{\bC}^{[s,r]} \to \tilde{\bC}^{[sp^{-ah}, rp^{-ah}]}.
\label{eq:phi rings4}
\end{gather}
In these cases, the target ring does not contain the source. However, if $s \leq rp^{-ah}$, then the maps
\begin{gather}
\tilde{\calR}^{[s,r]} \to \tilde{\calR}^{[s, rp^{-ah}]},
\tilde{\bC}^{[s,r]} \to \tilde{\bC}^{[s, rp^{-ah}]}.
\label{eq:phi rings5}
\end{gather}
do have this property.
\end{defn}

\begin{defn} \label{D:pseudocoherent phi-module}
For a ring $*$ listed in \eqref{eq:phi rings},
we define a \emph{projective (resp.\ pseudocoherent, fpd) $\varphi$-module} over $*$ to be 
a finite projective (resp.\ pseudocoherent, fpd) module $M$ over $*$ 
equipped with a semilinear $\varphi$-action (i.e., an isomorphism $\varphi^* M \cong M$). Similarly,
for a sheaf $*$ listed in \eqref{eq:phi rings},
we define a \emph{projective (resp.\ pseudocoherent, fpd) $\varphi$-module} over $*$ to be 
a finite locally free (resp.\ pseudocoherent, fpd) sheaf $M$ of $*$-modules
equipped with a semilinear $\varphi$-action (i.e., an isomorphism $\varphi^* M \cong M$), subject to topological restrictions as listed below.
We define the \emph{$\varphi$-(hyper)cohomology} of $M$ to be the (hyper)cohomology of the complex 
\[
0 \to M \stackrel{\varphi-1}{\to} M \to 0.
\]

For a morphism $*_1 \to *_2$ listed in \eqref{eq:phi rings2}, we define a
\emph{projective (resp.\ pseudocoherent, fpd) $\varphi$-module} over $*_1$ to be
a finite projective (resp.\ pseudocoherent, fpd) module $M$ over $*_1$ 
equipped with an isomorphism $\varphi^* M \cong M \otimes_{*_1} *_2$, subject to topological restrictions as listed below.
Note that this is equivalent to specifying a finite projective (resp.\ pseudocoherent, fpd) module $M$ over $*_2$  equipped with a semilinear $\varphi^{-1}$-action (again subject to the topological restrictions).
We define the \emph{$\varphi$-(hyper)cohomology} of $M$ to be the (hyper)cohomology of the complex 
\[
0 \to M \stackrel{\varphi-1}{\to} M \otimes_{*_1} *_2 \to 0.
\]

For a morphism $*_1 \to *_2$ listed in \eqref{eq:phi rings4}, let $*_1 \to *_3$ be the
corresponding morphism listed in \eqref{eq:phi rings5}; we define a
\emph{projective (resp.\ pseudocoherent, fpd) $\varphi$-module} over $*_1$ to be
a finite projective (resp.\ pseudocoherent, fpd) module $M$ over $*$ 
equipped with an isomorphism $\varphi^* M \otimes_{*_2} *_3 \cong M \otimes_{*_1} *_3$, 
subject to topological restrictions as listed below.
We define the \emph{$\varphi$-(hyper)cohomology} of $M$ to be the (hyper)cohomology of the complex 
\[
0 \to M \stackrel{\varphi-1}{\to} M \otimes_{*_1} *_3 \to 0.
\]
We write $H^i_{\varphi}$ to denote the functor of $i$-th cohomology groups.

We now return to the topological restrictions alluded to earlier.
\begin{itemize}
\item
For $* = \tilde{\bC}^{[s,r]}$, we insist that the module be complete for its natural topology and,
for every perfectoid subdomain $Y$ of $X_{\proet}$, 
the sections over $Y$ form an \'etale-stably pseudocoherent
$\tilde{\calR}^{[s,r]}_{\overline{\calO}(Y)}$-module.
\item
For $* = \tilde{\bC}^r$ (including $r = \infty$), we insist that  the module be complete for its natural topology and,
for every closed interval $I \subseteq [0,r)$, 
the base extension to $\tilde{\bC}^{[s,r]}$ satisfies the previous condition.
\item
For $* = \tilde{\bC}$, we insist that  the module be complete for its natural topology and,
for every perfectoid subdomain $Y$ of $X_{\proet}$, 
there exists $r>0$ such that the restriction of the sheaf to $Y$ arises by base extension from a sheaf over $\tilde{\bC}^r$ satisfying the previous condition.
\end{itemize}
\end{defn}

\begin{remark} 
Note that in Definition~\ref{D:pseudocoherent phi-module}, for modules
we only impose the pseudocoherent and fpd conditions at the algebraic level, without any of the topological restrictions that were needed in the previous discussions to pass freely between modules and sheaves. We will thus have to be more careful about that passage in this context.
\end{remark}

\begin{defn}
A \emph{projective (resp.\ pseudocoherent, fpd) $\varphi$-bundle} over $\tilde{\calR}_R$ is a collection consisting of one projective (resp. \'etale-stably pseudocoherent, \'etale-stably fpd) 
$\varphi$-module $M_I$ over $\tilde{\calR}^I_R$ for each closed interval $I = [s,r] \subset (0, +\infty)$ for which $s \leq rp^{-ah}$, together with an isomorphism
$M_I \otimes_{\tilde{\calR}^I_R} \tilde{\calR}^J_R \cong M_J$ for each containment
$J \subseteq I$ satisfying the cocycle condition. 
\end{defn}

\begin{convention} \label{conv:drop projective}
To conform with the usage in \cite{part1} and elsewhere, we will sometimes drop the adjective \emph{projective} in reference to a $\varphi$-module or $\varphi$-bundle.
\end{convention}

\begin{lemma} \label{L:lift action1}
For any ring (not sheaf) $*_R$ appearing in \eqref{eq:phi rings} or \eqref{eq:phi rings2},
every finitely generated $*_R$-module equipped with a semilinear $\varphi$-action (in particular, every pseudocoherent $\varphi$-module over $*_R$)
is a quotient of some projective $\varphi$-module over $*_R$.
\end{lemma}
\begin{proof}
For rings appearing in \eqref{eq:phi rings}, we may directly apply \cite[Lemma~1.5.2]{part1}.
For rings appearing in \eqref{eq:phi rings2}, we instead apply \cite[Lemma~1.5.2]{part1} to the $\varphi^{-1}$-action.
\end{proof}

\begin{remark}
Note that Lemma~\ref{L:lift action1} does not apply to $\tilde{\calR}^{[s,r]}_R$. This creates some minor complications when attempting to extend results about projective $\varphi$-modules to the pseudocoherent and fpd cases.
\end{remark}

\subsection{\texorpdfstring{$\varphi$}{Phi}-modules and local systems}

We next recall, following \cite[\S 8.3]{part1} (but with some variations), the relationship between $\varphi$-modules and local systems. In the case of integral local systems, we can upgrade the correspondence to include pseudocoherent objects on both sides.

\begin{defn}
Recall that for any topological space $T$, we obtain an associated \emph{constant sheaf}
$\underline{T}$ on $X_{\proet}$ given by $\underline{T}(Y) = \Cont(\left| Y \right|, T)$.
For $R$ a topological ring, we define an
\emph{\'etale projective/pseudocoherent/fpd $R$-local system} on $X_{\proet}$ to be a locally finite projective/pseudocoherent/fpd $\underline{R}$-module. By analogy with Convention~\ref{conv:drop projective}, we will sometimes drop the adjective \emph{projective}.
\end{defn}

\begin{remark} \label{R:GAGA for local systems}
We may similarly define local systems on schemes using the pro-\'etale topology, as in 
\cite[Definition~1.4.10]{part1}.
As in \cite[Remark~1.4.2]{part1}, we may see that \'etale pseudocoherent $\gotho_{E_a}/(\varpi^n)$-local systems and their cohomology on a scheme, a preadic space, or a reified adic space can be computed equally well in 
any topology which refines the finite \'etale topology and for which descent data for finite \'etale morphisms to $X$ are effective. This includes the finite \'etale topology, the \'etale topology, the pro-\'etale topology, and in the case of a perfectoid space the v-topology (by Theorem~\ref{T:vector bundle faithful descent}).

By \cite[Theorem~2.6.9]{part1},
we have $\Spec(R)_{\fet} \cong \Spa(R,R^+)_{\fet}$;
it follows that the categories of \'etale pseudocoherent $\gotho_{E_a}$-local systems
on $\Spec(R)_{\proet}$, $\Spa(R,R^+)_{\proet}$, and $\Spa(R,R^+)_{\faith}$ are equivalent to each other. (The reified analogue also holds.)
\end{remark}

\begin{lemma} \label{L:AS exact}
The sequences
\begin{gather}
0 \to \underline{\gotho_{E_a}}_X \to \tilde{\bA}_X \stackrel{\varphi-1}{\to} \tilde{\bA}_X \to 0 \label{eq:as1} \\
0 \to \underline{\gotho_{E_a}}_X \to \tilde{\bA}^\dagger_X \stackrel{\varphi-1}{\to} \tilde{\bA}^\dagger_X \to 0 \label{eq:as2} \\
0 \to \underline{E_a}_X \to \tilde{\bB}^\dagger_X \stackrel{\varphi-1}{\to} \tilde{\bB}^\dagger_X \to 0 \label{eq:as3} \\
0 \to \underline{E_a}_X \to \tilde{\bC}_X \stackrel{\varphi-1}{\to} \tilde{\bC}_X \to 0 \label{eq:as4} 
\end{gather}
are exact.
\end{lemma}
\begin{proof}
Exactness at the left is trivial. For exactness at the middle, see \cite[Corollary~5.2.4]{part1}. For exactness at the right in the case of \eqref{eq:as1}, note that surjectivity amounts to the fact that adjoining to $R$ a root of an Artin-Schreier equation creates a faithfully finite \'etale $R$-algebra. 

To check exactness at the right in the case of \eqref{eq:as2} (and hence \eqref{eq:as3}),
by the previous paragraph it is enough to check that for $x \in \tilde{\calR}^{\inte}_R$, $y \in \tilde{\calE}^{\inte}_R$ satisfying $(\varphi-1)(y) = x$, we have $y \in \tilde{\calR}^{\inte}_R$. To see this, write $x = \sum_{n=0}^\infty \varpi^n [\overline{x}_n]$, $y = \sum_{n=0}^\infty \varpi^n [\overline{y}_n]$; by induction on $n$, we have
\[
\max\{1, \left|\overline{y}_0 \right|, \dots, \left|\overline{y}_n \right|\} \leq 
\max\{1, \left|\overline{x}_0 \right|, \dots, \left|\overline{x}_n \right|\}.
\]
This proves the claim.

To check exactness at the right in the case of \eqref{eq:as4}, note that any $x \in \tilde{\calR}_R$ can be written as the sum of some $y \in \tilde{\calR}^{\bd}$ with a convergent sum $\sum_{i=0}^\infty \varpi^{n_i} [\overline{z}_i]$ for some $n_i \in \ZZ$,
$\overline{z}_i \in R^{\circ \circ}$. We then have $x = y + (\varphi-1)(w)$ for
\[
w = -\sum_{i=0}^\infty \sum_{j=0}^\infty \varpi^{n_i} [\overline{z}_i^{p^j}].
\]
We are thus reduced to the previous paragraph.
\end{proof}

It is obvious that every pseudocoherent $\gotho_E$-local system admits a projective resolution of length 1 by projective local systems. 
We next verify the corresponding statement on the side of $\varphi$-modules.

\begin{lemma} \label{L:pseudocoherent type A torsion}
Let $S$ be one of $\tilde{\calR}^{\inte}_R, \tilde{\calE}_R^{\inte}$.
Let $M$ be a finitely presented $S$-module equipped with a semilinear $\varphi$-action.
\begin{enumerate}
\item[(a)]
The Fitting ideals of $M$ are generated by elements of $S$ of the form $\varpi^n [\overline{e}_n]$ for $n$ a nonnegative integer and $\overline{e}_n \in R$ an idempotent.
\item[(b)]
The $\varpi$-power torsion submodule $T$ of $M$ is finitely generated, and in particular killed by some power of $\varpi$.
\item[(c)]
The quotient $M/T$ is projective as an $S$-module. In particular, $T$ is also a finitely presented $S$-module.
\item[(d)]
The module $T$ is a finite direct sum of $\varphi$-stable submodules, each of which is finite projective over $S/\varpi^n S$ for some (varying) positive integer $n$.
\item[(e)]
The module $M$ is strictly pseudocoherent over $S$.
\end{enumerate}
\end{lemma}
\begin{proof}
Part (a) is immediate from \cite[Proposition~3.2.13]{part1}.
By Lemma~\ref{L:lift action1}, we may construct a surjection $F \to M$ of $S$-modules with semilinear $\varphi$-actions such that $F$ is finite projective as an $S$-module. Put $M' = M/T$ and let $N$ be the kernel of $F \to M'$. Since $M'$ is $\varpi$-torsion-free, it is flat over $\gotho_E$, so 
\[
0 \to N/\varpi N \to F/\varpi F \to M'/\varpi M' \to 0
\]
is an exact sequence of $R$-modules.
By \cite[Proposition~3.2.13]{part1}, $M'/ \varpi M'$ is a finite projective $R$-module, as then is $N/ \varpi N$. 

In the case where $S = \tilde{\calE}^{\inte}_R$, we may show directly (without Nakayama's lemma, which is not applicable without \emph{a priori} knowledge of finiteness) that $N$ is generated as an $S$-module by any finite set lifting a set of generators of $N/\varpi N$: given $\bv \in N$, we may write it as the limit of a $\varpi$-adically convergent series of linear combinations of generators. By a similar argument (replacing the $\varpi$-adic topology with the Fr\'echet topology), we may deduce the same conclusion when $S = \tilde{\calR}^{\inte}_R$.

Let $T_n$ be the $\varpi^n$-torsion submodule of $M$, so that $T = \bigcup_{n=1}^\infty T_n$.
By Lemma~\ref{L:pseudocoherent 2 of 3}, $M'$ is a finitely presented $S$-module and $T = \ker(M \to M')$ is a finitely generated $S$-module.
In particular, $T = T_n$ for some $n$.
This proves (b). 

By (b) and  Lemma~\ref{L:pseudocoherent 2 of 3}, $M/T$ is a finitely presented $S$-module.
We may then apply (a) to deduce (c). From (a) and (c), we deduce (d), which implies (e)
using Lemma~\ref{L:pseudocoherent 2 of 3}.
\end{proof}

\begin{cor} \label{C:pseudocoherent type A flat}
Let $S$ be one of $\tilde{\calR}^{\inte}_R, \tilde{\calE}_R^{\inte}$.
Then every pseudocoherent $\varphi$-module over $S$ admits a projective resolution of length at most $1$ by projective $\varphi$-modules, and in particular is fpd of projective dimension at most $1$.
\end{cor}
\begin{proof}
Combine Lemma~\ref{L:lift action1} with Lemma~\ref{L:pseudocoherent type A torsion}.
\end{proof}

\begin{remark}
We have not investigated the structure of the underlying modules of pseudocoherent $\varphi$-modules over $\tilde{\calR}^{\bd}_R$ or $\tilde{\calE}_R$. In any case, we will have little use for such objects in what follows.
\end{remark}

\begin{theorem} \label{T:pseudocoherent type A}
Assume that $R$ is an $\gotho_{E_a}$-algebra (equivalently, an $\FF_{p^{ah}}$-algebra) and put $X = \Spa(R,R^+)$.
The following tensor categories are equivalent.
\begin{enumerate}
\item[(a)]
The category of \'etale projective (resp.\ pseudocoherent, fpd) $\gotho_{E_a}$-local systems on $X_{\proet}$.
\item[(b)]
The category of projective (resp.\ pseudocoherent, fpd) $\varphi$-modules over $\tilde{\calE}^{\inte}_R$.
\item[(c)]
The category of projective (resp.\ pseudocoherent, fpd) $\varphi$-modules over $\tilde{\calR}^{\inte}_R$.
\item[(d)]
The category of projective (resp.\ pseudocoherent, fpd) $\varphi$-modules over $\tilde{\bA}_{X}$, viewing the latter as a ring sheaf on any of the sites $X, X_{\et}, X_{\proet}, X_{\faith}$.
\item[(e)]
The category of projective (resp.\ pseudocoherent, fpd) $\varphi$-modules over $\tilde{\bA}^\dagger_{X}$, viewing the latter as a ring sheaf on any of the sites $X, X_{\et}, X_{\proet},  X_{\faith}$.
\end{enumerate}
In particular, the functors from (a) to (d) and from (a) to (e)
are base extensions of $\underline{\gotho_{E_a}}$-modules, with the quasi-inverse functors being $\varphi$-invariants;
and the functors from (d) to (b) and from (e) to (c) are global  sections.
Moreover, these equivalences induces an isomorphism between sheaf cohomology groups
(in case (a)), $\varphi$-cohomology groups (in cases (b)--(c)),
and $\varphi$-hypercohomology groups (in cases (d)--(e)).
(The reified analogue also holds.)
\end{theorem}
\begin{proof}
We first connect (c) and (e) (for any choice of site). In this case, the equivalence follows from 
Theorem~\ref{T:perfect Robba Kiehl}
and the comparison from Theorem~\ref{T:perfect Robba proetale cohomology}.
(Note that by Lemma~\ref{L:pseudocoherent type A torsion}, the underlying module of any pseudocoherent $\varphi$-module over $\tilde{\calR}^{\inte}_R$ is stably fpd.)

We next connect (b) and (d). Let $S,S^+$ be the completions of $R[T^{\pm p^{-\infty}}],
R[T^{p^{-\infty}}]$ for the $T$-adic norm (for any normalization; note that the norm on $R$ is now trivial). Since $R \to S$ admits a splitting given by the constant coefficient map, the \v{C}ech (simplicial) sequence
\[
0 \to R \to S \to S \widehat{\otimes}_R S \to S \widehat{\otimes}_R S \widehat{\otimes}_R S \to \cdots
\]
is exact. We may thus deduce the connection from the connection between (c) and (e) applied to $(S,S^+)$.

For an object $T$ of (a), the corresponding object of (d) (for the site $X_{\proet}$) is 
\begin{equation} \label{eq:type A base change1}
M = T \otimes_{\underline{\gotho_{E_a}}} \tilde{\bA}_X
\end{equation}
with the $\varphi$-action coming from the second factor; note that the base extension morphism is flat, so it acts on pseudocoherent and fpd objects. Similarly, the corresponding object of (e) is
\begin{equation} \label{eq:type A base change2}
M = T \otimes_{\underline{\gotho_{E_a}}} \tilde{\bA}^\dagger_X.
\end{equation}
By Lemma~\ref{L:AS exact}, the sequence 
\begin{equation} \label{eq:AS}
0 \to T \to M \stackrel{\varphi-1}{\to} M \to 0
\end{equation}
is exact.
It follows that the functors from (a) to (d) and from (a) to (e) are fully faithful.

We next check the comparison of cohomology from (a) to (e). Using the exactness of
\eqref{eq:AS}, this reduces to the acyclicity of $M$ from Theorem~\ref{T:perfect Robba proetale cohomology}. Note that this also immediately applies to the functor from (a) to (e) if we restrict consideration to $\varpi$-power-torsion objects.

We next check that the functor from (a) to (d) is essentially surjective.
By Lemma~\ref{L:pseudocoherent type A torsion}, any object $M$ of (d) fits into a short exact sequence
\[
0 \to M_0 \to M \to P \to 0
\]
of $\varphi$-modules in which $M_0$ is killed by some power of $\varpi$ and $P$ is projective. The extension is defined by some element of $H^1_\varphi(P^\dual \otimes M_0)$; note that $P^\dual \otimes M_0$ is again $\varpi$-power-torsion. By the fact that we have the comparison of cohomology from (a) to (e) on torsion objects, we may reduce to checking essential surjectivity in the torsion and projective cases. The latter itself reduces to the torsion case by taking inverse limits, using full faithfulness to match up the truncations. To treat the torsion case, we may again use comparison of cohomology to reduce to the case of $\varpi$-torsion objects, for which we may directly invoke \cite[Proposition~3.2.7]{part1}.

By the previous paragraph, the functor from (e) to (d) is essentially surjective
(because the equivalence from (a) to (d) factors through it).
However, from the previous calculations, this functor
induces a bijective map of $H^0$ groups and an injective map of $H^1$ groups, so in particular it is also fully faithful.
We thus have the full connection between (a) and (e).

We also have the equivalence between (b) and (c). For both of these categories, the cohomology vanishes by definition beyond degree 1, so it suffices to compare the Yoneda extension groups in degrees 0 and 1. By Lemma~\ref{L:pseudocoherent 2 of 3}, these extensions only involve pseudocoherent modules, so the comparison follows immediately from the equivalence. We therefore have the full connection between (b) and (c), thus completing the proof.
\end{proof}

\begin{cor} \label{C:pseudocoherent type A}
Let $X$ be a preadic space over $\gotho_{E_a}$. 
Then the category of \'etale projective (resp.\ pseudocoherent) $\gotho_{E_a}$-local systems $T$ on $X$ is equivalent to the category of projective (resp.\ pseudocoherent) $\varphi$-modules $M$ over $\bA^\dagger_X$ (viewing the latter as a ring sheaf on any of the sites $X, X_{\et}, X_{\proet}$; if $X$ is perfectoid we may also take $X_{\faith}$) via the functors
\[
T \mapsto M(T) = T \otimes_{\underline{E}} \tilde{\bA}^\dagger_X, \qquad
M \mapsto T(M) = M^{\varphi}.
\]
Moreover, this equivalence induces a functorial isomorphism of sheaf cohomology groups with $\varphi$-hypercohomology groups.
\end{cor}
\begin{proof}
Immediate from Theorem~\ref{T:pseudocoherent type A}.
\end{proof}

\begin{defn}
A $\varphi$-module $M$ over $\tilde{\calR}^{\bd}_R$ or $\tilde{\calR}_R$ is \emph{globally \'etale} if it arises by base extension from $\tilde{\calR}^{\inte}_R$.

A $\varphi$-module $M$ over $\tilde{\bB}^\dagger$ or $\tilde{\bC}$ is \emph{\'etale} if it arises by base extension from $\tilde{\bA}^\dagger$.
Beware that for $X = \Spa(R,R^+)$, a $\varphi$-module over $\tilde{\calR}^{\bd}_R$ or $\tilde{\calR}_R$ can be \'etale as a sheaf (i.e., as a module over  $\tilde{\bB}^\dagger_X$ or $\tilde{\bC}_X$) without being globally \'etale; this has to do with the fact that the descent to $\tilde{\calR}^{\inte}_R$ is not uniquely determined. See \cite[Remark~7.3.5]{part1} for further discussion.
\end{defn}

\begin{remark} \label{R:local systems lattice}
As in \cite[Corollary~8.4.7]{part1}, one sees that
any \'etale $E$-local system over a preadic space can locally (for the analytic topology) be realized as the base extension of an $\gotho_E$-local system.
If $X$ is a perfectoid space, then we may make the same argument for $E$-local systems on $X_{\faith}$ by replacing the invocation of \cite[Lemma~8.2.17(a)]{part1} in the proof of
\cite[Proposition~8.4.6]{part1} with Theorem~\ref{T:vector bundle faithful descent} 
(as used in Remark~\ref{R:GAGA for local systems}); alternatively, we may apply Corollary~\ref{C:same local systems} below to reduce to the case of $X_{\proet}$.
\end{remark}

\begin{theorem}  \label{T:etale type C}
Let $X$ be a preadic space over $\gotho_{E_a}$. 
Then the category of \'etale $E_a$-local systems $V$ on $X$ is equivalent to the category of \'etale projective $\varphi$-modules $M$ over $\tilde{\bC}_X$ (viewing the latter as a ring sheaf on any of the sites $X, X_{\et}, X_{\proet}$; if $X$ is perfectoid we may also take $X_{\faith}$) via the functors
\[
V \mapsto M(V) = V \otimes_{\underline{E}} \tilde{\bC}_X, \qquad
M \mapsto V(M) = M^{\varphi}.
\]
Moreover, this equivalence induces a functorial isomorphism of sheaf cohomology groups with $\varphi$-hypercohomology groups.
\end{theorem}
\begin{proof}
Given $V$, we may combine Remark~\ref{R:local systems lattice}
with the exactness of \eqref{eq:AS} to deduce that
\[
0 \to V \to V(M) \stackrel{\varphi-1}{\to} V(M) \to 0
\]
is exact. It follows that the functor $V \mapsto M(V)$ is fully faithful;
by Theorem~\ref{T:perfect Robba proetale cohomology}, we also see that it compares cohomology groups. Thanks to full faithfulness, we may check essential surjectivity locally on $X$, for which we may invoke Theorem~\ref{T:pseudocoherent type A} to conclude.
\end{proof}
\begin{cor} \label{C:globally etale BC}
The categories of globally \'etale $\varphi$-modules over $\tilde{\calR}^{\bd}_R$ and over $\tilde{\calR}_R$ are equivalent via base extension.
\end{cor}
\begin{proof}
Immediate from Theorem~\ref{T:etale type C}.
\end{proof}

\begin{cor} \label{C:same local systems}
For any perfectoid space $X$, the categories of \'etale $E$-local systems on $X_{\proet}, X_{\faith}$ are canonically equivalent.
\end{cor}
\begin{proof}
By the same argument as in the proof of Theorem~\ref{T:etale type C}, we see that these categories are equivalent to the respective categories of projective $\varphi$-modules over $\tilde{\bC}_X$, which we already know to be equivalent using Theorem~\ref{T:perfect Robba Kiehl}.
\end{proof}

\subsection{Pseudocoherent \texorpdfstring{$\varphi$}{phi}-modules of type \texorpdfstring{$\bC$}{C}}

We next study pseudocoherent $\varphi$-modules over sheaves of type
$\bC$ in more detail.

\begin{theorem} \label{T:perfect generalized phi-modules}
Put $X = \Spa(R,R^+)$. The following exact tensor categories are equivalent
to each other.
\begin{enumerate}
\item[(a)]
The category of strictly pseudocoherent $\varphi$-modules $M$ over $\tilde{\calR}^\infty_R$
such that for every closed interval $I \subset (0, \infty)$,
$M \otimes_{\tilde{\calR}^\infty_R} \tilde{\calR}^I_R$ is \'etale-stably pseudocoherent.
\item[(b)]
The category of pseudocoherent $\varphi$-modules $M$ over $\tilde{\calR}_R$
such that for some $r>0$, the underlying module of $M$ is the base extension of a 
strictly pseudocoherent module $M_0$ over $\tilde{\calR}^{r}_R$
for which for every closed interval $I \subset (0, r]$,
$M_0 \otimes_{\tilde{\calR}^{r}_R} \tilde{\calR}^I_R$ is \'etale-stably pseudocoherent.
\item[(c)]
The category of pseudocoherent $\varphi$-bundles over $\tilde{\calR}_R$.
\item[(d)]
The category of \'etale-stably pseudocoherent $\varphi$-modules over $\tilde{\calR}^{[s,r]}_R$,
provided that $0 < s \leq r/p^{ah}$.
\item[(e)]
The category of pseudocoherent sheaves on  $\FFC_{R,\et}$ (the adic relative Fargues-Fontaine curve).
\item[(f)]
The category of pseudocoherent sheaves on  $\FFC_{R,\proet}$.
\item[(g)]
The category of pseudocoherent $\varphi$-modules over $\tilde{\bC}^\infty_{X}$.
\item[(h)]
The category of pseudocoherent $\varphi$-modules over $\tilde{\bC}_{X}$.
\item[(i)]
The category of pseudocoherent $\varphi$-modules over $\tilde{\bC}^{[s,r]}_{X}$.
\end{enumerate}
More precisely, the functors from (a) to (b), (a) to (d), (g) to (h), and (g) to (i) are base extensions;
the functor from (f) is (g) is a pullback;
and the functors from (f) to (e) to (c) to (d) are restrictions.
(The reified analogue also holds.)
\end{theorem}

\begin{proof}
We first note that the functors from (f) to (e) to (c) are equivalences by Theorem~\ref{T:perfect Robba Kiehl}.
Using $\varphi$ to move the interval around, we may also add (d) and (i) to the equivalence.

By imitating the proof of Lemma~\ref{L:weak flatness1}, we see that base extension of
strictly pseudocoherent $\tilde{\calR}^\infty_R$-modules to $\tilde{\calR}^I_R$ is exact;
consequently, the functor from (a) to (c) is fully faithful.
On the other hand, let $M$ be a pseudocoherent $\varphi$-bundle over $\tilde{\calR}_R$. For any $r,s \in \QQ$
with $0 < s \leq r/p^{ah}$, as in Remark~\ref{R:Robba reified} we may form a reified quasi-Stein space admitting a uniform covering by the spaces $\Spra(\tilde{\calR}^{[sp^{nah}, rp^{nah}]}_R, \tilde{\calR}^{[sp^{nah}, rp^{nah}],\Gr}_R)$ for $n \in \ZZ$.
Since the modules $M_{[sp^{nah}, rp^{nah}]}$ are all pullbacks of each other via powers of $\varphi$, they can all be generated by the same number of elements;
we may thus apply Proposition~\ref{P:Stein space uniform covering} to see that $M$ is generated by finitely many global sections. By taking the kernel of the map to $M$ from the associated finite free $\varphi$-bundle and repeating the process, we see that the module of global sections is a pseudocoherent $\varphi$-module over $\tilde{\calR}^\infty_R$; this yields the equivalence of (a) and (c).
A similar argument adds (b) to this equivalence.

By a similar argument, (g), (h), (i) are equivalent.
\end{proof}

\begin{cor} \label{C:perfect fpd phi-modules}
The statement of Theorem~\ref{T:perfect generalized phi-modules} remains true if
we replace ``pseudocoherent'' with ``fpd'' everywhere, or if we restrict to projective modules and use the v-topology in place of the pro-\'etale topology.
\end{cor}
\begin{proof}
Combine Theorem~\ref{T:perfect generalized phi-modules} 
with Theorem~\ref{T:perfect Robba Kiehl}.
\end{proof}

\begin{cor} \label{C:lift action}
For each of the categories in Theorem~\ref{T:perfect generalized phi-modules}
or Corollary~\ref{C:perfect fpd phi-modules}, each object admits a strict epimorphism from an object corresponding to a projective $\varphi$-module over $\tilde{\calR}^{\infty}_R$; in fact, we may even ensure that the underlying module is free rather than merely projective.
\end{cor}
\begin{proof}
Apply Lemma~\ref{L:lift action1}, Theorem~\ref{T:perfect generalized phi-modules},
and Corollary~\ref{C:perfect fpd phi-modules}.
\end{proof}

\begin{theorem} \label{T:perfect generalized phi-modules cohomology}
The equivalences in Theorem~\ref{T:perfect generalized phi-modules}
and Corollary~\ref{C:perfect fpd phi-modules}
induce isomorphisms of $\varphi$-cohomology groups in cases (a)--(d),
sheaf cohomology groups in cases (e)--(f), and 
$\varphi$-hypercohomology groups on any of $X_{\et}, X_{\proet}$
(or $X_{\faith}$ in the context of Corollary~\ref{C:perfect fpd phi-modules}) in cases (g)--(i). (The reified analogue also holds.)
\end{theorem}
\begin{proof}
For projective $\varphi$-modules, the comparison in cases (a)--(d) follows as in \cite[Proposition~6.3.19]{part1},
while the comparison in cases (a) and (e) follows as in \cite[Theorem~8.7.13]{part1}. 
For pseudocoherent $\varphi$-modules, we may then deduce the comparison of cohomology in the same cases by forming projective resolutions using Corollary~\ref{C:lift action}; note that since we are using a projective resolution and not an injective resolution, this argument requires a uniform bound on the cohomology degrees, which in this case is 1. 
Finally, using Theorem~\ref{T:perfect Robba proetale cohomology},
we may compare (e) with (f) and (a)--(d) with (g)--(i).
\end{proof}

For arbitrary preadic spaces, we deduce the following corollary.

\begin{theorem} \label{T:perfect generalized phi-modules preadic}
The following exact tensor categories are equivalent.
\begin{enumerate}
\item[(a)]
The category of pseudocoherent $\varphi$-modules over $\tilde{\bC}^\infty_{X}$.
\item[(b)]
The category of pseudocoherent $\varphi$-modules over $\tilde{\bC}_{X}$.
\item[(c)]
The category of pseudocoherent $\varphi$-modules over $\tilde{\bC}^{[s,r]}_{X}$,
provided that $0 < s \leq r/p^{ah}$.
\end{enumerate}
These equivalences also induce isomorphisms of
$\varphi$-hypercohomology groups on $X_{\proet}$.
Moreover, similar statements hold if we replace ``pseudocoherent'' with ``fpd'' everywhere, 
or if we restrict to projective modules and replace $X_{\proet}$ with $X_v$.
(The reified analogue also holds.)
\end{theorem}
\begin{proof}
This is an immediate consequence of Theorem~\ref{T:perfect generalized phi-modules}
and Theorem~\ref{T:perfect generalized phi-modules cohomology}.
\end{proof}

\begin{remark}
Beware that we have not shown that an arbitrary pseudocoherent $\varphi$-module over $\tilde{\bC}_X$ admits a projective resolution by projective $\varphi$-modules
over $\tilde{\bC}_X$, even locally on $X$ (unless $X$ is perfectoid,
in which case Corollary~\ref{C:lift action} applies).
This adds some extra complications to a few arguments.
\end{remark}

\begin{remark}  \label{R:O as pseudocoherent}
By combining Definition~\ref{D:Robba theta}
with Theorem~\ref{T:perfect generalized phi-modules preadic}, we see that
$\calO(X)$ may be viewed as a pseudocoherent $\varphi$-module over 
$\tilde{\bC}_X^\infty$.
\end{remark}

The reader comparing this discussion with \cite[\S 8.7]{part1} may have noticed that we have so far  failed to mention the schematic Fargues-Fontaine curve. That is because the relationship between pseudocoherence on the adic and schematic curves is a bit fraught; we turn to this next.

\begin{defn}
Given a module $M$ over some $E$-algebra equipped with a semilinear $\varphi$-action,
for each $n \in \ZZ$ we define the twist 
$M(n)$ to be the same underlying module with the action of $\varphi$ multiplied by $\varpi^{-n}$.
\end{defn}

\begin{lemma} \label{L:pseudocoherent h1 vanishing}
For every finitely generated $\tilde{\calR}_R$-module $M$ equipped with a semilinear $\varphi$-action,
there exists an integer $N$ such that for all $n \geq N$,
$H^0_{\varphi}(M(n))$ generates $M$ as a module over $\tilde{\calR}_R$ and 
$H^1_{\varphi}(M(n)) = 0$ for all $n \geq N$.
\end{lemma}
\begin{proof}
By Lemma~\ref{L:lift action1}, the claim about $H^1$ reduces to the special case where $M$ is a projective $\varphi$-module, which follows as in \cite[Proposition~6.2.2]{part1}
(using Remark~\ref{R:homogeneity}). Given the claim about $H^1$, the claim about $H^0$ may be similarly deduced as in \cite[Proposition~6.2.4]{part1}.
\end{proof}

\begin{cor} \label{C:pseudocoherent to proj}
Let $0 \to M_1 \to M \to M_2 \to 0$ be a short exact sequence of finitely generated $\tilde{\calR}_R$-modules equipped with semilinear $\varphi$-actions. 
\begin{enumerate}
\item[(a)]
There exists $N \in \ZZ$ such that for all $n \geq N$,
\[
0 \to M_1(n)^{\varphi} \to M(n)^{\varphi} \to M_2(n)^{\varphi} \to 0
\]
is exact.
\item[(b)]
For $f \in \tilde{\calR}_R$ such that $\varphi(f) = \varpi^d f$ for some $d >0$,
put
\[
M_f = (M[f^{-1}])^{\varphi} = \bigcup_{n \in \ZZ} f^{-n} M(dn)^{\varphi}
\]
and define $M_{1,f}, M_{2,f}$ similarly. Then
\[
0 \to M_{1,f} \to M_f \to M_{2,f} \to 0
\]
is an exact sequence of finitely generated modules over $(\tilde{\calR}_R[f^{-1}])^{\varphi^a}$.
\item[(c)]
With notation as in (b), suppose that $M$ is a pseudocoherent $\varphi$-module over $\tilde{\calR}_R$. Then $M_f$ is a pseudocoherent module over $(\tilde{\calR}_R[f^{-1}])^{\varphi}$.
\end{enumerate}
\end{cor}
\begin{proof}
Part (a) is immediate from Lemma~\ref{L:pseudocoherent h1 vanishing}, as in 
\cite[Corollary~6.2.3]{part1}. Part (b) is immediate from (a),
as in \cite[Lemma~6.3.3]{part1}. To deduce (c), 
use Lemma~\ref{L:lift action1} to construct a projective resolution of $M$ consisting of projective $\varphi$-modules,
then apply (b) repeatedly (and Lemma~\ref{L:pseudocoherent 2 of 3}).
\end{proof}

We next recall the construction of \cite[Definition~6.3.10]{part1}.
\begin{defn} \label{D:proj functor}
Let $P_R$ be the graded ring defined in \cite[\S 6.3]{part1} in the case $\varpi = p$; that is, $P_{R,d}$ consists of those $f \in \tilde{\calR}^\infty_R$ for which $\varphi(f) = \varpi^d f$. The space $\Proj(P_R)$ is covered by the spaces
$\Spec(P_R[f^{-1}]_0)$ as $f$ runs over homogeneous elements $P_{R,+}$. For each such element, the map $P_R[f^{-1}]_0 \to \tilde{\calR}^\infty_R[f^{-1}]$ defines a map
$\Spec(\tilde{\calR}^{\infty}_R[f^{-1}]) \to \Spec(P_R[f^{-1}]_0)$.
By \cite[Lemma~6.3.7]{part1}, $P_{R,+}$ generates the unit ideal of $\tilde{\calR}_R$,
so we get a map $\Spec(\tilde{\calR}^\infty_R) \to \Proj(P_R)$ which is $\varphi$-equivariant for the trivial action on the target.
By pulling back and taking global sections, we obtain a right exact 
(by Theorem~\ref{T:perfect Robba Kiehl} and Corollary~\ref{C:pseudocoherent to proj})
functor from quasicoherent (resp. quasicoherent finitely generated, quasicoherent finitely presented) sheaves on $\Proj(P_R)$ to arbitrary (resp.\ finitely generated, finitely presented) $\tilde{\calR}^\infty_R$-modules equipped with semilinear $\varphi$-actions. Unfortunately, since we do not know that $P_R[f^{-1}]_0 \to \tilde{\calR}_R^\infty[f^{-1}]$ is flat, we cannot immediately see that the functor is left exact, that it takes pseudocoherent sheaves to pseudocoherent modules, or that the resulting modules are complete.
\end{defn}

\begin{theorem} \label{T:algebraic equivalence}
The construction of Definition~\ref{D:proj functor} defines an exact equivalence of categories between the category of pseudocoherent sheaves (in the sense of \cite{sga6}) 
on $\Proj(P_R)$ (the schematic relative Fargues-Fontaine curve)
and the category of pseudocoherent $\tilde{\calR}_R$-modules equipped with isomorphisms with their $\varphi$-pullbacks. Moreover, under this equivalence, sheaf cohomology corresponds to $\varphi$-cohomology.
\end{theorem}
\begin{proof}
Let $V$ be a pseudocoherent sheaf on $\Proj(P_R)$.
As in \cite[Lemma~8.8.4]{part1}, there exists some $n \geq 0$ such that $V(n)$ is generated by global sections; consequently, we can construct an exact sequence
\begin{equation} \label{eq:perfect generalized phi-modules1}
0 \to V_1 \to V_2 \to V \to 0
\end{equation}
of pseudocoherent sheaves on $\Proj(P_R)$ in which $V_2$ is a vector bundle. Apply
Definition~\ref{D:proj functor} to obtain an exact sequence
\[
M_1 \to M_2 \to M \to 0
\]
of $\tilde{\calR}_R$-modules equipped with semilinear $\varphi$-actions,
in which $M_1$ is finitely generated, $M_2$ is finite projective, and so $M$ is finitely presented. Put $K = \ker(M_1 \to \ker(M_2 \to M))$; by Corollary~\ref{C:pseudocoherent to proj}, for
$d>0$ and $f \in P_{R,d}$,
\[
0 \to K_f \to M_{1,f} \to M_{2,f} \to M_f \to 0
\]
is an exact sequence of $P_R[f^{-1}]_0$-modules.
Now consider the commutative diagram
\[
\xymatrix{
& 0 \ar[r] & V_{1,f} \ar[r] \ar[d] & V_{2,f} \ar[r] \ar[d] & V_f \ar[d] \ar[r] & 0\\
0 \ar[r] & K_f \ar[r] & M_{1,f} \ar[r] & M_{2,f} \ar[r] & M_f \ar[r] & 0
}
\]
in which the top row is obtained from \eqref{eq:perfect generalized phi-modules1}
by taking sections over the open affine subspace $\Spec(P_R[f^{-1}]_0)$ of $\Proj(P_R)$.
As in \cite[Theorem~6.3.12]{part1}, the map $V_{2,f} \to M_{2,f}$ is an isomorphism;
this implies that $V_f \to M_f$ is surjective. However, since $V_1$ is itself pseudocoherent, we may repeat the argument with $V$ replaced by $V_1$ to see that $V_{1,f} \to M_{1,f}$ is also surjective. By the five lemma, we may now deduce that $V_f \to M_f$ is injective, and then repeat the argument with $V$ replaced by $V_1$ to see that $V_{1,f} \to M_{1,f}$ is injective. That is, the vertical arrows are all isomorphisms, so $K_f = 0$. By Lemma~\ref{L:pseudocoherent h1 vanishing}, this implies that $K=0$;
that is, the functor from pseudocoherent sheaves on $\Proj(P_R)$ to $\tilde{\calR}_R$-modules is exact. Since \cite[Theorem~6.3.12]{part1} ensures that this functor takes vector bundles to finite projective modules, we may deduce that the essential image of this functor consists of pseudocoherent $\tilde{\calR}_R$-modules. 

At this point, we know that Definition~\ref{D:proj functor} defines an exact functor from pseudocoherent sheaves on $\Proj(P_R)$ to pseudocoherent  $\tilde{\calR}_R$-modules equipped with semilinear $\varphi$-actions, and that Corollary~\ref{C:pseudocoherent to proj} defines an exact functor in the opposite direction.
It is obvious that the composition from $\tilde{\calR}_R$-modules and back is an equivalence.
In the other direction from \eqref{eq:perfect generalized phi-modules1}, the fact that the composition from $\Proj(P_R)$ and back is a quasi-equivalence reduces to the case of vector bundles, for which we may appeal to \cite[Theorem~6.3.12]{part1}.
This completes the proof.
\end{proof}
\begin{cor} \label{C:pseudocoherent sheaf as vb quotient}
Every pseudocoherent sheaf on $\Proj(P_R)$ is a quotient of some vector bundle.
\end{cor}
\begin{proof}
Combine Lemma~\ref{L:lift action1} with Theorem~\ref{T:algebraic equivalence}. 
\end{proof}

\begin{remark}
The discrepancy between Theorem~\ref{T:algebraic equivalence} and Theorem~\ref{T:perfect generalized phi-modules} is that the modules coming out of the former are not guaranteed to be complete for the natural topology, and so cannot be viewed as sheaves on the adic Fargues-Fontaine curve using the framework we have developed. This difficulty disappears in the case of a point; see \S\ref{subsec:coherent FF}.
\end{remark}

\subsection{Coherent sheaves on Fargues-Fontaine curves}
\label{subsec:coherent FF}

We now restrict back to the case where the base space is a point. In this case,  additional noetherian properties make it possible to establish a GAGA-style theorem for coherent sheaves on Fargues-Fontaine curves (as promised in the introduction of \cite{kedlaya-noetherian}).
\begin{hypothesis}
Throughout \S\ref{subsec:coherent FF}, let $R = L$ be a perfect analytic field of characteristic $p$ and put $R^+ = \gotho_L$.
\end{hypothesis}

\begin{remark} \label{C:pseudocoherent type A fpd}
In this setting, the rings $\tilde{\calR}^{\inte}_L$ and $\tilde{\calE}_L$ are both discrete valuation rings. Consequently, the statement of Corollary~\ref{C:pseudocoherent type A flat} can be formally upgraded: every finitely generated module equipped with a semilinear $\varphi$-action is a pseudocoherent $\varphi$-module, and is moreover fpd of global dimension at most 1.
\end{remark}

\begin{theorem} \label{T:curve noetherian}
The following statements hold.
(Here $P_L = P_R$ is as in Definition~\ref{D:proj functor}.)
\begin{enumerate}
\item[(a)]
For every nonzero homogeneous $f \in P_L$, the ring $P_L[f^{-1}]_0$ is noetherian. Moreover, it is a Dedekind domain, and if $L$ is algebraically closed it is even a principal ideal domain.
\item[(b)]
For $0 < s \leq r$, the ring $\tilde{\calR}^{[s,r]}_L$ is strongly noetherian. Moreover,
for every rational localization (resp. \'etale morphism) $(\tilde{\calR}^{[s,r]}_L, \tilde{\calR}^{[s,r],+}_L)
\to (B,B^+)$, the ring $B$ is a finite direct sum of principal ideal domains (resp.\ Dedekind domains).
\item[(c)]
For $0 < s \leq r$ and $n$ a nonnegative integer, every maximal ideal of  $\tilde{\calR}^{[s,r]}_L\{T_1,\dots,T_n\}$ contracts to a maximal ideal of $\tilde{\calR}^{[s,r]}_L$.
\item[(d)]
For $0 < s \leq r$ and $n$ a nonnegative integer, the ring $\tilde{\calR}^{[s,r]}_L\{T_1,\dots,T_n\}$ is regular of dimension $n+1$.
\end{enumerate}
\end{theorem}
\begin{proof}
For (a), see \cite{fargues-fontaine}.
For (b), see \cite[Theorem~4.10, Theorem~7.11, Theorem~8.9]{kedlaya-noetherian}.
(Note that this derivation follows Remark~\ref{R:Robba reified2}, by first proving that
$\tilde{\calR}^{r}_L$ is strongly noetherian.)
For (c), see \cite{wear}. For (d), combine (a) and (c) with the Nullstellensatz for Tate algebras over an analytic field \cite[Proposition~7.1.1/3]{bgr}.
\end{proof}

\begin{theorem} \label{T:perfect generalized phi-modules at a point}
Let $F$ be a perfectoid analytic field corresponding to the perfect analytic field $L$,
and put $X = \Spa(F, \gotho_F)$. Then the following exact tensor categories are equivalent (and abelian).
\begin{enumerate}
\item[(a)]
The category of coherent $\varphi$-modules over $\tilde{\calR}^\infty_L$.
\item[(b)]
The category of coherent $\varphi$-modules over $\tilde{\calR}_L$.
\item[(c)]
The category of coherent $\varphi$-bundles over $\tilde{\calR}_L$.
\item[(d)]
The category of coherent $\varphi$-modules over $\tilde{\calR}^{[s,r]}_L$,
provided that $0 < s \leq r/p^{ah}$.
\item[(e)]
The category of coherent sheaves on $\Proj(P_L)$.
\item[(f)]
The category of coherent sheaves on $\FFC_L$.
\item[(g)]
The category of coherent sheaves on $\FFC_{L,\et}$.
\item[(h)]
The category of fpd sheaves on $\FFC_{L,\proet}$.
\item[(i)]
The category of fpd $\varphi$-modules over $\tilde{\bC}^\infty_{X}$.
\item[(j)]
The category of fpd $\varphi$-modules over $\tilde{\bC}_{X}$.
\item[(k)]
The category of coherent $\varphi$-bundles over $\tilde{\bC}_{X}$.
\item[(l)]
The category of coherent $\varphi$-modules over $\tilde{\bC}^{[s,r]}_{X}$,
provided that $0 < s \leq r/p^{ah}$.
\end{enumerate}
More precisely, the sheaves in (i)--(l) may be taken on any of $X,X_{\et},X_{\proet}$. If we restrict to projective modules, we may also replace the pro-\'etale topology with the v-topology.
\end{theorem}
\begin{proof}
Note that the following rings are Pr\"ufer domains, and hence coherent rings by
Remark~\ref{R:Bezout torsion}.
\begin{itemize}
\item
The ring $\tilde{\calR}^{[s,r]}_L$, by Theorem~\ref{T:curve noetherian}. Alternatively, apply \cite[Proposition~2.6.8]{kedlaya-revisited} in case $E$ is of mixed characteristic,
and otherwise realize $\tilde{\calR}^{[s,r]}_L$ as a Berkovich affinoid algebra over $L$.
\item
The rings $\tilde{\calR}^r_L$ for all $r$ (including $r = \infty$), by
Theorem~\ref{T:curve noetherian} plus Corollary~\ref{C:quasi-Stein Prufer}.
\item
The ring $\tilde{\calR}_L$, by virtue of being the direct limit of the Pr\"ufer domains $\tilde{\calR}^r_L$ via flat morphisms.
Alternatively, in case $E$ is of mixed characteristic, 
see \cite[Theorem 2.9.6]{kedlaya-revisited}.
\item
Any Banach ring $B$ arising from an etale morphism
$(\tilde{\calR}^{[s,r]}_L, \tilde{\calR}^{[s,r],+}_L) \to (B,B^+)$,
by Theorem~\ref{T:curve noetherian}.
\item
The ring $P_L[f^{-1}]_0$ for any homogeneous $f \in P_L$, by Theorem~\ref{T:curve noetherian} or \cite[Lemma~6.3.6]{part1}.
\end{itemize}
This allows us to reduce to the corresponding statement with every occurrence of ``coherent'' or ``fpd'' replaced by ``pseudocoherent''. We thus deduce the claim from
Theorem~\ref{T:perfect generalized phi-modules} as supplemented by
Corollary~\ref{C:perfect fpd phi-modules}
and Theorem~\ref{T:algebraic equivalence} (to bring (e) into the equivalence).
\end{proof}

\begin{theorem} \label{T:perfect generalized phi-modules cohomology at a point}
The equivalences in Theorem~\ref{T:perfect generalized phi-modules}
induce isomorphisms of $\varphi$-cohomology groups in cases (a)--(d),
sheaf cohomology groups in cases (e)--(h), and 
$\varphi$-hyper\-cohomology groups on any of $X,X_{\et}, X_{\proet}$ in cases (i)--(l). If we restrict to projective modules, we may also replace the pro-\'etale topology with the v-topology.
\end{theorem}
\begin{proof}
This is immediate from Theorem~\ref{T:perfect generalized phi-modules cohomology}.
\end{proof}

\begin{remark} \label{R:GAGA}
Theorems~\ref{T:perfect generalized phi-modules at a point}
and Theorem~\ref{T:perfect generalized phi-modules cohomology at a point} include an assertion in the spirit of the GAGA principle for rigid analytic spaces \cite[Example~3.2.6]{conrad}: the morphism $\FFC_L \to \Proj(P_L)$ induces an equivalence of categories of coherent sheaves, as well as isomorphisms between sheaf cohomology groups of coherent sheaves. 
However, this morphism does not obey the universal property associated with  the analytification of a scheme locally of finite type over either $\CC$ or a nonarchimedean analytic field \cite[Expos\'e~XII, \S 1]{sga1}; that is, for $Y$ an adic space over $E$, a morphism $Y \to \Proj(P_L)$ of locally ringed spaces is not guaranteed to factor uniquely through $\FFC_L$. To recover the universal property, one must add the extra assumption that the factorization exists at the level of points; that is, for each $y \in Y$ mapping to $x \in \Proj(P_L)$, the seminorm on $\kappa(x)$ induced by $y$ must correspond to some element of $\FFC_L$. (Similar considerations apply to the map $\FFC_R \to \Proj(P_R)$.)
\end{remark}

\begin{remark}
One consequence of Theorem~\ref{T:perfect generalized phi-modules} is that for any analytic field $F$ over $E$, for $X = \Spa(F, \gotho_F)$, the category of pseudocoherent $\varphi$-modules over $\tilde{\bC}_X$ is abelian (namely, this reduces to the case of a perfectoid field). One cannot hope for this to hold in general. For instance, suppose $X = \Spa(A,A^+)$ is affinoid perfectoid. Then every pseudocoherent $A$-module arises as a pseudocoherent $\varphi^a$-module over $\tilde{\bC}_X$ (see Remark~\ref{R:O as pseudocoherent}), and pseudocoherent $A$-modules do not form an abelian category if $A$ is not coherent
(see Remark~\ref{R:noetherian pseudoflat}). On the other hand, if $X$ is strongly noetherian,  the hope
may be more justified; see for example Theorem~\ref{T:abelian category phi}.
\end{remark}

\begin{remark} \label{R:phi-gamma point Robba}
In Theorem~\ref{T:curve noetherian}(b), the given reference shows that the ring $\tilde{\calR}^{[s,r]}_{L}$ is in fact really strongly noetherian (see Remark~\ref{R:really strongly noetherian}). With this observation,
the obvious reified analogues of 
Theorem~\ref{T:perfect generalized phi-modules at a point}
and Theorem~\ref{T:perfect generalized phi-modules cohomology at a point} hold, with corresponding proofs.
\end{remark}

\subsection{Pseudocoherent \texorpdfstring{$B$}{B}-pairs}
\label{subsec:pseudocoherent B-pairs}

Returning to the general setting, we attempt to extend the equivalence among categories of 
pseudocoherent $\varphi$-modules of type $\bC$ to an analogue of $B$-pairs. In light of Remark~\ref{R:no pseudoflat completion}, this requires some care.

\begin{hypothesis}
Throughout \S\ref{subsec:pseudocoherent B-pairs}, fix a Fontaine perfectoid adic Banach ring $(A,A^+)$
over $\gotho_E$ corresponding to $(R,R^+)$ as in Theorem~\ref{T:Fontaine perfectoid correspondence}.
\end{hypothesis}

\begin{defn} \label{D:B-pair}
Let $\bB^+_{\dR,R}, \bB_{\mathrm{e},R}, \bB_{\dR,R}$ be the rings denoted $R_1, R_2, R_{12}$
in \cite[Definition~8.9.4]{part1}. To make these more explicit, choose an open affine subscheme $U$ of $\Proj(P_R)$ containing $\Spec(A)$ 
such that for $S$ the coordinate ring of $U$, the ideal $J$ of $S$ cutting out $\Spec(A)$ is principal
(see \cite[Remark~8.7.6]{part1}).
Then $\bB^+_{\dR, R}$ is the $J$-adic completion of $S$;
$\bB_{\mathrm{e}, R}$ is the coordinate ring of the complement of $\Spec(A)$ in
$\Proj(P_R)$; and $\bB_{\dR,R}$ is the coordinate ring of the fibred product of 
$\Spec(\bB^+_{\dR, R})$ and $\Spec(\bB_{\mathrm{e},R})$ over $\Proj(P_R)$, which is also the localization of $S$ in which a generator of $J$ is inverted.

A \emph{projective (resp.\ pseudocoherent, fpd) $B$-pair} over $R$ is a tuple $(M_{\dR}^+, M_{\mathrm{e}}, M_{\dR}, \iota_{\dR}^+, \iota_{\mathrm{e}})$ in which $M_{\dR}^+, M_{\mathrm{e}}, M_{\dR}$
are finite projective (resp.\ pseudocoherent, fpd) modules over the respective rings $\bB^+_{\dR,R}, \bB_{\mathrm{e},R}, \bB_{\dR, R}$ and $\iota_{\dR}^+:M_{\dR}^+ \otimes_{\bB^+_{\dR,R}} \bB_{\dR,R} \to M_{\dR}$,
$\iota_{\mathrm{e}}:  M_{\mathrm{e}}\otimes_{\bB_{\mathrm{e},R}} \bB_{\dR,R} \to M_{\dR}$ are isomorphisms.
Define the \emph{$B$-cohomology groups} of a pseudocoherent $B$-pair to be the kernel and cokernel of the map
\[
M_{\dR}^+ \oplus M_{\mathrm{e}} \to M_{\dR}, \qquad (\bv_1, \bv_2)
\mapsto \iota_{\dR}^+(\bv_1) - \iota_{\mathrm{e}}(\bv_2).
\]
\end{defn}

\begin{lemma} \label{L:cover B-pair}
Every pseudocoherent $B$-pair over $R$ is a quotient of some projective $B$-pair.
\end{lemma}
\begin{proof}
We emulate the proof of \cite[Lemma~1.5.2]{part1}.
Set notation as in Definition~\ref{D:B-pair}. For some positive integer $n$, we can construct generators $\bv_1,\dots,\bv_n$ of $M_{\dR}^+$ and
$\bw_1,\dots,\bw_n$ of $M_{\mathrm{e}}$.
There then exist $n \times n$ matrices $A,B$ over $\bB_{\dR,R}$ such that
\[
\iota_{\mathrm{e}}(\bw_j) = \sum_i A_{ij} \iota_{\dR}^+(\bv_i), \qquad
\iota_{\dR}^+(\bv_j) = \sum_i B_{ij} \iota_{\mathrm{e}}(\bw_i).
\]
Let $F_{\dR}^+ \to M_{\dR}^+, F_{\mathrm{e}} \to M_{\mathrm{e}}$ be the surjective morphisms of modules corresponding to the chosen generating sets, so that
$F_{\dR}^+$, $F_{\mathrm{e}}$ are free modules of rank $n$ over $\bB_{\dR}^+$,
$\bB_{\mathrm{e}}$, respectively. Let $N_{\dR}^+$, $N_{\mathrm{e}}$ be the respective kernels of these morphisms.
By writing
\[
\iota_{\dR}^+(\bv_k) = \sum_j B_{jk} \iota_{\mathrm{e}}(\bw_j) = 
\sum_{i,j} A_{ij} B_{jk} \iota_{\dR}^+(\bv_i),
\]
we see that the columns of the matrix $C = AB-1$ belong to $N_{\dR}^+$.

Let $D$ be the block matrix $\begin{pmatrix} A & C \\ 1 & B \end{pmatrix}$;
by using row operations to clear the top left block, we see that
$\det(D) = \det(AB-C) = 1$. In particular, $D$ is invertible,
so it defines an isomorphism 
\begin{equation} \label{eq:B-pair construction}
(F_{\mathrm{e}} \oplus F_{\mathrm{e}}) \otimes_{\bB_{\mathrm{e}}} \bB_{\dR} \to
(F_{\dR}^+ \oplus F_{\dR}^+) \otimes_{\bB_{\dR}^+} \bB_{\dR}.
\end{equation}
This morphism carries $(N_{\mathrm{e}} \oplus F_{\mathrm{e}})  \otimes_{\bB_{\mathrm{e}}} \bB_{\dR}$ into
$(N_{\dR}^+ \oplus F_{\dR}^+) \otimes_{\bB_{\dR}^+} \bB_{\dR}$.
Consequently, we may interpret \eqref{eq:B-pair construction} as defining a projective $B$-pair mapping surjectively to the original one.
\end{proof}

\begin{defn} \label{D:pullback to B-pair}
Let $\calF$ be a pseudocoherent sheaf on $\Proj(P_R)$. By pulling back along the arrows in the cartesian square
\[
\xymatrix{
\Spec(\bB_{\dR,R}) \ar[r] \ar[d] & \Spec(\bB_{\dR,R}^+) \ar[d] \\
\Spec(\bB_{\mathrm{e},R}) \ar[r] & \Proj(P_R)
}
\]
we obtain a tuple $(M_{\dR}^+, M_{\mathrm{e}}, M_{\dR}, \iota_{\dR}^+, \iota_{\mathrm{e}})$ in which $M_{\dR}^+, M_{\mathrm{e}}, M_{\dR}$ as in the definition of a pseudocoherent $B$-pair, with one deficiency: the modules $M_{\dR}^+, M_{\dR}$ are not known to be pseudocoherent. The problem is that the horizontal arrows in the square are open immersions and hence flat, but the vertical arrows are not known to be flat
(Remark~\ref{R:no pseudoflat completion}).
\end{defn}

One circumstance in which we can guarantee pseudocoherence of $M_{\dR}^+$ is the following.
\begin{defn}
We say that a pseudocoherent sheaf on $\Proj(P_R)$ is \emph{$\theta$-finite} if its $J$-power-torsion submodule is finitely generated. We say that a pseudocoherent $B$-pair if the $J$-power-torsion submodule of $M_{\dR}^+$ is finitely generated.
\end{defn}

\begin{theorem} \label{T:pseudocoherent B-pairs}
The construction of Definition~\ref{D:pullback to B-pair} defines an equivalence of exact tensor categories
between $\theta$-finite pseudocoherent sheaves on $\Proj(P_R)$ 
and $\theta$-finite pseudocoherent $B$-pairs over $R$.
Moreover, this equivalence identifies $\varphi$-cohomology groups 
with $B$-cohomology groups.
\end{theorem}
\begin{proof}
Retain notation as in Definition~\ref{D:B-pair} and let $f$ be a generator of $J$; then $z$ is not a zero-divisor in $S$.The category of $\theta$-finite pseudocoherent sheaves on $\Proj(P_R)$ is equivalent to the category of descent data for pseudocoherent modules with finite $z$-torsion with respect to the diagram
$\bB_{\mathrm{e},R} \to S_z \leftarrow S$ (the $z$-torsion condition only being applied over $S$).
By Proposition~\ref{P:Beauville-Laszlo with limited torsion},
the category of pseudocoherent modules over $S$ with finite $z$-torsion is equivalent to the category of descent data for pseudocoherent modules with finite $z$-torsion with respect to the diagram
$\bB_{\dR,R}^+ \to \bB_{\dR,R} \leftarrow S_z$. Combining these results proves the claim.
\end{proof}

\subsection{Slopes of \texorpdfstring{$\varphi$}{phi}-modules}
\label{subsec:slopes}

We next recall the basic theory of slopes for $\varphi$-modules of type $\bC$, and use it to describe the relationship between projective $\varphi$-modules and \'etale $E$-local systems. 
In the process, we use descent for the v-topology, as provided by Remark~\ref{R:local systems lattice}, to simplify the arguments of \cite[\S 7]{part1}.

We begin by recalling the analogue of the Dieudonn\'e-Manin classification for $\varphi$-modules of type $\bC$ over a point, as in \cite[Proposition~4.2.16]{part1}.

\begin{defn}
For $m$ a positive integer, let $\Psi_m^*$ denote the restriction functor from $\varphi$-modules to $\varphi^m$-modules. Let $\Psi_{m*}$ denote the functor from $\varphi^m$-modules to $\varphi$-modules taking $M$ to $M \oplus \varphi^*(M) \oplus \cdots \oplus (\varphi^{m-1})^*(M)$.
\end{defn}

We next extend \cite[Proposition~4.2.15]{part1} to cover arbitrary $E$.

\begin{theorem} \label{T:DM}
Suppose that $R = L$ is an algebraically closed analytic field.
Then every $\varphi$-module over $\tilde{\calR}_L$ decomposes as a direct sum
of submodules, each of the form $\Psi_{d*} (\tilde{\calR}_L(c))$ for some pair of integers $c,d$ with $d>0$ and $\gcd(c,d) = 1$. (Such a decomposition is called a
\emph{Dieudonn\'e-Manin decomposition} of the $\varphi$-module.)
\end{theorem}
\begin{proof}
See \cite[Theorem~4.5.7]{kedlaya-revisited} in the case where $E$ is of mixed characteristic,
and \cite[Theorem~11.1]{hartl-pink} in the case where $E$ is of equal characteristic.
A distinct derivation will also be included in \cite{fargues-fontaine}.
\end{proof}

\begin{defn} \label{D:pure phi-module}
For $m$ a positive integer, let $\Psi_m^*$ denote the restriction functor from $\varphi$-modules to $\varphi^m$-modules. Let $\Psi_{m*}$ denote the functor from $\varphi^m$-modules to $\varphi$-modules taking $M$ to $M \oplus \varphi^*(M) \oplus \cdots (\varphi^{m-1})^*(M)$.

For $s = c/d$ a rational number, we say that a $\varphi$-module $M$ over $\tilde{\bC}_*$ is \emph{pure of slope $s$} if $(\Psi_m^*(M))(-d)$ is an \'etale $\varphi^m$-module.
Using Theorem~\ref{T:etale type C}, one checks easily that this definition depends only on $s$ and not on the representation of $s$ as $c/d$. In particular, \emph{\'etale} is the same as \emph{pure of slope $0$}.
\end{defn}

\begin{defn} \label{D:valuation function}
The sets of units in the rings $\tilde{\calR}^{\bd}_R$ and $\tilde{\calR}_R$ coincide
\cite[Corollary~5.2.3]{part1}. This means that each unit in $\tilde{\calR}_R$ has a well-defined valuation, as a continuous function from $\Spa(R,R^+)$ to $v(E^\times) \cong \ZZ$; note that the valuation map is $\varphi$-invariant.
\end{defn}

\begin{theorem} \label{T:pure locus is open}
Let $M$ be a $\varphi$-module over $\tilde{\bC}_X$ such that for some map $Y \to X$ where
$Y = \Spa(L, \gotho_L)$ for some perfectoid field $L$,
$M \otimes_{\tilde{\bC}_X} \tilde{\bC}_{Y}$ is pure of slope $s$. 
Then there exists a neighborhood $U$ of the image of $Y$ in $X$ such that $M|_U$ is pure of slope $s$. (The reified analogue also holds.)
\end{theorem}
\begin{proof}
We reduce immediately to the case where $s=0$;
by Lemma~\ref{L:perfectoid cover}, we may also assume that $X$ is perfectoid.
Let $x \in X$ be the image of the unique point of $Y$.
We may then construct a (not necessarily pro-\'etale) v-covering $X'$ of $X$ such that $x$ lifts to a point $x' \in X'$ with $\calH(x')$ algebraically closed and containing $L$.
By Theorem~\ref{T:etale type C} and Remark~\ref{R:local systems lattice},
we may reduce to the case $X' = X$, $x' = x$.

In this case, we may combine
Theorem~\ref{T:DM} with \cite[Lemma~7.1.2]{part1} (as extended to general $E$ using Remark~\ref{R:homogeneity}) to show that for some $U = \Spa(R,R^+)$, $M$ admits a basis on which $\varphi$ acts via an integral matrix over $\tilde{\calR}^{\inte}_R$;
this proves the claim.
\end{proof}

\begin{defn}
For $M$ a $\varphi$-module over $\tilde{\bC}_X$,  a \emph{slope filtration} of $M$
is a filtration $0 = M_0 \subset \cdots \subset M_l = M$ by $\varphi$-submodules such that each quotient $M_i/M_{i-1}$ is a $\varphi$-module which is pure of some slope $s_i$,
and $s_1 > \cdots > s_l$.
\end{defn}

\begin{lemma} \label{L:H0 by slope}
Let $M_1, M_2$ be $\varphi$-modules over $\tilde{\bC}_X$ which are pure of slopes $s_1, s_2$. If $s_1 > s_2$, then $\Hom_{\varphi}(M_1, M_2) = 0$.
(The reified analogue also holds.)
\end{lemma}
\begin{proof}
Put $M = M_1^\dual \otimes M_2$; then $M$ is pure of slope $s = s_2 - s_1 < 0$ and we wish to check that $H^0_\varphi(M) = 0$. By applying $\Psi_m^*$ as needed, we may further reduce to the case where $s \in \ZZ$; by making a pro-\'etale cover, we may further reduce to the case $M \cong \tilde{\bC}_X(s)$. So suppose then that $x \in \tilde{\calR}_R$ satisfies $\varphi(x) = \varpi^m x$ for some integer $m<0$.
By writing
\[
\lambda(\alpha^{rp^{-ah}})(x) = \lambda(\alpha^{rp^{-ah}})(\varpi^{-m} \varphi(x))
= p^m \lambda(\alpha^r)(x),
\]
we see that $\lambda(\alpha^r)(x)$ remains bounded as $r \to 0^+$, which by \cite[Lemma~5.2.2]{part1} forces $x \in \tilde{\calR}^{\bd}_R$. 
But the valuation of $x$ is $\varphi$-invariant (see Definition~\ref{D:valuation function}), so we get a contradiction against the fact that $m \neq 0$.
\end{proof}

\begin{cor} \label{C:unique slope filtration}
A slope filtration of a $\varphi$-module over $\tilde{\bC}_X$ is unique if it exists.
(The reified analogue also holds.)
\end{cor}

\begin{theorem} \label{T:slope filtration}
Suppose that $R = L$ is an analytic field.
Then every $\varphi$-module $M$ over $\tilde{\calR}_L$ admits a (unique) slope filtration.
\end{theorem}
\begin{proof}
By Corollary~\ref{C:unique slope filtration} the filtration is unique if it exists, and 
moreover we can check its existence after making a v-covering. We may thus assume that $L$ is algebraically closed, in which case Theorem~\ref{T:DM} proves the claim.
\end{proof}

\begin{defn} \label{D:HN polygon}
From Definition~\ref{D:valuation function}, we obtain a canonical map
\[
\Pic(\FFC_R) \to \Cont(\Spa(R, R^+), \ZZ),
\]
which we normalize so that twisting by $n$ increases the degree uniformly by $n$.
For $M$ a projective $\varphi$-module over $\tilde{\bC}_X$ of rank $n>0$,
we define the \emph{degree} of $M$ to be the image of $\wedge^n M$ in $\Cont(\left| X \right|, \ZZ)$; this definition
extends uniquely to a degree function on pseudocoherent $\varphi^a$-modules which is additive on short exact sequences.
For $M$ projective of everywhere nonzero rank, we define the \emph{slope} of $M$ to be $\mu(M) = \deg(M)/\rank(M)$.

Suppose now that $R = L$ is an analytic field. We say that a projective $\varphi^a$-module $M$ over $\tilde{\calR}_L$ is \emph{semistable} if there exists no nonzero proper $\varphi$-submodule $N$ of $M$ such that $\mu(N) > \mu(M)$.
By Lemma~\ref{L:H0 by slope} and Theorem~\ref{T:slope filtration},
$M$ is semistable if and only if it is pure of slope $\mu(M)$.

With notation as in Theorem~\ref{T:slope filtration}, we define the \emph{Harder-Narasimhan (HN) polygon} to be the convex polygonal line originating at $(0,0)$ associated to the multiset containing $\mu(M_i/M_{i-1})$ with multiplicity $\rank(M_i/M_{i-1})$ for $i=1,\dots,l$. This can be defined using only the definition of semistability; the comparison with purity implies the behavior of the polygon under tensor products, as well as stability under arbitrary base extension.
The latter ensures that for $M$ a $\varphi$-module over $\tilde{\bC}_X$, we may view the HN polygon as a function from $\left| X \right|$ to the space of polygons; note that the right endpoint is
$(\rank(M), \deg(M))$, which is locally constant. 
Note also that the HN polygon at a point $x \in X$ depends only on the analytic field $\calH(x)$; that is, the fibers of the HN polygon map are \emph{partially proper} sets.
\end{defn}

We next extend \cite[Proposition~7.4.3]{part1}.
\begin{lemma} \label{L:compare filtrations}
Let $R = L$ be an analytic field. Let $L'$ be an analytic field containing $L$, but for which the norm on $L'$ restricts trivially to $L$. Let 
$M$ be a $\varphi$-module over $\tilde{\calR}^{\bd}_L$. Then the HN polygon of
$M \otimes_{\tilde{\calR}^{\bd}_L} \tilde{\calR}_{L'}$ lies on or above the HN polygon of
$M \otimes_{\tilde{\calR}^{\bd}_L} \tilde{\calR}_{L}$, with the same endpoints.
\end{lemma}
\begin{proof}
We may assume without loss of generality that $L$ is algebraically closed
and that the HN polygon of $M \otimes_{\tilde{\calR}^{\bd}_L} \tilde{\calR}_{L'}$ has integral slopes.
By the usual Dieudonn\'e-Manin classification theorem (e.g., see \cite[Corollary~14.6.4]{kedlaya-course}), $M \otimes_{\tilde{\calR}^{\bd}_L} \tilde{\calE}_L$ admits a basis on which $\varphi$ acts via a diagonal matrix whose entries are powers of $\varpi$
(the powers of which are the slopes of $M \otimes_{\tilde{\calR}^{\bd}_L} \tilde{\calR}_{L'}$). Form the ``reverse filtration'' of $M \otimes_{\tilde{\calR}^{\bd}_L} \tilde{\calE}_L$ by taking the first subobject to be generated by the basis vectors of \emph{smallest} slope, and so on;
by calculating as in the proof of \cite[Lemma~5.4.1]{part1}, we see that this filtration
descends to $M$. By base-extending this filtration to $\tilde{\calR}_L$ and then comparing with the slope filtration as in \cite[Proposition~3.5.4]{part1},
we obtain the desired comparison.
\end{proof}

\begin{theorem}
Let $M$ be a $\varphi$-module over $\tilde{\bC}_X$.
\begin{enumerate}
\item[(a)]
The HN polygon of $M$ is bounded above and below, and lower semicontinuous as a function on $\left| X \right|$. Here the space of polygons is equipped with the discrete topology and the partial ordering for which $\calP_1 \leq \calP_2$ if the two polygons have the same width and endpoints and the pointwise comparison holds.
\item[(b)]
If the HN polygon of $M$ is a constant function, then $M$ admits a unique filtration pulling back to the HN filtration over any point. (By Theorem~\ref{T:pure locus is open},
the successive quotients of this filtration are pure.)
\end{enumerate}
(The reified analogue also holds.)
\end{theorem}
\begin{proof}
In both cases, we may work locally around some $x \in X$; we may thus argue as in Theorem~\ref{T:pure locus is open} to reduce to the case where $X$ is perfectoid and $\calH(x)$ is algebraically closed. 
We then argue as in the proofs of 
\cite[Theorem~7.4.5, Proposition~7.4.6]{part1}:
by combining Theorem~\ref{T:DM} and \cite[Lemma~7.1.2, Lemma~7.4.4]{part1},
we obtain a neighborhood $U = \Spa(R,R^+)$ of $x$ in $X$ on which $M$ admits a basis
on which $\varphi$ acts via a matrix over $\tilde{\calR}^{\bd}_R$
whose classical Dieudonn\'e-Manin slope polygon (i.e., as computed over $\tilde{\calE}_R$)
coincides with the HN polygon of $M$ at $x$. By Lemma~\ref{L:compare filtrations},
the HN polygon of $M$ at each point of $U$ lies on or above the HN polygon at $x$; this proves (a).

To prove (b), we argue as in \cite[Theorem~7.4.9]{part1}.
Let $M_0$ be the $\tilde{\calR}^{\bd}_R$-span of the basis constructed in (a);
then $M_0$ admits a ``reverse filtration'' as in the proof of Lemma~\ref{L:compare filtrations}, and it suffices to split this filtration under the hypothesis of matching polygons. As in \cite[Lemma~7.4.8]{part1}, this reduces to checking the splitting pointwise, which we confirm by following the proof of \cite[Theorem~5.5.2]{kedlaya-revisited}.
\end{proof}

\section{Imperfect period rings: an axiomatic approach}
\label{sec:axiomatic}

We next set up an axiomatic framework in which we can construct and study imperfect period rings. This plays a role analogous to the
\emph{Tate-Sen formalism} used by Berger-Colmez \cite{berger-colmez}, but to avoid complications with higher ramification theory, we impose some limitations on the types of towers to be considered.

\setcounter{theorem}{0}

\begin{hypothesis}
Throughout \S\ref{sec:axiomatic}, retain Hypothesis~\ref{H:towers}, and
 let $(A,A^+)$ be a uniform adic Banach algebra over $\gotho_E$ with spectral norm $\alpha$.
\end{hypothesis}

\begin{convention}
Throughout \S\ref{sec:axiomatic}, we work over an affinoid space rather than an arbitrary preadic space; we thus freely use the $\bA, \bB, \bC$ suite of notations to refer to period \emph{rings} rather than period \emph{sheaves}.
\end{convention}

\subsection{Perfectoid towers}
\label{subsec:perfectoid towers}

We begin by considering certain towers of ring extensions, mostly to fix notations.

\begin{defn}
By a \emph{tower} over $(A,A^+)$, we will mean a sequence 
\[
\psi = (\psi_n: (A_{\psi,n}, A_{\psi,n}^+) \to (A_{\psi,n+1}, A_{\psi,n+1}^+))_{n=0}^\infty
\]
of bounded homomorphisms of uniform adic Banach algebras with $(A_{\psi,0}, A_{\psi,0}^+) = (A,A^+)$, such that each of the induced maps
$\psi_n^*$ on adic spectra is surjective. 
We will refer to $(A,A^+)$ (resp.\ $\Spa(A,A^+)$) as the \emph{base} 
(resp.\ the \emph{base space}) of the tower $\psi$.

For $\sigma: (A,A^+) \to (B,B^+)$ a bounded homomorphism of uniform adic Banach algebras over $E$, define the base extension of $\psi$ along $\sigma$ to be the tower 
\[
\psi' = (\psi'_n: (B_{\psi,n}, B_{\psi,n}^+) \to (B_{\psi,n+1}, B_{\psi,n+1}^+))_{n=0}^\infty
\]
with
\[
(B_{\psi,n}, B_{\psi,n}^+) = (A_{\psi,n}, A_{\psi,n}^+) \widehat{\otimes}_{(A,A^+)} (B,B^+).
\]
Let $\alpha_{\psi,n}$ be the spectral norm on $A_{\psi,n}$;
by \cite[Theorem~2.3.10]{part1}, the maps $\psi_n$ are isometric with respect to spectral norms. 
Let $(A_{\psi}, A_{\psi}^+)$ be the direct limit of the $(A_{\psi,n}, A_{\psi,n}^+)$; 
the ring $A_{\psi}$ then admits a unique power-multiplicative norm $\alpha_{\psi}$
which restricts to $\alpha_{\psi,n}$ on $A_{\psi,n}$ for each $n$.
Let $(\tilde{A}_{\psi}, \tilde{A}_\psi^+)$ be the completion of $(A_{\psi}, A_{\psi}^+)$ with respect to $\alpha_\psi$.

We say that $\psi$ is \emph{perfectoid} if $\tilde{A}_{\psi}$ is a Fontaine perfectoid Banach algebra. In this case, we may apply the perfectoid correspondence 
(Theorem~\ref{T:Fontaine perfectoid correspondence})
to produce a perfect uniform Banach algebra $(\tilde{R}_{\psi}, \tilde{R}_\psi^+)$ over $\FF_p$ with spectral norm $\overline{\alpha}_\psi$.
We also refer to the affinoid perfectoid space $\Spa(\tilde{A}_\psi, \tilde{A}_\psi^+)$ as the \emph{total space} of $\psi$.
\end{defn}

\begin{defn}
We say that $\psi$ is \emph{finite \'etale} if each map $\psi_n$ is finite \'etale;
note that in this case, it is automatic that each $A_{\psi,n}$ is uniform 
\cite[Proposition~2.8.16]{part1}. In this case, $\psi$ defines a perfectoid subdomain in $X_{\proet}$ for $X = \Spa(A,A^+)$; however, unlike in \cite{part1}, we will be interested in the individual levels of the tower, not just the inverse system as a pro-object.
\end{defn}

\begin{prop} \label{P:perfectoid tower persistence}
Suppose that $\psi$ is a perfectoid tower.
Then the base extension $\psi'$ of $\psi$ along $\sigma$ is again a perfectoid tower whenever $\sigma$ is of one of the following forms.
\begin{enumerate}
\item[(a)]
A rational localization which is again uniform. In this case, $\tilde{R}_{\psi'}$ is a rational localization of $\tilde{R}_{\psi}$.
\item[(b)]
A finite \'etale morphism. In this case, $\tilde{R}_{\psi'}$ is a finite \'etale algebra over $\tilde{R}_\psi$.
\end{enumerate}

In the following cases, assume further that $\psi$ is a finite \'etale perfectoid tower.

\begin{enumerate}
\item[(c)]
A morphism in which $B^+$ is obtained by completing $A^+$ at a finitely generated ideal containing an ideal of definition.
In this case, $\tilde{R}_{\psi'}^+$ is obtained by completing
$\tilde{R}_{\psi}^+$ at a finitely generated ideal containing an ideal of definition.
\item[(d)]
A morphism in which $B^+$ is obtained by taking the completion of a (not necessarily finite) \'etale extension of $A^+$ with respect to an ideal of definition of $A^+$.
In this case, $\tilde{R}_{\psi'}^+$ is the completion of an \'etale algebra over $\tilde{R}_{\psi}^+$ with respect to an ideal of definition of $\tilde{R}_\psi^+$.
\item[(e)]
A morphism in which $B^+$ is obtained by taking the $p$-adic completion of a (not necessarily finite type) algebraic localization of $A^+$ with respect to an ideal of definition of $A^+$.
In this case, $\tilde{R}_{\psi'}^+$ is the completion of an algebraic localization of  $\tilde{R}_{\psi}^+$ with respect to an ideal of definition of $\tilde{R}_\psi^+$.
\item[(f)]
A morphism with dense image. In this case, $\tilde{R}_{\psi} \to \tilde{R}_{\psi'}$ also has dense image. Moreover, if $\sigma$ is strict surjective, then so is $\tilde{R}_{\psi} \to \tilde{R}_{\psi'}$.
\end{enumerate}
\end{prop}
\begin{proof}
Each of these properties implies the corresponding property for the homomorphism $\tilde{A}_{\psi} \to \tilde{B}_{\psi}$.
We may thus apply Theorem~\ref{T:Fontaine perfectoid compatibility} to conclude.
\end{proof}

\subsection{Weakly decompleting towers}
\label{subsec:reality checks}

We next define imperfect analogues of the period rings of types $\tilde{\bA}, \tilde{\bB}, \tilde{\bC}$ considered in \cite{part1}. In the process, we identify an additional restriction on perfectoid towers which is needed to make sense of the constructions.

\begin{defn} \label{D:period rings}
Throughout this definition, let $r,s$ denote arbitrary real numbers with $0 < s \leq r$.
To begin with, for any superscript $*$, define the rings $\tilde{\bA}^*_{\psi}, \tilde{\bB}^*_{\psi}, \tilde{\bC}^*_{\psi}$ to be the same as the ones in which the subscript $\psi$ is replaced by $\tilde{R}_\psi$.

Let $\bA^{r}_{\psi}$ be the subring of 
$\tilde{\bA}^{r}_{\psi}$ consisting of those $x$ for which
$\theta(\varphi_\varpi^{-n}(x)) \in A_{\psi,n}$ for all integers $n$ with $hn \geq -\log_p r$. Put $\bA^{\dagger}_{\psi} = \cup_{r>0} \bA^{r}_{\psi}$
and let $\bA_\psi$ be the $\varpi$-adic completion of $\bA^{\dagger}_\psi$.
Also put $R_\psi = \bA_\psi/(\varpi) \subseteq \tilde{R}_\psi$ and $R_\psi^+ = R_\psi \cap \tilde{R}_\psi^+$. For uniformity with subsequent notation, we denote
$R_\psi^{\perf}$ also as $\breve{R}_\psi$.

Put
\[
\breve{\bA}_\psi = \bigcup_{n=0}^\infty \varphi_\varpi^{-n}(\bA_\psi),
\qquad
\breve{\bA}^{r}_\psi = \bigcup_{n=0}^\infty \varphi_\varpi^{-n}(\bA^{p^{-hn}r}_\psi),
\qquad
\breve{\bA}^{\dagger}_\psi = \bigcup_{n=0}^\infty \varphi_\varpi^{-n}(\bA^{\dagger}_\psi).
\]
Let $\hat{\bA}_\psi$ be the $\varpi$-adic completion of $\breve{\bA}_\psi$ within $\tilde{\bA}_\psi$, so that $\hat{\bA}_\psi/(\varpi) \cong \breve{R}_\psi$.
(One could also reasonably denote this quotient by $\hat{R}_\psi$, but we will not do so.)
Let $\hat{\bA}^{r}_\psi$
denote the completion of $\breve{\bA}^{r}_\psi$
for the supremum of $\lambda(\overline{\alpha}_\psi^r)$ and the $\varpi$-adic norm.
Put $\hat{\bA}^{\dagger}_\psi = \cup_{r>0} \hat{\bA}^{r}_\psi$.

For each ring of type $\bA$ defined above, form a new ring with the symbol $\bA$ replaced by $\bB$ by inverting $\varpi$.

Let $\bC^{[s,r]}_\psi$ be the completion of $\bB^{r}_{\psi}$
with respect to $\max\{\lambda(\overline{\alpha}_\psi^s), \lambda(\overline{\alpha}_\psi^r)\}$ within $\tilde{\bC}^{[s,r]}_\psi$.
Let $\bC^r_\psi$ be the Fr\'echet completion of $\bB^{r}_{\psi}$
with respect to $\lambda(\overline{\alpha}_\psi^s)$ for all $s \in (0,r]$ within $\tilde{\bC}^{r}_{\psi}$. Put $\bC_\psi = \cup_{r>0} \bC^r_\psi$. Also put
\[
\breve{\bC}_\psi = \bigcup_{n=0}^\infty \varphi_\varpi^{-n}(\bC_\psi),
\qquad
\breve{\bC}^{r}_\psi = \bigcup_{n=0}^\infty \varphi_\varpi^{-n}(\bC^{p^{-hn}r}_\psi),
\qquad
\breve{\bC}^{[s,r]}_\psi = \bigcup_{n=0}^\infty \varphi_\varpi^{-n}(\bC^{[p^{-hn}s,p^{-hn}r]}_\psi).
\]
We do not define the notation $\hat{\bC}^*_\psi$; see Remark~\ref{R:C dense}.
\end{defn}

The following lemma can be used to help pin down the ring $R_\psi$.
\begin{lemma} \label{L:intermediate mod p subring}
Choose $\eta \in A_{\psi,n_0}$ for some $n_0 \geq 0$ satisfying
$\alpha_\psi(\eta) \in (p^{-1},1)$ and
$\alpha_\psi(\eta^{-1}) = \alpha_\psi(\eta)^{-1}$.
Identify $\tilde{R}_\psi^+$ with the inverse limit of 
$\tilde{A}_\psi^+/(\eta)$ under $\overline{\varphi}_\varpi$,
and let $R'_\psi$ be the subring of $\tilde{R}_\psi$
consisting of those sequences $(\overline{x}_n)$ for which $\overline{x}_n \in A_{\psi,n}^+/(\eta)$ for all sufficiently large $n$.
Then $R_\psi^+ \subseteq R'_\psi$.
\end{lemma}
\begin{proof}
Choose $x \in \bA^{r}_{\psi}$ for some $r>0$.
For $hn \geq -\log_p r$,
we have 
\begin{align*}
\lambda(\overline{\alpha}_\psi)(\varphi_\varpi^{-n}(x - [\overline{x}]))
&= 
\lambda(\overline{\alpha}_\psi^{p^{-hn}})(x - [\overline{x}]) \\
&\leq p^{-1+p^{-hn}} \lambda(\overline{\alpha}_\psi^r)(x - [\overline{x}]).
\end{align*}
Consequently, if we choose $n_0$ so that $-1 + p^{-hn_0}/r < \log_p \alpha_\psi(\eta)$, then for $n \geq n_0$,
\[
\alpha_\psi(\theta(\varphi_\varpi^{-n-1}(x))^{p^h} - \theta(\varphi_\varpi^{-n}(x)))
\leq p^{-1+p^{-hn}} \leq \alpha_\psi(\eta).
\]
The elements $\theta(\varphi_\varpi^{-n}(x))$ thus form a sequence in $R'_\psi$.
\end{proof}

\begin{defn} \label{D:weakly decompleting}
We say that $\psi$ is \emph{weakly decompleting} if the following conditions hold.
\begin{enumerate}
\item[(a)]
The ring $R_\psi^{\perf}$ is dense in $\tilde{R}_{\psi}$.
\item[(b)]
For some $r>0$, the map $\bA^{r}_{\psi} \to R_{\psi}$
induced by reduction modulo $\varpi$ is strict surjective for the norms
$\lambda(\overline{\alpha}_\psi^r)$ and $\overline{\alpha}_\psi^r$.
In particular, $R_\psi$ is complete under $\overline{\alpha}_\psi^r$.
\end{enumerate}
\end{defn}

\begin{example}
Let $k_0 \to k_1 \to \cdots$ be a tower of finite extensions of perfect fields of characteristic $p$, and define the tower $\psi$ with $A_{\psi,n} = \Frac W_\varpi(k_n), A_{\psi,n}^+ = W_\varpi(k_n)$. Then $\psi$ is weakly decompleting if and only if $\cup_n k_n$ is finite over $k_0$, as otherwise condition (b) fails.
\end{example}

\begin{remark}
For $0 < s \leq 1 \leq r$, the map $\theta: \tilde{\bA}^{r}_\psi
\to \tilde{\bA}^{1}_\psi \to \tilde{A}_\psi$ extends by continuity to a surjective map $\theta: \tilde{\bC}^{[s,r]}_\psi \to \tilde{A}_\psi$.
This map restricts to a map $\breve{\bC}^{[s,r]}_\psi \to A_\psi$, but this map is not \emph{a priori} known to be surjective; see Remark~\ref{R:theta not surjective}.
\end{remark}

\begin{remark} \label{R:theta not surjective}
Note that the definition of a weakly decompleting tower does not guarantee that the map $\theta \circ \varphi_\varpi^{-n}: \bA_{\psi}^{r} \to A_{\psi,n}$ is surjective for some pair $r,n$ with $hn \geq -\log_p r$. However, this occurs in all of the special cases we consider in this paper.
When $\psi$ is decompleting, one can at least say that the map
$\theta: \breve{\bA}_{\psi}^{1} \to A_{\psi}$
is surjective; see Proposition~\ref{P:Ax}.
\end{remark}

For the remainder of \S\ref{subsec:reality checks},
assume that $\psi$ is weakly decompleting.
The next few lemmas constitute a series of \emph{reality checks} analogous to \cite[\S 5.2]{part1}. The first of these is a formal refinement of the lifting condition from Definition~\ref{D:weakly decompleting}(b); this provides an alternative to the use of Teichm\"uller lifts, which do not generally appear within the imperfect period rings.

\begin{lemma} \label{L:optimal lifts}
There exists $r_0 > 0$ such that for $r \in (0, r_0]$, 
every $\overline{x} \in R_{\psi}$ lifts to some $x \in \bA^{r}_{\psi}$ with $\lambda(\overline{\alpha}_\psi^r)(x - [\overline{x}]) \leq p^{-1/2} \overline{\alpha}_{\psi}^r(\overline{x})$.
In particular, the map $\bA^{r}_{\psi} \to R_{\psi}$ induced by reduction modulo $p$ is optimal surjective (not just strict surjective) for all sufficiently small $r>0$ (not just for a single such $r$).
\end{lemma}
\begin{proof}
Since $\psi$ is weakly decompleting, we may choose $r_1 > 0$ and $c>1$ such that every $\overline{x} \in R_{\psi}$ lifts to some $x \in \bA^{r_1}_{\psi}$ with $\lambda(\overline{\alpha}_\psi^{r_1})(x) \leq c \overline{\alpha}_{\psi}^{r_1}(\overline{x})$.
In particular, we have $\lambda(\overline{\alpha}_\psi^{r_1})(x - [\overline{x}]) \leq c \overline{\alpha}_{\psi}^{r_1}(\overline{x})$.
For $r \in (0,r_1]$, we then have
\[
\lambda(\overline{\alpha}_\psi^{r})(x - [\overline{x}]) 
\leq p^{-1+r/r_1} \lambda(\overline{\alpha}_\psi^{r})(x - [\overline{x}])^{r/r_1} 
\leq 
p^{-1+r/r_1} c^{r/r_1} \overline{\alpha}_{\psi}^{r}(\overline{x}).
\]
We thus deduce the claim by choosing $r_0 \in (0,r_1]$ so that
$(c/p)^{r_0/r_1} \leq p^{1/2}$.
\end{proof}
\begin{cor} \label{C:lift presentation}
For $r_0$ as in Lemma~\ref{L:optimal lifts}, for $r \in (0, r_0]$,
any $x \in \bB^{r}_\psi$ can be written as a $\lambda(\overline{\alpha}_\psi^r)$-convergent sum
$\sum_{n=m}^\infty \varpi^n y_n$ with $y_n \in \bA^{r_0}_\psi$ such that, if we write
$\overline{y}_n$ for the image of $y_n$ in $A_\psi$, then
$\lambda(\overline{\alpha}_\psi^{r_0})(y_n - [\overline{y}_n]) \leq p^{-1/2} \lambda(\overline{\alpha}_\psi^{r_0})(y_n)$ for all $n \geq 0$. Moreover, for any such representation, for all $s \in (0,r]$,
\begin{equation} \label{eq:lift presentation}
\lambda(\overline{\alpha}_\psi^s)(x) = \max_n \{p^{-n} \overline{\alpha}_\psi^s(\overline{y}_n)\}.
\end{equation}
\end{cor}
\begin{proof}
We may assume $x \in \bA^{r}_\psi$.
The existence of such a representation within $\bA_\psi$ is immediate from Lemma~\ref{L:optimal lifts}.
To prove convergence with respect to $\lambda(\overline{\alpha}_\psi^r)$,
write $x = \sum_{n=0}^\infty \varpi^n [\overline{x}_n]$ and note that
$\lambda(\overline{\alpha}_\psi^r)(\varpi^n y_n)$ is bounded above by the greater
of $\lambda(\overline{\alpha}_\psi^r)(\varpi^n [\overline{x}_n])$ and the maximum of
$p^{-1/2} \lambda(\overline{\alpha}_\psi^r)(\varpi^i y_i)$ over $i < n$.
By an easy induction argument, it follows that there exists $n$ such that
\[
\lambda(\overline{\alpha}_\psi^r)\left(x - \sum_{i=0}^{n-1} \varpi^i y_i \right)
\leq p^{-1/2} \lambda(\overline{\alpha}_\psi^r)(x).
\]
Applying this argument repeatedly yields the desired convergence,
which in turn immediately implies \eqref{eq:lift presentation}.
\end{proof}

\begin{remark} \label{R:C dense}
Another corollary of Lemma~\ref{L:optimal lifts} is that 
$\breve{\bC}^*_\psi$ is dense in $\tilde{\bC}^*_\psi$. For this reason, there is no reason to introduce a separate notation $\hat{\bC}^*_\psi$ for the completion of $\breve{\bC}^*_\psi$.
\end{remark}

We next obtain
analogues of \cite[Lemmas~5.2.5 and~5.2.7]{part1}.
\begin{lemma} \label{L:intersect dagger}
For $r_0$ as in Lemma~\ref{L:optimal lifts}, for $0 < s \leq r \leq r_0$,
within $\tilde{\bC}^{[s,s]}_\psi$ we have $\tilde{\bA}^{s}_\psi \cap \bC^{[s,r]}_\psi = \bA^{r}_\psi$.
In particular, within $\tilde{\bC}_\psi$ we have $\tilde{\bA}^{\dagger}_\psi \cap \bC_\psi = \bA^{\dagger}_\psi$.
\end{lemma}
\begin{proof}
Take $x$ in the intersection, and write $x$ as the limit in $\bC^{[s,r]}_\psi$ of a sequence
$x_0, x_1,\dots$ with $x_i \in \bB^{r}_\psi$. For each positive integer $j$, we can find
$N_j > 0$ such that
\[
\lambda(\overline{\alpha}_\psi^t)(x_i - x) \leq p^{-j} \qquad (i \geq N_j, t \in [s,r]).
\]
Write $x_i = \sum_{l=m}^\infty \varpi^l x_{il}$ as in Corollary~\ref{C:lift presentation}.
Put $y_i = \sum_{l=0}^\infty \varpi^l x_{il} \in \bA^{r}_\psi$.
For $i \geq N_j$, having $x \in \tilde{\bA}^{\dagger,s}_\psi$ and $\lambda(\overline{\alpha}_\psi^s)(x_i - x)
\leq p^{-j}$ implies that $\lambda(\overline{\alpha}_\psi^s)(\varpi^l x_{il}) \leq p^{-j}$ for $l < 0$ by
\eqref{eq:lift presentation}. That is,
\[
\overline{\alpha}_\psi(\overline{x}_{il}) \leq p^{(l-j)/s} \qquad (i \geq N_j, l < 0).
\]
Since $p^{-l} p^{(l-j)r/s} \leq p^{1 + (1-j)r/s}$ for $l \leq -1$, we deduce that
$\lambda(\overline{\alpha}_\psi^r)(x_i - y_i) \leq p^{1 + (1-j)r/s}$ for $i \geq N_j$.
Consequently, the sequence $y_0, y_1, \dots$ converges to $x$ under $\lambda(\overline{\alpha}_\psi^r)$;
it follows that $x \in \bA^{r}_\psi$ as desired.
\end{proof}

\begin{lemma} \label{L:decompose big}
For $r_0$ as in Lemma~\ref{L:optimal lifts}, for $0 < s \leq r \leq r_0$,
each $x \in \bC_{\psi}^{[s,s]}$ can be decomposed as $y+z$
with $y \in \bA_\psi^{\dagger,s}$,
$z \in \bC_\psi^{[s,r]}$, and
\[
\lambda(\overline{\alpha}_\psi^t)(z) \leq p^{1-t/s} \lambda(\overline{\alpha}_\psi^s)(x)^{t/s}
\qquad (t \in [s,r]).
\]
\end{lemma}
\begin{proof}
As in the proof of \cite[Lemma~5.2.7]{part1}, we may reduce to the case
$x \in \bB^{s}_\psi$.
Write $x = \sum_{n=m}^\infty \varpi^n y_n$ as in Corollary~\ref{C:lift presentation},
and put $y = \sum_{n=0}^\infty \varpi^n y_n$ and $z = x - y$.
For $n<0$, we then have $\lambda(\overline{\alpha}_\psi^s)(\varpi^n x_n) \leq \lambda(\overline{\alpha}_\psi^s)(x)$ and so
\begin{align*}
\lambda(\overline{\alpha}_\psi^t)(\varpi^n x_n) &= \lambda(\overline{\alpha}_\psi^t)(\varpi^n [\overline{x}_n]) \\
&= p^{-n(1-t/s)} \lambda(\overline{\alpha}_\psi^s)(\varpi^n x_n)^{t/s} \\
&\leq p^{1-t/s} \lambda(\overline{\alpha}_\psi^s)(x)^{t/s}.
\end{align*}
This proves the claim.
\end{proof}

This in turn yields an analogue of \cite[Lemma~5.2.9]{part1}.
\begin{lemma} \label{L:interval intersection}
For $r_0 > 0$ as in Lemma~\ref{L:optimal lifts} and $0 < s \leq s' \leq r \leq r' \leq r_0$, inside $\bC^{[s',r]}_\psi$ we have
\[
\bC^{[s,r]}_\psi \cap \bC^{[s',r']}_\psi = \bC^{[s,r']}_\psi.
\]
\end{lemma}
\begin{proof}
Given $x$ in the intersection, apply Lemma~\ref{L:decompose big} to write $x = y+z$ with $y \in \bA^{r}_\psi$, $z \in \bC^{[s,r']}_\psi$.
By Lemma~\ref{L:intersect dagger},
\[
y = z-x \in \bA^{r}_\psi \cap \bC^{[s',r']}_\psi
= \bA^{r'}_\psi \subset \bC^{[s,r']}_\psi,
\]
so $x \in \bC^{[s,r']}_\psi$ as desired.
\end{proof}

For weakly decompleting towers, we may expand the net of equivalences between finite \'etale algebras emerging from \cite{part1}
so as to include the imperfect period rings.
\begin{theorem} \label{T:big etale}
The functors in the diagram
\[
\xymatrix@R=20pt@!C=80pt{
& \varphi_\varpi^{-1}\text{-}\FEt(\breve{\bA}^{\dagger,1}_\psi) \ar[rd] \ar[dd] \ar[rrr] &&& \FEt(A_{\psi}) \ar[d] \\
& & \varphi_\varpi^{-1}\text{-}\FEt(\hat{\bA}^{\dagger,1}_\psi) \ar[r]\ar[d] & \varphi_\varpi^{-1}\text{-}\FEt(\tilde{\bA}^{\dagger,1}_\psi) \ar[d] \ar[r] & \FEt(\tilde{A}_{\psi}) \\
 (\varphi_\varpi^{-1}\text{-})\FEt(\bA^{\dagger}_\psi) \ar[r] \ar[d]  & (\varphi_\varpi^{-1}\text{-})\FEt(\breve{\bA}^{\dagger}_\psi) \ar[d] \ar[r]
 &  (\varphi_\varpi^{-1}\text{-})\FEt(\hat{\bA}^{\dagger}_\psi) \ar[d] \ar[r] & (\varphi_\varpi^{-1}\text{-})\FEt(\tilde{\bA}^\dagger_\psi) \ar[d] & \\
 (\varphi_\varpi^{-1}\text{-})\FEt(\bA_\psi) \ar[r] \ar[d] & (\varphi_\varpi^{-1}\text{-})\FEt(\breve{\bA}_\psi) \ar[r] & (\varphi_\varpi^{-1}\text{-})\FEt(\hat{\bA}_\psi) \ar[r] \ar[d] & (\varphi_\varpi^{-1}\text{-})\FEt(\tilde{\bA}_\psi) \ar[d] &\\
 (\varphi_\varpi^{-1}\text{-})\FEt(R_\psi) \ar[rr] & & (\varphi_\varpi^{-1}\text{-})\FEt(R_\psi^{\perf}) \ar[r] & (\varphi_\varpi^{-1}\text{-})\FEt(\tilde{R}_\psi)  &
}
\]
induced by the evident ring homomorphisms are rank-preserving equivalences of tensor categories. Here $\FEt(S)$ denotes the category of finite \'etale $S$-algebras, $\varphi_\varpi^{-1}\text{-}\FEt(S)$ denotes the category of finite \'etale $S$-algebras equipped with isomorphisms with their $\varphi_\varpi^{-1}$-pullbacks, and
$(\varphi_\varpi^{-1}\text{-})\FEt(S)$ means either $\FEt(S)$ or $\varphi_\varpi^{-1}\text{-}\FEt(S)$ (i.e., the two categories are equivalent to each other).
\end{theorem}
\begin{proof}
The arrows in the bottom row of the diagram, with and without $\varphi_\varpi^{-1}$, are equivalences by \cite[Theorem~3.1.15]{part1}.
It is thus sufficient to link each entry of the diagram with the bottom row.

Consider first the bottom three rows of the diagram (in which case we may ignore $\varphi_\varpi^{-1}$).
Each of the rings 
\[
\bA^\dagger_\psi, \hat{\bA}^{\dagger}_\psi,
\tilde{\bA}^\dagger_\psi, \bA_\psi, \hat{\bA}_\psi,
\tilde{\bA}_\psi
\]
is henselian with respect to $(p)$;
consequently, quotienting any of these rings by $(p)$ induces an equivalence
by \cite[Theorem~1.2.8]{part1}.
The homomorphisms
$\breve{\bA}^\dagger_\psi \to \hat{\bA}^{\dagger}_\psi$,
$\breve{\bA}_\psi \to \hat{\bA}_\psi$
thus induce functors which are essentially surjective, but also fully faithful by
\cite[Lemma~2.2.4(a)]{part1}, and thus equivalences.

It remains to link the top two rows to the rest of the diagram.
To begin with, the homomorphisms out of $\tilde{\bA}^{\dagger,1}_\psi$ induce equivalences by
\cite[Corollary~5.5.6]{part1},
while the homomorphism $A_{\psi} \to \tilde{A}_{\psi}$ induces an equivalence
by \cite[Proposition~2.6.8]{part1}.

We next link $\breve{\bA}^{\dagger,1}_\psi$ to the bottom of the diagram.
For $0 <s \leq r$, we have a functor
$\varphi_\varpi^{-1}\text{-}\FEt(\breve{\bA}^{\dagger,r}_\psi) \to 
\varphi_\varpi^{-1}\text{-}\FEt(\breve{\bA}^{\dagger,s}_\psi)$,
which is evidently an equivalence for $s = p^{-h} r$; it follow easily that all of the functors are both fully faithful and essentially surjective.
By \cite[Remark~1.2.9]{part1}, we deduce that
$\varphi_\varpi^{-1}\text{-}\FEt(\breve{\bA}^{\dagger,1}_\psi) \to 
\varphi_\varpi^{-1}\text{-}\FEt(\breve{\bA}^{\dagger}_\psi)$ is an equivalence.

We finally link $\hat{\bA}^{\dagger,1}_\psi$ to the bottom of the diagram.
The functor induced by the homomorphism from this ring to $\hat{\bA}^{\dagger}_\psi$
is essentially surjective, because we have an equivalence that factors through it. We may check full faithfulness by imitating the argument given for
$\breve{\bA}^{\dagger,1}_\psi$.
\end{proof}

We next study the effect of base extensions on weakly decompleting towers.

\begin{lemma} \label{L:weakly decompleting affinoid base extension}
Let $\psi'$ be a base extension of $\psi$ such that that $R_{\psi'}^{\perf}$ is dense in $\tilde{R}_{\psi'}$ and $R_\psi \to R_{\psi'}$ is an affinoid homomorphism. Then $\psi'$ is also weakly decompleting.
\end{lemma}
\begin{proof}
Write $R_{\psi'}$ as a strict quotient of $R_\psi\{\overline{T}_1,\dots,\overline{T}_n\}$ for some $\overline{T}_1,\dots,\overline{T}_n$, and lift the images of $\overline{T}_1,\dots,\overline{T}_n$ to some $T_1,\dots,T_n \in R_{\psi'}$. Then for 
any $r>0$ for which $\bA^{r}_\psi \to R_\psi$ is strict surjective,
so is $\bA^{r}_\psi\{T_1,\dots,T_n\} \to R_{\psi'}$
This surjection factors through $\bA^r_{\psi'}$, proving the claim.
\end{proof}

\begin{prop} \label{P:weakly decompleting tower persistence}
The base extension $\psi'$ of $\psi$ along any morphism $(A,A^+) \to (B,B^+)$ of any one of the following forms is again weakly decompleting.
\begin{enumerate}
\item[(a)]
A rational localization which is again uniform.
In this case, $R_{\psi'}$ is a rational localization of $R_{\psi}$.
\item[(b)]
A finite \'etale morphism. In this case, $R_{\psi'}$ is a finite \'etale algebra over $R_\psi$.
\end{enumerate}

In the following cases, assume further that $\psi$ is a finite \'etale perfectoid tower.

\begin{enumerate}
\item[(c)]
A morphism in which $B^+$ is obtained by completing $A^+$ at a finitely generated ideal containing $p$.
In this case, $R_{\psi'}^+$ is obtained by completing
$R_{\psi}^+$ at a finitely generated open ideal.
\item[(d)]
A morphism in which $B^+$ is obtained by taking the $p$-adic completion of a (not necessarily finite) \'etale extension of $A^+$.
In this case, $R_{\psi'}^+$ is the completion of an \'etale algebra over $R_{\psi}^+$.
\item[(e)]
A morphism in which $B^+$ is obtained by taking the $p$-adic completion of a (not necessarily finite type) algebraic localization of $A^+$.
In this case, $R_{\psi'}^+$ is the completion of an algebraic localization of  $R_{\psi}^+$.
\end{enumerate}
\end{prop}
\begin{proof}
To deduce (a), note that any rational subspace of
$\Spa(\tilde{R}_\psi, \tilde{R}_\psi^+)$ can be specified using parameters in the dense subring $R_\psi^{\perf}$ (by \cite[Remark~2.4.8]{part1}) and hence also using some parameters in $R_\psi$ (by raising each parameter to the $p^n$-th power for some large $n$). 
Let $\overline{f}_1,\dots,\overline{f}_n,\overline{g} \in R_{\psi}$ be some such parameters;
then choose lifts $f_1,\dots,f_n,g$ of these parameters to $\bA^\dagger_\psi$. For all sufficiently small $r>0$, the lifts belong to $\bA^r_\psi$ and generate the unit ideal,
and the rational localization of $\tilde{\bA}^r_\psi$ defined by these parameters may be identified with $\tilde{\bA}^r_{\psi'}$. 
Since $R_{\psi'}^{\perf}$ contains $R_\psi[(\overline{f}_1/\overline{g})^{p^{-\infty}}, \dots,
(\overline{f}_n/\overline{g})^{p^{-\infty}}]$, it is dense in $\tilde{R}_{\psi'}$; this guarantees condition (a) of Definition~\ref{D:weakly decompleting}. Condition (b) of the definition follows from Lemma~\ref{L:weakly decompleting affinoid base extension}
and \cite[Lemma~2.4.13]{part1}.

To deduce (b), note that $\FEt(A_{\psi})
\cong \FEt(\bA^\dagger_\psi)$ by 
Theorem~\ref{T:big etale}. This guarantees condition (a)  of Definition~\ref{D:weakly decompleting}; condition (b) again holds 
by Lemma~\ref{L:weakly decompleting affinoid base extension}.

To deduce (c), note that $R_\psi$ maps isometrically to a dense subring of $R_{\psi'}$ and that $R_{\psi}^{\perf}$ is dense in $\tilde{R}_{\psi'}$. This implies both conditions of Definition~\ref{D:weakly decompleting}.

To deduce (d), note that $R_{\psi'}^+$ is the completion of an \'etale extension of $R_\psi^+$. This implies both conditions of Definition~\ref{D:weakly decompleting}.

To deduce (e), note that $R_{\psi'}^+$ is the completion of an algebraic localization of $R_\psi^+$. This implies both conditions of Definition~\ref{D:weakly decompleting}.
\end{proof}

In order to deal with pseudocoherent $\varphi$-modules in type $\bC$, we will need the following additional condition.
\begin{defn} \label{D:sheafy tower}
We say that a tower $\psi$ is \emph{noetherian} if there exists $r_0 > 0$ such that for $0 < s \leq r \leq r_0$,
the Banach ring $\bC^{[s,r]}_\psi$ is strongly noetherian (and hence sheafy by \cite[Proposition~2.4.16]{part1}).
This implies that the ring $\breve{\bC}^{[s,r]}_\psi$ is coherent.
\end{defn}

\begin{remark} \label{R:affinoid slice}
Let $\psi$ be a tower satisyfing the following conditions.
\begin{enumerate}
\item[(a)]
The ring $R_\psi^{\perf}$ is dense in $\tilde{R}_\psi$.
\item[(b)]
The ring $A$ is a Banach algebra over an analytic field $K$ with perfect residue field $k$.
\item[(c)]
The ring $R_\psi$ is an affinoid algebra over some field of the form $k((\overline{\pi}))$ with $\overline{\alpha}_{\psi}(\overline{\pi}) = p^{-1}$.
\end{enumerate}
By arguing as in Lemma~\ref{L:weakly decompleting affinoid base extension}, we may see that $\psi$ is perfectoid and weakly decompleting, so we may
choose $r_0$ as in Lemma~\ref{L:optimal lifts}. For $0 < s \leq r \leq r_0$
with $r,s \in \log_p \left| K^\times \right|$, 
the ring $\bC^{[s,r]}_\psi$ is an affinoid algebra over $K$.
If we do not require that $r,s \in \log_p \left| K^\times \right|$,
then $\bC^{[s,r]}_\psi$ is still an affinoid algebra over $K$ in the sense of Berkovich.
We conclude that $\psi$ is a noetherian tower.
\end{remark}

\subsection{Modules and sheaves}
\label{subsec:projective modules}

In preparation for the study of $\varphi$-modules and $\Gamma$-modules, we study modules over some of the rings we have just introduced, retracing the steps of \S\ref{sec:perfect period sheaves}.

\begin{hypothesis}
Throughout \S\ref{subsec:projective modules}, assume that the tower $\psi$ is weakly decompleting, and fix $r_0 > 0$ as in Lemma~\ref{L:optimal lifts}. In case $\psi$ is noetherian, we also assume that $r_0$ satisfies the conditions
of Definition~\ref{D:sheafy tower}.
\end{hypothesis}

\begin{defn} \label{D:not quasi-Stein}
For $0 < s \leq r \leq r_0$,
let $\bC^{[s,r],+}_\psi$ be the completion in $\bC^{[s,r]}_\psi$ of the subring generated by
\[
\{x \in \bC^{[s,r]}_\psi: \lambda(\alpha^t)(x) < 1 \quad (t \in [s,r])\}
\cup 
\{z \in \bA^{r}_\psi: \lambda(\alpha^r)(z) \leq 1, \overline{z} \in R_\psi^+\}.
\]
\end{defn}

\begin{lemma} \label{L:imperfect Robba interval localization}
For $0 < s \leq r \leq r_0$ and $\rho>0$, let $\bC^r_\psi\{T/\rho\}$ (resp.\ $\bC^{[s,r]}_\psi\{T/\rho\}$)
be the Fr\'echet completion of $\bC^r_\psi[T]$ (resp.\ $\bC^{[s,r]}_\psi[T]$) for the $\rho$-Gauss extensions of $\lambda(\overline{\alpha}_\psi^u)^{1/u}$ for all $u \in (0,r]$ (resp.\ all $u \in [s,r]$). Then for $t \in [s,r]$, the maps
\begin{align*}
\bC^r_\psi\{T/p^{-1/s}\}/(T-\varpi) &\to \bC^s_\psi \\
\bC^{[s,r]}_\psi\{T/p^{-1/t}\}/(T-\varpi) &\to \bC^{[s,t]}_\psi \\
\bC^r_\psi\{T/p^{-1/s}\}/(\varpi T-1) &\to \bC^{[s,r]}_\psi \\
\bC^{[s,r]}_\psi\{T/p^{-1/t}\}/(\varpi T-1) &\to \bC^{[t,r]}_\psi
\end{align*}
are strict isomorphisms.
\end{lemma}
\begin{proof}
As in Remark~\ref{R:Robba reified}.
\end{proof}

\begin{lemma} \label{L:weak flatness imperfect Robba}
For $0 < s \leq t \leq r \leq r_0$, the morphisms
\[
\bC^{[s,r]}_\psi \to \bC^{[s,t]}_\psi, \bC^{[s,r]}_\psi \to \bC^{[t,r]}_\psi
\]
are $2$-pseudoflat.
\end{lemma}
\begin{proof}
Using Lemma~\ref{L:imperfect Robba interval localization}, we may emulate the proof of 
Proposition~\ref{P:weak flatness perfect Robba}.
\end{proof}

\begin{prop}  \label{P:glueing projective modules over intervals}
Let $I$ be a closed subinterval of $(0,r_0]$ and let $I_1,\dots,I_n$
be closed subintervals of $I$ which cover $I$. 
\begin{enumerate}
\item[(a)]
For any stably pseudocoherent $\bC^I_\psi$-module $M$, the augmented \v{C}ech complex
\[
0 \to M \to \bigoplus_{i=1}^n M \otimes_{\bC^I_\psi} \bC^{I_i}_\psi \to \cdots
\]
is exact.
\item[(b)]
The morphism
$\bC^I_\psi \to \bigoplus_{i=1}^n \bC^{I_i}_\psi$ is an effective descent morphism both for 
finite projective modules over Banach rings
and for \'etale-stably pseudocoherent modules over Banach rings.
\end{enumerate}
\end{prop}
\begin{proof}
Both statements reduce immediately to the case of a covering of $I = [s,r]$ by two subintervals $I_1 = [s,t]$, $I_2 = [t,r]$. In this case, (a) follows from Lemma~\ref{L:weak flatness imperfect Robba} as in the proof of Theorem~\ref{T:pseudocoherent acyclicity}; (b) follows from the proof of Lemma~\ref{L:refined Kiehl},
using Lemma~\ref{L:weak flatness imperfect Robba} in place of
Theorem~\ref{T:weak flatness}.
\end{proof}

\begin{defn} \label{D:artificial vector bundle}
By a \emph{vector bundle} (resp.\ \emph{pseudocoherent sheaf})
over $\bC^r_\psi$ for some $r>0$, we will mean a collection consisting of a finite projective (resp. \ \'etale-stably pseudocoherent) module $M_{[t,s]}$ over $\bC^{[t,s]}_\psi$ for each pair $t,s$ with $0 < t \leq s \leq r$, together with isomorphisms
$M_{[t,s]} \otimes_{\bC^{[t,s]}_{\psi}} \bC^{[t',s']}_{\psi} \to M_{[t',s']}$
satisfying the cocycle condition for $t \leq t' \leq s' \leq s$. 
The inverse limit of the $M_{[t,s]}$ forms the $\bC^r_\psi$-module of \emph{global sections} of the bundle/sheaf.
\end{defn}

\begin{remark} \label{R:pseudocoherent sheaf C}
We will mostly use the definition of a pseudocoherent sheaf in the case where $\psi$ is noetherian.
In this case, any finitely generated module over $\bC^{[t,s]}_\psi$ is 
\'etale-stably pseudocoherent; the union of the spaces
$\Spa(\bC^{[t,s]}_\psi, \bC^{[t,s],+}_\psi)$ is a strongly noetherian quasi-Stein space with $\bC^r_\psi$ as its rings of global sections (see Lemma~\ref{L:global sections C} below); and a finite locally free sheaf (resp.\ coherent sheaf) on this space is the same as a vector bundle (resp.\ pseudocoherent sheaf) over $\bC^r_\psi$ in the sense of Definition~\ref{D:artificial vector bundle}.
\end{remark}

\begin{lemma} \label{L:global sections C}
The natural map $\bC^r_\psi \to \varprojlim_{s>0} \bC^{[s,r]}_\psi$ is an isomorphism.
\end{lemma}
\begin{proof}
The map is injective because the map $\bC^r_\psi \to \bC^{[s,r]}_\psi$ is an isometry with respect to
$\lambda(\overline{\alpha}_\psi^t)$ for all $t \in [s,r]$. To see that it is surjective, note that $\bC^r_\psi$ has dense image in $\bC^{[s,r]}_\psi$ for each $s$. Hence for any $x \in \varprojlim_{s>0} \bC^{[s,r]}_\psi$, we
can find $x_n \in \bC^r_\psi$ so that 
\[
\lambda(\overline{\alpha}^t_\psi)(x-x_n) \leq 2^{-n} \qquad (t \in [2^{-n}r, r]).
\]
The sequence $\{x_n\}$ is a Cauchy sequence in $\bC^r_\psi$ with limit $x \in \bC^r_\psi$.
\end{proof}

\begin{lemma} \label{L:global sections are dense}
Suppose either that $(M_{[t,s]})$ is a vector bundle over $\bC^r_\psi$ for some $r \in (0,r_0]$,
or that $\psi$ is noetherian and $(M_{[t,s]})$ is a pseudocoherent sheaf over $\bC^r_\psi$ for some $r \in (0,r_0]$.
\begin{enumerate}
\item[(a)]
The global sections of this bundle are dense in each $M_{[t,s]}$.
\item[(b)]
We have $R^1 \varprojlim M_{[t,s]} = 0$.
\end{enumerate}
\end{lemma}
\begin{proof}
In the noetherian case, this follows from Lemma~\ref{L:Kiehl calc}; the general case follows by a similar argument.
\end{proof}

\begin{lemma} \label{L:uniformly pseudocoherent}
Suppose that $\psi$ is noetherian.
Let $\calF$ be  a pseudocoherent sheaf over $\bC^r_\psi$ for some $r \in (0,r_0]$. Then the following conditions are equivalent.
\begin{enumerate}
\item[(a)]
The global sections of $\calF$ form a pseudocoherent $\bC^r_\psi$-module.
\item[(b)]
There exists a sequence $m_0, m_1, \dots$ of integers such that for all $0 < t \leq s \leq r$, $\calF(\Spa(\bC^{[t,s]}_\psi, \bC^{[t,s],+}_\psi))$ admits a projective resolution $\cdots \to M_1 \to M_0 \to 0$ such that $M_i$ has rank at most $m_i$ for all $i$.
\item[(c)]
There exist a locally finite covering of $(0,r]$ by closed subintervals $\{I_j\}_j$ 
and a sequence $m_0, m_1, \dots$ of integers such that for all $j$,
$\calF(\bC^{I_j}_\psi)$ admits a projective resolution $\cdots \to M_1 \to M_0 \to 0$ such that $M_i$ has rank at most $m_i$ for all $i$.
\end{enumerate}
\end{lemma}
\begin{proof}
By Remark~\ref{R:pseudocoherent sheaf C} and Lemma~\ref{L:global sections C}, we may apply 
Corollary~\ref{C:quasi-Stein flat} to deduce that $\bC^r_\psi \to \bC^{[t,s]}_\psi$ is flat; from this plus Corollary~\ref{C:quasi-Stein tensor}, we deduce the implication from (a) to (b).
The implication from (b) to (c) is trivial. The implication from (c) to (a)
follows from Proposition~\ref{P:Stein space uniform covering}.
\end{proof}

\begin{defn}
For $r \in (0,r_0]$,
a \emph{uniformly pseudocoherent sheaf} over  $\bC^r_\psi$
is a pseudocoherent sheaf over $\bC^{r}_\psi$ satisfying any of the equivalent conditions of Lemma~\ref{L:uniformly pseudocoherent}.
\end{defn}

\begin{lemma} \label{L:Robba base extension pseudocoherent}
Suppose that $\psi$ is noetherian.
For $r \in (0,r_0]$,
base extension defines an equivalence of categories between pseudocoherent modules over $\bC^r_\psi$ 
and uniformly pseudocoherent 
sheaves over $\bC^r_\psi$, with the quasi-inverse being the global sections functor.
\end{lemma}
\begin{proof}
The base extension functor exists and is fully faithful because
$\bC^r_\psi \to \bC^{[t,s]}_\psi$ is flat (see the proof of Lemma~\ref{L:uniformly pseudocoherent}).
It is essentially surjective by Corollary~\ref{C:quasi-Stein tensor} and Lemma~\ref{L:uniformly pseudocoherent}.

\end{proof}

\begin{lemma} \label{L:Robba base extension vector bundle}
For $r \in (0, r_0]$, let $(M_{[t,s]})$ be a vector bundle over $\bC^r_\psi$ which is uniformly pseudocoherent.
Then $M = \varprojlim_{s>0} M_{[s,r]}$
is a finite projective $\bC^r_\psi$-module.
\end{lemma}
\begin{proof}
By (the proof of) Proposition~\ref{P:Stein space uniform covering}, $M$ is finitely generated.
We may thus emulate the proof of Corollary~\ref{C:quasi-Stein finitely generated} to conclude.
\end{proof}

\subsection{\texorpdfstring{$\varphi$}{phi}-modules}
\label{subsec:phi-modules}

We next extend the concepts of projective and pseudocoherent $\varphi$-modules to the case of imperfect period rings.

\begin{hypothesis} \label{H:weakly decompleting Frobenius}
Throughout \S\ref{subsec:phi-modules}, assume that the tower $\psi$ is weakly decompleting, let $a$ be a positive integer, put $\varphi = \varphi_\varpi^a$,
and put $q = p^{ah}$.
\end{hypothesis}

\begin{defn} \label{D:imperfect phi-modules}
Let $*_\psi$ be one of the perfect or imperfect period rings of Definition~\ref{D:period rings} on which $\varphi$ acts. 
By a \emph{projective (resp.\ pseudocoherent, fpd) $\varphi$-module} 
over $*_\psi$, we will mean
a finite projective (resp.\ \'etale-stably pseudocoherent, \'etale-stably fpd) module $M$ over $*_\psi$ equipped with an isomorphism $\varphi^* M \cong M$, subject to topological conditions as listed below.

[Moved from 5.4.10].
For $r>0$, a \emph{projective (resp.\ pseudocoherent, fpd) $\varphi$-module} over $\bC^{r}_\psi$
is a finite projective (resp.\ pseudocoherent, fpd) module $M$ over $\bC^{r}_\psi$ equipped with an isomorphism
$\varphi^* M \cong M \otimes_{\bC^{r}_\psi} \bC^{r/q}_\psi$
of modules over $\bC^{r/q}_\psi$, subject to topological conditions as listed below. We similarly define projective (resp.\ pseudocoherent, fpd) $\varphi$-modules over $\breve{\bC}^{r}_\psi$ and $\tilde{\bC}_\psi$.

For $r>0$ and $s \in (0,r/q]$, a \emph{projective (resp.\ pseudocoherent, fpd) $\varphi$-module} over $\bC^{[s,r]}_\psi$
is a finite projective  (resp.\ pseudocoherent, fpd) module $M$ over $\bC^{[s,r]}_\psi$ equipped with an isomorphism
$\varphi^* M  \otimes_{\bC^{[s/q,r/q]}_\psi} \bC^{[s,r/q]}_\psi \cong M \otimes_{\bC^{[s,r]}_\psi} \bC^{[s,r/q]}_\psi$
of modules over $\bC^{[s,r/q]}_\psi$, subject to topological conditions as listed below. We similarly define projective (resp.\ pseudocoherent, fpd) $\varphi$-modules over $\breve{\bC}^{[s,r]}_\psi$ and $\tilde{\bC}^{[s,r]}_\psi$.

As in Definition~\ref{D:pseudocoherent phi-module}, we also impose topological restrictions as follows.
\begin{itemize}
\item
For $* = \bC^{[s,r]}$ (resp.\ $* = \tilde{\bC}^{[s,r]}$), we insist that the module be \'etale-stably pseudocoherent.
(Recall that this condition is empty when $\psi$ is noetherian and $* = \bC^{[s,r]}$.)
\item
For $* = \bC^r$ (resp.\ $* = \tilde{\bC}^r$), we insist that  the module be complete for its natural topology and,
for every closed interval $I \subseteq [0,r)$, 
the base extension to $\bC^{[s,r]}$ (resp.\ $\tilde{\bC}^{[s,r]}$) satisfies the previous condition.
\item
For $* = \bC$ (resp.\ $* = \tilde{\bC}$), we insist that the module be complete for its natural topology and,
for some $r>0$, arises by base extension from a module over $\bC^r$ (resp.\ $\tilde{\bC}^r$)
satisfying the previous condition.
\end{itemize}

As usual, we omit the adjective \emph{projective} except in cases of emphasis.
\end{defn}

\begin{remark} \label{R:imperfect pseudocoherent}
Over a ring on which $\varphi$ is not bijective, the concept of a pseudocoherent $\varphi$-module requires some care. For example, if $R_\psi = K\{T\}$ for some perfect analytic field $K$, then the trivial $\varphi$-module structure on $R_\psi$ induces a bijective action of $\varphi$ on the pseudocoherent $R_\psi$-module $M = K\{T\}/(T)$, but
the map $\varphi^* M \to M$ is not injective.
\end{remark}

\begin{remark} \label{R:breve phi-modules}
If $*_\psi$ is a symbol with no accent on top, then base extension of $\varphi$-modules from $*_\psi$ to $\breve{*}_\psi$ is an equivalence of categories. The same is true for pseudocoherent $\varphi$-modules provided that $\breve{*}_\psi$ is pseudoflat over $*_\psi$; see \S\ref{subsec:Frobenius splittings}.
\end{remark}

\begin{lemma} \label{L:DM}
Take $R$ to be one of $\bA_\psi, \hat{\bA}_\psi, \tilde{\bA}_\psi$.
Let $M$ be a finite projective module over $R/(p^m)$
for some positive integer $m$ equipped with a semilinear $\varphi$-action.
Then there exists a
faithfully finite \'etale $R$-algebra $U$
such that $M \otimes_{R/(p^m)} U/(p^m)$ admits a basis fixed by $\varphi$.
More precisely, if $l<m$ is another positive integer
and $M \otimes_{R/(p^m)} R/(p^l)$ admits a $\varphi$-fixed basis, then $U$ can be chosen so that
this basis lifts
to a $\varphi$-fixed basis of $M \otimes_{R/(p^m)} U/(p^m)$.
\end{lemma}
\begin{proof}
As in \cite[Lemma~3.2.6]{part1}.
\end{proof}
\begin{cor} \label{C:fully faithful to extended}
The base change functors among the categories of $\varphi$-modules over the rings
\[
 \bA_\psi \to \breve{\bA}_\psi \to \hat{\bA}_\psi  \to
 \tilde{\bA}_\psi
\]
are exact tensor equivalences.
\end{cor}
\begin{proof}
It is enough to check that for each ring $R$ in the diagram, the base change functor
from $R$ to $\tilde{\bA}_\psi$ is an equivalence of categories.
To prove full faithfulness,
it suffices (as in \cite[Remark~4.3.4]{part1}) to prove that for any $\varphi$-module $M$ over $R$, any $\bv \in M \otimes_{R}
\tilde{\bA}_\psi$ fixed by $\varphi$ must belong to $M$.
For this, we may omit the case $R = \breve{\bA}_\psi$ thanks to
Remark~\ref{R:breve phi-modules}; we may thus assume that $R$ is $\varpi$-adically complete.
Apply Lemma~\ref{L:DM} to construct
a $R$-algebra $U$ which is the $\varpi$-adic completed direct limit of faithfully finite \'etale
$R$-subalgebras, such that $M \otimes_{R} U$ admits a $\varphi$-fixed basis.
Put $\tilde{U} = U \widehat{\otimes}_{R} \tilde{\bA}_\psi$,
so that we may view $\bv$ as an element of $M \otimes_{R} \tilde{U}$.
By writing $\bv$ in terms of a $\varphi$-fixed basis of $M \otimes_{R} U$
and noting that $U^{\varphi} = \tilde{U}^{\varphi}$,
we deduce that $\bv \in M \otimes_{R} U$. Since within $\tilde{U}$
we have $\tilde{\bA}_\psi \cap U = R$, we deduce that $\bv \in M$ as desired.

To prove essential surjectivity, it is now enough to check this for $\bA_\psi \to \tilde{\bA}_\psi$. From Theorem~\ref{T:pseudocoherent type A}, we obtain an equivalence of categories between $\gotho_{E_a}$-local systems on $\Spec (\tilde{R}_\psi \otimes_{\FF_p} \FF_{q})$
and $\varphi$-modules over $\tilde{\bA}_\psi \otimes_{\gotho_E} \gotho_{E_a}$.
Using Lemma~\ref{L:DM}, we may imitate the construction
in the proof of \cite[Theorem~8.5.3]{part1} to factor this functor through
the category of $\varphi$-modules over $\bA_\psi \otimes_{\gotho_E} \gotho_{E_a}$.
We then complete the proof by faithfully flat descent.
\end{proof}

\begin{defn}
Choose $R$ from among $\bB_\psi, \bB^\dagger_\psi, \bC_\psi, \breve{\bB}_\psi, \breve{\bB}^\dagger_\psi, \breve{\bC}_\psi$, then define $R_0$ to be $\bA_\psi,
\bA^\dagger_\psi,
\bA^{\dagger}_{\psi}, \breve{\bA}_\psi, \breve{\bA}^\dagger_\psi, \breve{\bA}^\dagger_\psi$, respectively.
We say a $\varphi$-module $M$ over $R$ is \emph{globally \'etale} if it arises by base extension from $R_0$. If instead this holds locally with respect to some rational covering of $\Spa(A,A^+)$ (whose terms inherit the weakly decompleting property via Proposition~\ref{P:weakly decompleting tower persistence}), we say that $M$ is \emph{\'etale}. One can also define \emph{globally pure} and \emph{pure} $\varphi$-modules by analogy with Definition~\ref{D:pure phi-module}, but we will not use them here.
\end{defn}

\begin{prop} \label{P:fully faithful imperfect}
The following functors are fully faithful.
\begin{enumerate}
\item[(a)]
The base change functor from $\varphi$-modules over $\bA_\psi^\dagger$
(resp.\ $\breve{\bA}_\psi^\dagger$)
 to $\varphi$-modules
over $\bA_\psi$ (resp.\ $\breve{\bA}_\psi$).
\item[(b)]
The base change functor from globally \'etale $\varphi$-modules over $\bB_\psi^\dagger$ (resp.\ $\breve{\bB}_\psi^\dagger$)
to $\varphi$-modules over
$\bC_\psi$ (resp.\ $\breve{\bC}_\psi$).
\end{enumerate}
\end{prop}
\begin{proof}
Again as in \cite[Remark~4.3.4]{part1}, part (a) reduces  to checking that
any $\varphi$-invariant element $\bv$ of $M \otimes_{\bA_\psi^\dagger} \bA_\psi$ (resp.\ $M \otimes_{\breve{\bA}_\psi^\dagger} \breve{\bA}_\psi$)
belongs to $M$ itself.
By extending scalars to $\tilde{\bA}_\psi$ and applying 
Theorem~\ref{T:pseudocoherent type A},
we deduce that $\bv$ belongs to $M \otimes_{\bA_\psi^\dagger} \tilde{\bA}^{\dagger}_\psi$
(resp.\ $M \otimes_{\breve{\bA}_\psi^\dagger} \tilde{\bA}^{\dagger}_\psi$).
Since $M$ is projective and $\bA_\psi \cap \tilde{\bA}_\psi^\dagger = \bA_\psi^\dagger$ (resp.\ $\breve{\bA}_\psi \cap \tilde{\bA}_\psi^\dagger = \breve{\bA}_\psi^\dagger$) 
within $\tilde{\bA}_\psi$, we deduce that $\bv \in M$ as desired.

Similarly, part (b) reduces to checking that any $\varphi$-invariant element $\bv$ of $M \otimes_{\bB_\psi^\dagger} \bC_\psi$
(resp.\ $M \otimes_{\breve{\bB}_\psi^\dagger} \breve{\bC}_\psi$)
belongs to $M$ itself.
By extending scalars to $\tilde{\bB}_\psi^\dagger$ and applying Corollary~\ref{C:globally etale BC},
we deduce that $\bv$ belongs to $M \otimes_{\bB_\psi^\dagger} \tilde{\bB}^{\dagger}_\psi$
(resp.\ $M \otimes_{\breve{\bB}_\psi^\dagger} \tilde{\bB}^{\dagger}_\psi$).
Since $M$ is projective and $\bC_\psi \cap \tilde{\bB}_\psi^\dagger = \bB_\psi^\dagger$
(resp.\ $\breve{\bC}_\psi \cap \tilde{\bB}_\psi^\dagger = \breve{\bB}_\psi^\dagger$)
within $\tilde{\bC}_\psi$
by Lemma~\ref{L:intersect dagger}, we deduce that $\bv \in M$ as desired.
\end{proof}

\begin{remark}
In general, base extension of $\varphi$-modules from
$\bA^\dagger_\psi$ to $\bA_\psi$ is not essentially surjective, even if $\psi$ is a decompleting tower. 
For the cyclotomic tower, a simple example
appears in \cite[\S 3.3]{tsuzuki-over}.
\end{remark}

\begin{defn} \label{D:type C imperfect phi-module}
For $r > 0$, a \emph{projective (resp.\ pseudocoherent, fpd) $\varphi$-bundle} 
over $\bC^{r}_\psi$
consists of a projective (resp.\ \'etale-stably pseudocoherent, \'etale-stably fpd) $\varphi^a$-module $M_{[t,s]}$ over $\bC^{[t,s]}_\psi$ for each pair $t,s$ with $s \in (0,r], t \in (0,s/q]$, together with isomorphisms
$M_{[t,s]} \otimes_{\bC^{[t,s]}_{\psi}} \bC^{[t',s']}_{\psi} \to M_{[t',s']}$
satisfying the cocycle condition for $t \leq t' \leq s' \leq s$. 
By taking a direct 2-limit \cite[Remark~1.2.9]{part1}
over all $r > 0$, we obtain the category of \emph{projective (resp.\ pseudocoherent, fpd) $\varphi$-bundles} over $\bC_\psi$.
We similarly define projective (resp.\ pseudocoherent, fpd) $\varphi$-bundles over $\breve{\bC}^{r}_\psi$ and $\breve{\bC}_\psi$.

Thanks to Lemma~\ref{L:weak flatness imperfect Robba}, the morphisms between these rings give rise to base extension functors between the various types of $\varphi$-modules and the corresponding types of $\varphi$-bundles. As usual, we omit the adjective \emph{projective} unless it is needed for emphasis.
\end{defn}

\begin{lemma} \label{L:phi-modules to phi-bundles}
If $\psi$ is noetherian, the following exact tensor categories are equivalent.
\begin{enumerate}
\item[(a)]
The category of projective (resp.\ pseudocoherent, fpd) $\varphi$-modules over $\bC^{r_0}_\psi$.
\item[(b)]
The category of projective (resp.\ pseudocoherent, fpd)  $\varphi$-bundles over $\bC^{r_0}_\psi$.
\item[(c)]
The category of projective (resp.\ pseudocoherent, fpd) $\varphi$-modules over $\bC^{[s,r_0]}_\psi$ for any $s \in (0, r_0/q]$.
\end{enumerate}
More precisely, the functor from (a) to (b) is base extension, the functor from (b) to (a) is the global sections functor, and the functor from (b) to (c) is restriction.
For general $\psi$, the same statements hold for projective $\varphi$-modules.
\end{lemma}
\begin{proof}
Full faithfulness of the base extension functor from (a) to (b)
follows from Lemma~\ref{L:Robba base extension pseudocoherent}
in the case of $\psi$ noetherian, and Lemma~\ref{L:global sections C} for general $\psi$.
To prove that the essential image includes a 
given projective (resp.\ pseudocoherent, fpd) $\varphi$-bundle $\{M_{[s,r]}\}$, 
note that any such bundle is uniformly pseudocoherent:
given a projective resolution of $M_{[rq^{-1},r]}$, we may pull back
using powers of $\varphi$ to get a projective resolution of $M_{[rq^{-n-1},rq^{-n}]}$
with modules of the same rank for $n=0,1,\dots$.
We may thus apply Lemma~\ref{L:Robba base extension pseudocoherent}
and Lemma~\ref{L:Robba base extension vector bundle} to deduce that
the $\varphi$-bundle arises from a projective (resp.\ pseudocoherent, fpd) module over $\bC_\psi$;
this proves the equivalence of (a) and (b). 
The equivalence of (b) and (c) follows from
Proposition~\ref{P:glueing projective modules over intervals}.
\end{proof}

We have the following analogue of \cite[Proposition~6.3.19]{part1}.
\begin{prop} \label{P:same cohomology}
For $r \in (0, r_0]$, either let $M_r$ be a $\varphi$-module over $\bC^r_\psi$,
or assume that $\psi$ is noetherian and let $M_r$ be a pseudocoherent $\varphi$-module over $\bC^r_\psi$.
For $s \in (0, r/q]$, put $M_{[s,r]} = M_r \otimes_{\bC^r_\psi} \bC^{[s,r]}_\psi$.
Then the vertical arrows in the diagram
\[
\xymatrix{
0 \ar[r] & M_r \ar^{\varphi-1}[r] \ar[d] & M_{r/q} \ar[r] \ar[d] & 0 \\
0 \ar[r] & M_{[s,r]} \ar^{\varphi-1}[r] & M_{[s,r/q]} \ar[r] & 0
}
\]
induce an isomorphism on the cohomology of the horizontal complexes.
\end{prop}
\begin{proof}
By interpreting the cohomology groups as Yoneda extension groups, we may deduce the claim immediately from Lemma~\ref{L:phi-modules to phi-bundles}.
\end{proof}

\subsection{\texorpdfstring{$\Gamma$}{Gamma}-modules}
\label{subsec:gamma-modules}

We next define $\Gamma$-modules over period rings. In the case of Galois towers these will indeed be certain modules equipped with continuous group actions, but we will phrase the definition in such a way as to apply also to non-Galois towers.

\begin{hypothesis}
Throughout \S\ref{subsec:gamma-modules}, assume that the tower $\psi$ is finite \'etale and weakly decompleting.
\end{hypothesis}

\begin{defn} \label{D:Cech complex}
For $k \geq 0$, let $\psi^k$ be the tower with
$A_{\psi^k, n}$ equal to the completion of the $(k+1)$-fold tensor product of $A_{\psi,n}$ over $A$. Note that this tower is perfectoid by Corollary~\ref{C:nonuniform perfectoid tensor product} but need not be weakly decompleting.

Let $*_\psi$ be one of the perfect or imperfect period rings described in Definition~\ref{D:period rings}. For $j=0,\dots,k+1$, let $i_{k,j}: *_{\psi^k} \to *_{\psi^{k+1}}$ be the map induced by the bounded homomorphism $\tilde{A}_{\psi^k} \to \tilde{A}_{\psi^{k+1}}$ taking
$x_0 \otimes \cdots \otimes x_k$ to 
$x_0 \otimes \cdots \otimes x_{j-1} \otimes 1 \otimes x_j \otimes \cdots \otimes x_k$. We may assemble the rings $*_{\psi^\bullet}$ into a complex by taking the differential $d^k: *_{\psi^k} \to *_{\psi^{k+1}}$ to be $\sum_{j=0}^{k+1} (-1)^j i_{k,j}$.
\end{defn}

\begin{lemma} \label{L:equalizer}
The equalizer of the two maps $i_{0,0}, i_{0,1}: R_{\psi} \to R_{\psi^1}$ is contained in $\overline{\varphi}(R_\psi)$.
\end{lemma}
\begin{proof}
Suppose $x$ belongs to the equalizer. Then for all $n$ sufficiently large,
$\theta(\varphi^{-n}([x]))$ belongs to the equalizer of $i_{0,0}, i_{0,1}: A_{\psi} \to A_{\psi^1}$; however, by faithfully flat descent this equalizer is equal to $A$. It follows that $\overline{\varphi}^{-1}(x) \in R_{\psi}$.
\end{proof}

\begin{defn} \label{D:Gamma-module}
Let $*_\psi$ be as in Definition \ref{D:Cech complex}. 
By a \emph{projective (resp.\ pseudocoherent, fpd) $\Gamma$-module}  over $*_\psi$, we will mean a pair $(M, \iota)$ in which $M$ is a finite projective (resp.\ pseudocoherent, fpd) $*_\psi$-module and $\iota: M \otimes_{i_{0,0}} *_{\psi^1} \to M \otimes_{i_{0,1}} *_{\psi^1}$ is a $*_{\psi^1}$-linear isomorphism satisfying the \emph{cocycle condition}: 
\[
(\iota \otimes i_{1,2}) \circ (\iota \otimes i_{1,0})
= \iota \otimes i_{1,1}
\]
where $\iota \otimes i_{1,j}$ denotes the base change to $*_{\psi^2}$ 
of $\iota$ along $i_{1,j}$.
Thanks to the cocycle condition, we may unambiguously define the complex $M^{\bullet} = M \otimes_{\iota} *_{\psi^\bullet}$
with differentials $d^k = \sum_{j=0}^{k+1} (-1)^j i_{k,j}$
as in Definition~\ref{D:Cech complex}.
Let $H^\bullet_{\Gamma}(M)$ denote the cohomology of $M^{\bullet}$.

We will have occasion later to consider data $\iota$ as above associated to somewhat more general modules $M$.
When this occurs, we will refer to $\iota$ as a \emph{$\Gamma$-structure} on $M$.
\end{defn}

The term \emph{$\Gamma$-module} is suggested by the following example, which includes the setup of traditional $p$-adic Hodge theory.
\begin{defn}
For $\Gamma$ a topological group and $A$ a topological left $\Gamma$-module, let $C_{\cont}^{\bullet}(\Gamma,A)$ be the complex of inhomogeneous continuous $\Gamma$-cochains with values in $A$. That is,
for $k \geq 0$, $C_{\cont}^{k}(\Gamma,A) = \Cont(\Gamma^k, A)$
and the differential $d^k: C_{\cont}^k(\Gamma,A) \to C_{\cont}^{k+1}(\Gamma,A)$ is given by
\begin{align*}
d^k(f)(\gamma_1,\dots,\gamma_{k+1})
& = \gamma_1(f(\gamma_2,\dots,\gamma_{k+1})) \\
&+ (-1)^i \sum_{i=1}^k (-1)^i f(\gamma_1,\dots,\gamma_{i-1},
\gamma_i \gamma_{i+1}, \gamma_{i+2}, \dots, \gamma_{k+1})\\
&+ (-1)^{k+1} f(\gamma_1,\dots,\gamma_k).
\end{align*}
\end{defn}

\begin{example} \label{exa:Galois tower}
Suppose $\psi$ is Galois with group $\Gamma$.
For each $k \geq 0$, there is an isomorphism of topological rings
\[
*_{\psi^k} \cong \Cont(\Gamma^k, *_\psi)
\]
induced by sending $x_0 \otimes \cdots \otimes x_k$ to $(\gamma_1,\dots,\gamma_k)\mapsto x_0\gamma_1(x)\cdots(\gamma_1\cdots\gamma_k)(x_k)$. It is then clear that under this isomorphism $i_{k,0} \circ \cdots  \circ i_{0,0}$ identifies $*_\psi$ with the constant maps in $\Cont(\Gamma^k, *_\psi)$
and evaluation at the identity in $\Gamma^k$ induces the multiplication map $*_{\psi^k} \to *_\psi$.
Using this identification, we may identify a projective (resp.\ pseudocoherent, fpd) $\Gamma$-module over $*_{\psi}$ with a finite projective (resp.\ pseudocoherent, fpd) module $M$ over $*_{\psi}$
equipped with a continuous semilinear $\Gamma$-action;
note that the continuity of the $\Gamma$-action is a formal consequence of the structure of the ring $*_{\psi^1}$ and does not need to be imposed separately. Similarly,
the complex $M^{\bullet}$ may be naturally identified with
$C_{\cont}^{\bullet}(\Gamma, M)$.
\end{example}

We may generalize the previous construction as follows. 
\begin{example} \label{exa:Galois tower2}
Suppose there exist a finite \'etale tower
$\psi'$ which is Galois with group $\Gamma_1$ and a closed subgroup $\Gamma_2$ of $\Gamma_1$ with $A_{\psi',n}^{\Gamma_2} = A_{\psi,n}$ for all $n \geq 0$.
Since $\psi$ is perfectoid, we have $*_{\psi'}^{\Gamma_2} = *_\psi$.

The action of $\Gamma_1^{k+1}$ on $*_{\psi^k}$ 
corresponds to the action of $\Gamma_1$ on $C_{\cont}^k(\Gamma_1,*_\psi)$ given by
\begin{equation} \label{eq:group action on cochains}
(\eta_0,\dots,\eta_k)(f)(\gamma_1,\dots,\gamma_k) = 
\eta_0(f(\eta_0^{-1} \gamma_1 \eta_1, \cdots, \eta_{k-1}^{-1} \gamma_k \eta_k)).
\end{equation}
In this setting, a projective (resp.\ pseudocoherent, fpd) $\Gamma$-module 
over $*_\psi$ is the same as a finite projective (resp.\ pseudocoherent, fpd) $*_\psi$-module $M$ equipped with a continuous semilinear $\Gamma_1$-action on $M' = M \otimes_{*_{\psi}} *_{\psi'}$ whose restriction to
$\Gamma_2$ fixes $M$. We have an identification
\[
M^k = C_{\cont}^k(\Gamma_1,M')^{\Gamma_2^{k+1}}.
\]
Note that this construction is naturally independent of $\psi'$;
in particular, in case $\Gamma_2$ is normal in $\Gamma_1$, a $\Gamma$-module is the same as a $(\Gamma_1/\Gamma_2)$-module in the sense of 
Example~\ref{exa:Galois tower}.
\end{example}

\begin{defn}
For each nonnegative integer $n$, let $\psi_{(n)}$ be the truncated tower $(A_{\psi,n} \to A_{\psi,n+1} \to \cdots)$.
\end{defn}

\begin{remark} \label{R:small connection}
Suppose that $*_\psi$ is equipped with a norm, and let $M$ be a pseudocoherent $\Gamma$-module over $*_\psi$. Equip $M$ with the quotient norm induced by some finite presentation of its underlying module.
Then for any $k \geq 0$ and $\epsilon \in (0,1)$, there exists $n_0 \geq 0$ such that for $n \geq n_0$, if we replace $\psi$ with $\psi_{(n)}$ by base extension,
then the differential $M \otimes_{*_\psi} *_{\psi^k}
\to M \otimes_{*_\psi} *_{\psi^{k+1}}$
differs from the base extension of the differential 
$*_{\psi^k} \to *_{\psi^{k+1}}$ by a map whose operator norm $c$ is less than $\epsilon$.
\end{remark}

\begin{remark} \label{R:small connection Galois}
In the setting of Example~\ref{exa:Galois tower}, Remark~\ref{R:small connection} can be interpreted as follows.
For $M$ a pseudocoherent $*_\psi$-module, the data of a $\Gamma$-module structure on $M$ amounts to an element of the pointed set $H^1_{\cont}(\Gamma, \Aut_{*_\psi}(M))$. By fixing a norm on $M$ and replacing $\Gamma$ with a suitable open subgroup, for any $\epsilon > 0$ we obtain a 1-cocycle with values in the subgroup of $\Aut_{*_\psi}(M)$ consisting of elements which differ from the identity by an endomorphism of operator norm less than $\epsilon$.

This argument can be phrased even more concretely as follows.
Concretely, for any $\epsilon>0$, for each $\bv$ in a set of generators of $M$, $(\gamma-1)(\bv)$ has norm at most $\epsilon$ for all $\gamma$ in some open subgroup of $\Gamma$;
using the ``Leibniz rule'' identity
\begin{equation}
(\gamma-1)(\overline{x} \bv)
= (\gamma-1)(\overline{x}) \bv + \gamma(\overline{x})
(\gamma-1)(\bv)
\end{equation}
together with continuity of the action of $\Gamma$ on $*_\psi$,
one reaches the same conclusion for arbitrary $\bv \in M$.
\end{remark}

\subsection{Decompleting towers}
\label{subsec:decompleting towers}

In order to descend $(\varphi, \Gamma)$-modules, we must augment the definition of a weakly decompleting tower with one extra hypothesis.

\begin{defn} \label{D:decompleting tower}
We say that the tower $\psi$ is \emph{decompleting} if it is finite \'etale and weakly decompleting, and additionally for each sufficiently large $n$, the complex
$\overline{\varphi}_\varpi^{-1}(R_{\psi_{(n)}^\bullet})/R_{\psi_{(n)}^\bullet}$ is exact.
This formally implies that each of the $\psi_{(n)}$ is also decompleting.
\end{defn}

The property of a tower being decompleting is less stable under base extension than the property of being weakly decompleting. For finite \'etale extensions,
we will establish stability after some further study of $\Gamma$-modules; see
Corollary~\ref{C:finite etale decompleting}. For rational localizations, we are forced to build this into the definition.
\begin{defn}
We say that a tower $\psi$ is \emph{locally decompleting} if $(A,A^+)$ is stably uniform
and the base extension of $\psi$ along every rational localization of $(A,A^+)$
is decompleting.
\end{defn}

For the remainder of \S\ref{subsec:decompleting towers}, let $\psi$ be a decompleting tower. 
Our next few lemmas are formal refinements of the definition.
\begin{lemma} \label{L:decompleting all}
For each nonnegative integer $n$, the complex $\overline{\varphi}^{-1}_{\varpi}(R_{\psi_{(n)}^\bullet})/R_{\psi_{(n)}^\bullet}$ is strict exact.
\end{lemma}
\begin{proof}
Exactness follows from the fact that the complex $\overline{\varphi}^{-1}_{\varpi}(R_{\psi_{(n')}^{\bullet}})/R_{\psi_{(n')}^\bullet}$ is obtained from  $\overline{\varphi}^{-1}_{\varpi}(R_{\psi_{(n)}^\bullet})/R_{\psi_{(n)}^\bullet}$ by tensoring along the flat morphism $R_{\psi_{(n)}} \to  R_{\psi_{(n')}}$.
Strictness follows from exactness via the open mapping theorem
(Theorem~\ref{T:open mapping}).
\end{proof}

\begin{lemma} \label{L:decompleting uniform}
For each nonnegative integer $n$,  the complex $\tilde{R}_{\psi_{(n)}^\bullet}/R_{\psi_{(n)}^\bullet}$ is strict exact.
\end{lemma}
\begin{proof}
We may simplify notation by considering only the case $n=0$.
It suffices to check that the complexes $\overline{\varphi}^{-m}_{\varpi}(R_{\psi^\bullet})/R_{\psi^\bullet}$ are uniformly strict exact over all $m > 0$. To check this, fix
$k \geq 0$
and choose $\epsilon, c \geq 1$ so that every cochain of $\overline{\varphi}^{-1}_{\varpi}(R_{\psi^k})/R_{\psi^k}$ can be lifted to a cochain of $\overline{\varphi}^{-1}_{\varpi}(R_{\psi^k})$ at most $\epsilon$ times its norm and every cocycle in $\overline{\varphi}^{-1}_{\varpi}(R_{\psi^k})/R_{\psi^k}$
can be lifted to a cochain of norm at most $c$ times its norm.
Let $x$ be a cocycle in $\overline{\varphi}^{-m}_{\varpi}(R_{\psi^k})/R_{\psi^k}$ of norm $t$.
We may then lift the projection of $x$ to $\overline{\varphi}^{-m}_{\varpi}(R_{\psi^k})/\overline{\varphi}^{-m+1}_{\varpi}(R_{\psi^k})$
to a cochain of 
norm at most $c^{p^{1-m}} t$.
This cochain in turn lifts to a cochain $y$ in $\overline{\varphi}^{-m}_{\varpi}(R_{\psi^{k-1}})/R_{\psi^{k-1}}$ of norm at most $(c \epsilon)^{p^{1-m}} t$.
Now $x-d^{k-1}(y)$ is a cocycle in  $\overline{\varphi}^{1-m}_{\varpi}(R_{\psi^{k}})/R_{\psi^{k}}$; repeating the process, we lift $x$ to a cochain of norm at most
$(c \epsilon)^{p^{-1} + \cdots + p^{1-m}} t < (c \epsilon)^{1/(p-1)}t$.
This proves the claim.
\end{proof}

\begin{cor} \label{C:decompleting uniform}
There exists $r_0 > 0$ such that for $0 < s \leq r \leq r_0$,
for each nonnegative integer $n$, 
the complexes
$\varphi_\varpi^{-1}({\bA}^{\dagger,r/p^h}_{\psi_{(n)}^\bullet})/\bA^{r}_{\psi_{(n)}^\bullet}$ 
and
$\tilde{\bA}^{\dagger,r}_{\psi_{(n)}^\bullet}/\bA^{r}_{\psi_{(n)}^\bullet}$ 
are strict exact for the norms $\lambda(\overline{\alpha}_{\psi_{(n)}^\bullet}^r)$,
and the complexes
$\varphi_\varpi^{-1}(\bC^{[s/p^h,r/p^h]}_{\psi_{(n)}^\bullet})/\bC^{[s,r]}_{\psi_{(n)}^\bullet}$
and
$\tilde{\bC}^{[s,r]}_{\psi_{(n)}^\bullet}/\bC^{[s,r]}_{\psi_{(n)}^\bullet}$ 
are strict exact for the norms 
$\max\{\lambda(\overline{\alpha}_{\psi_{(n)}^\bullet}^s),
\lambda(\overline{\alpha}_{\psi_{(n)}^\bullet}^r)\}$.
\end{cor}
\begin{proof}
This is immediate from Lemma~\ref{L:optimal lifts}
and Lemma~\ref{L:decompleting uniform}.
\end{proof}

We next indicate how to execute some descent constructions, generalizing the approach to decompletion described in \cite[\S 2]{kedlaya-new-phigamma}.
To make corresponding statements about pseudocoherent $\Gamma$-modules, we need some extra restrictions on $\psi$; see \S\ref{subsec:Frobenius splittings}.

\begin{lemma} \label{L:Gamma quasi-isomorphism}
For $r_0$ as in Corollary~\ref{C:decompleting uniform},
choose $r,s$ with $0 < s \leq r \leq r_0$ and $k \geq 0$.
\begin{enumerate}
\item[(a)]
For $M$ a $\Gamma$-module over $R_\psi$, 
there exists $m_0 \geq 0$ such that for all $m \geq m_0$,
the complexes
\[
M \otimes_{R_\psi} (\overline{\varphi}_\varpi^{-m-1}(R_{\psi^{\bullet}})/\overline{\varphi}_\varpi^{-m}(R_{\psi^\bullet})),
\qquad
M \otimes_{R_\psi} (\tilde{R}_{\psi^{\bullet}}/\overline{\varphi}_\varpi^{-m}(R_{\psi^\bullet}))
\]
are strict exact in degree $k$.
\item[(b)]
For $M$ a $\Gamma$-module over $\bA^r_{\psi}$,
there exists $m_0 \geq 0$ such that for all $m \geq m_0$, 
the complexes
\[
M \otimes_{\bA^r_{\psi}} (\varphi_\varpi^{-m-1}(\bA^{rp^{-h(m+1)}}_{\psi^{\bullet}})/\varphi_\varpi^{-m}(\bA^{rp^{-hm}}_{\psi^{\bullet}})),
\,
M \otimes_{\bA^r_{\psi}} 
(\tilde{\bA}^{r}_{\psi^{\bullet}}/\varphi_\varpi^{-m}(\bA^{rp^{-hm}}_{\psi^{\bullet}}))
\]
are strict exact in degree $k$.
\item[(c)]
For $M$ a $\Gamma$-module over $\bC^{[s,r]}_{\psi}$,
there exists $m_0 \geq 0$ such that for all $m \geq m_0$, 
the complexes
\begin{gather*}
M \otimes_{\bC^{[s,r]}_{\psi}} (\varphi_\varpi^{-m-1}(\bC^{[sp^{-h(m+1)},rp^{-h(m+1)}]}_{\psi^{\bullet}})/\varphi_\varpi^{-m}(\bC^{[sp^{-hm},rp^{-hm}]}_{\psi^{\bullet}})),
\\
M \otimes_{\bC^{[s,r]}_{\psi}} 
(\tilde{\bC}^{[s,r]}_{\psi^{\bullet}}/\varphi_\varpi^{-m}(\bC^{[sp^{-hm},rp^{-hm}]}_{\psi^{\bullet}}))
\end{gather*}
are strict exact in degree $k$.
\end{enumerate}
\end{lemma}
\begin{proof}
We prove only (a), as (b) and (c) follow by similar arguments
thanks to Corollary~\ref{C:decompleting uniform}.
We may freely replace $\psi$ with $\psi_{(n)}$ for some conveniently large $n$
depending on $k$, so as to land in the setting of Remark~\ref{R:small connection}. By Lemma~\ref{L:decompleting all}, we may then choose $m$ so that for the differentials in and out of $\overline{\varphi}_\varpi^{-m-1}(R_{\psi^{k}})/\overline{\varphi}_\varpi^{-m}(R_{\psi^k})$, any element of the image having norm $t$ lifts to an element of the domain having norm at most $c^{-1/3} t$. 
We then find that every cocycle in $M \otimes_{R_\psi} (\overline{\varphi}_\varpi^{-m-1}(R_{\psi^{k}})/\overline{\varphi}_\varpi^{-m}(R_{\psi^k}))$ of norm $t$ differs from a coboundary by an element of norm at most $c^{1/3}t$.
This implies the first claim because $M$ is complete; the second claim follows by a similar argument in which Lemma~\ref{L:decompleting all} is replaced with 
Lemma~\ref{L:decompleting uniform}.
\end{proof}
\begin{cor} \label{C:finite etale decompleting}
Let $\psi'$ be the base extension of $\psi$ along any finite \'etale morphism. Then $\psi'$ is also decompleting.
\end{cor}
\begin{proof}
View $R_{\psi'_{(n)}}$ as a $\Gamma$-module over $R_{\psi_{(n)}}$ 
and apply Lemma~\ref{L:Gamma quasi-isomorphism}.
\end{proof}

\begin{lemma}  \label{L:projective dense subring}
Let $(A_i, \alpha_i)_{i \in I}$ be a directed system in the category of Banach rings and submetric homomorphisms. Let $A$ be the direct limit of the $A_i$ equipped with the infimum seminorm.
Let $A \to R$ be an isometric homomorphism with dense image from $A$ to a Banach ring $R$. Then every finite projective module over $R$ arises by base extension from some finite projective module over some $A_i$.
\end{lemma}
\begin{proof}
Let $M$ be a finite projective $R$-module, and choose an $R$-linear surjection $F \to M$ with $F$ a finite free $R$-module. Choose a projector on $F$ corresponding to a splitting of $F \to M$, and represent this projector by the matrix $U$.
We may then choose a matrix $V$ over some $A_i$ such that $\left| U - V \right| < \left| U \right|^{-3}$. Define the sequence $W_0, W_1, \dots$ by Newton iteration, taking $W_0 = V$ and $W_{l+1} = 3W_l^2 - 2W_l^3$. Since
\begin{align*}
W_{l+1} - W_l &= (W_l^2 - W_l)(1-2W_l) \\
W_{l+1}^2 - W_{l+1} &= (W_l^2 - W_l)^2(4W_l^2 - 4W_l - 3),
\end{align*}
by induction on $l$ we have
\begin{align*}
\left|W_l - U\right| &< \left|U\right|^{-2} \\
\left|W_l^2 - W_l\right| &\leq \left|U \right|^{-2} \left( \left| V^2-V \right| \left| U\right|^2 \right)^{2^l}.
\end{align*}
Consequently, the $W_l$ converge to a projector $W$ over $A_i$ satisfying
$\left| U - W \right| < \left| U \right|^{-2}$; 
this implies that $\left| UW + (1-U)(1-W) - 1 \right| < 1$,
whence the image of $W$ is a finite projective module $M_i$ over $A_i$ satisfying $M_i \otimes_{A_i} R \cong M$.
\end{proof}

\begin{lemma} \label{L:Gamma equivalences}
Choose $r,s$ with $0<s \leq r$. For $* \in \{R, \bA^{r}, \bC^{[s,r]}\}$, base extension of $\Gamma$-modules from $\breve{*}_\psi$ to $\tilde{*}_{\psi}$
is an equivalence of tensor categories.
\end{lemma}
\begin{proof}
The base extension functor is fully faithful by
Lemma~\ref{L:Gamma quasi-isomorphism},
so it suffices to check essential surjectivity. 
We explain only the case $* = R$ in detail, the other cases being similar.
Let $\tilde{M}$ be a $\Gamma$-module over $\tilde{R}_\psi$.
By Lemma~\ref{L:projective dense subring}, for some $m \geq 0$ we can write the underlying module of $\tilde{M}$ as $M \otimes_{\overline{\varphi}_\varpi^{-m}(R_\psi)} \tilde{R}_\psi$ for some finite projective module $M$ over $\overline{\varphi}_\varpi^{-m}(R_\psi)$;
to simplify notation let us assume that $m=0$.
Thanks to full faithfulness, we are free to replace $\psi$ with $\psi_{(n)}$ for some suitable $n$, so we may put ourselves in the setting of
Remark~\ref{R:small connection} for $k=0$ using a presentation of $\tilde{M}$ induced by a presentation of $M$.  
That is, the given differential on $\tilde{M}$ differs from the base extension of some differential on $M$ (initially the trivial one) by a map of operator norm $c<1$.
We may then apply Lemma~\ref{L:Gamma quasi-isomorphism} (for some sufficiently large $m$) to modify the base extension to reduce $c$ by a constant factor (depending only on $m$). By iterating this construction, we reach the desired conclusion.
\end{proof}

\subsection{\texorpdfstring{$(\varphi, \Gamma)$}{(phi, Gamma)}-modules}
\label{subsec:phigamma}

We now incorporate the action of Frobenius and use our results on $\Gamma$-modules to descend relative $(\varphi, \Gamma)$-modules over perfect period rings (which arise from the global relative $(\varphi, \Gamma)$-modules described in \cite[\S 9]{part1}) to imperfect period rings.

\begin{hypothesis} \label{H:phigamma}
Throughout \S\ref{subsec:phigamma}, 
retain Hypothesis~\ref{H:weakly decompleting Frobenius},
and assume in addition that the tower $\psi$ is decompleting.
\end{hypothesis}

\begin{defn}
Let $*_\psi$ be one of the perfect or imperfect period rings of Definition~\ref{D:period rings} on which $\varphi$ acts. 
By a \emph{projective (resp.\ pseudocoherent, fpd) $(\varphi, \Gamma)$-module} over $*_\psi$, we will mean
a projective (resp.\ pseudocoherent, fpd) $\Gamma$-module $M$ over $*_\psi$ equipped with an isomorphism $\varphi^* M \cong M$ of projective (resp.\ pseudocoherent, fpd) $\Gamma$-modules. We say that a $(\varphi, \Gamma)$-module is \emph{(globally) \'etale} if this is true of the underlying $\varphi$-module. We similarly define
\emph{projective (resp.\ pseudocoherent, fpd) $(\varphi, \Gamma)$-modules} over $\bC^r_\psi, \breve{\bC}^r_\psi, \tilde{\bC}^r_\psi, \bC^{[s,r]}_\psi, \breve{\bC}^{[s,r]}_\psi, \tilde{\bC}^{[s,r]}_\psi$
for $r>0$ and $s \in (0,r/q]$, as well as
\emph{projective (resp.\ pseudocoherent, fpd) $(\varphi, \Gamma)$-bundles} over $\bC^r_\psi, \breve{\bC}^r_\psi, \tilde{\bC}^r_\psi, \bC_\psi, \breve{\bC}_\psi, \tilde{\bC}_\psi$. 
We also impose topological restrictions as in Definition~\ref{D:imperfect phi-modules}.

As in Definition~\ref{D:Gamma-module}, we will sometimes have occasion to consider the situation where $M$ is a module with $\Gamma$-structure together with a compatible semilinear $\varphi$-action. We will refer to such extra data on $M$ as a \emph{$(\varphi, \Gamma)$-structure} on $M$.
\end{defn}

\begin{theorem} \label{T:perfect equivalence1}
The (exact tensor) categories of (projective) $(\varphi, \Gamma)$-modules over the rings in the diagram
\[
\xymatrix@R=30pt@!C=60pt{
 \bA^{\dagger}_\psi \ar[r] \ar[d]  & \breve{\bA}^\dagger_\psi \ar[d] \ar[r]
 &  \hat{\bA}^\dagger_\psi \ar[d] \ar[r] &
 \tilde{\bA}^\dagger_\psi \ar[d] \\
 \bA_\psi \ar[r] & \breve{\bA}_\psi \ar[r] & 
 \hat{\bA}_\psi \ar[r] &\tilde{\bA}_\psi
}
\]
are equivalent via the apparent base change functors.
(If $A$ is a Banach algebra over $\gotho_{E_a}$,
then by Theorem~\ref{T:pseudocoherent type A} these categories are equivalent to the category of $\gotho_{E_a}$-local systems on $\Spa(A,A^+)$.)
\end{theorem}
\begin{proof}
The categories of $(\varphi, \Gamma)$-modules over $\tilde{\bA}^\dagger_\psi$
and $\tilde{\bA}_\psi$ are equivalent to each other by Theorem~\ref{T:pseudocoherent type A}.
It thus remains to produce enough equivalences of categories to link the rings in the right column
to all of the other rings in the diagram.
Among the rings in the bottom row of the diagram (which all carry the weak topologies), all of the
claimed equivalences follow from Corollary~\ref{C:fully faithful to extended}.

We next check that the base change functors on $(\varphi, \Gamma)$-modules
from all of the rings in the top
row of the diagram down to $\tilde{\bA}_\psi$ are fully faithful.
It suffices to check the corresponding statement for $\varphi$-modules,
which follows from Proposition~\ref{P:fully faithful
imperfect}(a) in the case of $\bA^\dagger_\psi$ and from similar arguments in the other cases.

It finally remains to check that base extension of $(\varphi,\Gamma)$-modules
from $\bA^{\dagger}_\psi$ to $\tilde{\bA}^\dagger_\psi$ is essentially surjective. For this, it suffices to apply Lemma~\ref{L:Gamma equivalences}
to descend the underlying $\Gamma$-module.
\end{proof}

\begin{theorem} \label{T:perfect equivalence1a}
The (exact tensor) categories of globally \'etale $(\varphi, \Gamma)$-modules over the rings in the diagram
\begin{equation} \label{eq:big diagram}
\xymatrix@R=30pt@!C=60pt{
\bC_\psi \ar[r] & \breve{\bC}_\psi \ar[rr] & & \tilde{\bC}_\psi \\
 \bB^{\dagger}_\psi \ar[d] \ar[r] \ar[u] & \breve{\bB}^\dagger_\psi \ar[d] \ar[r]
 \ar[u] &  \hat{\bB}^\dagger_\psi \ar[d] \ar[r] & \tilde{\bB}^\dagger_\psi \ar[d] \ar[u]\\
 \bB_\psi \ar[r] & \breve{\bB}_\psi \ar[r] & \hat{\bB}_\psi \ar[r] & \tilde{\bB}_\psi
}
\end{equation}
are equivalent via the apparent base change functors.
(If $A$ is a Banach algebra over $\gotho_{E_a}$,
then by Theorem~\ref{T:pseudocoherent type A} these categories are equivalent to the category of isogeny $\gotho_{E_a}$-local systems on $\Spa(A,A^+)$.)

\end{theorem}
\begin{proof}
If we omit the first row from the diagram, then
the claim follows from Theorem~\ref{T:perfect equivalence1}.
We may join $\tilde{\bC}_\psi$ to the rest of the diagram using
Corollary~\ref{C:globally etale BC}.
To connect $\bC_\psi$ (and hence $\breve{\bC}_\psi$, thanks to Remark~\ref{R:breve phi-modules})
to the rest of the diagram,
note that base extension of arbitrary projective $\Gamma$-modules
from $\bC_\psi$ to $\tilde{\bC}_\psi$ is fully faithful by
Lemma~\ref{L:Gamma quasi-isomorphism}, and hence the same is true of 
globally \'etale $(\varphi, \Gamma)$-modules; on the other hand, 
this functor is also essentially surjective because base extension from $\bB^\dagger_\psi$ to $\tilde{\bC}_\psi$ is known to be essentially surjective.
This completes the proof.
\end{proof}

\begin{theorem} \label{T:add tilde1}
For any $r,s$ with $0 < s \leq r/q$, the following exact tensor categories are equivalent.
\begin{enumerate}
\item[(a)]
The category of $(\varphi, \Gamma)$-modules over $\bC_\psi$.
\item[(b)]
The category of $(\varphi,\Gamma)$-bundles over $\bC_\psi$.
\item[(c)]
The category of $(\varphi, \Gamma)$-modules over $\breve{\bC}_\psi$.
\item[(d)]
The category of $(\varphi, \Gamma)$-bundles over $\breve{\bC}_\psi$.
\item[(e)]
The category of $(\varphi, \Gamma)$-modules over $\breve{\bC}^{[s,r]}_\psi$.
\item[(f)]
The category of $(\varphi, \Gamma)$-modules over $\tilde{\bC}_\psi$.
\item[(g)]
The category of $(\varphi,\Gamma)$-bundles over $\tilde{\bC}_\psi$.
\item[(h)]
The category of $(\varphi, \Gamma)$-modules over $\tilde{\bC}^{[s,r]}_\psi$.
\item[(i)]
The category of vector bundles on the relative Fargues-Fontaine curve
$\FFC_{\tilde{R}_\psi}$ equipped with a semilinear action of $\Gamma$ which is continuous on sections over any $\Gamma$-invariant adic affinoid subspace.
\item[(j)]
The category of $\varphi$-modules over $\tilde{\bC}_{\Spa(A,A^+)}$.
\item[(k)]
The category of $\varphi$-bundles over  $\tilde{\bC}_{\Spa(A,A^+)}$.
\item[(l)]
The category of $\varphi$-modules over $\tilde{\bC}^{[s,r]}_{\Spa(A,A^+)}$.
\end{enumerate}
\end{theorem}
\begin{proof}
The equivalences among (f)--(l) follow from
Theorem~\ref{T:perfect generalized phi-modules}.
By Lemma~\ref{L:phi-modules to phi-bundles}, the global sections functor from (b) to (a) is an equivalence.
By Remark~\ref{R:breve phi-modules}, the base extension functor from (a) to (c) is an equivalence.
By Lemma~\ref{L:Gamma equivalences}, the base extension functors from (b) 
to (g) and from (e) to (h) are equivalences.
We may now see that the base extension functor from (d) to (g) is essentially surjective (because the equivalence from (b) to (g) factors through (d))
and fully faithful (because the functor from (e) to (h) is fully faithful), hence an equivalence.
\end{proof}
\begin{cor}
If $\psi$ is locally decompleting, then
the category of $(\varphi, \Gamma)$-modules over $\bC_\psi$
admits glueing over coverings of (i.e., is a stack on) $\Spa(A,A^+)$.
\end{cor}
\begin{proof}
Thanks to Theorem~\ref{T:add tilde1},
this follows from the corresponding statement for $\varphi$-modules over $\tilde{\bC}_\psi$, which follows from Theorem~\ref{T:perfect Robba Kiehl}.
\end{proof}

\begin{theorem} \label{T:perfect equivalence2b}
If $\psi$ is locally decompleting, then
base extension of \'etale $(\varphi, \Gamma)$-modules from $\bC_\psi$
to $\tilde{\bC}_\psi$ is an equivalence of exact tensor categories.
(If $A$ is a Banach algebra over $\gotho_{E_a}$,
then by Theorem~\ref{T:pseudocoherent type A} these categories are equivalent to the category of $E_a$-local systems on $\Spa(A,A^+)$.)
\end{theorem}
\begin{proof}
Full faithfulness follows from Theorem~\ref{T:add tilde1}.
Essential surjectivity follows by combining Theorem~\ref{T:add tilde1}
with Theorem~\ref{T:perfect equivalence1a}.
\end{proof}

\begin{remark}
We have no natural imperfect analogue of a $B$-pair; for this reason, we do not include an analogue of Theorem~\ref{T:pseudocoherent B-pairs} here.
\end{remark}

We next compare the cohomology of $(\varphi, \Gamma)$-modules with the (pro-)\'etale cohomology of local systems.

\begin{defn}
For $M$ a $(\varphi, \Gamma)$-module over a ring on which $\varphi$ acts, we define the cohomology groups
$H^i_{\varphi, \Gamma}(M)$ to be the cohomology groups of the total complex in the category of $\Gamma$-modules associated to the double complex
\[
0 \to M^{\bullet} \stackrel{\varphi-1}{\to} M^{\bullet} \to 0,
\]
where $M^{\bullet}$ is as defined in Definition~\ref{D:Gamma-module}.
\end{defn}

\begin{theorem} \label{T:Galois cohomology1}
In Theorem~\ref{T:perfect equivalence1},
the equivalences of categories induce isomorphisms of $(\varphi, \Gamma)$-cohomology groups and (in case $A$ is a Banach algebra over $\gotho_{E_a}$) cohomology of the corresponding $\gotho_{E_a}$-local system on $\Spa(A,A^+)$.
\end{theorem}
\begin{proof}
For the base rings $\tilde{\bA}_\psi, \tilde{\bA}^{\dagger}_{\psi}$, this 
follows from Theorem~\ref{T:pseudocoherent type A}.
For the other base rings, we may reduce to the previous  case using Lemma~\ref{L:Gamma quasi-isomorphism}.
\end{proof}

\begin{theorem} \label{T:Galois cohomology1c}
Let $M$ be a $(\varphi, \Gamma)$-module over $\bC_\psi$.
\begin{enumerate}
\item[(a)]
Put $\breve{M} = M \otimes_{\bC_\psi} \breve{\bC}_\psi$,
$\tilde{M} = M \otimes_{\bC_\psi} \tilde{\bC}_\psi$.
Then the natural maps $H^i_{\varphi, \Gamma}(M) \to H^i_{\varphi,\Gamma}(\breve{M}) \to H^i_{\varphi,\Gamma}(\tilde{M})$ are bijections for all $i \geq 0$.
\item[(b)]
Choose $r_0>0$ such that $M$ descends to a $(\varphi, \Gamma)$-module $M_r$ over $\bC^{r_0}_\psi$. For $0 < s \leq r \leq r_0$, put $M_{[s,r]} = M_r \otimes_{\bC^{r_0}_\psi} \bC^{[s,r]}_\psi$,
$\breve{M}_{[s,r]} = M_r \otimes_{\bC^{r_0}_\psi} \breve{\bC}^{[s,r]}_\psi$.
$\tilde{M}_{[s,r]} = M_r \otimes_{\bC^{r_0}_\psi} \tilde{\bC}^{[s,r]}_\psi$.
Then for any $r,s$ with $0 < r \leq r_0$, $0 < s \leq r/q$, there exists $m_0 \geq 0$ such that for any integer $m \geq m_0$, 
the cohomology groups of the total complexes associated to each of the complexes
\begin{gather*}
0 \to \varphi^{-m}(M_{[s/q^{m}, r/q^{m}]}^\bullet) \stackrel{\varphi-1}{\to}
\varphi^{-m}(M_{[s/q^{m}, r/q^{m+1}]}^\bullet) \to 0 \\
0 \to \breve{M}_{[s,r]}^\bullet \stackrel{\varphi-1}{\to}
\breve{M}_{[s,r/q]}^\bullet \to 0 \\
0 \to \tilde{M}_{[s,r]}^\bullet \stackrel{\varphi-1}{\to}
\tilde{M}_{[s,r/q]}^\bullet \to 0
\end{gather*}
are isomorphic to $H^i_{\varphi, \Gamma}(M)$ via the natural comparison morphisms.
\end{enumerate}
\end{theorem}
\begin{proof}
Apply \cite[Proposition~6.3.19]{part1},
Proposition~\ref{P:same cohomology},
and Lemma~\ref{L:Gamma quasi-isomorphism}.
\end{proof}

\begin{theorem} \label{T:Galois cohomology1b}
Suppose that $A$ is a $\gotho_{E_a}$-algebra.
Let $E$ be an \'etale $E_a$-local system on $\Spa(A,A^+)$.
Let $M$ be the \'etale $(\varphi, \Gamma)$-module over one of $\bC_\psi, \breve{\bC}_\psi, \tilde{\bC}_\psi$
corresponding to $E$ via Theorem~\ref{T:perfect equivalence2b}.
Then for $i \geq 0$, there is a natural (in $E$ and $A$) bijection
$H^i_{\proet}(\Spa(A,A^+)_,E) \cong H^i_{\varphi, \Gamma}(M)$.
\end{theorem}
\begin{proof}
For the base ring $\tilde{\bC}_\psi$, this follows from
Theorem~\ref{T:perfect generalized phi-modules cohomology}.
For the base rings $\bC_\psi, \breve{\bC}_\psi$, this reduces to the previous case via Theorem~\ref{T:Galois cohomology1c}.
\end{proof}

As a corollary, we obtain a version of the Ax-Sen-Tate theorem
conditioned on the existence of a suitable tower $\psi$.

\begin{defn} \label{D:canonical section}
Recall that $\FFC_{\tilde{R}_\psi}$ admits a distinguished line bundle $\calL_\psi$ and a distinguished section $t_\psi$ of $\calL_\psi$ whose zero locus is naturally isomorphic to $\Spec(\tilde{A}_\psi)$ (see \cite[\S 8.8]{part1}). 
By Theorem~\ref{T:add tilde1}, the inclusion defined by $t_\psi$  of $\calL_\psi^{-1}$ into the trivial line bundle on $\FFC_{\tilde{R}_\psi}$ corresponds to an inclusion
$M_\psi \to \bC_\psi$ of $(\varphi, \Gamma)$-modules over $\bC_\psi$.
If we put $\tilde{M}_\psi = M_\psi \otimes_{\bC_\psi} \tilde{\bC}_\psi$, we then have a canonical identification
\begin{equation} \label{eq:tilde section map}
(\tilde{\bC}_\psi/\tilde{M}_\psi)^{\varphi=1} \cong \tilde{A}_\psi.
\end{equation}
\end{defn}

\begin{prop} \label{P:Ax}
Set notation as in Definition~\ref{D:canonical section} and put $X = \Spa(A,A^+)$.
\begin{enumerate}
\item[(a)]
Via the identification \eqref{eq:tilde section map}, we have
\begin{equation} \label{eq:section map}
(\bC_\psi/M_\psi)^{\varphi=1} \cong A_\psi.
\end{equation}
\item[(b)]
We have $H^0_{\proet}(X, \widehat{\calO}) = A$.
\end{enumerate}
\end{prop}
\begin{proof}
We may view 
$H^i(X_{\proet}, \widehat{\calO})$ as the cohomology of the complex $\tilde{A}_{\psi^{\bullet}}$.
By applying Theorem~\ref{T:Galois cohomology1c} to $\bC_\psi$ and $M_\psi$ and using \eqref{eq:tilde section map},
we may view
$H^i(X_{\proet}, \widehat{\calO})$ also as the cohomology of the complex $(\bC_{\psi^\bullet}/M_{\psi^\bullet})^{\varphi=1}$.
In particular, 
$H^0(X_{\proet}, \widehat{\calO}) \subseteq (\bC_{\psi}/M_{\psi})^{\varphi=1}
\subseteq
A_{\psi}$,
so by faithfully flat descent we have
$H^0(X_{\proet}, \widehat{\calO}) \subseteq A$.
Since the other inclusion is obvious, we must have
$H^0(X_{\proet}, \widehat{\calO}) = A$.
This yields (b); by applying the same reasoning to each tower $\psi_{(n)}$, we deduce (a).
\end{proof}

\begin{cor} \label{C:stay pseudocoherent}
For $0 < s \leq 1 \leq r \leq r_0$ (resp.\ $0 < s \leq 1 \leq r$),
$A_\psi$ (resp.\ $\tilde{A}_\psi$) is pseudocoherent as a module over $\breve{\bC}^{[s,r]}_\psi$ (resp.\ $\tilde{\bC}^{[s,r]}_\psi$).
Consequently (by Remark~\ref{R:need flatness}), pseudocoherent
modules over $A_\psi$ (resp.\ $\tilde{A}_\psi$) remain pseudocoherent when viewed as modules over $\bC^{[s,r]}_\psi$ (resp.\ $\tilde{\bC}^{[s,r]}_\psi$).
\end{cor}

\begin{remark}
In some cases,
\eqref{eq:section map} can be shown to extend to an isomorphism
\[
(\bC_{\psi^\bullet}/M_{\psi^\bullet})^{\varphi=1} \cong A_{\psi^\bullet}
\]
of complexes; for instance, this holds in the context of Example~\ref{exa:Galois tower2}.
However, even when such an isomorphism exists, one can still have
$H^i(X_{\proet}, \widehat{\calO}) \neq 0$ for some $i>0$;
see for example Remark~\ref{R:nonzero higher cohomology}.
This suggests that one should treat one's intuition about the pro-\'etale cohomology of coherent sheaves with extreme caution.
For another cautionary tale in a similar vein, see Example~\ref{exa:bad cohomology over perfectoid}.
\end{remark}

\subsection{Frobenius splittings and pseudocoherent \texorpdfstring{$(\varphi, \Gamma)$}{(phi, Gamma)}-modules}
\label{subsec:Frobenius splittings}

In practice, the decompleting towers we consider have additional useful properties that can be further exploited, especially to make statements about pseudocoherent $(\varphi, \Gamma)$-modules. These properties arise from properties of the Frobenius actions on $\Fp$-algebras, some of which already play an important role in commutative algebra;
see for example \cite{hochster-f-purity} for a recent survey.
\begin{defn}
Let $R$ be an $\FF_{p^h}$-algebra, and let $S$ be a copy of $R$ viewed as an $R$-algebra via $\overline{\varphi}$. (If $R$ is reduced, we may identify $S$ with $\overline{\varphi}_\varpi^{-1}(R)$ inside $R^{\perf}$.)
We say\footnote{Of these terms, only \emph{$F$-finite}, \emph{$F$-pure}, and \emph{$F$-split} seem to be commonly used.} that $R$ is \emph{$F$-finite}, \emph{$F$-coherent}, \emph{$F$-pure}, \emph{$F$-split}, \emph{$F$-flat}, or \emph{$F$-projective} if as an $R$-module $S$ is respectively finitely generated,  finitely presented, pure (for every $R$-module $M$ the map $M \to M \otimes_R S$ is injective), split (for the map $R \to S$), flat, or projective.
We will concatenate these definitions, writing for instance \emph{$F$-(finite flat)} to mean $F$-finite and $F$-flat.
We make the following observations.
\begin{itemize}
\item
All of these conditions hold if $R$ is perfect, as then $R \to S$ is an isomorphism.
\item
A finitely generated algebra over an $F$-finite field  is always $F$-finite,
and is  $F$-projective if and only if it is regular (Kunz's criterion; see \cite{kunz}).
\item
$F$-coherent implies $F$-finite, and conversely if $R$ is noetherian.
\item
$F$-flat implies $F$-pure.
\item
$F$-pure implies reduced.
\item
$F$-split implies $F$-pure, but not conversely\footnote{Karl Schwede suggests the example $R = k \llbracket x \rrbracket \otimes_k k^{1/p}$ for $k$ a field which is not $F$-finite.} unless $R$ is $F$-coherent 
\cite[Corollary~5.2]{hochster-roberts}.
\item
$F$-(coherent flat) is equivalent to $F$-(finite projective).
\item
$F$-(finite projective) implies $F$-(coherent split).
\end{itemize}
\end{defn}

\begin{lemma} \label{L:lift generators}
Suppose that $\psi$ is weakly decompleting and $R_\psi$ is $F$-finite.
Let $\overline{x}_1,\dots,\overline{x}_n$ be generators of $\overline{\varphi}^{-1}(R_\psi)$
as a module over $R_\psi$, and lift them to $x_1,\dots,x_n \in \varphi^{-1}(\bA^{r_0}_\psi)$ for some $r_0 > 0$. Then for some $r_1 \in (0,r_0]$, for all $r \in (0,r_1]$,
$x_1,\dots,x_n$ generate  $\varphi^{-1}(\bA^r_\psi)$ as a module over $\bA^r_\psi$.
\end{lemma}
\begin{proof}
As in \cite[Lemma~5.5.2]{part1}.
\end{proof}

\begin{lemma} \label{L:splitting1}
Suppose that $\psi$ is weakly decompleting and $R_\psi$ is $F$-(finite projective).
Then for some $r_0 > 0$, for all $r \in (0,r]$,
$\varphi^{-1}(\bA^{r}_\psi)$ is finite projective over $\bA^{r}_\psi$.
\end{lemma}
\begin{proof}
If $R_\psi$ is $F$-(finite free), then Lemma~\ref{L:lift generators} implies that
for any $r>0$ sufficiently small,
$\varphi^{-1}(\bA^{r}_\psi)$ is finite free over $\bA^{r}_\psi$.
To deduce the general case, by \cite[Proposition~2.4.20]{part1} it suffices to show that we can infer the finite projective property for $\psi$ for the same property for the towers $\psi_1, \psi_2, \psi_{12}$ obtained from $\psi$ along a simple Laurent covering.
By \cite[Proposition~2.7.5]{part1}, it suffices to check that for $r>0$ sufficiently small,
the diagram
\[
\xymatrix{
\bA^r_\psi \ar[r] \ar[d] & \bA^r_{\psi_1} \ar[d] \\
\bA^r_{\psi_2} \ar[r] & \bA^r_{\psi_{12}} 
}
\]
is a glueing square in the sense of \cite[Definition~2.7.3]{part1}.
Using Lemma~\ref{L:optimal lifts}, this reduces to checking that
\[
\xymatrix{
R_\psi \ar[r] \ar[d] & R_{\psi_1} \ar[d] \\
R_{\psi_2} \ar[r] & R_{\psi_{12}} 
}
\]
is a glueing square. We may deduce this as follows:
$(\tilde{R}_\psi, \tilde{R}_\psi^+)$ is stably uniform by \cite[Proposition~3.1.7]{part1}; 
thus $(R_\psi, R_\psi^+)$ is stably uniform by \cite[Remark~2.8.12]{part1};
thus $(R_\psi, R_\psi^+)$ is sheafy by \cite[Theorem~2.8.10]{part1};
thus the structure sheaf on $\Spa(R_\psi, R_\psi^+)$ is acyclic on rational subspaces
by \cite[Theorem~2.4.23]{part1}.
\end{proof}

\begin{hypothesis} \label{H:phigamma split}
For the remainder of \S\ref{subsec:Frobenius splittings},
retain Hypothesis~\ref{H:phigamma}, but assume further that $R_\psi$ is $F$-(finite projective).
\end{hypothesis}

\begin{remark} \label{R:descent pseudocoherent}
In general, one must distinguish between the category of pseudocoherent modules over $R_\psi^{\perf}$ and the direct 2-limit of the categories of pseudocoherent modules over $\overline{\varphi}^{-m}(R_\psi)$ for all $m$. However, by hypothesis the map $R_\psi \to R_\psi^{\perf}$ is faithfully flat, so a module over $R_\psi$ becomes pseudocoherent over $R_\psi^{\perf}$ if and only if it was already pseudocoherent to begin with.
\end{remark}

\begin{prop} \label{P:perfect equivalence2a}
The (exact tensor) categories of pseudocoherent (resp.\ fpd) $\varphi$-modules over the rings in the diagram
\[
 \bA_\psi \to \breve{\bA}_\psi \to
 \hat{\bA}_\psi  \to \tilde{\bA}_\psi
\]
are equivalent via the apparent base change functors; in particular, by Corollary~\ref{C:pseudocoherent type A flat} and Theorem~\ref{T:pseudocoherent type A}, these categories are abelian and every object is fpd of projective dimension at most $1$.
Moreover, these equivalences induce isomorphisms of $\varphi$-cohomology groups.
\end{prop}
\begin{proof}
By using the flatness of $R_\psi \to R_\psi^{\perf}$ and applying \cite[Proposition~3.2.13]{part1} to $R_\psi^{\perf}$, we see that any finitely presented $\varphi$-module over $R_\psi$ admitting a semilinear $\varphi$-action is projective.
Using this fact, we may imitate the proof of Lemma~\ref{L:pseudocoherent type A torsion} to see that for any pseudocoherent $\varphi$-module $M$ over any of the rings in the diagram, the $\varpi$-power-torsion submodule $T$ of $M$ is again a pseudocoherent $\varphi$-module and $M/T$ is projective. We may thus deduce the claim from
Corollary~\ref{C:fully faithful to extended}.
\end{proof}

\begin{lemma} \label{L:full splitting}
There exists $r_0>0$ such that for $0 < s \leq r \leq r_0$,
there exist isomorphisms
\begin{align*}
\tilde{R}_\psi &\cong R_\psi \oplus \widehat{\bigoplus_{m=1}^\infty} \overline{\varphi}^{-m}(R_\psi)/\overline{\varphi}^{-m+1}(R_\psi) \\
\tilde{\bA}^r_\psi &\cong \bA^r_\psi \oplus \widehat{\bigoplus_{m=1}^\infty} \varphi^{-m}(\bA^{r/p^{hm}}_\psi)/\varphi^{-m+1}(\bA^{r/p^{h(m-1)}}_\psi) \\
\tilde{\bC}^{[s,r]}_\psi &\cong \bC^{[s,r]}_\psi \oplus \widehat{\bigoplus_{m=1}^\infty} \varphi^{-m}(\bC^{[s/p^{hm}, r/p^{hm}]}_\psi)/\varphi^{-m+1}(\bC^{[s/p^{h(m-1)}, r/p^{h(m-1)}]}_\psi).
\end{align*}
More precisely, in each case,
for $R$ the first summand on the right-hand side, each subsequent summand on the right-hand side is a finite projective $R$-module topologized using the quotient topology
(for the respective norms $\overline{\alpha}_\psi$, $\lambda(\overline{\alpha}_\psi^r)$,
$\max\{\lambda(\overline{\alpha}_\psi^s), \lambda(\overline{\alpha}_\psi^r)\}$);
the completion is with respect to the supremum norm; and the isomorphism 
is an isomorphism of normed $R$-modules.
\end{lemma}
\begin{proof}
To obtain the first isomorphism, we start with a splitting 
\begin{equation} \label{eq:splitting1}
\overline{\varphi}^{-1}(R_\psi) \cong R_\psi \oplus \overline{\varphi}^{-1}(R_\psi)/ R_\psi
\end{equation}
and then repeatedly apply $\overline{\varphi}^{-1}$ to obtain splittings
\begin{equation} \label{eq:splitting2}
\overline{\varphi}^{-m}(R_\psi) \cong \overline{\varphi}^{-m+1}(R_\psi) \oplus \overline{\varphi}^{-m}(R_\psi)/ \overline{\varphi}^{-m+1}(R_\psi).
\end{equation}
Composing these yields an isomorphism
\begin{equation} \label{eq:splitting3}
R_\psi^{\perf} \cong R_\psi \oplus \bigoplus_{m=1}^\infty \overline{\varphi}^{-m}(R_\psi)/\overline{\varphi}^{-m+1}(R_\psi).
\end{equation}
To see that this is a topological isomorphism, equip $\overline{\varphi}^{-m}/\overline{\varphi}^{-m+1}(R_\psi)$ with the quotient norm. There then exists $c>1$ such that in \eqref{eq:splitting1}, the norm of an element on the right-hand side is at most $c$ times its image on the left-hand side.
Then in \eqref{eq:splitting2}, the norm of an element on the right-hand side is at most $c^{p^{1-m}}$ times its image on the left-hand side. Consequently, in 
\eqref{eq:splitting3}, the norm of an element on the right-hand side is at most $c^{1+1/p+1/p^2+\cdots} = c^{p/(p-1)}$ times its image on the left-hand side.

To obtain the second isomorphism, apply Lemma~\ref{L:splitting1} to produce a splitting
\begin{equation} \label{eq:splitting4}
\varphi^{-1}(\bA^{r_0/p^h}_\psi)  \cong \bA^{r_0}_\psi \oplus \varphi^{-1}(\bA^{r_0/p^h}_\psi)/ \bA^{r_0}_\psi
\end{equation}
and then repeatedly apply $\varphi^{-1}$ to obtain splittings
\begin{equation} \label{eq:splitting5}
\varphi^{-m}(\bA^{r_0/p^{hm}}_\psi)  \cong \varphi^{-m+1}(\bA^{r_0/p^{h(m-1)}}_\psi) \oplus \varphi^{-m}(\bA^{r_0/p^{hm}}_\psi)/ \varphi^{-m+1}(\bA^{r_0/p^{h(m-1)}}_\psi).
\end{equation}
Composing these yields an isomorphism
\begin{equation} \label{eq:splitting6}
\breve{\bA}^{r_0}_\psi \cong \bA^{r_0}_\psi \oplus \bigoplus_{m=1}^\infty {\varphi}^{-m}(\bA^{r_0/p^{hm}}_\psi)/\varphi^{-m+1}(\bA^{r_0/p^{h(m-1)}}_\psi).
\end{equation}
By Lemma~\ref{L:optimal lifts} and the previous paragraph, there exist $c>0$ and $r_0>0$ such that for each $m$, the isomorphism \eqref{eq:splitting5} has the property that for $r \in (0,r_0]$, the norm with respect to $\lambda(\overline{\alpha}_\psi^r)$ of an element of the right-hand side is at most $c^{r/p^{h(m-1)}}$ times the norm on the left-hand side. Consequently, in \eqref{eq:splitting6}, the norm with respect to $\lambda(\overline{\alpha}_\psi^r)$ of an element of the right-hand side is at most $c^{p^hr/(p^h-1)}$ times the norm on the left-hand side. 
This yields the second isomorphism; the third isomorphism follows similarly.
\end{proof}

\begin{cor} \label{C:split base extension1}
There exists $r_0>0$ such that for $0 < s \leq r \leq r_0$,
the morphisms $R_\psi \to \tilde{R}_\psi$,
$\bA^{r}_\psi \to \tilde{\bA}^{r}_\psi$,
$\bC^{[s,r]}_\psi \to \tilde{\bC}^{[s,r]}_\psi$
are $2$-pseudoflat, and also split in the category of Banach modules over the source.
Also, the map $\bA_\psi \to \tilde{\bA}_\psi$ is $2$-pseudoflat and split in the category of $\bA_\psi$-modules.
\end{cor}

\begin{cor} \label{C:perfect cohomology base change}
There exists $r_0>0$ such that for $0 < s \leq r \leq r_0$,
for any strict complex $C^{\bullet}$ of Banach modules over $\bC^{[s,r]}_\psi$, the complex
$\tilde{C}^{\bullet} = C^{\bullet} \widehat{\otimes}_{\bC^{[s,r]}_\psi} \tilde{\bC}^{[s,r]}_\psi$ is strict and the maps
\[
h^{\bullet}(C) \widehat{\otimes}_{\bC^{[s,r]}_\psi} \tilde{\bC}^{[s,r]}_\psi
\to
h^{\bullet}(\tilde{C})
\]
are isomorphisms.
\end{cor}

\begin{cor} \label{C:imperfect stably uniform}
For $0 < s \leq r \leq r_0$,
the adic affinoid algebra $(\bC^{[s,r]}_\psi,\bC^{[s,r],+}_\psi)$ is stably uniform,
and hence sheafy by \cite[Theorem~2.8.10]{part1}.
Moreover, this remains true upon passage to any finite \'etale cover of any rational localization of 
$(\bC^{[s,r]}_\psi,\bC^{[s,r],+}_\psi)$.
\end{cor}
\begin{proof}
Since $(\tilde{\bC}^{[s,r]}_\psi, \tilde{\bC}^{[s,r],+}_\psi)$ is stably uniform by
Corollary~\ref{C:Robba stably uniform},
this follows from
Corollary~\ref{C:split base extension1}. Explicitly, if $(\bC^{[s,r]}_\psi,\bC^{[s,r],+}_\psi) \to (B, B^+)$ is 
a composition of finite \'etale morphisms of rational localizations, then
\[
(\tilde{\bC}^{[s,r]}_\psi,\tilde{\bC}^{[s,r],+}_\psi) \to (\tilde{B}, \tilde{B}^+) = (B, B^+) \widehat{\otimes}_{(\bC^{[s,r]}_\psi,\bC^{[s,r],+}_\psi)}
(\tilde{\bC}^{[s,r]}_\psi,\tilde{\bC}^{[s,r],+}_\psi)
\]
is also such a composition,
and so $\tilde{B}$ is uniform as in the proof of Corollary~\ref{C:Robba stably uniform}.
Since the inclusion $\bC^{[s,r]}_\psi \to \tilde{\bC}^{[s,r]}_\psi$ splits in the category of Banach modules over $\bC^{[s,r]}_\psi$, so then does the inclusion $B \to \tilde{B}$; consequently, $B$ is uniform.
\end{proof}

\begin{cor} \label{C:split base extension2}
There exists $r_0>0$ such that for $0 < s \leq r \leq r_0$,
the base extension of an \'etale-stably pseudocoherent module 
along any of the morphisms  $R_\psi \to \tilde{R}_\psi$,
$\bA_\psi \to \tilde{\bA}_\psi$,
$\bA^{r}_\psi \to \tilde{\bA}^{r}_\psi$,
$\bC^{[s,r]}_\psi \to \tilde{\bC}^{[s,r]}_\psi$
is again \'etale-stably pseudocoherent.
\end{cor}

\begin{remark}
Beware that the splittings described in Lemma~\ref{L:full splitting} are not guaranteed to be $\Gamma$-equivariant. However, for toric towers one does obtain a $\Gamma$-equivariant splitting; see Remark~\ref{R:toric splitting}.
\end{remark}

\begin{remark}
For $r \in (0,r_0]$, let $\Spa(\bC^r_\psi, \bC^{r,+}_\psi)$ be the space
\[
\bigcup_{0<t\leq s \leq r_0} \Spa(\bC^{[t,s]}_\psi, \bC^{[t,s],+}_\psi).
\]
By Corollary~\ref{C:imperfect stably uniform}, $\Spa(\bC^r_\psi, \bC^{r,+}_\psi)$ is a quasi-Stein space; in particular, the category of vector bundles (resp.\ pseudocoherent sheaves) over $\bC^r_\psi$ in the sense of Definition~\ref{D:artificial vector bundle} is equivalent to the category of vector bundles (resp.\ pseudocoherent sheaves) on 
$\Spa(\bC^r_\psi, \bC^{r,+}_\psi)$.
Consequently, as indicated in Definition~\ref{D:not quasi-Stein}, in this context 
the results of \S\ref{subsec:projective modules} become direct corollaries of the corresponding results of \S\ref{subsec:quasi-Stein}.
\end{remark}

\begin{lemma} \label{L:Gamma quasi-isomorphism2}
The statements of Lemma~\ref{L:Gamma quasi-isomorphism} also hold for
pseudocoherent $\Gamma$-modules.
\end{lemma}
\begin{proof}
Thanks to Corollary~\ref{C:split base extension1}, we may argue as in 
the proof of Lemma~\ref{L:Gamma quasi-isomorphism}.
\end{proof}

\begin{theorem} \label{T:Galois cohomology1c2}
For $* \in \{R, \bA, \bA^\dagger, \bB, \bB^\dagger, \bC\}$,
let $M$ be a pseudocoherent $(\varphi, \Gamma)$-module over $*_\psi$
and put $\breve{M} = M \otimes_{*_\psi} \breve{*}_\psi$,
$\tilde{M} = M \otimes_{*_\psi} \tilde{*}_\psi$.
Then the natural maps $H^i_{\varphi, \Gamma}(M) \to H^i_{\varphi,\Gamma}(\breve{M}) \to H^i_{\varphi,\Gamma}(\tilde{M})$ are bijections for all $i \geq 0$.
\end{theorem}
\begin{proof}
Apply Proposition~\ref{P:perfect equivalence2a} and Lemma~\ref{L:Gamma quasi-isomorphism2}.
\end{proof}

The descent of pseudocoherent $(\varphi, \Gamma)$-modules requires a bit more care.
The case of type $\bA$ poses no serious issues.
\begin{theorem} \label{T:perfect equivalence2}
The (exact tensor) categories of pseudocoherent (resp.\ fpd) $(\varphi, \Gamma)$-modules over the rings in the diagram
\[
\xymatrix@R=30pt@!C=60pt{
 \bA^{\dagger}_\psi \ar[r] \ar[d]  & \breve{\bA}^\dagger_\psi \ar[d] \ar[r]
 &  \hat{\bA}^\dagger_\psi \ar[d] \ar[r] &
 \tilde{\bA}^\dagger_\psi \ar[d] \\
 \bA_\psi \ar[r] & \breve{\bA}_\psi \ar[r] & 
 \hat{\bA}_\psi \ar[r] &\tilde{\bA}_\psi
}
\]
are equivalent via the apparent base change functors.
\end{theorem}
\begin{proof}
By Proposition~\ref{P:perfect equivalence2a} and Theorem~\ref{T:pseudocoherent type A}, the claim holds for each of the arrows in the bottom row and for the rightmost vertical arrow. It thus suffices to link 
$\bA^\dagger_\psi, \breve{\bA}^\dagger_\psi, \hat{\bA}^\dagger_\psi$ to the other rings in the diagram. As in the proof of Proposition~\ref{P:perfect equivalence2a}, we may show that each pseudocoherent $\varphi$-module over one of these rings is an extension of a projective $\varphi$-module by a $\varpi$-power-torsion pseudoocoherent $\varphi$-module;
we may thus reduce the claim to Theorem~\ref{T:perfect equivalence1}.
\end{proof}

In type $\bC$, we obtain only a weaker analogue of Lemma~\ref{L:Gamma equivalences}. The full analogue requires a stronger hypothesis; see Lemma~\ref{L:Gamma equivalences2}. 
\begin{lemma} \label{L:Gamma equivalences2a}
Choose $r,s$ with $0<s \leq r$.
\begin{enumerate}
\item[(a)]
Base extension of pseudocoherent $\Gamma$-modules from $\breve{\bC}^{[s,r]}_\psi$ to $\tilde{\bC}^{[s,r]}_{\psi}$
is a fully faithful functor of exact tensor categories.
\item[(b)]
The essential image of the functor in (a) includes all pseudocoherent $\Gamma$-modules whose underlying modules descend to $\breve{\bC}^{[s,r]}_\psi$.
\end{enumerate}
\end{lemma}
\begin{proof}
Part (a) follows from Remark~\ref{R:descent pseudocoherent} and Corollary~\ref{C:split base extension2} (which guarantee the existence of the base extension functor)
and Corollary~\ref{C:split base extension1} (which gives full faithfulness).
To check (b), let $\tilde{M}$ be a pseudocoherent $\Gamma$-module over $\tilde{\bC}^{[s,r]}_\psi$
whose underlying module has the form $M \otimes_{\breve{\bC}^{[s,r]}_\psi} \tilde{\bC}^{[s,r]}_\psi$
for some pseudocoherent module $M$ over $\breve{\bC}^{[s,r]}_\psi$. Using Lemma~\ref{L:Gamma quasi-isomorphism2}, we may imitate the proof of 
Lemma~\ref{L:Gamma equivalences} to replace $\tilde{M}$ with an isomorphic $\Gamma$-module structure induced from $M$.
\end{proof}

\subsection{Pseudocoherent \texorpdfstring{$(\varphi, \Gamma)$}{(phi, Gamma)}-modules of type \texorpdfstring{$\bC$}{C}}
\label{subsec:pseudocoherent type C}

To extend the results of \S\ref{subsec:Frobenius splittings} to pseudocoherent $(\varphi, \Gamma)$-modules over imperfect period rings of type $\bC$, we need to assume that the tower $\psi$ is noetherian.
Fortunately, this hypothesis will be satisfied in most cases of interest.

\begin{hypothesis}
Throughout \S\ref{subsec:pseudocoherent type C},
assume Hypothesis~\ref{H:phigamma} (in particular, $\psi$ is a decompleting tower) and Hypothesis~\ref{H:phigamma split} (i.e., $R_\psi$ is $F$-(finite projective)), and assume further that $\psi$ is noetherian.
\end{hypothesis}

Using the extra hypothesis, we may refine Lemma~\ref{L:Gamma equivalences2a} to obtain an analogue of Lemma~\ref{L:Gamma equivalences} for pseudocoherent $\Gamma$-modules over $\bC^{[s,r]}_\psi$.
\begin{lemma} \label{L:lift to psc}
For $r,s$ with $0<s \leq r$, let $\tilde{M}$ be a finite $\tilde{\bC}^{[s,r]}_{\psi}$-module with $\Gamma$-structure. 
Then there exists $n_0 \geq 0$ such that for $n \geq n_0$,
there exist a pseudocoherent $\breve{\bC}^{[s,r]}_{\psi_{(n)}}$-module $M'$ with $\Gamma$-structure
and a surjective morphism $M' \otimes_{\breve{\bC}^{[s,r]}_{\psi_{(n)}}} \tilde{\bC}^{[s,r]}_{\psi_{(n)}} \to M \otimes_{\tilde{\bC}^{[s,r]}_{\psi}} \tilde{\bC}^{[s,r]}_{\psi_{(n)}}$ of finite  $\tilde{\bC}^{[s,r]}_{\psi_{(n)}}$-modules with $\Gamma$-structure.
\end{lemma}
\begin{proof}
Choose a surjection $\tilde{F} \to \tilde{M}$ of $\tilde{\bC}^{[s,r]}_\psi$-modules
in which $\tilde{F} = F \otimes_{\breve{\bC}^{[s,r]}_\psi} \tilde{\bC}^{[s,r]}_\psi$ for some
finite free module $F$ over $\breve{*}_\psi$. We may lift the map 
\[
\iota_{\tilde{M}}:
 \tilde{M} \otimes_{i_{0,0}} \tilde{\bC}^{[s,r]}_{\psi^0} \to \tilde{M} \otimes_{i_{0,1}} \tilde{\bC}^{[s,r]}_{\psi^1}
 \]
defining the $\Gamma$-structure on $\tilde{M}$ to a morphism 
\[
\iota_{\tilde{F}}: \tilde{F} \otimes_{i_{0,0}} \tilde{\bC}^{[s,r]}_{\psi^0} \to \tilde{F} \otimes_{i_{0,1}} \tilde{\bC}^{[s,r]}_{\psi^1},
\]
but this morphism need neither be an isomorphism nor satisfy the cocycle condition. Nonetheless, using Lemma~\ref{L:Gamma quasi-isomorphism2} in place of Lemma~\ref{L:Gamma quasi-isomorphism}, we may
(after replacing $\psi$ with $\psi_{(n)}$ for some suitably large $n$) imitate the proof of Lemma~\ref{L:Gamma equivalences} to replace $\iota_{\tilde{F}}$ with a sequence of lifts of $\iota_{\tilde{M}}$ converging to a lift induced by a morphism
\[
\iota_{F}: F \otimes_{i_{0,0}} \breve{\bC}^{[s,r]}_{\psi^0} \to F \otimes_{i_{0,1}} \breve{\bC}^{[s,r]}_{\psi^1}.
\]
Since $\iota_{\tilde{F}}$ lifts the map $\iota_{\tilde{M}}$ which is an isomorphism satisfying the cocycle condition, there exists a finitely generated submodule $K$ of $\ker(F \to \tilde{M})$ such that for $M' = F/K$, the map
\[
\iota_{M'}: M' \otimes_{i_{0,0}} \breve{\bC}^{[s,r]}_{\psi^0} \to M' \otimes_{i_{0,1}} \breve{\bC}^{[s,r]}_{\psi^1}
\]
induced by $\iota_F$ is an isomorphism satisfying the cocycle condition.
Since $\breve{\bC}^{[s,r]}_\psi$ is coherent, Remark~\ref{R:noetherian pseudoflat} implies that
$M'$ is a pseudocoherent $\Gamma$-module over $\breve{\bC}^{[s,r]}_\psi$.
This yields the desired result.
\end{proof}

\begin{lemma} \label{L:Gamma equivalences2}
For $r,s$ with $0<s \leq r$, the following statements hold.
\begin{enumerate}
\item[(a)]
Base extension of pseudocoherent $\Gamma$-modules from $\breve{\bC}^{[s,r]}_\psi$ to $\tilde{\bC}^{[s,r]}_{\psi}$
is an equivalence of exact tensor categories. 
\item[(b)]
Every finitely presented $\tilde{\bC}^{[s,r]}_{\psi}$-module with $\Gamma$-structure is \'etale-stably pseudocoherent.
\end{enumerate}
\end{lemma}
\begin{proof}
By Lemma~\ref{L:Gamma equivalences2a}, the base extension functor is fully faithful 
and its essential image include all objects whose underlying modules descent to
pseudocoherent modules over $\breve{\bC}^{[s,r]}_\psi$.
In particular, in the course of proving essential surjectivity, we may pass freely from $\psi$ to $\psi_{(n)}$ for some conveniently large $n$.

Now let $\tilde{M}$ be a finitely presented $\Gamma$-module over $\tilde{\bC}^{[s,r]}_\psi$. By Lemma~\ref{L:lift to psc}, after replacing $\psi$ with some $\psi_{(n)}$,
we can find a pseudocoherent $\Gamma$-module $M'$ over $\breve{\bC}^{[s,r]}_\psi$
such that for $\tilde{M}' = M' \otimes_{\breve{\bC}^{[s,r]}_\psi} \tilde{\bC}^{[s,r]}_\psi$, there exists a surjective morphism
$\tilde{M}' \to \tilde{M}$ of finitely presented $\Gamma$-modules.
By Lemma~\ref{L:pseudocoherent 2 of 3}, the kernel $\tilde{N}$ of this map is finitely generated. We may thus apply Lemma~\ref{L:lift to psc} again to obtain a
surjective $\Gamma$-equivariant morphism $\tilde{N}' \to \tilde{N}$ of finitely generated modules such that 
\[
\tilde{N}' \to \tilde{M}' \to \tilde{M} \to 0
\]
is exact, $\tilde{N}'$ is induced from a pseudocoherent $\Gamma$-module $N$ over $\breve{\bC}^{[s,r]}_\psi$, and the map $\tilde{N}' \to \tilde{M}'$ is induced from a morphism
$N' \to M'$ of pseudocoherent $\Gamma$-modules. Since $\breve{\bC}^{[s,r]}_\psi$ is coherent, Remark~\ref{R:noetherian pseudoflat} implies that
$M = \coker(\tilde{N}' \to \tilde{M}')$ is a pseudocoherent $\Gamma$-module over $\breve{\bC}^{[s,r]}_\psi$ and $\tilde{M} \cong M \otimes_{\breve{\bC}^{[s,r]}_\psi} \tilde{*}_\psi$.
(Here we use the splitting aspect of Corollary~\ref{C:split base extension1} to ensure that $M$ is a pseudocoherent $\breve{\bC}^{[s,r]}_\psi$-module.)
This completes the proof of both (a) and (b).
\end{proof}

We then obtain a pseudocoherent analogue of
Theorem~\ref{T:add tilde1}. 

\begin{theorem} \label{T:add tilde2}
For all $r,s$ with $0 < s \leq r/q$, the following exact tensor categories are equivalent,
and in addition are abelian categories.
\begin{enumerate}
\item[(a)]
The category of pseudocoherent $(\varphi, \Gamma)$-modules over $\bC_\psi$.
\item[(b)]
The category of pseudocoherent $(\varphi,\Gamma)$-bundles over $\bC_\psi$.
\item[(c)]
The category of pseudocoherent $(\varphi, \Gamma)$-modules over $\breve{\bC}_\psi$.
\item[(d)]
The category of pseudocoherent $(\varphi, \Gamma)$-bundles over $\breve{\bC}_\psi$.
\item[(e)]
The category of pseudocoherent $(\varphi, \Gamma)$-modules over $\breve{\bC}^{[s,r]}_\psi$.
\item[(f)]
The category of pseudocoherent $(\varphi, \Gamma)$-modules over $\tilde{\bC}_\psi$.
\item[(g)]
The category of pseudocoherent $(\varphi,\Gamma)$-bundles over $\tilde{\bC}_\psi$.
\item[(h)]
The category of pseudocoherent $(\varphi, \Gamma)$-modules over $\tilde{\bC}^{[s,r]}_\psi$.
\item[(i)]
The category of pseudocoherent sheaves on the relative Fargues-Fontaine curve
$\FFC_{\tilde{R}_\psi}$.
\item[(j)]
The category of pseudocoherent $\varphi$-modules over $\tilde{\bC}_{\Spa(A,A^+)}$.
\item[(k)]
The category of pseudocoherent $\varphi$-bundles over  $\tilde{\bC}_{\Spa(A,A^+)}$.
\item[(l)]
The category of pseudocoherent $\varphi$-modules over $\tilde{\bC}^{[s,r]}_{\Spa(A,A^+)}$.
\end{enumerate}
\end{theorem}
\begin{proof}
This follows from the proof of Theorem~\ref{T:add tilde1} upon replacing
Lemma~\ref{L:Gamma equivalences}
with Lemma~\ref{L:Gamma equivalences2}.
\end{proof}

\begin{cor} \label{C:toric tower coherent}
There exists $N \geq 0$ depending only on $(A,A^+)$ such that 
for every pseudocoherent $\Gamma$-module $\tilde{M}$ over $\tilde{A}_\psi$, the groups
$H^i_\Gamma(\tilde{M})$ are finitely generated $A$-modules for all $i \geq 0$ and vanish for all $i > N$.
\end{cor}
\begin{proof}
By Corollary~\ref{C:stay pseudocoherent},
$M$ is a pseudocoherent $\varphi$-module over $\tilde{\bC}_\psi$; we may thus
apply Theorem~\ref{T:add tilde2} to descend $\calF$ to a pseudocoherent $(\varphi,\Gamma)$-module $M$ over $\bC_\psi$, which is in fact a pseudocoherent $\Gamma$-module over $A_{\psi}$. By Theorem~\ref{T:Galois cohomology1c2},
we have $H^i_{\Gamma}(M) = H^i_{\Gamma}(\tilde{M})$ for all $i$.
For any sufficiently large $n$, we may realize $M$ as a module $M_n$ over $A_{\psi,n}$
in such a way that the $\Gamma$-module structure is defined over $A_{\psi^1,n}$.
By faithfully flat descent, we may compute $H^i_{\Gamma}(M)$ as 
the cohomology of the complex $M_n \otimes_{A_{\psi,n}} A_{\psi^\bullet,n}$.
Since this complex consists of finitely generated $A$-modules, the claim follows at once.
\end{proof}

\begin{remark} \label{R:higher vanishing}
In Corollary~\ref{C:toric tower coherent}, it is not guaranteed that there exists $N \geq 0$ such that
$H^i_{\Gamma}(\tilde{M}) = 0$ for all $i > N$. One circumstance in which this is guaranteed is when
$\psi$ is Galois with group $\Gamma$ being a $p$-adic Lie group, as in this case $\Gamma$-cohomology can be computed by a finite complex.
\end{remark}

\begin{remark} \label{R:decompleting conditions}
To summarize, one can apply all of the results of \S\ref{subsec:phigamma}, \S\ref{subsec:Frobenius splittings}, and \S\ref{subsec:pseudocoherent type C}
 provided that the tower $\psi$ satisfies the following conditions.
\begin{enumerate}
\item[(a)]
The tower is locally decompleting (and in particular, finite \'etale, perfectoid, and weakly decompleting).
\item[(b)]
The ring $R_\psi$ is $F$-(finite projective).
\item[(c)]
For $0 <s \leq r \leq r_0$, the  ring $\breve{\bC}^{[s,r]}_\psi$ is noetherian.
\end{enumerate}
\end{remark}

\section{Towers over fields}
\label{sec:towers fields}

We next apply the axiomatic formalism described in \S\ref{sec:axiomatic} to
exhibit several known and conjectured classes of examples of decompleting towers
over analytic fields.
The case of the cyclotomic tower recovers the Cherbonnier-Colmez theorem on overconvergence of $p$-adic Galois representations
\cite{cherbonnier-colmez} via the technique described in \cite{kedlaya-new-phigamma},
while the case of the Kummer tower suggests a conjecture of a similar flavor.

\setcounter{theorem}{0}
\begin{hypothesis} \label{H:p-adic field}
Throughout \S\ref{sec:towers fields},
let $K$ be a field of characteristic $0$ complete for a discrete valuation, whose residue field $\kappa_K$ is perfect of characteristic $p$. We refer to such an object hereafter as a \emph{$p$-adic field}. 
\end{hypothesis}

\subsection{The cyclotomic tower over a \texorpdfstring{$p$}{p}-adic field}
\label{subsec:cyclotomic}

We begin with the case of the $p$-cyclotomic tower over a $p$-adic field. 
By applying Theorem~\ref{T:perfect equivalence1}, we may then recover the
theorem of Cherbonnier-Colmez \cite{cherbonnier-colmez}
in essentially the same fashion as in \cite{kedlaya-new-phigamma}.
By applying Theorem~\ref{T:Galois cohomology1b},
we also recover Herr's description of Galois cohomology in terms of $(\varphi, \Gamma)$-modules \cite{herr, herr-tate}.

\begin{defn} \label{D:cyclotomic tower}
Put $K_0 = K$ and $\epsilon_0 = 1$.
For $n \geq 1$, put $K_n = K[\epsilon_n]/(\Phi_{p^n}(\epsilon_n))$, where
$\Phi_{p^n}(T) = (T^{p^n}-1)/(T^{p^{n-1}}-1)$ is the $p^n$-th cyclotomic polynomial.
Set notation as in Hypothesis~\ref{H:towers} with
$h=1$, $E = \Qp$, $\varpi = p$. 
Form the tower $\psi$ over $K$ by taking $A_{\psi,n} = K_n$, $A_{\psi,n}^+ = A_{\psi,n}^{\circ}$,
and $A_{\psi,n} \to A_{\psi,n+1}$ to be the map identifying $\epsilon_n$ with $\epsilon_{n+1}^p$.
This tower is finite \'etale and Galois with group $\Gamma = \ZZ_p^\times$
acting via the cyclotomic character; we are thus in the setting of Example~\ref{exa:Galois tower}.
\end{defn}

\begin{lemma} \label{L:cyclotomic perfectoid}
The tower $\psi$ is perfectoid.
\end{lemma}
\begin{proof}
By Proposition~\ref{P:perfectoid tower persistence}, we need only check this in case $K$ is absolutely unramified (i.e., $K = \Frac(W(\kappa_K))$). In this case, $\tilde{A}_{\psi}$ is an analytic field, so to check that $(\tilde{A}_\psi, \tilde{A}_\psi^+)$ is perfectoid, by \cite[Proposition~3.6.2]{part1} it suffices to check that $\tilde{A}_\psi$ is nondiscretely valued (which is clear) and that Frobenius is surjective on $\gotho_{\tilde{A}_\psi}/(p)$.
But $\gotho_{A_{\psi}} = \gotho_K[\epsilon_1, \epsilon_2, \dots]$ and 
\[
\gotho_{\tilde{A}_{\psi}}/(p) = \gotho_K[\overline{\epsilon}_1, \overline{\epsilon}_2, \dots]/(\overline{\epsilon}_1^p, \overline{\epsilon}_2^p - \overline{\epsilon}_1, \dots),
\]
so the surjectivity of Frobenius is clear.
\end{proof}

\begin{defn}
Since $\psi$ is perfectoid, we may set notation as in \S\ref{sec:axiomatic}.
Let $\epsilon \in \tilde{R}_{\psi}$ be the element $(\dots, \epsilon_1, \epsilon_0)$. Put $\pi = [\epsilon] - 1$;
note that $\overline{\alpha}_\psi(\overline{\pi}) = p^{-p/(p-1)}$.
\end{defn}

\begin{lemma} \label{L:cyclotomic Rpsi}
The ring $R_{\psi}$ is a finite separable extension of $\kappa_K((\overline{\pi}))$
and $\tilde{R}_{\psi}$ is the completion of $R_{\psi}^{\perf}$.
In particular, $\psi$ is weakly decompleting and $R_\psi$ is $F$-(finite projective).
\end{lemma}
\begin{proof}
By Proposition~\ref{P:weakly decompleting tower persistence} we may again reduce to the case where $K$ is absolutely unramified, in which case we will check that
in fact $R_{\psi} = \kappa_K((\overline{\pi}))$.
Since $\pi \in \calR^{\inte}_\psi$, we have $\kappa_K((\overline{\pi})) \subseteq R_{\psi}$. 
We may then see that $\tilde{R}_{\psi}$ is the completion of $R_{\psi}^{\perf}$ by comparing truncated valuation rings,
and then see that $R_{\psi} \subseteq \kappa_K((\overline{\pi}))$
using Lemma~\ref{L:intermediate mod p subring}.

To deduce that $\psi$ is weakly decompleting, it remains to check that 
$\bA^{r}_\psi$ maps strict surjectively onto $R_\psi$ for some $r>0$. For this, it suffices to choose $r$ so that 
$\lambda(\overline{\alpha}^r_\psi)(\pi - [\overline{\pi}]) \leq p^{-1/2} \overline{\alpha}^r_{\psi}(\overline{\pi})$.
\end{proof}

\begin{lemma} \label{L:cyclotomic coherent}
There exists some $r_0 > 0$ such that for $0 < s \leq r \leq r_0$, the ring
$\bC^{[s,r]}_\psi$ is noetherian (hence coherent). That is, the tower $\psi$ is noetherian.
\end{lemma}
\begin{proof}
This ring is a direct sum of subrings, each of which is the ring of rigid analytic functions on a closed annulus. It is well-known that any such ring is a principal ideal domain (see for example \cite[Proposition~8.3.2]{kedlaya-course}).
\end{proof}

\begin{lemma} \label{L:cyclotomic decompleting}
The tower $\psi$ is decompleting.
\end{lemma}
\begin{proof}
By Corollary~\ref{C:finite etale decompleting}, we may once again reduce to the case
where $K$ is absolutely unramified.
Since we are in the setting of Example~\ref{exa:Galois tower}, we need only check that the continuous $H$-cohomology of 
$\overline{\varphi}^{-1}(R_\psi)/R_\psi$ vanishes for every sufficiently small open subgroup $H$ of $\Gamma$. However, 
for $H \subseteq 1 + p \ZZ_p$ we have an $H$-equivariant decomposition
\[
\overline{\varphi}^{-1}(R_\psi)/R_\psi \cong \bigoplus_{i=1}^{p-1}
(1 + \overline{\pi})^{i/p} R_\psi
\]
of finite-dimensional vector spaces over $R_\psi$, and hence of Banach modules
over $R_\psi$ by Theorem~\ref{T:open mapping}.
Consequently, it is in fact enough to check that the continuous $H$-cohomology of $\overline{y} R_\psi$ vanishes in case $\overline{y} = (1 + \overline{\pi})^{i/p}$ for some $i \in \{1,\dots,p-1\}$.
Put $\gamma = 1 + p^2 \in \Gamma$, so that $\gamma$ generates a procyclic pro-$p$-subgroup of $\Gamma$. Then
\[
(\gamma^{p^n}-1)(\overline{\pi}) = (\gamma^{p^n}-1)(1 + \overline{\pi}) = 
((1 + \overline{\pi})^{\gamma^{p^n} - 1} -1)(1 + \overline{\pi}).
\]
Note that $\gamma^{p^n} - 1$ has $p$-adic valuation $n+2$,
so $(1 + \overline{\pi})^{\gamma^{p^n} - 1} -1$ is a power series in $\overline{\pi}^{p^{n+2}}$ with zero constant term and nonzero linear term.
Consequently, on one hand, using \eqref{eq:Leibniz rule} we obtain
\[
\overline{\alpha}_{\psi}((\gamma^{p^n}-1)(\overline{x}))
\leq p^{-p^{n+3}/(p-1)} \overline{\alpha}_{\psi}(\overline{x})
\qquad (\overline{x} \in R_\psi);
\]
on the other hand,
\[
\overline{\alpha}_{\psi}((\gamma^{p^n}-1)(\overline{y}))
= p^{-p^{n+2}/(p-1)} \overline{\alpha}_{\psi}(\overline{y}).
\]
By writing
\begin{equation} 
(\gamma^{p^n}-1)(\overline{x} \overline{y})
= (\gamma^{p^n}-1)(\overline{x}) \overline{y} + \gamma^{p^n}(\overline{x})
(\gamma^{p^n}-1)(\overline{y}),
\end{equation}
we see that for all $\overline{z} \in \overline{y} R_\psi$ we have
\[
\overline{\alpha}_{\psi}((\gamma^{p^n}-1)(\overline{z}))
= p^{-p^{n+2}/(p-1)} \overline{\alpha}_{\psi}(\overline{z}).
\]
Consequently, $\gamma^{p^n}-1$ acts bijectively on $\overline{y} R_\psi$, so the continuous $H$-cohomology vanishes for $H = 1 + p^{n+2} \ZZ_p$ for all $n \geq 0$.
This proves the claim.
\end{proof}

Putting everything together, we have the following.
\begin{theorem} \label{T:cyclotomic decompleting}
The tower $\psi$ satisfies the conditions of Remark~\ref{R:decompleting conditions},
so all of the results of \S\ref{subsec:phigamma}, \S\ref{subsec:Frobenius splittings},
\S\ref{subsec:pseudocoherent type C}
apply to it.
\end{theorem}
\begin{proof}
The tower $\psi$ is decompleting (and hence automatically locally decompleting) by Lemma~\ref{L:cyclotomic decompleting}.
The ring $R_\psi$ is $F$-(finite projective) by Lemma~\ref{L:cyclotomic Rpsi}.
The ring $\breve{\bC}^{[s,r]}_\psi$ is coherent by
Lemma~\ref{L:cyclotomic coherent}.
\end{proof}

\begin{remark} \label{R:infinite orbit}
When applying Theorem~\ref{T:add tilde2} in the setting of Theorem~\ref{T:cyclotomic decompleting}, we may also keep in mind Theorem~\ref{T:perfect generalized phi-modules at a point}, which guarantees that every pseudocoherent $(\varphi,\Gamma)$-module over $\bC_\psi$ is fpd with projective dimension at most 1.
Moreover, the torsion in any such module is killed by some power of $\log([1+\overline{\pi}])$, because any point of $\FFC_{\tilde{R}_\psi}$ not in the zero locus of this function has infinite $\Gamma$-orbit. (See \cite[Proposition~4.1]{liu-herr} for an equivalent calculation.)
\end{remark}

\begin{theorem} \label{T:pseudocoherent resolution}
Any pseudocoherent $(\varphi, \Gamma)$-module over $\tilde{\bC}_\psi$ is a quotient of some projective
$(\varphi, \Gamma)$-module.
\end{theorem}
\begin{proof}
Let $M$ be a pseduocoherent $(\varphi, \Gamma)$-module.
Put $t = \log([1+\overline{\pi}])$; by Remark~\ref{R:infinite orbit}, 
we have an exact sequence
\[
0 \to T \to M \to N \to 0
\]
of pseudocoherent $(\varphi, \Gamma)$-modules with $T$ killed by a power of $t$
and $N$ projective.

To prove that $M$ admits a surjection from a projective $(\varphi, \Gamma)$-module, we first treat the case $M = T$. 
Since $T$ is killed by a power of $t$, we may also view it as a module over $\tilde{\bC}^{[p^{-1/2}, p^{1/2}]}_\psi$.
Let $\gamma$ be an element of $\Gamma$ of infinite order.
Choose a finite free $\tilde{\bC}^{[p^{-1/2}, p^{1/2}]}_\psi$-module $F_0$ admitting a surjection $F_0 \to T$. Since the action of $\Gamma$ on $T$ is continuous, by replacing $\gamma$ by a suitable power we may ensure that
the actions of $\gamma, \gamma^{-1}$ of $T$  may be represented by matrices $B,C$
such that $|B-1|, |C-1| < 1$. By arguing as in \cite[Lemma~1.5.2]{part1}, we obtain a semilinear $\gamma$-action on $F_1 = F_0 \oplus F_0$ and a $\gamma$-equivariant surjection $F_0 \to T$; by the conditions on $B$ and $C$, this action extends continuously
to an action of the closure $\Gamma'$ of the subgroup of $\Gamma$ generated by $\Gamma'$.
Put $F = \Ind^{\Gamma}_{\Gamma'} F_1$; then $F_1 \to M$ corresponds to a $\Gamma$-equivariant surjection $F \to M$. We may choose any matrix of action for $\varphi$
with coefficients in $\tilde{\bC}^{[p^{1/2}, p^{1/2}]}_\psi$ to promote $F$ to a projective $(\varphi, \Gamma)$-module over $\tilde{\bC}^{[p^{-1/2}, p^{1/2}]}_\psi$
admitting a $(\varphi, \Gamma)$-equivariant surjection $F \to T$. Moreover, for $n \in \ZZ$, we may twist the action to obtain a surjection $F(n) \to T$.

We now return to the general case. For $n \in \ZZ$, define maps $F(n) \to T$ as in the previous paragraph, and let $K_n$ be the kernel of $F(n) \to T$. Let $s_0$ be the maximum slope of $N$ and let $s_1$ be the minimum slope of $F$. Let $P_n$ be the final quotient in the slope filtration of $K_n$. By pushing out, we form the exact sequence
\[
0 \to P_n \to F(n)/\ker(K_n \to P_n) \to T \to 0
\]
in which the middle term has rank $r = \rank(P_n)$ and degree $r \mu(P_n) + \deg(T)$.
Since this term is a quotient of $F(n)$, its average slope is at least $\mu(F(n)) \geq s_1 + n $;
therefore $\mu(P_n) \geq s_1 + n - \deg(T) /r \geq s_1 + n - \deg(T)$.
In particular, if $n > \deg(T) - s_1 + s_0$, then the slopes of $N^\dual \otimes K_n$ are all positive, so $H^1_{\varphi}(N^\dual \otimes K_n) = 0$ 
(apply \cite[Theorem~8.8.15]{part1}) and hence $H^2_{\varphi, \Gamma}(N^\dual \otimes K_n) = 0$. Consequently, $H^1_{\varphi, \Gamma}(N^\dual \otimes F(n)) \to
H^1_{\varphi, \Gamma}(N^\dual \otimes T)$ is surjective, so the extension class defining $M$ lifts to define a commutative diagram
\[
\xymatrix{
0 \ar[r] & F(n) \ar[r] \ar[d] & M' \ar[r] \ar[d] & N \ar[r] \ar@{=}[d] & 0 \\
0 \ar[r] & T \ar[r] & M \ar[r] & N \ar[r] & 0
}
\]
with exact rows. By the five lemma, $M' \to M$ is surjective, proving the claim.
\end{proof}

\begin{cor} 
The category of pseudocoherent $(\varphi, \Gamma)$-modules over $\tilde{\bC}_\psi$ is 
the abelian completion of the exact category of
$(\varphi, \Gamma)$-modules.
\end{cor}

\begin{remark} \label{R:nonzero higher cohomology}
The theory of $(\varphi, \Gamma)$-modules in this setting reproduces the $(\varphi, \Gamma)$-modules of traditional $p$-adic Hodge theory;
the theory of pseudocoherent $(\varphi, \Gamma)$-modules in this setting reproduces 
the \emph{generalized $(\varphi, \Gamma)$-modules} described in \cite[\S 3]{liu-herr} and
\cite[Appendix]{kedlaya-bordeaux}.
(By Theorem~\ref{T:pseudocoherent resolution},
the distinction made in \textit{loc.\ cit.}\ between generalized $(\varphi, \Gamma)$-modules
and \emph{$(\varphi, \Gamma)$-quotients} collapses.)
It is useful to keep in mind some facts from the usual theory by translating them into the language of period sheaves. For example, 
\[
H^1(\Spa(\Qp,\Zp)_{\proet}, \widehat{\calO}) 
= H^1_{\varphi, \Gamma}(\tilde{\bC}_\psi/\tilde{\bC}_\psi(-1))
= H^1_{\varphi, \Gamma}(\bC_\psi/\bC_\psi(-1))
= \Qp
\]
where the second equality follows from Theorem~\ref{T:Galois cohomology1c}
(plus Theorem~\ref{T:cyclotomic decompleting}) and the third from \cite[\S 4.1]{liu-herr}.
\end{remark}

\begin{remark}
One important actor from traditional $p$-adic Hodge theory that does not make an appearance in our framework is the reduced trace of $\varphi$, often called $\psi$ in the literature. We have avoided introducing the reduced trace for two reasons. One is that it does not admit an analogue in the realm of perfect period rings. The other is that  its key properties are quite specific to the cyclotomic tower; for instance, for a general decompleting tower one does not expect the reduced trace of $\varphi$ to preserve
$\bA_\psi$. For these reasons, we broadly advocate to rework applications of $p$-adic Hodge theory in which the reduced trace plays a central role, such as in the study of Iwasawa cohomology (e.g., see \cite{kpx, pottharst} and especially
\cite{kedlaya-pottharst}), to obtain constructions more compatible with the perfectoid approach to $p$-adic Hodge theory.
\end{remark}

\subsection{Kummer towers}
\label{subsec:Kummer}

We next consider a class of non-Galois towers, for which we establish the weakly decompleting property and conjecture the decompleting property. The theory of $(\varphi, \Gamma)$-modules in this setting reproduces Caruso's theory of $(\varphi, \tau)$-modules
\cite{caruso}; the decompleting property would thus establish an analogue of the Cherbonnier-Colmez theorem in Caruso's theory.

\begin{hypothesis} \label{H:Kummer}
Throughout \S\ref{subsec:Kummer}, set notation as in Hypothesis~\ref{H:towers} with
$h=1$, $E = \Qp$, $\varpi = p$,
then fix a uniformizer $\varpi_0$ of $K$.
\end{hypothesis}

\begin{defn}
Put $K_0 = K$.
For $n \geq 1$, put $K_n = K_{n-1}[\varpi_n]/(\varpi_n^p - \varpi_{n-1})$.
This tower is finite \'etale but not Galois.
However, it does satisfy the hypotheses of Example~\ref{exa:Galois tower2}
by taking $\psi'$ to be the tower with $A_{\psi',n} = K_n[\epsilon_n]$
for $\epsilon_n$ as in Definition~\ref{D:cyclotomic tower};
namely, the tower $\psi'$ is Galois with Galois group $\Gamma_1 = \Zp \rtimes \Zp^\times$, and $A_{\psi,n} = A_{\psi',n}^{\Gamma_2}$ for $\Gamma_2 = \Zp^\times$.
\end{defn}

\begin{lemma} \label{L:Kummer perfectoid}
The tower $\psi$ is perfectoid.
\end{lemma}
\begin{proof}
Again, $\tilde{A}_\psi$ is an analytic field, so by
\cite[Proposition~3.6.2]{part1} we need only check that it is nondiscretely valued (which is clear) and that Frobenius is surjective on
$\gotho_{\tilde{A}_\psi}/(\varpi_0)$. The latter is clear because
\[
\gotho_{A_\psi} = \gotho_K[\varpi_1, \varpi_2, \dots]
\]
and so
\[
\gotho_{\tilde{A}_\psi}/(\varpi_0) = \kappa_K[\overline{\epsilon}_1, \overline{\epsilon}_2, \dots]/(\overline{\epsilon}_1^p, \overline{\epsilon}_2^p - \overline{\epsilon}_1, \dots).
\]
\end{proof}

\begin{defn}
Since $\psi$ is perfectoid, we may now set notation as in \S\ref{sec:axiomatic}.
By \cite[Lemma~3.4.2]{part1}, we have an identification
\[
\gotho_{\tilde{R}_\psi}
\cong \varprojlim_{\overline{\varphi}} \gotho_{\tilde{A}_\psi}/(\varpi);
\]
let $\overline{\varpi} \in \gotho_{\tilde{R}_\psi}$ be the element corresponding to the sequence $(\dots, \varpi_1, \varpi_0)$.
Note that $\overline{\alpha}_\psi(\overline{\varpi}) = p^{-1/e}$ where $e$ is the absolute ramification index of $K$.
\end{defn}

\begin{lemma} \label{L:Kummer Rpsi}
The ring $R_{\psi}$ equals $\kappa_K(( \overline{\varpi}))$
and $\tilde{R}_{\psi}$ is the completion of $R_{\psi}^{\perf}$.
In particular, $\psi$ is weakly decompleting and $R_\psi$ is $F$-(finite projective).
\end{lemma}
\begin{proof}
Since $\varpi \in \tilde{\calR}^{\inte}_\psi$, we have $\kappa_K((\overline{\varpi})) \subseteq R_\psi$. We may then see that $\tilde{R}_\psi$ is the completion of $R_\psi^{\perf}$ by comparing truncated valuation rings,
and then see that $R_{\psi} \subseteq \kappa_K((\overline{\varpi}))$
using Lemma~\ref{L:intermediate mod p subring}.

To deduce that $\psi$ is weakly decompleting, it remains to check that 
$\bA^{r}_\psi$ maps strict surjectively onto $R_\psi$ for some $r>0$. However, this is true for any $r>0$ because $[\overline{\varpi}] \in \bA^{r}_\psi$.
\end{proof}

\begin{conj} \label{conj:Kummer decompleting}
The tower $\psi$ is decompleting. 
\end{conj}

\begin{remark}
The main difficulty in proving that $\psi$ is decompleting 
is the fact that $\psi'$ is not weakly decompleting, so the structure of $R_{\psi'}$ is quite complicated. 
To wit, let $\epsilon \in \tilde{R}_{\psi'}$ be the element $(\dots, \epsilon_1, \epsilon_0)$
and put $\pi = [\epsilon] - 1$. 
Note that $R_{\psi'}$ contains both $R_\psi$ and $\kappa_K((\overline{\pi}))$.
Let $L$ be the completion of $\kappa_K((\overline{\pi}))^{\perf}$; then
the induced map $\tilde{R}_\psi \widehat{\otimes}_{\kappa_K} L \to \tilde{R}_{\psi'}$ has dense image but is not an isomorphism, due to a mismatch of norms.
Similarly, the map
$R_\psi \widehat{\otimes}_{\kappa_K} \kappa_K((\overline{\pi})) \to R_{\psi'}$ is not an isomorphism.

In any case, if we let $\gamma_n$ denote the element $1+p^n$ of $\Gamma_2$
and let $\tau_n$ be the element $p^n$ of $\Zp \subset \Gamma_1$, we have
\[
\gamma_n(\overline{\pi}) = \overline{\pi} + (1 + \overline{\pi}) \overline{\pi}^{p^n}, \quad
\tau_n(\overline{\pi}) = \overline{\pi}, \quad
\gamma_n(\overline{\varpi}) = \overline{\varpi}, \quad
\tau_n(\overline{\varpi}) = (1 + \overline{\pi}^{p^n}) \overline{\varpi}.
\]
It would suffice to check that for any sufficiently large $n$, the map
\[
\tau_n - 1: \varphi^{-1}(R_\psi)/R_\psi \to
\{\overline{x} \in \varphi^{-1}(R_{\psi'})/R_{\psi'}: \gamma(\overline{x}) = \frac{\tau_n^{\gamma}-1}{\tau_n-1}(\overline{x}) \quad (\gamma \in \Zp^\times \cap \ZZ) \}
\]
is an isomorphism, or even surjective.
\end{remark}

\begin{remark}
One consequence of Conjecture~\ref{conj:Kummer decompleting} which is known is that
base extension of $(\varphi, \Gamma)$-modules from $\bA^\dagger_\psi$ to $\bA_\psi$
is an equivalence of categories; this was recently established by Gao and Liu \cite{gao-liu}. The relationship between the proof technique and the methods used here remains to be determined.
\end{remark}

\subsection{\texorpdfstring{$\Phi$}{Phi}-iterate towers}
\label{subsec:phi-iterate}

We exhibit some additional examples of weakly decompleting towers which are not at all related to $p$-adic Lie groups, using the work of Cais and Davis \cite{cais-davis}.
This construction provides an instructive model for considering Lubin-Tate towers,
but the latter will be postponed to a subsequent paper.

\begin{hypothesis} \label{H:cais-davis}
Throughout \S\ref{subsec:Kummer}, fix $E, h, \varpi$ as in Hypothesis~\ref{H:towers} and assume that $K$ contains $E$.
In addition, fix a power series $\Phi(T) \in \gotho_E \llbracket T \rrbracket$
satisfying $\Phi(0)= 0$ and $\Phi(T) \equiv T^p \pmod{\varpi}$.
\end{hypothesis}

\begin{defn}
Let $\psi$ be the tower over $K$ in which $K_{\psi,n+1} = K_{\psi,n}[\varpi_{n+1}]/(\Phi(\varpi_{n+1}) - \varpi_n)$.
By \cite[Lemma~3.4.2]{part1}, we have an identification
\[
\gotho_{\tilde{R}_\psi}
\cong \varprojlim_{\overline{\varphi}} \gotho_{\tilde{A}_\psi}/(\varpi);
\]
let $\overline{\varpi}$ be the element of $\gotho_{\tilde{R}_\psi}$ corresponding to the sequence $(\dots, \varpi_1, \varpi_0)$.
\end{defn}

\begin{lemma} \label{L:Frobenius equivariant embedding}
Let $R$ be a $\varpi$-torsion-free $\gotho_E$-algebra.
Let $f: R \to R$ be an $\gotho_E$-algebra homomorphism that induces the $p^h$-power Frobenius map modulo $\varpi$. Then there exists a unique $\gotho_F$-algebra homomorphism
$\lambda_f: R \to W_\varpi(R)$ which is a section of the first projection $W_\pi(R) \to R$ and which satisfies $\lambda_f \circ f = \varphi_\varpi \circ \lambda_f$.
In particular, the construction is functorial in the pair $(R,f)$.
\end{lemma}
\begin{proof}
See \cite[Proposition~2.13]{cais-davis}.
\end{proof}

\begin{prop} \label{P:cais-davis}
The tower $\psi$ is weakly decompleting
with $R_\psi = \kappa_K((\overline{\varpi}))$. In particular, $R_\psi$ is $F$-(finite projective).
\end{prop}
\begin{proof}
The tower $\psi$ is perfectoid because it defines a strictly APF extension of $K$,
as shown in \cite{cais-davis-lubin}. 
To check the other claims, it suffices to exhibit a lift of $\overline{\varpi}$ to $\bA^+_\psi$, as then one may argue as in the proofs of Lemma~\ref{L:cyclotomic Rpsi} and Lemma~\ref{L:Kummer Rpsi}.
For this, we recall the proof of \cite[Proposition~7.14]{cais-davis}.
Identify $\Phi$ with a ring endomorphism of
$\gotho_E \llbracket T \rrbracket$.
By Lemma~\ref{L:Frobenius equivariant embedding},
there exists a unique $\gotho_E$-algebra homomorphism 
$\lambda_\Phi: \gotho_E \llbracket T \rrbracket \to W_\varpi(\gotho_E \llbracket T \rrbracket)$ which is a section to the first projection and satisfies
$\lambda_\Phi \circ \Phi = \varphi_\varpi \circ \lambda_\Phi$. By
\cite[Corollary~2.14]{cais-davis}, for $r \in \gotho_E \llbracket T \rrbracket$,
the ghost components of $\lambda_\Phi(r)$ are $r, \Phi(r), \Phi^2(r),\dots$.
For each $n$, consider the $\gotho_E$-algebra homomorphism 
$\gotho_E \llbracket T \rrbracket \to \gotho_{A_{\psi,n}}$ taking $T$ to $\varpi_n$.
The composition $\gotho_E \llbracket T  \rrbracket \to W_\varpi(\gotho_E \llbracket T \rrbracket) \to W_\varpi(\gotho_{A_{\psi,n}})$ then takes $T$ to 
an element $x_n$ with ghost components $(\varpi_n,\varpi_{n-1},\dots,\varpi_0, \Phi(\varpi_0), \dots)$. In particular, the $x_n$ form an element of $\varprojlim_{\Phi} W_\varpi(\gotho_{A_{\psi,n}})$, which maps to $\bA^+_\psi$ via \eqref{eq:theta identification}. The resulting element of $\bA^+_\psi$ maps to $\overline{\varpi}$ in $R_\psi$, proving the claim.
\end{proof}

\begin{conj} \label{C:cais-davis}
The tower $\psi$ is decompleting.
\end{conj}

\section{Toric towers}
\label{sec:toric}

We next consider towers derived from the standard perfectoid cover of a torus, which is obtained by taking roots of the coordinate functions. The standard towers
are sufficient for many applications to the theory of relative $(\varphi, \Gamma)$-modules on smooth analytic spaces, such as those discussed in \S\ref{sec:applications}. For additional applications, we also introduce an analogue of the Andreatta-Brinon construction of relative toric towers, which combines a cyclotomic tower over a base $p$-adic field with the extraction of $p$-th roots of local coordinates. We also describe logarithmic variants of both constructions.

\setcounter{theorem}{0}
\begin{hypothesis}
Throughout \S\ref{sec:toric}, we take $E=\Qp, \varpi=p$ in Hypothesis~\ref{H:towers}.
\end{hypothesis}

\subsection{Standard toric towers}
\label{subsec:standard toric towers}

We begin by constructing toric towers over a base ring which is already perfectoid.
\begin{defn} \label{D:standard toric tower1}
Let $(A_0, A_0^+)$ be an adic perfectoid algebra over $\Qp$
corresponding to the perfect uniform adic Banach ring $(R_0,R_0^+)$ via
Theorem~\ref{T:Fontaine perfectoid correspondence}.
For $d \geq 0$ and $q_1,\dots,q_d,r_1,\dots,r_d$ real numbers with $0 < q_i \leq r_i$, the \emph{standard toric tower} over $(A_0, A_0^+)$ 
with parameters $q_1,\dots,q_d, r_1,\dots,r_d$
is the finite \'etale tower 
$\psi$ with 
\begin{align*}
A &= A_0\{q_1/T_1, \dots, q_d/T_d, T_1/r_1, \dots, T_d/r_d\} \\
A^+ &= A_0^+\{q_1/T_1, \dots, q_d/T_d, T_1/r_1, \dots, T_d/r_d\} \\
A_{\psi,n} &= A_0\{q_1^{p^{-n}}/T_1^{p^{-n}}, \dots, q_d^{p^{-n}} / T_d^{p^{-n}}, T_1^{p^{-n}}/r_1^{p^{-n}},\dots,T_d^{p^{-n}}/r_d^{p^{-n}} \}.
\end{align*}
If $q_1,\dots,q_d,r_1,\dots,r_d$ are not mentioned, we take them to equal $1$.
If $d=1$, we usually omit it from the subscripts. For $i=1,\dots,d$, let $\overline{t}_i \in  R_\psi$ be the element corresponding to the sequence $T_i, T_i^{1/p}, \dots$.
\end{defn}

\begin{lemma} \label{L:standard toric weakly decompleting}
With notation as in Definition~\ref{D:standard toric tower1}, the tower $\psi$ is weakly decompleting and 
\begin{equation} \label{eq:reduction of standard toric tower}
R_\psi = R_0\{
q_1/\overline{t}_1, \dots,q_d/\overline{t}_d, 
\overline{t}_1/r_1, \dots,\overline{t}_d/r_d\}.
\end{equation}
\end{lemma}
\begin{proof}
It is clear that $R_0\{
q_1/\overline{t}_1, \dots,q_d/\overline{t}_d, 
\overline{t}_1/r_1, \dots,\overline{t}_d/r_d\} \subseteq R_\psi$.
The inclusion in the other direction follows from
Lemma~\ref{L:intermediate mod p subring}.
This implies the stated results.
\end{proof}

\begin{remark} \label{R:affinoid1}
With notation as in Definition~\ref{D:standard toric tower1}, let $\alpha$ denote the spectral norm on $R_0$. Then for $0 < s \leq r$,
we have the identification
\begin{equation} \label{eq:affinoid1}
\bC^{[s,r]}_\psi \cong \tilde{\calR}^{[s,r]}_{R_0}
\{q_1/[\overline{t}_1], \dots, q_d/[\overline{t}_d], [\overline{t}_1]/r_1, \dots, [\overline{t}_d]/r_d\}
\end{equation}
provided that we equip $\tilde{\calR}^{[s,r]}_{R_0}$
with the norm $\max\{\lambda(\alpha^s)^{1/s}, \lambda(\alpha^r)^{1/r}\}$
as in Remark~\ref{R:Robba norms} (instead of our usual choice $\max\{\lambda(\alpha^s), \lambda(\alpha^r)\}$).

Suppose in addition that $A_0$ is a perfectoid field. By 
\cite[Theorem~4.10]{kedlaya-noetherian}, we may deduce from \eqref{eq:affinoid1} that $\bC^{[s,r]}_\psi$ is really strongly noetherian.
\end{remark}

\begin{defn} \label{D:toric tower1}
Let $(A_0, A_0^+)$ be an adic perfectoid algebra over $\QQ_p$.
By a \emph{toric tower} over $(A_0, A_0^+)$, we will mean a tower obtained from some standard toric tower over $(A_0, A_0^+)$ by a sequence of base extensions
along morphisms $(A,A^+) \to (B,B^+)$ each of one of the following forms:
\begin{itemize}
\item[(i)]
a rational localization;
\item[(ii)]
a finite \'etale extension;
\item[(iii)]
a morphism in which $B^+$ is obtained by completing $A^+$ at a finitely generated ideal containing $p$;
\item[(iv)]
a morphism in which $B^+$ is obtained by taking the $p$-adic completion of a (not necessarily finite) \'etale extension of $A^+$;
\item[(v)]
a morphism in which $B^+$ is obtained by taking the $p$-adic completion of a (not necessarily of finite type) algebraic localization of $A^+$.
\end{itemize}
By Lemma~\ref{L:standard toric weakly decompleting} and Proposition~\ref{P:weakly decompleting tower persistence}, any such tower is weakly decompleting,
and moreover $R_\psi$ is $F$-(finite projective).
If only base extensions of type (i) and (ii) are used, we will refer to the resulting tower also as a \emph{restricted toric tower}.
\end{defn}

\begin{remark} \label{R:standard group action}
Suppose that $A_0$ contains $\QQ_p(\mu_{p^\infty})$. Then any toric tower is Galois with group $\Gamma = \ZZ_p^d$; the action of an element $\gamma$ of the $i$-th copy of $\ZZ_p$ in $\Gamma$ fixes $T_j^{p^{-n}}$ for $j \neq i$ and takes $T_i^{p^{-n}}$ to $\zeta_{p^n} T_i^{p^{-n}}$. Consequently, the cohomology of $\Gamma$-modules can be computed as continuous group cohomology as per
Example~\ref{exa:Galois tower}.
\end{remark}

\begin{remark} \label{R:strongly noetherian imperfect}
Let $\psi$ be a restricted toric tower obtained by base extension from a standard toric tower $\psi_0$, and suppose moreover that $A_0$ is a perfectoid field. For $0 < s \leq r$, by Remark~\ref{R:affinoid1}, $\bC^{[s,r]}_{\psi_0}$ is really strongly noetherian. There then exists $r_0>0$ such that for $r \leq r_0$, $\bC^{[s,r]}_\psi$ can be obtained from $\bC^{[s,r]}_{\psi_0}$ by a sequence of rational localizations and finite \'etale morphisms, and therefore is also really strongly noetherian. For similar reasons, $R_\psi$ and $\bA^{r}_{\psi}$ are also really strongly noetherian.

As suggested in \cite{kedlaya-noetherian}, one expects Tate algebras over $\bC^{[s,r]}_{\psi_0}$ to satisfy additional ring-theoretic properties by analogy with affinoid algebras over a field; for instance, they should be regular and excellent and admit a relative form of the Nullstellensatz. Such results would imply similar statements about the rings $\bC^{[s,r]}_\psi$. See \cite{wear} for some results in this direction.
\end{remark}

We may obtain the decompleting property for standard toric towers by directly imitating the proof of Theorem~\ref{T:cyclotomic decompleting}. In order to make the argument applicable to general toric towers, we need a small additional computation.
\begin{lemma} \label{L:standard toric decompleting}
Any toric tower is decompleting. (The reified analogue also holds.)
\end{lemma}
\begin{proof}
Suppose first that $A_0$ contains $\Qp(\mu_{p^\infty})$,
in which case $\psi$ is Galois (see Remark~\ref{R:standard group action}).
As in Definition~\ref{D:toric tower1},  $\psi$ is weakly decompleting and
we have the $\Gamma$-equivariant decomposition
\begin{equation} \label{eq:ab-decomposition1}
\overline{\varphi}^{-1}(R_{\psi^\bullet})/R_{\psi^\bullet}
\cong  
\bigoplus_{e_1,\dots,e_d} \overline{t}_1^{e_1/p} \cdots \overline{t}_d^{e_d/p} R_{\psi^\bullet},
\end{equation}
where $(e_1,\dots,e_d)$ runs over $\{0,\dots,p-1\}^d \setminus \{(0,\dots,0)\}$;
this also constitutes a decomposition of Banach modules by the open mapping theorem (Theorem~\ref{T:open mapping}).
 For each index $(e_1,\dots,e_d)$, put $\overline{y} = \overline{t}_1^{e_1/p} \cdots \overline{t}_d^{e_d/p}$,
choose $i \in \{1,\dots,d\}$ for which $e_i \neq 0$, and let $\gamma$ be the canonical generator of the $i$-th copy of $\Zp$ within $\Gamma$. Define $\overline{\pi}$ as in Theorem~\ref{T:cyclotomic decompleting};
then $(\gamma^{p^n}-1)(\overline{y}) = \overline{\pi}^{p^{n-1}} \overline{y}$, so 
\[
\overline{\alpha}_\psi((\gamma^{p^n}-1)(\overline{y})) = p^{-p^{n}/(p-1)} \overline{y}.
\]
On the other hand, by Lemma~\ref{L:exponential convergence gamma} below, there exists $c>0$ such that
\[
\overline{\alpha}_\psi((\gamma^{p^n}-1)(\overline{x}))
\leq c p^{-p^{n+1}/(p-1)}
\overline{\alpha}_\psi(\overline{x}) \qquad (n \geq 0, \overline{x} \in R_\psi).
\]
Using \eqref{eq:Leibniz rule}, we see that
for $n$ sufficiently large, 
$\overline{\pi}^{-p^{n-1}} (\gamma^{p^n}-1)$
acts on $\overline{y} R_\psi$ as the identity plus an operator of norm less than 1. It follows that $\gamma^{p^n}-1$ acts invertibly on $\overline{y} R_\psi$, 
from which it follows that $\psi$ is decompleting (by any of Lemma~\ref{L:hochschild-serre}, Theorem~\ref{T:kill analytic cohomology}, or an elementary calculation using the fact that $\Gamma$ is abelian).

To treat the general case, equip $A_n = A_0 \otimes_{\Qp} \Qp(\mu_{p^n})$ with the spectral norm. Since $A_n$ is finite \'etale over $A_0$, it is again perfectoid by
Theorem~\ref{T:Fontaine perfectoid compatibility}.
By Lemma~\ref{L:perfectoid tower splitting}, we may split the inclusions $A_0 \to A_n$ in such a way that the splitting extends continuously to the completed direct limit of the $A_n$. This then allows us to reduce to the previous case.
\end{proof}

\begin{lemma} \label{L:exponential convergence gamma}
Let $\psi$ be a toric tower such that $A_0$ contains $\Qp(\mu_{p^\infty})$,
so that $\psi$ is Galois with group $\Gamma = \Zp^d$. Then for any $\gamma \in \Gamma$,
there exists $c>0$ such that 
\begin{equation} \label{eq:exponential convergence gamma}
\overline{\alpha}_\psi((\gamma^{p^n}-1)(\overline{x}))
\leq c p^{-p^{n+1}/(p-1)}
\overline{\alpha}_\psi(\overline{x}) \qquad (n \geq 0, \overline{x} \in R_\psi).
\end{equation}
\end{lemma}
\begin{proof}
For $\psi$ a standard toric tower, we may read the claim directly off of \eqref{eq:reduction of standard toric tower}. It thus suffices to check that if $\psi'$ is obtained from a toric tower $\psi$ by a single operation of one of the forms (i)--(v) enumerated in Definition~\ref{D:toric tower1} and the action of $\Gamma$ on $R_{\psi}$ is analytic, then so is the action on $R_{\psi'}$.
This is obvious for type (iii); for the other types, we may argue as in the proof of Proposition~\ref{P:extend analyticity along affinoid}.
\end{proof}

\begin{theorem} \label{T:standard toric decompleting}
Let $\psi$ be a toric tower.
\begin{enumerate}
\item[(a)]
The results of \S\ref{subsec:phigamma} and \S\ref{subsec:Frobenius splittings} hold for $\psi$.
\item[(b)]
If $\psi$ is a restricted toric tower over a perfectoid field,
then the results of \S\ref{subsec:pseudocoherent type C} also hold for $\psi$.
\end{enumerate}
(The reified analogue also holds.)
\end{theorem}
\begin{proof}
By Lemma~\ref{L:standard toric decompleting}, the tower $\psi$ is locally decompleting;
this proves (a).
By Remark~\ref{R:strongly noetherian imperfect},
the ring $R_\psi$ is $F$-(finite projective)  and the rings $\breve{\bC}^{[s,r]}_\psi$ are coherent; this proves (b)
by Remark~\ref{R:decompleting conditions}.
\end{proof}

\begin{cor} \label{C:pseudocoherent fpd local}
Suppose that $\psi$ is a restricted toric tower over a perfectoid field.
For $a$ a positive integer, every pseudocoherent $\varphi^a$-module over $\tilde{\bC}_\psi$ is fpd of projective dimension at most $d+1$.
\end{cor}
\begin{proof}
This follows from Theorem~\ref{T:standard toric decompleting},
Theorem~\ref{T:add tilde2}, and Theorem~\ref{T:curve noetherian}(c).
\end{proof}

Let us make explicit a point which has already come up implicitly in the preceding discussion. 
\begin{remark} \label{R:toric splitting}
For toric towers, the inclusion $A_\psi \to \tilde{A}_\psi$ admits a unique $\Gamma$-equivariant splitting: in the case of a standard toric tower this is the map
\[
\sum_{i_1,\dots,i_d \in \ZZ[p^{-1}]} a_{i_1,\dots,i_d} T_1^{i_1} \cdots T_d^{i_d}
\mapsto
\sum_{i_1,\dots,i_d \in \ZZ} a_{i_1,\dots,i_d} T_1^{i_1} \cdots T_d^{i_d}.
\]
Similarly, the inclusion $\bC_\psi \to \tilde{\bC}_\psi$ admits a unique $\Gamma$-equivariant splitting.
\end{remark}

\subsection{Relative toric towers}
\label{subsec:andreatta-brinon}

We next describe a modification to the construction of toric towers over a $p$-adic field that incorporates the cyclotomic tower.
The resulting \emph{relative toric towers} are decompleting; the ensuing results extend the work of Andreatta and Brinon \cite{andreatta-brinon},
in particular their analogue of the theorem of Cherbonnier-Colmez. It also follows that the categories of $(\varphi, \Gamma)$-modules considered in \cite{andreatta-brinon}
are canonically equivalent to the categories of relative $(\varphi, \Gamma)$-modules
on the same underlying spaces in the sense of \cite{part1}, so any results about the objects of \cite{andreatta-brinon} can be immediately reinterpreted in our present framework. 

\begin{defn} \label{D:standard toric tower}
Let $K$ be a $p$-adic field. Let $\psi_0$ denote the cyclotomic tower over $K$. For $d \geq 0$ and $q_1,\dots,q_d,r_1,\dots,r_d \in p^{\QQ}$ with $q_i \leq r_i$, the \emph{standard relative toric tower} over $K$ with parameters $q_1,\dots,q_d, r_1,\dots,r_d$
is the finite \'etale tower 
$\psi$ with 
\[
A = K\{q_1/T_1, \dots, q_d/T_d, T_1/r_1, \dots, T_d/r_d\},
\]
$A^+ = A^\circ$, and
\[
A_{\psi,n} = A_{\psi_0,n} \{q_1^{p^{-n}}/T_1^{p^{-n}}, \dots, q_d^{p^{-n}} / T_d^{p^{-n}}, T_1^{p^{-n}}/r_1^{p^{-n}},\dots,T_d^{p^{-n}}/r_d^{p^{-n}} \}.
\]
If $q_1,\dots,q_d,r_1,\dots,r_d$ are not mentioned, we take them to equal $1$.
If $d=1$, we usually omit it from the subscripts. For $i=1,\dots,d$, let $\overline{t}_i \in  R_\psi$ be the element corresponding to the sequence $T_i, T_i^{1/p}, \dots$.
\end{defn}

We have the following generalizations of Lemma~\ref{L:standard toric weakly decompleting}.
\begin{lemma}
With notation as in Definition~\ref{D:standard toric tower},
the tower $\psi$ is perfectoid, and  $\tilde{R}_\psi$ is the completion of
\[
\tilde{R}_{\psi_0}\{q_1/\overline{t}_1, \dots,q_d/\overline{t}_d,
\overline{t}_1/r_1, \dots,\overline{t}_d/r_d
\}^{\perf}.
\]
\end{lemma}
\begin{proof}
Straightforward.
\end{proof}

\begin{lemma} \label{L:standard toric weakly decompleting2}
With notation as in Definition~\ref{D:standard toric tower},
the tower $\psi$ is weakly decompleting and
\[
R_\psi = R_{\psi_0}\{q_1/\overline{t}_1, \dots,q_d/\overline{t}_d,
\overline{t}_1/r_1, \dots,\overline{t}_d/r_d
\}.
\]
\end{lemma}
\begin{proof}
It is apparent that $R_{\psi_0}\{q_1/\overline{t}_1, \dots,q_d/\overline{t}_d,
\overline{t}_1/r_1, \dots,\overline{t}_d/r_d
\} \subseteq R_\psi$.
The inclusion in the other direction follows from
Lemma~\ref{L:intermediate mod p subring}. This implies both claims.
\end{proof}

\begin{defn} \label{D:toric tower}
By a \emph{relative toric tower} over a $p$-adic field $K$,
we will mean a tower obtained by starting with a standard relative toric tower and then
performing  a finite number of base extensions along morphisms
as in Definition~\ref{D:toric tower1}.
Any such tower is Galois with group $\Gamma = \ZZ_p^\times \rhd \ZZ_p^d$,
with $\ZZ_p^d$ acting as in Remark~\ref{R:standard group action} fixing $\tilde{A}_{\psi_0}$ and $\ZZ_p^\times$ acting on $\tilde{A}_{\psi_0}$
as in Definition~\ref{D:cyclotomic tower} fixing $T_i^{p^{-n}}$.
Consequently, the cohomology of $\Gamma$-modules can be computed as continuous group cohomology as per
Example~\ref{exa:Galois tower}.
\end{defn}

\begin{prop} \label{P:toric weakly decompleting}
Any relative toric tower $\psi$ over a $p$-adic field is weakly decompleting.
\end{prop}
\begin{proof}
This follows from Proposition~\ref{P:weakly decompleting tower persistence}
and Lemma~\ref{L:standard toric weakly decompleting}.
\end{proof}

\begin{remark} \label{R:toric tower}
For a standard relative toric tower, $A^+$ is noetherian (because we assumed $q_i, r_i \in p^{\QQ}$), as is $R_\psi^+$ by Lemma~\ref{L:standard toric weakly decompleting2};
also, $R_\psi$ is $F$-(finite projective).
Moreover, there exists $r_0 > 0$ such that for $r,s \in p^{\QQ}$
with $0 < s \leq r \leq r_0$, $\bA^{r,+}_{\psi}$ and $\bC^{[s,r],+}_{\psi}$ are noetherian.
It follows that for any relative toric tower, the same statements hold for some choice of $r_0 > 0$; therefore, the adic Banach rings
\[
(A, A^+), (R_\psi, R_\psi^+), (\bA^r_\psi, \bA^{r,+}_{\psi}), (\bC^{[s,r]}_\psi, \bC^{[s,r],+}_{\psi})
\]
are all strongly noetherian. (Note that it is not immediately obvious how to adapt this argument to the reified case.)
\end{remark}

\begin{lemma} \label{L:andreatta-brinon}
Any relative toric tower $\psi$ over a $p$-adic field $K$ is decompleting.
\end{lemma}
\begin{proof}
Put $\Gamma_0 = (1 + p \ZZ_p) \rhd \ZZ_p^d$.
By Proposition~\ref{P:toric weakly decompleting},
$\psi$ is weakly decompleting 
and we have a $\Gamma_0$-equivariant decomposition
\begin{equation} \label{eq:ab-decomposition}
\overline{\varphi}^{-1}(R_\psi)
\cong 
\bigoplus_{e_0,\dots,e_d=0}^{p-1}
(1 + \overline{\pi})^{e_0/p} \overline{t}_1^{e_1/p} \cdots \overline{t}_d^{e_d/p}
R_\psi
\end{equation}
of modules over $R_\psi$, which is also a decomposition of Banach modules by the open mapping theorem (Theorem~\ref{T:open mapping}).
It thus suffices to check acyclicity for each summand with $e_0,\dots,e_d$ not all zero.
Put $\overline{y} = 
(1 + \overline{\pi})^{e_0/p} \overline{t}_1^{e_1/p} \cdots \overline{t}_d^{e_d/p}$.

Suppose first that $e_i \neq 0$ for some $i>0$.
Let $\gamma$ be the canonical generator of the $i$-th copy of $\ZZ_p^d$.
As in the proof of Theorem~\ref{T:standard toric decompleting}, we see that
for $n$ sufficiently large, 
 $\gamma^{p^n}-1$ acts invertibly on $\overline{y} R_\psi$. The acyclicity property for this summand follows by Lemma~\ref{L:hochschild-serre}.

Suppose next that $e_0 \neq 0$ but $e_1 = \cdots = e_d = 0$. 
Put $\gamma = 1 +p^2 \in \ZZ_p^\times$.
Then $(\gamma^{p^n}-1)(\overline{y}) = ((1 + \overline{\pi})^{p^{n+1} m} - 1) \overline{y}$ for some integer $m$ coprime to $p$, so 
\[
\overline{\alpha}_\psi((\gamma^{p^n}-1)(\overline{y})) = p^{-p^{n+2}/(p-1)} \overline{y};
\]
meanwhile, we may imitate the proof of Lemma~\ref{L:exponential convergence gamma}
to deduce that there exists $c>0$ such that
\[
\overline{\alpha}_\psi((\gamma^{p^n}-1)(\overline{x}))
\leq c p^{-p^{n+3}/(p-1)}
\overline{\alpha}_\psi(\overline{x}) \qquad (\overline{x} \in R_\psi).
\]
Using \eqref{eq:Leibniz rule}, we see that
for $n$ sufficiently large, 
$((1 + \overline{\pi})^{p^{n+1}m}-1)^{-1} (\gamma^{p^n}-1)$
acts on $\overline{y} R_\psi$ as the identity plus an operator of norm less than 1. It follows again that $\gamma^{p^n}-1$ acts invertibly on $\overline{y} R_\psi$.
We cannot continue as in the previous case because $\Zp^\times$ is not normal in $\Gamma$; instead, we apply Theorem~\ref{T:kill analytic cohomology} to conclude.
\end{proof}

\begin{theorem} \label{T:andreatta-brinon}
Let $\psi$ be a relative toric tower over a $p$-adic field. Then the results of \S\ref{subsec:phigamma}, \S\ref{subsec:Frobenius splittings},  \S\ref{subsec:pseudocoherent type C} hold for $\psi$.
\end{theorem}
\begin{proof}
By Lemma~\ref{L:andreatta-brinon}, the tower $\psi$ is locally decompleting.
By Remark~\ref{R:toric tower},
the ring $R_\psi$ is $F$-(finite projective)  and the rings $\breve{R}_\psi, \breve{\bA}^r_\psi, \breve{\bC}^{[s,r]}_\psi$ are coherent. By Remark~\ref{R:decompleting conditions}, this proves the claim.
\end{proof}

\begin{cor} \label{C:pseudocoherent fpd local2}
For $a$ a positive integer, every pseudocoherent $\varphi^a$-module over $\tilde{\bC}_\psi$ is fpd of projective dimension at most $d+1$.
\end{cor}
\begin{proof}
Analogous to Corollary~\ref{C:pseudocoherent fpd local}.
\end{proof}

\begin{remark} \label{R:andreatta-brinon}
The towers considered by Andreatta-Brinon \cite{andreatta-brinon} are those obtained by
starting with a standard relative toric tower in which $q_1 = \cdots = q_d= r_1 = \cdots = r_d= 1$ and then performing a sequence of operations of types (ii)--(v) in Definition~\ref{D:toric tower1}.
In that setting, the statement of Theorem~\ref{T:perfect equivalence1}
is \cite[Th\'eor\`eme~4.35]{andreatta-brinon} while the other statements
of \S\ref{subsec:phigamma}, \S\ref{subsec:Frobenius splittings},
\S\ref{subsec:pseudocoherent type C} are new.
\end{remark}

\subsection{Logarithmic structures on toric towers}
\label{subsec:logarithmic toric}

We next consider some modified constructions in which we allow some limited ramification within the towers themselves, starting with the nonrelative case. To simplify matters slightly, we assume that the base ring contains the $p$-cyclotomic extension; this restriction is likely to be harmless in applications. We also limit ourselves to projective modules rather than pseudocoherent modules.

\begin{defn} \label{D:standard ramified toric tower}
Let $K$ be the completion of $\Qp(\mu_{p^\infty})$.
Let $(A_0, A_0^+)$ be an adic perfectoid algebra over $K$
corresponding to the perfect uniform adic Banach ring $(R_0,R_0^+)$ via
Theorem~\ref{T:Fontaine perfectoid correspondence}.
For $d\geq a\geq 0$, and positive real numbers $q_1, \dots, q_a, r_1,\dots,r_d$ with $q_i\leq r_i$, the \emph{standard ramified toric tower} with parameters $q_1,\dots, q_a, r_1,\dots,r_d$ is the tower $\psi$ with 
\begin{align*}
A &= A_0\{q_1/T_1, \dots, q_a/ T_a, T_1/r_1,\dots,T_d/r_d\}, \\
A^+ &= A_0^+\{q_1/T_1, \dots, q_a/ T_a, T_1/r_1,\dots,T_d/r_d\}, \\
A_{\psi,n} &= A_0 \{q_1^{p^{-n}}/T_1^{p^{-n}}, \dots, q_a^{p^{-n}} / T_a^{p^{-n}}, T_1^{p^{-n}}/r_1^{p^{-n}},\dots,T_d^{p^{-n}}/r_d^{p^{-n}}\}, \\
A_{\psi,n}^+ &= A_0^+ \{q_1^{p^{-n}}/T_1^{p^{-n}}, \dots, q_a^{p^{-n}} / T_a^{p^{-n}}, T_1^{p^{-n}}/r_1^{p^{-n}},\dots,T_d^{p^{-n}}/r_d^{p^{-n}}  \}.
\end{align*}
If $q_1, \dots, q_a, r_1,\dots,r_d$ are not mentioned, we take them to equal 1.
For $i=1,\dots,d$, let $\overline{t}_i \in  R_\psi$ be the element corresponding to the sequence $T_i, T_i^{1/p}, \dots$.
\end{defn}

\begin{lemma}
With notation as in Definition~\ref{D:standard ramified toric tower},
the tower $\psi$ is perfectoid, and $\tilde{R}_\psi$ is the completion of
\[
R_0\{q_1/\overline{t}_1, \dots, q_a/\overline{t}_a, \overline{t}_1/r_1, \dots,\overline{t}_d/r_d
\}^{\perf}.
\]
\end{lemma}
\begin{proof}
Straightforward.
\end{proof}

\begin{lemma} \label{L:standard ramified toric weakly decompleting2}
With notation as in Definition~\ref{D:standard ramified toric tower},
the tower $\psi_0$ is weakly decompleting and
\[
R_\psi = R_{\psi_0}\{
q_1/\overline{t}_1, \dots, q_a/\overline{t}_a, \overline{t}_1/r_1, \dots,\overline{t}_d/r_d
\}.
\]
\end{lemma}
\begin{proof}
Analogous to Lemma~\ref{L:standard toric weakly decompleting2}.
\end{proof}

\begin{hypothesis} \label{H:ramified1}
For the remainder of \S\ref{subsec:logarithmic toric}, let $\psi$ be a tower derived from a standard ramified toric tower by a sequence of operations of the forms (i)--(v) of Definition~\ref{D:toric tower1}.
By Proposition~\ref{P:weakly decompleting tower persistence}
and
Lemma~\ref{L:standard ramified toric weakly decompleting2},
$\psi$ is weakly decompleting, and $R_\psi$ is $F$-(finite projective) with $p$-basis
$\{\overline{t}_1^{-1},\dots,\overline{t}_a^{-1}, \overline{t}_1,\dots,\overline{t}_d\}$.
\end{hypothesis}

Since $\psi$ is not a finite \'etale tower, we do not have an \emph{a priori} definition of $\Gamma$-modules and $(\varphi, \Gamma)$-modules over $*_\psi$. We instead use the following \emph{ad hoc} definition.

\begin{defn}
Define the group $\Gamma = \ZZ_p^d$; it acts on $\psi$
as in Remark~\ref{R:standard group action}, although the tower is no longer Galois.
Nonetheless, we may still define a \emph{$\Gamma$-module} (resp.\ \emph{$(\varphi, \Gamma)$-module}) over $*_\psi$ to be a finite projective module (resp.\ $\varphi$-module) over $*_\psi$ equipped with a continuous semilinear action of $\Gamma$.
Put $\Gamma_0 = \Gamma$ and $\Gamma_n = p^n \ZZ_p^d$
for $n>0$.
\end{defn}

\begin{lemma}
For each nonnegative integer $n$, the complex $C^\bullet_{\cont}(\Gamma_n, \overline{\varphi}^{-1}(R_\psi)/R_\psi))$ is strict exact.
\end{lemma}
\begin{proof}
The decomposition \eqref{eq:ab-decomposition1} from the proof of Lemma~\ref{L:standard toric decompleting} is valid in this context and $\Gamma_n$-equivariant for $n>0$, so the remainder of that proof carries over.
\end{proof}

\begin{cor} \label{C:ramified Gamma equivalences}
The statements of
Lemma~\ref{L:Gamma quasi-isomorphism} and
Lemma~\ref{L:Gamma equivalences}
hold in this setting.
\end{cor}

\subsection{Logarithmic structures on relative toric towers}
\label{subsec:logarithmic relative toric}

We now redo the logarithmic construction in the setting of relative toric towers. Again, we restrict attention from pseudocoherent modules to projective modules for simplicity.

\begin{defn} \label{D:ramified relative toric tower}
Let $K$ be a $p$-adic field. Let $\psi_0$ denote the cyclotomic tower over $K$. For $d \geq a\geq  0$ and $q_1,\dots,q_a,r_1,\dots,r_d \in p^{\QQ}$ with $q_i \leq r_i$, the \emph{standard ramified relative toric tower} over $K$ with parameters $q_1,\dots,q_a, r_1,\dots,r_d$
is the tower 
$\psi$ with 
\begin{align*}
A &= K\{q_1/T_1, \dots, q_a/T_a, T_1/r_1, \dots, T_d/r_d\} \\
A^+ &= \gotho_K\{q_1/T_1, \dots, q_a/T_a, T_1/r_1, \dots, T_d/r_d\} \\
A_{\psi,n} &= (K \otimes_{\QQ_p} \QQ_p(\mu_{p^n})) \{q_1^{p^{-n}}/T_1^{p^{-n}}, \dots, q_a^{p^{-n}} / T_a^{p^{-n}}, T_1^{p^{-n}}/r_1^{p^{-n}},\dots,T_d^{p^{-n}}/r_d^{p^{-n}} \} \\
A_{\psi,n}^+ &= (\gotho_K \otimes_{\ZZ_p} \ZZ_p[\mu_{p^n}])\{q_1^{p^{-n}}/T_1^{p^{-n}}, \dots, q_a^{p^{-n}} / T_a^{p^{-n}}, T_1^{p^{-n}}/r_1^{p^{-n}},\dots,T_d^{p^{-n}}/r_d^{p^{-n}} \}.
\end{align*}
If $q_1,\dots,q_a,r_1,\dots,r_d$ are not mentioned, we take them to equal $1$.
If $d=1$, we usually omit it from the subscripts. For $i=1,\dots,d$, let $\overline{t}_i \in  R_\psi$ be the element corresponding to the sequence $T_i, T_i^{1/p}, \dots$.
\end{defn}

\begin{lemma}
With notation as in Definition~\ref{D:ramified relative toric tower},
the tower $\psi$ is perfectoid, and $\tilde{R}_\psi$ is the completion of
\[
\tilde{R}_{\psi_0}\{q_1/\overline{t}_1, \dots, q_a/\overline{t}_a, \overline{t}_1/r_1, \dots,\overline{t}_d/r_d\}^{\perf}.
\]
\end{lemma}
\begin{proof}
Straightforward.
\end{proof}

\begin{lemma} \label{L:standard ramified relative toric weakly decompleting2}
With notation as in Definition~\ref{D:standard ramified toric tower},
the tower $\psi$ is weakly decompleting and
\[
R_\psi = R_{\psi_0}\{q_1/\overline{t}_1, \dots, q_a/\overline{t}_a,
\overline{t}_1/r_1, \dots,\overline{t}_d/r_d
\}.
\]
\end{lemma}
\begin{proof}
Analogous to Lemma~\ref{L:standard toric weakly decompleting2}.
\end{proof}

\begin{hypothesis} \label{H:ramified2}
For the remainder of \S\ref{subsec:logarithmic relative toric}, let $\psi$ be a tower derived from a standard ramified relative toric tower by a sequence of operations of the forms (i)--(v) of Definition~\ref{D:toric tower1}.
By Proposition~\ref{P:weakly decompleting tower persistence}
and
Lemma~\ref{L:standard ramified relative toric weakly decompleting2},
$\psi$ is weakly decompleting, and $R_\psi$ is $F$-(finite projective) with $p$-basis
$\{\overline{\pi}, \overline{t}_1^{-1}, \dots, \overline{t}_a^{-1}, \overline{t}_1,\dots,\overline{t}_d\}$.
\end{hypothesis}

\begin{defn}
Define the group $\Gamma = \ZZ_p^\times \rhd \ZZ_p^d$; it acts on $\psi$
as in Definition~\ref{D:toric tower}, although the tower is no longer Galois.
Nonetheless, we may still define a \emph{$\Gamma$-module} (resp.\ \emph{$(\varphi, \Gamma)$-module}) over $*_\psi$ to be a finite projective module (resp.\ $\varphi$-module) over $*_\psi$ equipped with a continuous semilinear action of $\Gamma$.
Put $\Gamma_0 = \Gamma$ and $\Gamma_n = (1 + p^n \ZZ_p)^\times \rhd p^n \ZZ_p^d$
for $n>0$.
\end{defn}

\begin{lemma}
For each nonnegative integer $n$, the complex $C^\bullet_{\cont}(\Gamma_n, \overline{\varphi}^{-1}(R_\psi)/R_\psi))$ is strict exact.
\end{lemma}
\begin{proof}
The decomposition \eqref{eq:ab-decomposition} from the proof of Lemma~\ref{L:andreatta-brinon} is valid in this context and $\Gamma_n$-equivariant for $n>0$, so the remainder of that proof carries over.
\end{proof}

\begin{cor} \label{C:ramified Gamma equivalences2}
The statements of
Lemma~\ref{L:Gamma quasi-isomorphism} and
Lemma~\ref{L:Gamma equivalences}
hold in this setting.
\end{cor}

\subsection{Ramified towers and \texorpdfstring{$(\varphi, \Gamma)$}{(phi, Gamma)}-modules}
\label{subsec:ramified phi-gamma}

We now give the main statements about $(\varphi, \Gamma)$-modules for nonrelative and relative ramified toric towers.

\begin{hypothesis}
Throughout \S\ref{subsec:ramified phi-gamma}, let $\psi$ be a tower as in either
Hypothesis~\ref{H:ramified1} or Hypothesis~\ref{H:ramified2}.
Let $a$ be a fixed positive integer.
\end{hypothesis}

\begin{theorem} \label{T:ramified perfect equivalence1}
The following statements hold.
\begin{enumerate}
\item[(a)]
The exact tensor categories of $(\varphi^a, \Gamma)$-modules over the rings in the diagram
\[
\xymatrix@R=30pt@!C=60pt{
 \bA^{\dagger}_\psi \ar[r] \ar[d]  & \breve{\bA}^\dagger_\psi \ar[d] \ar[r]
 &  \hat{\bA}^\dagger_\psi \ar[d] \ar[r] &
 \tilde{\bA}^\dagger_\psi \ar[d] \\
 \bA_\psi \ar[r] & \breve{\bA}_\psi \ar[r] & 
 \hat{\bA}_\psi \ar[r] &\tilde{\bA}_\psi
}
\]
are equivalent via the apparent base change functors.
\item[(b)]
The exact tensor categories of globally \'etale $(\varphi^a, \Gamma)$-modules over the rings in the diagram
\[
\xymatrix@R=30pt@!C=60pt{
\bC_\psi \ar[r] & \breve{\bC}_\psi \ar[rr] & & \tilde{\bC}_\psi \\
 \bB^{\dagger}_\psi \ar[d] \ar[r] \ar[u] & \breve{\bB}^\dagger_\psi \ar[d] \ar[r]
 \ar[u] &  \hat{\bB}^\dagger_\psi \ar[d] \ar[r] & \tilde{\bB}^\dagger_\psi \ar[d] \ar[u]\\
 \bB_\psi \ar[r] & \breve{\bB}_\psi \ar[r] & \hat{\bB}_\psi \ar[r] & \tilde{\bB}_\psi
}
\]
are equivalent via the apparent base change functors.
\item[(c)]
The exact tensor categories of $(\varphi^a, \Gamma)$-modules
over $\bC_\psi$, $\breve{\bC}_\psi$, and $\tilde{\bC}_\psi$ are equivalent via the apparent base change functors.
\item[(d)]
Within each of (a)--(c), the $(\varphi^a, \Gamma)$-cohomology of a projective (resp.\ pseudocoherent) $(\varphi^a, \Gamma)$-module is invariant under the apparent base change functors.
\end{enumerate}
\end{theorem}
\begin{proof}
This follows from Corollary~\ref{C:ramified Gamma equivalences}
as in the proofs of Theorems~\ref{T:perfect equivalence1},
\ref{T:perfect equivalence1a},
 \ref{T:add tilde1}, \ref{T:Galois cohomology1}, and \ref{T:Galois cohomology1c}.
\end{proof}

What is notably absent from Theorem~\ref{T:ramified perfect equivalence1} is the relationship between $(\varphi, \Gamma)$-modules and local systems. We address this next.
\begin{defn}
We say that a morphism $X \to \Spa(A_0,A_0^{+})$ is (faithfully) (finite) \emph{Kummer-\'etale} if it is (faithfully) (finite) flat and becomes (faithfully) (finite) \'etale after pullback along the map
$T_1 \mapsto T_1^{p^n}, \dots, T_d \mapsto T_d^{p^n}$ for some nonnegative integer $n$. By analogy with the pro-\'etale topology of \cite[\S 9.1]{part1}, we may define the \emph{pro-Kummer-\'etale topology} of $\Spa(A_0, A_0^+)$.
\end{defn}

\begin{theorem}
The categories of \'etale $(\varphi^a, \Gamma)$-modules over $\tilde{\bA}_\psi$,
globally \'etale $(\varphi^a, \Gamma)$-modules over $\tilde{\bB}_\psi$,
and \'etale $(\varphi^a, \Gamma)$-modules over $\tilde{\bC}_\psi$ are respectively equivalent to the categories of $\ZZ_{p^a}$-local systems, isogeny $\ZZ_{p^a}$-local systems, and $\QQ_{p^a}$-local systems on $\Spa(A_0,A_0^{+})$ for the  
pro-Kummer-\'etale topology. Moreover, $(\varphi^a, \Gamma)$-cohomology corresponds to pro-Kummer-\'etale cohomology of local systems.
\end{theorem}
\begin{proof}
By Theorem~\ref{T:pseudocoherent type A}
Theorem~\ref{T:etale type C},
and Corollary~\ref{C:globally etale BC}, the given statement would hold if we replaced
$(\varphi^a, \Gamma)$-modules with $\varphi^a$-modules and $\Spa(A_0,A_0^+)$
with $\Spa(\tilde{A}_{\psi}, \tilde{A}_{\psi}^+)$.
By \cite[Proposition~2.6.8]{part1}, the functor 
$\FEt(A_\psi) \to \FEt(\tilde{A}_\psi)$
is an equivalence; this proves the original statements about $\ZZ_{p^a}$-local systems and isogeny $\ZZ_{p^a}$-local systems. For the statement about $\QQ_{p^a}$-local systems, we must additionally invoke \cite[Lemma~8.4.2]{part1} to see that \'etale $\QQ_{p^a}$-local systems locally descend to isogeny $\ZZ_{p^a}$-local systems; we may then deduce this case (for all $\psi$ at once) from the previous cases (for all $\psi$ at once).
\end{proof}

One can also use this logarithmic construction to study the category of relative $(\varphi, \Gamma)$-modules over the base space, by building a suitable logarithmic condition into the definition of a $(\varphi, \Gamma)$-module.

\begin{defn}
We say a pseudocoherent $(\varphi^a, \Gamma)$-module $M$ over $*_\psi$ is \emph{effective}
(or that the action of $\Gamma$ on $M$ is \emph{effective}) if
for $i =1,\dots,d$, the $i$-th factor of $\ZZ_p$ in $\Gamma$ acts trivially on $M/T_i M$. 
Similarly, a $\Gamma$-cochain $f: \Gamma^{n+1} \to M$ is \emph{effective} 
if for $i=1,\dots,d$, the composite $\Gamma^{n+1} \to M/T_i M$ factors through 
the quotient of $\Gamma$ by the $i$-th factor of $\ZZ_p$.
\end{defn}

\begin{theorem} \label{T:functor to log}
Put $X = \Spa(A_{\psi,0}, A_{\psi,0}^+)$
and choose $* \in \{\bA, \bA^\dagger, \bC\}$.
\begin{enumerate}
\item[(a)]
There is a equivalence of exact tensor categories between $\varphi^a$-modules over $\tilde{*}_X$ and effective $(\varphi^a, \Gamma)$-modules over $*_\psi$.
\item[(b)]
In (a), the sheaf cohomology of a 
$\varphi^a$-module over $\tilde{*}_X$ is canonically identified with the $\varphi^a$-(hyper)cohomology of the complex of continuous effective $\Gamma$-cochains with values in the corresponding $(\varphi^a, \Gamma)$-module over $*_\psi$.
\end{enumerate}
\end{theorem}
\begin{proof}
By Theorem~\ref{T:ramified perfect equivalence1}, it suffices to check the analogous statements with $*_\psi$ replaced by $\tilde{*}_\psi$ everywhere.
Part (a) follows from Theorem~\ref{T:vector bundle faithful descent} via the usual splitting argument (Definition~\ref{D:Robba fractional}).
Part (d) follows similarly from Theorem~\ref{T:faithful acyclic}.
\end{proof}

\begin{remark} \label{R:projective space using ramified}
As an exercise, we suggest using the aforementioned results to compute the cohomology
of the sheaves $\widehat{\calO}(d)$ on the pro-\'etale site of a projective space over a perfectoid field, using the obvious coordinates.
This can also be done without ramified towers (see Remark~\ref{R:cover smooth by toric} for the key idea), but the calculation becomes somewhat more difficult;
see Example~\ref{exa:bad cohomology over perfectoid} for an example.
\end{remark}

\begin{remark}
The implicit role of affine $d$-space in the preceding constructions can also be assumed by a more general (possibly singular) affine toric variety. We omit further details.
\end{remark}

\subsection{Application to semistable comparison isomorphisms}

As a sample illustration of relative toric towers, we derive a statement used in recent work of Colmez and Nizio\l\ on semistable comparison isomorphisms for rigid analytic spaces
\cite{colmez-niziol}. Note that this example does not require rings of type $\bA^\dagger$
as in Cherbonnier-Colmez \cite{cherbonnier-colmez} or Andreatta-Brinon \cite{andreatta-brinon}, but only rings of type $\bA$ as in Andreatta-Iovi\cb{t}\u{a} \cite{andreatta-iovita};
however, the arguments do ultimately rely on the perfectoid correspondence in a stronger form than that provided by the original almost purity theorem of Faltings.

\begin{example} \label{exa:colmez-niziol}
Let $K$ be a $p$-adic field, and choose a uniformizer $\pi_K$ of $\gotho_K$.
Put
\[
A^+ = \gotho_K\{T_1^{\pm},\dots,T_a^{\pm}, T_{a+1},\dots,T_{d}\}/(T_{a+1} \cdots T_{a+b} - \pi_K^h), \qquad A= A^+[\pi_K^{-1}]
\]
for some nonnegative integers $a,b,d,h$ with $d \geq a+b$.
Let $B^+$ be the completion of an \'etale algebra over $A^+$ and put $B = B^+[\pi_K^{-1}]$. Then $(B,B^+)$ occurs as the base of a ramified relative toric tower $\psi$ over $K$.
Let $\overline{B}$ be the maximal subring of an absolute integral closure of $B$
which is \'etale over $B[(T_{a+b+1} \cdots T_d)^{-1}]$
(or equivalently by Abhyankar's lemma, which becomes \'etale after base extension from $B$ to $B[T_{a+b+1}^{1/m}, \dots, T_d^{1/m}: m >0]$ followed by flattening).
This ring is perfectoid, so we may apply Theorem~\ref{T:Fontaine perfectoid correspondence} to produce a corresponding perfect uniform Banach ring $S$.
\end{example}

\begin{theorem} \label{T:colmez-niziol}
With notation as in Example~\ref{exa:colmez-niziol},
for $r \in \ZZ$, the natural morphisms
\[
\RR\Gamma_{\cont}(\Gamma, \bA_\psi(r))
\to \RR\Gamma_{\cont}(\Gamma, \tilde{\bA}_\psi(r))
\to \RR\Gamma_{\cont}(\Aut(\overline{B}/B_{\psi,\infty}), \tilde{\bA}_{S}(r))
\]
are quasi-isomorphisms.
\end{theorem}
\begin{proof}
To check the first quasi-isomorphism, we may replace $\bA_\psi$ with $\breve{\bA}_\psi$; the claim then follows from Corollary~\ref{C:ramified Gamma equivalences2}
(specifically its extension of Lemma~\ref{L:Gamma quasi-isomorphism}).
The second claim follows from the comparison of cohomology with hypercohomology in Theorem~\ref{T:pseudocoherent type A}, using the v-topology
on $\Spa(\tilde{A}_\psi, \tilde{A}_\psi^+)$.
\end{proof}

\begin{remark} \label{R:colmez-niziol}
Theorem~\ref{T:colmez-niziol} is used in \cite[Proposition~2.20(ii)]{colmez-niziol} as part of an extended calculation of syntomic cohomology groups. This, combined with Scholze's proof of the finite-dimensionality of \'etale cohomology for smooth proper rigid spaces
\cite{scholze2}, allows Colmez and Nizio\l\ to recover the comparison isomorphism for proper semistable formal schemes over $\gotho_K$ \cite[Corollary~1.5]{colmez-niziol}.
\end{remark}

\begin{remark}
The case $h=0$ of Theorem~\ref{T:colmez-niziol} is also covered by the prior results of 
Andreatta-Iovi\cb{t}\u{a} \cite{andreatta-iovita}. (Strictly speaking, this is only true if one also assumes $d = a+b$ to avoid logarithmic structures, but their methods apply equally well in the logarithmic case.) However, the case $h>0$ is crucial for the intended application described in Remark~\ref{R:colmez-niziol}; this example, together with Scholze's comparison isomorphism \cite{scholze2}, can be taken as strong evidence for the importance in $p$-adic Hodge theory of the generalization of the original almost purity theorem of Faltings to the context of perfectoid rings.
\end{remark}

\section{Applications to pseudocoherent sheaves}
\label{sec:applications}

Using imperfect period rings associated to (nonrelative) toric towers, we establish a number of key properties of pseudocoherent sheaves of $\widehat{\calO}$-modules 
and pseudocoherent $(\varphi, \Gamma)$-modules on rigid analytic spaces over an arbitrary (and in particular, not necessarily discretely valued) analytic field of mixed characteristics. 
Consequences for the cohomology of pro-\'etale local systems will be discussed in a subsequent paper.
(Similar arguments can be made in the equal characteristic case, but some care needs to be taken to avoid the use of resolution of singularities and to handle inseparable base field extensions.) 

\setcounter{theorem}{0}

\begin{hypothesis}
Throughout \S\ref{sec:applications}, let $K$ be an analytic field containing $\Qp$,
let $X$ be a rigid analytic space over $K$ (viewed as an adic space locally of tft over $K$), and let $a$ be a positive integer.
Let $\nu_{\proet}: X_{\proet} \to X$ be the canonical morphism.
\end{hypothesis}

\begin{remark}
It is not difficult to carry over all of the following results in the context of reified adic spaces. However, we will generally be free to replace $K$ with a conveniently large overfield (e.g., a perfectoid or even algebraically closed field); consequently, the reified results can all be formally deduced from the context of standard adic spaces. We thus omit any further mention of reified spaces in what follows.
\end{remark}

\subsection{Geometric arguments}

We begin with some key geometric arguments about analytic spaces.
\begin{remark} \label{R:cover smooth by toric}
For any adic Banach algebra $(A,A^+)$, the disc $\Spa(A\{T\}, A^+\{T\})$ is covered by the annuli $\Spa(A\{T^{\pm}\}, A^+\{T^{\pm}\})$ and $\Spa(A\{(T-1)^{\pm}\}, A^+\{(T-1)^{\pm}\})$. By induction on dimension, we see that the closed unit polydisc 
over $(A,A^+)$
\[
\Spa(A\{T_1,\dots,T_n\}, A^+\{T_1,\dots,T_n\})
\]
is covered by finitely many copies of the closed unit polyannulus over $(A,A^+)$
\[
\Spa(A\{T_1^{\pm}, \dots, T_n^{\pm}\}, A^+\{T_1^{\pm}, \dots, T_n^{\pm}\}).
\]
A related observation is that the map $(T_1, \dots, T_n) \mapsto (1+pT_1, \dots, 1+pT_n)$
identifies the polydisc with the rational subdomain
of the polyannulus defined by the conditions
\[
v(T_1 - 1) \leq v(p), \dots, v(T_n-1) \leq v(p).
\]
Either construction has the following consequence.
Let $Y \to X$ be a smooth morphism of uniform adic spaces. Then there exist a covering $\{U_i\}_i$ of $X$ by rational subspaces
and, for each $i$, a covering $\{V_{i,j}\}_j$ of $Y \times_X U_i$ by rational subspaces each of which is \'etale over some closed unit polyannulus over $U_i$. In particular, for any perfectoid subdomain $W = \Spa(B,B^+)$ of $U_{i,\proet}$, $Y \times_X W$
can be covered by the base spaces of finitely many restricted toric towers over $W$.
We may thus use Theorem~\ref{T:standard toric decompleting} to apply the results of 
\S\ref{subsec:phigamma} to various questions about smooth morphisms of adic spaces.
Of course one could also use ramified towers as in \S\ref{subsec:logarithmic toric}
 to obtain similar results, but avoiding ramified towers leads to some technical simplifications in the ensuing arguments.
\end{remark}

\begin{remark} \label{R:resolution}
Suppose that $X = \Spa(A,A^+)$ where $A$ is a classical affinoid algebra over $K$ (i.e., a quotient of $K\{T_1,\dots,T_n\}$ for some $n$) which is reduced. 
The ring $A$ is known to be excellent
\cite[Theorem~1.1.3]{conrad-irreducible}. We may thus apply Temkin's desingularization theorem for quasiexcellent $\QQ$-schemes \cite{temkin-resolution1} to produce a projective birational morphism $f_0: Y_0 \to \Spec(A)$ of schemes which analytifies to a morphism $f: Y \to X$ of rigid analytic spaces with $Y$ smooth. In particular, $f_0$ is a covering in the $h$-topology (Definition~\ref{D:h-topology}) and the pullback of $f$ to any perfectoid space is a covering in the v-topology
(Definition~\ref{D:v-topology}). 
Moreover, the construction satisfies certain functoriality properties; for example, the open subspace of $X$ over which $f$ is an isomorphism is precisely the subspace at which $X$ is smooth.
\end{remark}

We also need an analogue of Remark~\ref{R:resolution} for embedded resolution of singularities. 
\begin{remark} \label{R:resolution2}
Let $Z \hookrightarrow X$ be a closed immersion of rigid analytic spaces and suppose that $X = \Spa(A,A^+)$ is a smooth affinoid space. Temkin has shown \cite{temkin-resolution2} that there exists a projective birational morphism $f_0: Y_0 \to \Spec(A)$
(actually a composition of blowups along regular centers)
 which analytifies to a morphism $f: Y \to X$ of rigid analytic spaces with $Y$ smooth and $f^{-1}(Z)$ a normal crossings divisor. Again, the construction satisfies certain functoriality properties; for example, the open subspace of $X$ over which $f$ is an isomorphism is precisely the subspace at which $Z$ is a normal crossings divisor.
\end{remark}

\begin{remark} \label{R:inverse image ideal sheaf}
Suppose that $X$ is seminormal in the sense of \S\ref{subsec:seminormality adic}, and set notation as in Remark~\ref{R:resolution}.
By Proposition~\ref{P:h-topology seminormal} plus rigid GAGA \cite[Example~3.2.6]{conrad},
the map $\calO_X \to f_* \calO_Y$ is an isomorphism.

With the same notation, let  $\calI$ be an ideal subsheaf of $\calO_X$ and let $\calJ$ be the inverse image ideal subsheaf of $\calO_Y$ (i.e., the image of $f^* \calI$ in $\calO_Y$). Then $f_* \calJ$ is an ideal subsheaf of $\calO_X$ containing $\calI$, but it need not equal $\calI$; all that one can say in general is that $\calI$ and $f_* \calJ$ have the same integral closure. In particular, they have the same support, and moreover there exists an integer $m$ such that $(f_* \calJ)^m \subseteq \calI$.
\end{remark}

\begin{remark} \label{R:not closed immersion}
Let $X' \to X$ be a closed immersion and let $Y$ be a perfectoid subdomain of $X_{\proet}$. Then the corresponding perfectoid subdomain $Y'$ of $X'_{\proet}$ is not $Y \times_X X'$ but rather the uniformization thereof, which is indeed perfectoid by Theorem~\ref{T:Fontaine perfectoid compatibility}(b). 
Consequently, it is not immediate that $Y' \to Y$ is a closed immersion, but this does in fact
hold because of Theorem~\ref{T:uniform closure}.
\end{remark}

To indicate the extent to which pro-\'etale cohomology fails to obey the intuition associated with coherent cohomology, we offer an example of the failure of the GAGA principle in this context.
\begin{example} \label{exa:bad cohomology over perfectoid}
Let $K$ be a perfectoid field containing $\Qp(\mu_{p^\infty})$. Put $A = K\{T^{\pm}\}$ and let $\psi$ be the standard toric cover over $A$.
As in Example~\ref{exa:Galois tower}, we may identify
$H^i(\Spa(A,A^{\circ})_{\proet}, \widehat{\calO})$ with $H^i_{\cont}(\Zp, \tilde{A}_\psi)$;
one thus computes 
\[
H^i(\Spa(A,A^\circ)_{\proet}, \widehat{\calO}) = \begin{cases} A & i=0,1 \\ 0 & i>1.
\end{cases}
\]
Moreover, there is a distinguished isomorphism between the $H^0$ and $H^1$ groups. 

Next, put $A = K\{T\}$. We may compute 
$H^i(\Spa(A,A^\circ)_{\proet}, \widehat{\calO})$
either using a ramified toric tower, or by the following argument (a form of
``Abhyankar's trick''). The morphism
\[
K\{T\} \to K\{T^{\pm}\}, \qquad
T \mapsto T^p - T^{-1}
\]
expresses $\Spa(K\{T,T^{-1}\}, K\{T,T^{-1}\}^\circ)$ as a finite \'etale cover of $\Spa(K\{T\},K\{T\}^\circ)$.
By viewing $K\{T,T^{-1}\}$ as a finite projective module $M$ over $K\{T\}$ in this fashion and splitting the inclusion $K\{T\} \to M$, we may deduce again that
\[
H^i(\Spa(A,A^\circ)_{\proet}, \widehat{\calO}) = \begin{cases} A & i=0,1 \\ 0 & i>1
\end{cases}
\]
with a distinguished isomorphism between $H^0$ and $H^1$.

Finally, put $X = \PP^1_K$. By covering $X$ with $\Spa(K\{T\},K\{T\}^\circ)$ and $\Spa(K\{T^{-1}\},K\{T^{-1}\}^\circ)$ and using the above considerations, we compute that
\[
H^i(X_{\proet}, \widehat{\calO}) = \begin{cases} K & i=0,2 \\ 0 & i=1, i>1.
\end{cases}
\]
In particular, the map $H^2(X, \calO) \to H^2(X_{\proet}, \widehat{\calO})$ is not an isomorphism.
\end{example}

\begin{remark} \label{R:HT spectral sequence}
David Hansen has observed that Example~\ref{exa:bad cohomology over perfectoid} admits a more conceptual explanation. 
Let $K$ be a perfectoid field containing $\Qp(\mu_{p^\infty})$ and let $X$ be any smooth adic space over $K$. Then there exists a Hodge-Tate spectral sequence
\[
E_2^{ij} = H^i(X, \Omega^j_X)(-j) \Longrightarrow H^{i+j}(X_{\proet}, \widehat{\calO}_X)
\]
arising from a canonical identification $R^j \nu_{\proet *}\widehat{\calO}_X \cong \Omega^j_X(-j)$ (obtained, for example, 
by reducing to the case where $K$ is algebraically closed,
then applying \cite[Proposition~3.23]{scholze-cdm}).
This plus acyclicity explains the first two calculations.

To explain the third calculation, one may use the fact that the Hodge-Tate spectral sequence degenerates at $E_2$ if $X$ is proper. When $X$ is the analytification of a scheme over $K$, as in the case $X= \PP^1_K$, this degeneration may be deduced from the corresponding degeneration for schemes using the GAGA principle. 
If $X$ descends to a discretely valued subfield of $K$, degeneration follows from the Galois-equivariance of the differentials together with the vanishing of $\CC_K(i)^{G_K}$ for any discretely valued extension $K$ of $\QQ_p$ and any $i \neq  0$.
The general case can be reduced to the discretely valued case using a spreading-out argument of Conrad--Gabber (in preparation).
\end{remark}

We also include a genuinely analytic example.
\begin{example}
Suppose that $K$ is algebraically closed. Let
\[
\pi: \GG_{m,K} \to \GG_{m,K}/q^\ZZ \cong E
\]
be the Tate uniformization of an elliptic curve $E$ over $K$ with multiplicative reduction.
For $\lambda \in \QQ_p^\times$, we obtain a rank 1 $\QQ_p$-local system $V$ on $E$ by equipping the trivial local system on $\GG_{m,K}$ with the translation action scaled by $\lambda$. Let $\calF$ be the associated $(\varphi, \Gamma)$-module over $\tilde{\bC}^{[s,r]}_E$, which can itself by obtained by equipping the trivial $(\varphi, \Gamma)$-module on $\GG_{m,K}$ with the translation action scaled by $\lambda$.

We compute the cohomology of the trivial local system on $\GG_{m,K}$. For $n \geq 0$, let $\GG_{m,K,n}$ be a copy of $\GG_{m,K}$ mapping to the original copy via the $p^n$-power map. The spaces $\GG_{m,K,n}$ then form a toric tower $\psi$ over $\GG_{m,K}$; let $\tilde{\GG}_{m,K}$ denote the limit of this tower, which is a perfectoid space.
Using this cover, we compute the $(\varphi, \Gamma)$-cohomology of the terms in the exact sequence
\[
0 \to \tilde{\bC}^{[s,r]} \to t_\theta^{-1} \tilde{\bC}^{[s,r]} \to t_\theta^{-1} \tilde{\bC}^{[s,r]}/\tilde{\bC}^{[s,r]} \to 0.
\]
We may identify the elements of $H^0(\tilde{\GG}_{m,K}, \tilde{\bC}^{[s,r]})$ with certain series
$\sum_{n \in \ZZ[p^{-1}]} c_n T^n$ with coefficients in $\tilde{\bC}^{[s,r]}_K$.
If we restrict to the subset of these series with zero constant term, then $\gamma-1$ acts invertibly for any nontrivial $\gamma \in \ZZ_p$. This yields
\[
H^i_{\varphi,\Gamma}(\GG_{m,K,\proet}, t_\theta^{-1} \tilde{\bC}^{[s,r]}) =
\begin{cases}
H^0_{\varphi,\Gamma}(\Spa(K,K^+), t_\theta^{-1} \tilde{\bC}^{[s,r]}) & i=0,1  \\
0 & i>1;
\end{cases}
\]
note that we have an exact sequence
\[
0 \to \QQ_p \to H^0_{\varphi,\Gamma}(\Spa(K,K^+), t_\theta^{-1} \tilde{\bC}^{[s,r]}) \to K \to 0.
\]
Meanwhile, note that $\varphi-1$ acts surjectively on $H^0(\tilde{\GG}_{m,K},  t_\theta^{-1} \tilde{\bC}^{[s,r]}/\tilde{\bC}^{[s,r]})$; we thus have
\[
H^i_{\varphi,\Gamma}(\GG_{m,K,\proet}, t_\theta^{-1} \tilde{\bC}^{[s,r]}/\tilde{\bC}^{[s,r]}) =
\begin{cases}
H^0(\GG_{m,K}, \calO) & i=0,1 \\ 0 & i>1.
\end{cases}
\]
Putting this together, we obtain the following.
\begin{itemize}
\item
We have $H^0(\GG_{m,K,\proet}, \QQ_p)= \QQ_p$.
\item
We have an exact sequence
\[
0 \to H^0(\GG_{m,K}, \calO) / K \to H^1(\GG_{m,K,\proet}, \QQ_p) \to \QQ_p \to 0.
\]
This agrees with a calculation of Colmez--Dospinescu--Nizio\l\ \cite{colmez-dospinescu-niziol}
made in a somewhat different fashion: they interpret $H^0(\GG_{m,K}, \calO) \to H^1(\GG_{m,K,\proet}, \QQ_p)$ as the exponential map followed by the Kummer map.

\item
We have $H^2(\GG_{m,K,\proet}, \QQ_p) \cong H^0(\GG_{m,K}, \calO) / K$.
\item
We have $H^i(\GG_{m,K,\proet}, \QQ_p) = 0$ for $i>2$.
\end{itemize}
We may now compute $H^i(E_{\proet}, V)$ using a spectral sequence
whose $E_2$-term consists of the kernel and cokernel of $\lambda (q^*) - 1$ on $H^i(\GG_{m,K,\proet}, \QQ_p)$.
If $\lambda = 1$, then we get the usual values
\[
H^i(E_{\proet}, V) = \begin{cases} \QQ_p & i =0, 2 \\
\QQ_p \oplus \QQ_p & i=1 \\
0 & i > 2. \end{cases}
\]
If $\lambda \notin q^\ZZ$, then the action of $\lambda (q^*) - 1$ on $H^i(\GG_{m,K,\proet}, \QQ_p)$ is invertible, so $H^i(E_{\proet}, V) = 0$ for all $i$.

This leaves the case where $\lambda = q^m$ for some nonzero $m \in \ZZ$. In this case, the
kernel and cokernel of $\lambda (q^*) - 1$ on $H^0(\GG_{m,K}, \calO)$ are both of the form $K \cdot T^{-m}$,
so from the spectral sequence we obtain
\[
H^i(E_{\proet}, V) = \begin{cases} 0 & i=0 \\
K & i =1,3 \\
K \oplus K & i=2 \\
0 & i>3.
\end{cases}
\]
\end{example}

\subsection{Ax-Sen-Tate for rigid analytic spaces}
\label{subsec:ax-sen-tate}

As suggested in \cite{part1}, we extend the Ax-Sen-Tate theorem to adic spaces of finite type over an analytic field of mixed characteristics.
To begin with, we record a consequence of the main theorem of \cite{ax}.
\begin{prop} \label{P:ax-sen-tate}
In case $X = \Spa(K, K^+)$  for some $K^+$, we have $H^0(X_{\proet}, \widehat{\calO}_X) = K$.
\end{prop}
\begin{proof}
Let $\CC_K$ be a completed algebraic closure of $K$; the absolute Galois group $G_K$ acts on $\CC_K$ by continuity. Let $\overline{K}$ be the integral closure of $K$ in $\CC_K$. Since $\CC_K$ is perfectoid, by 
Theorem~\ref{T:faithful acyclic} we
have
$H^0(X_{\proet}, \widehat{\calO}_X) = \CC_K^{G_K}$; by the main theorem of \cite{ax}, we have $\CC_K^{G_K} = \overline{K}^{G_K} = K$.
\end{proof}

We next recall some facts about seminormality from \S\ref{subsec:seminormality adic} and point out their significance.
\begin{remark} \label{R:effect of seminormalization}
Let $X'$ be the seminormalization of $X$.
For any perfectoid subdomain $Y = \Spa(A,A^+)$ of $X_{\proet}$, by Remark~\ref{R:sections of completed structure sheaf}, $A$ is seminormal and so $Y \to X$ factors through $X'$; it follows that  $Y \times_X X' \cong Y$ is uniform, so the perfectoid subdomain of $X'$ corresponding to $Y$ coincides with $Y$. That is, we have an isomorphism of sites $X_{\proet} \cong X'_{\proet}$ inducing an isomorphism of rings $\widehat{\calO}_X \cong \widehat{\calO}_{X'}$.

This means that relative $p$-adic Hodge theory over $X'$ is \emph{de facto} equivalent to relative $p$-adic Hodge theory over $X$.
For example, for any perfect period sheaf $*$, the classes of relative $(\varphi, \Gamma)$-modules over $*_X$ and $*_{X'}$ are canonically equivalent.
\end{remark}

\begin{theorem} \label{T:sections of completed structure sheaf}
The space $X$ is seminormal if and only if 
the map
$\calO_X \to \nu_{\proet *} \widehat{\calO}_X$ is an isomorphism.
\end{theorem}
\begin{proof}
We first check that if $K$ is a perfectoid field and $X$ is smooth over $K$, then $\calO_X \cong \nu_{\proet *} \widehat{\calO}_X$. Using Remark~\ref{R:cover smooth by toric}, we may reduce to the case where
$X$ is the base of a restricted toric tower; thanks to
Theorem~\ref{T:standard toric decompleting},
the claim follows from Proposition~\ref{P:Ax}(b).

We next check that the conclusion of the previous paragraph remains true for arbitrary $K$. Since the claim is local, we may assume that $X = \Spa(A,A^+)$.
Let $K'$ be a perfectoid extension of $K$
and put $A' = A \widehat{\otimes}_K K'$.
Let $K''$ be the uniform completion of $K' \otimes_K K'$
and put $A'' = A \widehat{\otimes}_K K''$.
Since the claim is known for $A'$ by the previous paragraph, we need only check that the \v{C}ech sequence $0 \to A \to A' \to A''$ is exact; however, by \cite[Lemma~2.2.10(b)]{part1}
this reduces to the corresponding statement for $A = K$, which
is Proposition~\ref{P:ax-sen-tate}.

We next check that if $X$ is seminormal, then $\calO_X \cong \nu_{\proet *} \widehat{\calO}_X$.
The claim is again local on $X$, so by Lemma~\ref{L:perfectoid cover} we may assume that
there exists a covering of $X_{\proet}$ by a single perfectoid subdomain $Y$.
Construct a proper birational morphism $f: X' \to X$ with $X'$ smooth by applying Remark~\ref{R:resolution},
and let $Y'$ be the perfectoid subdomain of $X'$ corresponding to $Y$ by pullback.
Given a section of $\nu_{\proet *} \widehat{\calO}_X$, we may evaluate it on some perfectoid subdomain of $Y$, pull back to $Y'$, then apply the previous paragraph to get a section of $\calO_{X'}$. Since $X$ is seminormal, by Remark~\ref{R:inverse image ideal sheaf} we have $\calO_X \cong f_* \calO_{X'}$; we thus deduce that $\calO_X \to \nu_{\proet *} \widehat{\calO}_X$ is surjective, as required.

We finally check that if $\calO_X \cong \nu_{\proet *} \widehat{\calO}_X$, then $X$ is seminormal. 
Let $f: X' \to X$ be the seminormalization of $X$.
By Remark~\ref{R:effect of seminormalization} and the previous paragraph, we have $f_* \calO_{X'} \cong \nu_{\proet *} \widehat{\calO}_X \cong \calO_X$. By Proposition~\ref{P:h-topology seminormal}, it follows that $X$ is seminormal.
\end{proof}

\begin{cor} \label{C:sections of completed vector bundle}
Suppose that $X$ is seminormal.
Then for any vector bundle $\calF$ over $X$,
the map
$\calF \to \nu_{\proet*} \nu_{\proet}^* \calF$ is an isomorphism.
\end{cor}

\begin{remark}
A related assertion to Theorem~\ref{T:sections of completed structure sheaf}
would state that if $X = \Spa(A,A^+)$ and $\psi$ is a perfectoid finite \'etale tower over $X$, then $A \to \tilde{A}_\psi$ splits in the category
of Banach modules over $A$ if and only if $A$ is seminormal. This will be addressed in upcoming work of the first author with David Hansen \cite{hansen-kedlaya}.
\end{remark}

We next use the preceding arguments to make some inferences about ideals of $\widehat{\calO}_X$.
\begin{lemma} \label{L:stable submodules}
Assume that 
$X = \Spa(A,A^+)$ is the base of a restricted toric tower $\psi$.
Let $\psi'$ be a finite \'etale tower factoring through $\psi$.
Let $F$ be a finite projective $A$-module.
Let $M$ be a submodule of $F \otimes_A \tilde{A}_{\psi'}$ 
such that the two base extensions $M \otimes_{\tilde{A}_{\psi'}, \iota_{0,0}} \tilde{A}_{\psi^{\prime 1}}$,
$M \otimes_{\tilde{A}_{\psi'}, \iota_{0,1}} \tilde{A}_{\psi^{\prime 1}}$
have the same image in $F \otimes_A \tilde{A}_{\psi^{\prime 1}}$.
Then $M$ is the base extension of a submodule of $F$, and is complete for the subspace topology.
\end{lemma}
\begin{proof}
The map $A \to \tilde{A}_\psi$ admits an $A$-linear splitting (e.g., Remark~\ref{R:toric splitting}); by Lemma~\ref{L:perfectoid tower splitting}, the same is true of the map
$A \to \tilde{A}_{\psi'}$. 
Let $J$ be the submodule of $F$ generated by $\pi(\bv)$
as $\pi$ varies over splittings of $A \to \tilde{A}_{\psi'}$ and $\bv$ varies over $M$.
We check that the evident inclusion $M \subseteq J \otimes_A \tilde{A}_{\psi'}$ is in fact an equality.

We first treat the case $J = F$; we may reduce to the case where $F$ is finite free,
and then to the case $F = A$. By \cite[Corollary~2.3.5]{part1}, it suffices to check that for each $\delta \in \calM(\tilde{A}_{\psi'})$, there exists $\bv \in M$ such that $\delta(\bv) > 0$. Let $\beta \in \calM(A)$ be the restriction of $\delta$; by the construction of $J$,
$M$ has nonzero image in $\tilde{A}_{\psi'} \widehat{\otimes}_A \calH(\beta)$. 
Consequently, for some $\delta \in \calM(\tilde{A}_{\psi'})$ lifting $\beta$, there exists $\bv \in M$ such that $\delta(\bv) > 0$; however, the same then holds for all $\delta$ lifting $\beta$ (for different $\bv$) because of the equality of the base extensions.

We now treat the general case.
Since $A$ is a noetherian ring, we may choose a strict surjection $F' \to J$ of $A$-modules with $F'$ finite free.
Let $M'$ be the inverse image of $M$ in $F' \otimes_A \tilde{A}_{\psi'}$.
By the previous paragraph, $M' = F' \otimes_A \tilde{A}_{\psi'}$ and hence $M = J \otimes_A \tilde{A}_{\psi'}$.

We have now established that $M$ is the base extension of a submodule of $F$; it remains to verify that $M$ is complete for the natural topology. Let $\overline{M}$ be the closure of $M$ for the natural topology on $F \otimes_A \tilde{A}_{\psi'}$; we may then apply the preceding argument to $\overline{M}$ to deduce that it is the base extension
of a submodule of $F$. In particular, it is finitely generated, and so by Corollary~\ref{C:closure finitely generated}
it coincides with $M$. 
\end{proof}

\begin{lemma} \label{L:subsheaf from pro}
If $\calG$ is a coherent subsheaf of a vector bundle $\calF$ over $\calO_X$, then  
$\calG=\nu_{\proet *}(\image(\nu_{\proet}^*\calG\to \nu_{\proet}^*\calF))$.
\end{lemma}
\begin{proof}
We may work locally to suppose that $\calG=\ker(\calF\to\calF')$, where $\calF'$ is another vector bundle over $\calO_X$. Pushing forward the exact sequence
\[
0\to\image(\nu_{\proet}^*\calG\to \nu^*_{\proet}\calF)\to\nu^*_{\proet}\calF\to\nu^*_{\proet}\calF'
\]
and using Corollary \ref{C:sections of completed vector bundle}, we get 
\[
0\to\nu_{\proet*}(\image(\nu_{\proet}^*\calG\to \nu^*_{\proet}\calF))\to\calF\to\calF'.
\]  
This yields $\calG=\nu_{\proet *}(\image(\nu_{\proet}^*\calG\to \nu_{\proet}^*\calF))$.
\end{proof}

\begin{lemma} \label{L:provector bundle submodule}
For $X$ seminormal and $\calF$ a locally finite free $\calO_X$-module,
the functor $\calG \mapsto \image(\nu_{\proet}^* \calG \to \nu_{\proet}^* \calF)$
defines an equivalence of categories 
(with quasi-inverse $\nu_{\proet *}$)
between coherent subsheaves of $\calF$ and $\widehat{\calO}_X$-submodules of $\nu_{\proet}^* \calF$.
\end{lemma}
\begin{proof}
Note that full faithfulness of the functor in question follows from 
Lemma~\ref{L:subsheaf from pro}. To check essential surjectivity,
we may work locally. Suppose first that $X$ is smooth. Using Remark~\ref{R:cover smooth by toric}, we may assume that $X = \Spa(A,A^+)$ is the base of a restricted toric tower $\psi$. 
Using \cite[Proposition~9.2.6]{part1}, we may further reduce to considering only $\widehat{\calO}_X$-submodules which are generated by their sections on a perfectoid subdomain which is a tower of faithfully finite \'etale covers over $X$. We may then deduce the claim from Lemma~\ref{L:stable submodules}.

To treat the case of general $X$, we may assume that $\calF$ is free; by induction on $\rank(\calF)$, we may further reduce to the case $\calF = \calO_X$.
Set notation as in Remark~\ref{R:inverse image ideal sheaf} and let $\calJ$ be an ideal of $\widehat{\calO}_X$.
We construct a decreasing sequence $\calO_X = \calI_0 \supseteq \calI_1 \supseteq \cdots$ of ideal subsheaves of $\calO_X$ whose inverse images in $\widehat{\calO}_X$ contain $\calJ$, as follows. Given $\calI_j$, construct a morphism $\calF_j \to \calI_j$ of $\calO_X$-modules with $\calF_j$ finite free;
form the inverse image $\calJ_j$ of $\calJ$ in $\nu_{\proet}^* \calF_j$; apply the previous paragraph to identify $f_{\proet}^* \calJ_j$ with the inverse image of a coherent subsheaf of $f^* \calF_j$; push this subsheaf forward into $\calF_j$ via adjunction; then map to $\calO_X$ to obtain $\calJ_{j+1}$. This construction has the property that if $\calJ \neq \calJ_j$, then $\calI_{j+1} \neq \calI_j$; more precisely, the support of $\calJ_j/\calJ$ coincides with the support of $\calI_j/\calI_{j+1}$ (i.e., for any perfectoid subdomain $Y$ of $X$,
the support of $(\calJ_j/\calJ)(Y)$ on $Y$ is the inverse image under $Y \to X$ of the support of $\calI_j/\calI_{j+1}$ on $X$). By noetherian induction, it follows that this support eventually becomes empty, yielding the claim.
\end{proof}

\begin{prop} \label{P:coherent pullback}
For $X$ seminormal, $\nu_{\proet}^*$ defines an exact fully faithful functor from
coherent sheaves on $X$ to pseudocoherent $\widehat{\calO}_X$-modules.
\end{prop}
\begin{proof}
Let $\calE \to \calF$ be an injective morphism of coherent sheaves on $X$
with $\calF$ locally free.
Choose a surjective morphism $\calF' \to \calE$ of coherent sheaves on $X$.
By Lemma~\ref{L:provector bundle submodule},
$\ker(\nu_{\proet}^* \calF' \to \nu_{\proet}^* \calF)$
is the image of $\nu_{\proet}^* \calG$ for some coherent subsheaf
$\calG$ of $\calF'$.
From the exact sequence
\[
0\to\image(\nu_{\proet}^* \calG \to \nu_{\proet}^* \calF')\to\nu_{\proet}^* \calF' \to \nu_{\proet}^* \calF
\]
we may push forward and apply Corollary~\ref{C:sections of completed
vector bundle} and Lemma \ref{L:provector bundle submodule}
to obtain an exact sequence
\[
0 \to \calG \to \calF' \to \calF
\]
which pulls back to the previous one; that is, we have $\calG = \ker(\calF' \to \calF)
= \ker(\calF' \to \calE)$. By writing
\[
\nu_{\proet}^* \calE = \nu_{\proet}^* (\calF'/\calG) = \coker(\nu_{\proet}^* \calG \to \nu_{\proet}^* \calF'),
\]
we see that $\nu_{\proet}^* \calE \to \nu_{\proet}^* \calF$ is injective.
This proves that $\nu_{\proet}^*$ is exact on coherent sheaves,
and that $\calE \cong \nu_{\proet *} \nu_{\proet}^* \calE$; the latter implies that
$\nu_{\proet}^*$ is fully faithful. Since $\nu_{\proet}^*$ is exact, it pulls back projective resolutions.

For any perfectoid subdomain of $Y$, 
we may apply Lemma~\ref{L:provector bundle submodule} to deduce that
the completion of $(\nu_{\proet}^* \calE)(Y)$ in $(\nu_{\proet}^* \calF)(Y)$ is finitely generated;
by Corollary~\ref{C:closure finitely generated}, we deduce that $(\nu_{\proet}^* \calE)(Y)$ is itself complete for the subspace topology. In conjunction with the previous paragraph,
this implies that $\nu_{\proet}^* (\calF/\calE)$ is a pseudocoherent $\widehat{\calO}_X$-module,
and completes the proof.
\end{proof}

\begin{cor} \label{C:pseudocoherent reduction}
For any coherent sheaf $\calF$ on $X$, 
$\calF \otimes_{\calO_X} \bullet$ is an endofunctor on the category of pseudocoherent $\widehat{\calO}_X$-modules.
\end{cor}
\begin{proof}
Since the seminormalization morphism is finite, this reduces immediately to the case
where $X$ is seminormal. In this case, the claim follows from Proposition~\ref{P:coherent pullback} on account of the stability of pseudocoherence under tensor products (Remark~\ref{R:need flatness}).
\end{proof}

\begin{cor} \label{C:provector bundle injective}
Suppose that $X$ is seminormal.
For any locally finite free $\widehat{\calO}_X$-module $\calF$,
the map
$\nu_{\proet}^* \nu_{\proet *} \calF \to \calF$ is injective.
\end{cor}
\begin{proof}
Put $\calE = \nu_{\proet *} \calF$. By 
Lemma~\ref{L:provector bundle submodule},
the kernel of $\nu_{\proet}^* \calE \to \calF$
is itself the pullback of a coherent subsheaf of $\calE$;
however, every nonzero section of that subsheaf
has nonzero image in $\calF$. Hence the kernel must vanish.
\end{proof}

As an aside, we observe that the corresponding situation for $\overline{\calO} = \tilde{\bA}/(p)$ is somewhat simpler.
\begin{lemma} \label{L:stable submodules type A}
Assume that 
$X = \Spa(A,A^+)$ is connected
and is the base of a restricted toric tower $\psi$.
Let $\psi'$ be a finite \'etale tower factoring through $\psi$.
Let $I$ be an ideal of $\tilde{R}_{\psi'}$
such that $\varphi(I) = I$ and
 the two base extensions $I \otimes_{\tilde{R}_{\psi'}, \iota_{0,0}} \tilde{R}_{\psi^{\prime 1}}$,
$I \otimes_{\tilde{R}_{\psi'}, \iota_{0,1}} \tilde{R}_{\psi^{\prime 1}}$
have the same image in $\tilde{R}_{\psi^{\prime 1}}$.
Then $I$ is either the zero ideal or the unit ideal.
\end{lemma}
\begin{proof}
By \cite[Corollary~2.3.5]{part1}, it suffices to derive a contradiction under the assumptions that $I \neq 0$ and the support of $I$ in $\calM(\tilde{R}_{\psi'})$ is nonempty. 
The quotient map $\tilde{R}_{\psi'} \to \tilde{R}_{\psi'}/I$ corresponds via Theorem~\ref{T:Fontaine perfectoid compatibility}(iii) to quotienting by a
$\Gamma$-stable ideal of
$\tilde{A}_{\psi'}$ in the sense of Lemma~\ref{L:stable submodules}, which thus must be the base extension of an ideal $J$ of $A$ which is neither the zero ideal nor the unit ideal.

By pulling back along some closed immersion, we may reduce to the case where $A$ is one-dimensional, irreducible, and seminormal, hence smooth. By working locally and taking norms, we may further reduce to the case where $A = K \langle T^{\pm} \rangle$ and $J = (T-1)$.
Note that the initial claim of the lemma does not depend on the choice of the restricted toric tower $\psi$ as long as $\psi'$ factors through $\psi$; we may thus enlarge $\psi'$
so that it factors through the standard toric tower $\psi$ over $A$. We may also assume that $K$ contains a full set of $p$-power roots of unity.

Since
$\tilde{A}_\psi/J \tilde{A}_\psi$ is a subring of the perfectoid ring
$\tilde{A}_{\psi'}/J \tilde{A}_{\psi'}$, it is itself uniform
and hence by Theorem~\ref{T:Fontaine perfectoid compatibility}(iii) corresponds 
to a quotient of $\tilde{R}_\psi$, necessarily by the ideal $I \cap \tilde{R}_\psi$.
We may thus assume at this point that $\psi = \psi'$;
this would mean that the quotient of the ring $K \langle T^{\pm p^{-\infty}} \rangle$
by the ideal $(T-1)$ is uniform.
This contradicts Example~\ref{exa:perfectoid quotient}.
\end{proof}

\begin{remark}
A corollary of Lemma~\ref{L:stable submodules type A} is that if $X$ is connected and $J$
is a $\Gamma$-stable ideal of $\tilde{A}_{\psi'}$ 
which is neither the zero ideal nor the unit ideal, then the quotient $\tilde{A}_{\psi'}/J$
cannot be uniform.
\end{remark}

\begin{cor}\label{C:type A mod p no ideals}
Suppose that $X$ is connected. Then the only $\varphi^{\pm 1}$-invariant ideal subsheaves of $\tilde{\bA}/(p)$
are $(0)$ and $(1)$.
\end{cor}

\begin{cor} \label{C:pseudocoherent mod p}
Let $M$ be an $\tilde{R}_\psi$-module contained in a finitely generated module.
Suppose that $M$ formally admits the structure of a $(\varphi, \Gamma)$-module:
that is, there exist a $\tilde{R}_{\psi^1}$-linear isomorphism $M \otimes_{i_{0,0}} \tilde{R}_{\psi^1} \to M \otimes_{i_{0,1}} \tilde{R}_{\psi^1}$
satisfying the cocycle condition and a compatible isomorphism $\varphi^*(M) \cong M$.
Then $M$ is a finite projective $\tilde{R}_\psi$-module.
\end{cor}
\begin{proof}
We may assume that $X$ is connected.
Let $I$ be the set of $f \in \tilde{R}_\psi$ for which
exists an injective homomorphism $M \to N$ of $\tilde{R}_\psi$-modules such that $N$ is finitely generated and $M_f \to N_f$ splits. By Corollary~\ref{C:splitting ideal}, this is an ideal not contained in any minimal prime ideal. Since $I$ is defined intrinsically in terms of $M$, it is itself stable under $\varphi^{\pm 1}$ and $\Gamma$; by Corollary~\ref{C:type A mod p no ideals}, $I$ must be the unit ideal.
This proves the claim.
\end{proof}

\subsection{Fitting ideals and finite generation}
\label{subsec:Fitting}

Using the structure of ideal sheaves of $\widehat{\calO}_X$, we analyze finite generation properties of more general $\widehat{\calO}_X$-modules. The general principle is that 
while $\widehat{\calO}_X$-modules do not all arise by pullback from $X$, their module-theoretic properties behave as if this were the case.

\begin{defn} \label{D:pseudocoherent support}
Suppose that $X$ is seminormal, and let $\calF$ be a locally finite $\widehat{\calO}_X$-module. We may then apply the Fitting ideal construction to obtain submodules of $\widehat{\calO}_X$. By Lemma~\ref{L:provector bundle submodule}, these submodules descend to ideals of $\calO$, which we denote $\Fitt_n(\calF)$.
In particular, we may define the \emph{support} of $\calF$ as the support of $\Fitt_0(\calF)$; it is a closed subspace of $X$.

Recall that for a finitely generated module over a ring, the sequence of Fitting ideals
eventually stabilizes at the unit ideal. Consequently, if $X$ is quasicompact,
the sequence $\{\Fitt_n(\calF)\}_n$ also stabilizes at the unit ideal.
\end{defn}

\begin{hypothesis} \label{H:Fitting}
For the remainder of \S\ref{subsec:Fitting}, suppose that $X = \Spa(A,A^+)$ is seminormal and affinoid, and let $\psi$ be a finite \'etale perfectoid tower over $A$;
in addition, let $J$ be an ideal of $A$ and put $\overline{A} = A/J$, $\overline{B} = \tilde{A}_\psi/J \tilde{A}_\psi$. 
We will consider $\tilde{A}_\psi$-modules with $\Gamma$-structure in the sense of Definition~\ref{D:Gamma-module}; when these modules are killed by $J$, we will also refer to them as \emph{$\overline{B}$-modules with $\Gamma$-structure}.
\end{hypothesis}

\begin{lemma} \label{L:remains torsion-free}
The ring homomorphism $A \to \tilde{A}_\psi$ is faithfully flat,
as then is $\overline{A} \to \overline{B}$.
In particular, if $f \in A$ is not a zero-divisor in $\overline{A}$,
then $f$ is not a zero-divisor in $\overline{B}$ either.
\end{lemma}
\begin{proof}
To check (a), it is enough to check that for every ideal $I$ of $A$, the sequence
\[
0 \to I \otimes_A \tilde{A}_\psi \to \tilde{A}_\psi \to \tilde{A}_\psi/I\tilde{A}_\psi \to 0
\]
is exact. This follows by applying
Theorem~\ref{T:refined Kiehl} (to equate pseudocoherent modules and sheaves) and 
Proposition~\ref{P:coherent pullback}.
From (a), we immediately deduce (b).
\end{proof}

\begin{lemma} \label{L:localize to projective}
Let $M$ be a finitely generated $\Gamma$-module over $\overline{B}$.
Then 
there exists $f \in A$ which does not map to a zero-divisor in $\overline{A}$ or $\overline{B}$ such that $M_f$ is a finite projective module over $\overline{B}_f$;
moreover, $f$ may be chosen so as to depend only on the Fitting ideals of $M$.
\end{lemma}
\begin{proof}
Let $\gothp_1,\dots,\gothp_k$ be the minimal primes of $A$ containing $J$
(there are only finitely many of these because $A$ is noetherian).
As in Definition~\ref{D:pseudocoherent support}, we may identify the Fitting ideals
$\Fitt_n(M)$ with ideals of $A$.

For $i=1,\dots,k$, we claim that there exists $f_i \in A \setminus \gothp_i$ 
such that $M \otimes_A (A/\gothp_i)_{f_i}$ is a finite projective module over $(\tilde{A}_\psi/\gothp_i \tilde{A}_\psi)_{f_i}$. To check this, let $m_i$ be the smallest nonnegative integer for which $\Fitt_{m_i}(M) \not\subseteq \gothp_i$. Then for any $f_i \in \Fitt_{m_i}(M) \setminus \gothp_i$,
$M \otimes_A (A/\gothp_i)_{f_i}$ is a finite projective module of rank $m_i$
over $(\tilde{A}_\psi/\gothp_i \tilde{A}_\psi)_{f_i}$.

For $i,j = 1,\dots,k$ distinct, choose $s_{i,j} \in \gothp_j \setminus \gothp_i$. Then
\[
f = \sum_{i=1}^k f_i \prod_{j \neq i} s_{i,j}
\]
does not map to a zero-divisor in $\overline{A}$ (by construction)
or $\overline{B}$ (by Lemma~\ref{L:remains torsion-free}), and
$M_f$ is a finite projective module over $\overline{B}_f = \bigoplus_{i=1}^k (\tilde{A}_\psi/\gothp_i \tilde{A}_\psi)_f$.
This proves the claim.
\end{proof}

\begin{prop} \label{P:stable finite}
Every finitely generated $\tilde{A}_\psi$-module with $\Gamma$-structure is \'etale-stably pseudocoherent
(and hence a pseudocoherent $\Gamma$-module).
\end{prop}
\begin{proof}
Let $M$ be a finitely generated $\tilde{A}_\psi$-module with $\Gamma$-structure.
We prove that $M$ is pseudocoherent under the additional assumption that $M$ is annihilated by a fixed ideal $J$ of $A$, working by noetherian induction on $J$; the case $J=0$ will give the desired result. 
Using Lemma~\ref{L:pseudocoherent 2 of 3}, we may further reduce to the case where $\overline{A}$ is reduced and connected.
By Corollary~\ref{C:pseudocoherent reduction}, $\overline{B}$ is a pseudocoherent $\tilde{A}_\psi$-module, so an $\tilde{A}_\psi$-module annihilated by $J$ is pseudocoherent if and only if it is pseudocoherent as a $\overline{B}$-module;
similarly, an $\tilde{A}_\psi$-module annihilated by $J$ and $f$ is pseudocoherent if and only if it is  pseudocoherent as a $\overline{B}/f\overline{B}$-module.

By Lemma~\ref{L:localize to projective}, we may choose $f \in A$ which is not a zero-divisor in $\overline{A}$ or $\overline{B}$ such that $M_f$ is a finite projective module over $\overline{B}_f$.
Let $T$ be the $f$-power torsion submodule of $M$. By
the induction hypothesis, $(M/T)/f(M/T)$ is pseudocoherent; since
$M_f \cong (M/T)_f$ and $(M/T)[f] = 0$, we may apply Lemma~\ref{L:pseudocoherent glueing}(a) to see that $M/T$ is pseudocoherent. By Lemma~\ref{L:pseudocoherent 2 of 3},
$T$ is finitely generated; it is thus annihilated by $f^m$ for some positive integer $m$.
By repeatedly applying the induction hypothesis and Lemma~\ref{L:pseudocoherent 2 of 3},
we deduce that $f^{i-1} T/f^i T$ is pseudocoherent for $i=1,\dots,m$.
It follows that $M$ is pseudocoherent.

Choose an $A$-linear surjection $F \to M$ with $F$ finite free,
let $N$ be the kernel of $F \to M$,
and let $\widehat{N}$ be the completion of $N$ in $F$.
Then $M/\widehat{N}$ is the completion of $M$, which is again a finitely generated $\tilde{A}_\psi$-module with $\Gamma$-structure and hence a pseudocoherent module. It follows that $\widehat{N}$ is finitely generated, so we may apply Corollary~\ref{C:closure finitely generated} to deduce that $N=\widehat{N}$. It follows that $M$ is strictly pseudocoherent, as desired.

Since any \'etale morphism $(\tilde{A}_\psi, \tilde{A}^+_\psi) \to (C,C^+)$ can be obtained by
base extension from some term in the tower $\psi$, the previous logic also implies that $M$ is \'etale-stably pseudocoherent. This completes the proof.
\end{proof}
\begin{cor} \label{C:stable finite}
Within the  category of $\tilde{A}_\psi$-modules with $\Gamma$-structures, the finitely generated modules form a Serre subcategory.
\end{cor}
\begin{proof}
This follows from Proposition~\ref{P:stable finite} and Lemma~\ref{L:pseudocoherent 2 of 3}.
\end{proof}

\begin{remark}
Recall that by Theorem~\ref{T:refined Kiehl}
and Corollary~\ref{C:refined Kiehl perfectoid}, \'etale-stably pseudocoherent modules over $\tilde{A}_\psi$ are equivalent to pseudocoherent sheaves on $\Spa(\tilde{A}_\psi, \tilde{A}_\psi^+)$
and $\Spa(\tilde{A}_\psi, \tilde{A}_\psi^+)_{\proet}$.
When applying
Proposition~\ref{P:stable finite},
we will use this equivalence frequently and without further comment.
\end{remark}

\begin{prop} \label{P:ascending sequence}
Let $M$ be a pseudocoherent $\Gamma$-module over $\tilde{A}_\psi$. Then
any ascending sequence of $\Gamma$-submodules of $M$ stabilizes.
\end{prop}
\begin{proof}
It is equivalent to verify that any sequence $M = M_0 \to M_1 \to \cdots$ of surjective morphisms of $\Gamma$-modules stabilizes. We again work under the additional assumption
that $M$ is annihilated by a fixed ideal $J$ of $A$, working by noetherian induction on $J$ with the case $J=0$ delivering the final result.

For each $m$, the sequence $\{\Fitt_m(M_i)\}_i$ of ideals of $A$ is nondecreasing in $i$;
since $A$ is noetherian, these sequences all stabilize.
By discarding some initial terms of the sequence, we may ensure that the sequences
are all constant. Choose $f$ as per Lemma~\ref{L:localize to projective}; since the construction depends only on Fitting ideals, it can be made uniformly across the $M_i$.
That is, for each $i$, $(M_i)_f$ is a finite projective module over $\overline{B}_f$.

By the induction hypothesis, the sequence $M_0/fM_0 \to M_1/fM_1 \to \cdots$ also stabilizes. By Corollary~\ref{C:stable finite}, the quotient of $M_i$ by its $f$-power-torsion submodule is pseudocoherent;
we may thus apply Lemma~\ref{L:stable sequence} to deduce that the original sequence stabilizes.
\end{proof}

\begin{prop} \label{P:ext tor stability}
Let $M$ and $N$ be pseudocoherent $\Gamma$-modules over $\tilde{A}_\psi$. Then
for all $i \geq 0$, $\Ext^i_{\tilde{A}_\psi}(M, N)$ and
$\Tor_i^{\tilde{A}_\psi}(M,N)$ are pseudocoherent $\Gamma$-modules over $\tilde{A}_\psi$.
\end{prop}
\begin{proof} 
For brevity, we omit the label $\tilde{A}_\psi$ on the $\Ext$ and $\Tor$ groups.
We again work under the additional assumption that $M$ is annihilated by a fixed ideal $J$ of $A$, working by noetherian induction on $J$. Since our goal is the case $J=0$,
we may assume in addition that $J$ is generated by a regular sequence; consequently,
$\overline{B}$ is fpd as an $\tilde{A}_\psi$-module, resolved by the Koszul complex.
In particular, we have a finite resolution
\[
0 \to F_n \to \cdots \to F_0 \to \overline{B} \to 0
\]
of $\overline{B}$ by finite free $\Gamma$-modules.
We may thus compute the modules $\Ext^i(\overline{B}, N)$ as the homology groups of the complex
\[
0 \to \Hom(F_0, N) \to \cdots \to \Hom(F_n, N) \to 0,
\]
of pseudocoherent $\Gamma$-modules; by Proposition~\ref{P:stable finite} and Corollary~\ref{C:stable finite}, the homology groups
are pseudocoherent. Similarly, the modules $\Tor_i(\overline{B}, N)$ are all pseudocoherent.

Choose $f$ as in Lemma~\ref{L:localize to projective}. By the previous paragraph (plus Proposition~\ref{P:stable finite}), $\Ext^i(M,N)_f$ and $\Tor_i(M, N)_f$ are pseudocoherent $\Gamma$-modules over $(\tilde{A}_\psi/J)_f$; 
by applying Lemma~\ref{L:localize to projective} to these modules, 
we may adjust the choice of $f$ (depending on $i$) to ensure that $\Ext^i(M,N)_f$ and $\Tor_i(M, N)_f$
are projective over $(\tilde{A}_\psi/J)_f$.
Hence by Lemma~\ref{L:pseudocoherent glueing}, it suffices to check that
the kernel and cokernel of multiplication by $f$ on $\Ext^i(M, N),
\Tor_i(M, N)$ 
are pseudocoherent $\Gamma$-modules.
We work out the case of $\Ext$ in detail, the case of $\Tor$ being similar.

By Remark~\ref{R:transfer pseudocoherence to quotient}, 
$M[f]$ and $M/fM$ are pseudocoherent $\Gamma$-modules.
By the induction hypothesis, the modules
\[
\Ext^i(M[f], N),
\Ext^i(M/fM, N)
\]
are pseudocoherent $\Gamma$-modules. Now apply derived functors to the exact sequences
\[
0 \to M[f] \to M \to fM \to 0, \qquad 0 \to fM \to M \to M/fM \to 0
\]
to obtain exact sequences
\begin{gather*}
\Ext^{i-1}(M[f], N) \to \Ext^i(fM, N) \to \Ext^i(M, N) \to \Ext^i(M[f], N),
\\
\Ext^i(M/fM, N) \to \Ext^i(M, N) \to \Ext^i(fM, N) \to \Ext^{i+1}(M/fM, N).
\end{gather*}
By Corollary~\ref{C:stable finite}, pseudocoherent $\Gamma$-modules over $\tilde{A}_\psi$ form a Serre subcategory of $\tilde{A}_\psi$-modules with semilinear $\Gamma$-action; in the quotient category, the morphisms $\Ext^i(fM, N) \to \Ext^i(M, N)$,
$\Ext^i(M, N) \to \Ext^i(fM, N)$ become isomorphisms.
Hence multiplication by $f$ on $\Ext^i(M,N)$ becomes an isomorphism in the quotient category, that is, its kernel and cokernel are pseudocoherent $\Gamma$-modules.
\end{proof}

For some applications, it will be useful to deduce pseudocoherence from an even weaker condition than finite generation.

\begin{prop} \label{P:stable finite sub}
Let $M$ be a $\overline{B}$-module satisfying the following conditions.
\begin{enumerate}
\item[(a)]
There exists a $\Gamma$-structure on $M$.
\item[(b)]
For every $f \in A$ which is not a zero-divisor in $\overline{A}$, $M/fM$ is finitely generated
(and hence \'etale-stably pseudocoherent by condition (a) and Proposition~\ref{P:stable finite}).
\item[(c)]
The module $M$ can be realized as a submodule of some finite $\overline{B}$-module
(not necessarily admitting a $\Gamma$-structure).
\item[(d)]
One of the following conditions holds.
\begin{enumerate}
\item[(i)]
In (c), the finite $\overline{B}$-module can be taken to be complete for the natural topology
and $M$ can be taken to be a closed submodule.
\item[(ii)]
For every $f \in A$ which is not a zero-divisor in $\overline{A}$, $M[f^\infty]$ is finitely generated.
\end{enumerate}
\end{enumerate}
Then $M$ is a pseudocoherent $\Gamma$-module.
\end{prop}
\begin{proof}
By Proposition~\ref{P:stable finite}, we only need to show that $M$ is a pseudocoherent (or even finitely generated) $\overline{B}$-module.
Let $I$ be the set of $f \in \tilde{A}_\psi$ for which
exists a homomorphism $M \to N$ of $\overline{B}$-modules such that $N$ is finitely generated 
(but not necessarily complete for the natural topology) and $M_f \to N_f$ is a split inclusion. By condition (c) and Corollary~\ref{C:splitting ideal}, we see that $I$ is an ideal whose extension to $\overline{B}$ is not contained in any minimal prime ideal. Since $I$ is defined intrinsically in terms of $M$, it inherits a $\Gamma$-structure by condition (a); by Lemma~\ref{L:provector bundle submodule}, $I$ is the extension of an ideal of $A$. We may argue as in the proof of Lemma~\ref{L:localize to projective} to produce $f \in A \cap I$ which is not a zero-divisor in $\overline{A}$ or $\overline{B}$.
By condition (b) and Lemma~\ref{L:splitting to finite}, $M/M[f^\infty]$ is finitely generated,
and hence \'etale-stably pseudocoherent by Proposition~\ref{P:stable finite}.

If condition (d)(ii) holds, we are done. Otherwise, there exists an injective homomorphism $M \to N$ of $\overline{B}$-modules such that $N$ is complete for the natural topology and $M$ is closed in $N$. (It is not guaranteed that $M_f \to N_f$ splits, but this is not needed here.) By the previous paragraph and Theorem~\ref{T:open mapping}, $M[f^\infty]$ is also a closed submodule of $N$; however, it is the union of the closed submodules $M[f^n]$ over all $n$, and an ascending sequence of closed submodules of $N$ cannot have closed union unless the sequence stabilizes (by the Baire category theorem or an elementary argument). Hence there exists some $n$ for which $M[f^\infty] = M[f^n]$. For such $n$, each of $M[f]/M[f^2], \dots, M[f^{n-1}]/M[f^n]$ injects into $M/fM$, and hence is pseudocoherent by condition (a) and Proposition~\ref{P:stable finite}; by Lemma~\ref{L:pseudocoherent 2 of 3}, $M$ is pseudocoherent, which yields the desired result.
\end{proof}

\subsection{Categories of pseudocoherent modules}

Using the preceding calculations, we deduce a number of categorical properties of 
pseudocoherent $\widehat{\calO}_X$-modules.

\begin{defn}
Let $\calC_X$ denote the category of pseudocoherent $\widehat{\calO}_X$-modules.
Note that if $f: X' \to X$ is the seminormalization of $X$, then by Remark~\ref{R:effect of seminormalization}, 
$f_{\proet}^*: \calC_X \to \calC_{X'}$ is an equivalence of categories
and $L_i f_{\proet}^*$ vanishes for all $i>0$.
\end{defn}

\begin{remark} \label{R:study pseudocoherent locally}
When studying local properties of $\calC_X$, we will use the fact that $X$ is covered by affinoid open subspaces which are the base spaces of finite \'etale perfectoid towers, e.g., by pullback from restricted toric towers along closed immersions.
\end{remark}

\begin{theorem} \label{T:pseudocoherent noetherian}
The category $\calC_X$ has the following properties.
\begin{enumerate}
\item[(a)]
It is a full abelian subcategory of the category of $\widehat{\calO}_X$-modules. Moreover, the formation of kernels and cokernels in $\calC_X$ is compatible with the larger category.
\item[(b)]
If $X$ is quasicompact, then the ascending chain condition holds: given any sequence $\calF_0 \to \calF_1 \to \cdots$
of epimorphisms in $\calC_X$, there exists $i_0 \geq 0$ such that for all $i \geq i_0$, the map $\calF_i \to \calF_{i+1}$ is an isomorphism.
\item[(c)]
The functors $\Tor_i$ take $\calC_X \times \calC_X$ into $\calC_X$. (In particular, this is true for $i=0$, so $\calC_X$ admits tensor products.)
\item[(d)]
The functors $\Ext^i$ take $\calC_X \times \calC_X$ into $\calC_X$. (In particular, this is true for $i=0$, so $\calC_X$ admits internal Homs.)
\end{enumerate}
\end{theorem}
\begin{proof}
As per Remark~\ref{R:study pseudocoherent locally}, we need only consider $X$ as in 
Hypothesis~\ref{H:Fitting}.
We then deduce (a) from Proposition~\ref{P:stable finite},
(b) from Proposition~\ref{P:ascending sequence},
and (c) and (d) from Proposition~\ref{P:ext tor stability}.
\end{proof}

\begin{theorem} \label{T:make projective}
For any $\calF \in \calC_X$, there exists a Zariski open dense subspace $U$ of $X$ such that $\calF|_U$ is projective.
\end{theorem}
\begin{proof}
Again, we need only consider $X$ as in Hypothesis~\ref{H:Fitting};
the claim then follows from Lemma~\ref{L:localize to projective}.
\end{proof}

\begin{theorem} \label{T:vector bundle pullback}
For $f: X' \to X$ a morphism and $i \geq 0$, the functors
$L_i f_{\proet}^*$ take $\calC_X$ into $\calC_{X'}$.
\end{theorem}
\begin{proof}
Again, we need only consider $X$ as in Hypothesis~\ref{H:Fitting}.
We may also assume that $X' = \Spa(A', A^{\prime +})$ is affinoid; let
$\psi'$ be a finite \'etale perfectoid tower factoring through
the pullback of the tower $\psi$ along $f$. To lighten notation, write
$B,B'$ in place of $\tilde{A}_\psi, \tilde{A}_{\psi'}$.

Let $M$ be a pseudocoherent $\Gamma$-module over $B$,
and construct a projective resolution
\[
\cdots \to F_1 \to F_0 \to M  \to 0
\]
by finite projective $B$-modules (without $\Gamma$-action). 
The groups $\Tor_i^{B}(M, B')$ are then computed by the cohomology groups of the complex
\[
\cdots \to F_1 \otimes_B B' \to F_0 \otimes_B B' \to 0
\]
and moreover carry the structure of $\Gamma$-modules (although $\Gamma$ does not act on the complex).
We show by induction on $i$ that $\Tor_i^B(M, B')$ is \'etale-stably pseudocoherent.
For $i=0$, this is immediate from Proposition~\ref{P:stable finite}. Suppose that the claim holds up to some value $i$. Using Lemma~\ref{L:pseudocoherent 2 of 3},
we may check inductively on $j$ that $\ker(F_{j+1} \otimes_B B' \to F_j \otimes_B B')$ is pseudocoherent for $j=0,\dots,i$: namely, for $0<j \leq i$, pseudocoherence of 
$\ker(F_{j} \otimes_B B' \to F_{j-1} \otimes_B B')$
and $\Tor_{j}^B(M, B')$ together imply pseudocoherence of
$\image(F_{j+1} \otimes_B B' \to F_{j} \otimes_B B')$
and hence of $\ker(F_{j+1} \otimes_B B' \to F_{j} \otimes_B B')$.
Since $\Tor_{i+1}^B(M,B')$ is both a $\Gamma$-module and
a quotient of $\ker(F_{i+1} \otimes_B B' \to F_{i} \otimes_B B')$,
we may apply Proposition~\ref{P:stable finite} to see that it is \'etale-stably pseudocoherent.
This completes the induction and proves the desired result.
\end{proof}

\begin{cor} \label{C:closed immersion equivalence surjective}
Let $f: X' \to X$ be a morphism of affinoid spaces.
Then for every $\calF' \in \calC_{X'}$, there exist an object $\calF \in \calC_X$
and a surjective morphism $f_{\proet}^* \calF \to \calF'$ in $\calC_{X'}$.
\end{cor}
\begin{proof}
By Remark~\ref{R:effect of seminormalization}, we may assume that $X$ is seminormal.
Choose a finite \'etale perfectoid tower $\psi$ over $X$
and let $\psi'$ a finite \'etale perfectoid tower factoring through the pullback of $\psi$ along $f$.
Let $M'$ be a pseudocoherent $\Gamma$-module over $\tilde{A}_{\psi'}$.
As in the proof of Lemma~\ref{L:lift to psc},
we may construct a finitely generated $\tilde{A}_{\psi}$-module $M$ with $\Gamma$-structure
and a $\Gamma$-equivariant surjection $M \otimes_{\tilde{A}_\psi} \tilde{A}_{\psi'} \to M'$ of pseudocoherent $\Gamma$-modules over $\tilde{A}_{\psi'}$.
To wit, choose a finite free module $F$ and a surjection $F \otimes_{\tilde{A}_\psi} \tilde{A}_{\psi'} \to M'$ of $\tilde{A}_{\psi'}$-modules.
Then specify an action of $\Gamma$ on each generators; this action becomes well-defined on a certain quotient $M$ of $F$ which then admits a $\Gamma$-equivariant surjection as desired. Finally, apply Proposition~\ref{P:stable finite} to see that $M$ is \'etale-stably pseudocoherent over $\tilde{A}_\psi$ 
and Theorem~\ref{T:vector bundle pullback} to see that 
$M \otimes_{\tilde{A}_\psi} \tilde{A}_{\psi'}$ is \'etale-stably pseudocoherent over $\tilde{A}_{\psi'}$.
\end{proof}

\begin{cor} \label{C:resolution adjunction}
Let $f: X' \to X$ be a proper birational morphism of smooth rigid spaces. Then for any 
$\calF \in \calC_X$, the adjunction morphism
$\calF \to \RR f_{\proet *} \LL f_{\proet}^{*} \calF$ is an isomorphism.
(The composition of left and right derived functors makes sense because $\calF$ is fpd by 
Corollary~\ref{C:pseudocoherent fpd local2}.)
\end{cor}
\begin{proof}
We may assume that $K$ is perfectoid.
By Theorem~\ref{T:vector bundle pullback}, we have $L_i f_{\proet}^{*} \calF \in \calC_{X'}$ for all $i$.
If $f': X'' \to X'$ is another such morphism, then proving the claim for 
$f'$ and $f \circ f'$ implies the claim for $f$. Using Temkin's embedded resolution of singularities (see Remark~\ref{R:resolution2}), we may thus reduce to the case where $f$ is the blowup along a smooth subspace $Z$. In this case, locally we may write $X$ as the base of a nonrelative ramified toric tower $\psi$ in such a way that $Z$ is the zero locus of some of the coordinate functions. 
Using Theorem~\ref{T:ramified perfect equivalence1}, we may realize $\calF$ as a module over $A_{\psi,n}$ for some $n$; in this way, we reduce to checking that
for any coherent sheaf $\calF$ on $X$, 
the map $\calF \to \RR f_* \LL f^* \calF$ is an isomorphism. 
Since $\calF$ is fpd, we may reduce to the case
$\calF = \calO_X$, for which $\LL f^* \calO_X = \calO_{X'}$
and an explicit calculation 
(or application of \cite[Theorem~1.1]{chatzistamatiou-rulling})
shows that $\calO_X \cong \RR f_* \calO_{X'}$.
\end{proof}

\begin{remark} \label{R:bad pushforward2}
By analogy with the case of schemes, one cannot expect  $R^i f_{\proet *}$ to take $\calC_{X'}$ into $\calC_X$ if $f: X' \to X$ is a morphism which is not proper.
On the other hand, properness alone is not enough. 

For example, suppose that $f$ is a closed immersion; then 
despite Remark~\ref{R:not closed immersion}, $f_{\proet *} \widehat{\calO}_Z$ is typically locally finitely generated, not locally finitely presented, and so does not belong to $\calC_X$.
On the positive side, by computing \v{C}ech cohomology, we may see that the functor
$R^i f_{\proet *}$ is identically zero for every $i>0$.
See Theorem~\ref{T:closed immersion equivalence} for a related result,
and Remark~\ref{R:bad pushforward1} for more discussion.
\end{remark}

There is also a local form of the ascending chain condition.

\begin{defn} \label{D:local limit}
For $S$ a subset of $X$, let $\calC_{X,S}$ be the direct 2-limit of $\calC_{U}$ over all open subspaces $U$ of $X$ containing $S$. Since the transition functors are exact by Theorem~\ref{T:weak flatness},
we may apply Theorem~\ref{T:pseudocoherent noetherian} to deduce that $\calC_{X,S}$ 
is an abelian category admitting $\Tor$ and $\Ext$ functors.
\end{defn}

\begin{theorem} \label{T:local acc}
Let $S$ be a subset of $X$ such that the ideal sheaves of the localization of $X$ at $S$ satisfy the ascending chain condition.
Then the ascending chain condition holds for $\calC_{X,S}$: given any sequence $\calF_0 \to \calF_1 \to \cdots$
of epimorphisms in $\calC_{X,S}$, there exists $i_0 \geq 0$ such that for all $i \geq i_0$, the map $\calF_i \to \calF_{i+1}$ is an isomorphism.
\end{theorem}
\begin{proof}
We may assume from the outset that $X$ is affinoid, and set notation as in Hypothesis~\ref{H:Fitting}.
Since the formation of Fitting ideals commutes with base extension, we may define Fitting ideals $\Fitt_m(\calF_i)$ contained in the local ring $\calO_{X,S}$; similarly, we may define the annihilator of $\calF_i$ as an ideal of $\calO_{X,x}$.
Since these ideal sheaves satisfy the ascending chain condition by hypothesis,
we may adapt the proof of Proposition~\ref{P:ascending sequence} without change in order to conclude.
\end{proof}

Note that the hypothesis of Theorem~\ref{T:local acc} holds when $S$ is a singleton set.
\begin{prop} \label{P:stalk noetherian}
For $x \in X$, the stalk $\calO_{X,x}$ is a noetherian ring.
\end{prop}
\begin{proof}
For $x$ of height 1, one may emulate the proof of \cite[Theorem~2.1.4]{berkovich2}.
For general $x$, see \cite[Proposition~15.1.1]{conrad-seminar}.
\end{proof}

\begin{remark}
In general, the condition of Theorem~\ref{T:local acc} should translate into some sort of finiteness condition on the geometric complexity of the set $S$. To wit, the stalk of a Stein compact subset $S$ of a complex-analytic space is noetherian if and only if every analytic subvariety of a neighborhood of $S$ meets $S$ in finitely many topological components \cite{siu}.
\end{remark}

\subsection{An embedded construction}
\label{subsec:additional statements}

We next consider a variant of the preceding construction, in which one works within an ambient seminormal space.

\begin{hypothesis}
Throughout \S\ref{subsec:additional statements}, let $Y,Z,S$ denote arbitrary rigid analytic spaces over $K$. All morphisms among $X,Y,Z,S$ will be in the category of rigid analytic spaces over $K$.
\end{hypothesis}

\begin{defn}
For $j: Z \to X$ a closed immersion,
put $\calJ_Z = \ker(\calO_X \to j_* \calO_Z)$,
and let $\calC_{Z \subset X}$ be the subcategory of $\calC_X$ consisting of sheaves killed by $\calJ_Z$. 
\end{defn}

\begin{theorem} \label{T:closed immersion equivalence}
Let $j: Z \to X$ be a closed immersion of seminormal spaces. Then 
the functor $j_{\proet}^*: \calC_{Z \subset X} \to \calC_Z$ of Theorem~\ref{T:vector bundle pullback} is an exact equivalence.
\end{theorem}
\begin{proof}
Set notation as in Hypothesis~\ref{H:Fitting}.
We again work under the additional assumption that $M$ is annihilated by a fixed 
ideal $J$ of $A$, working by noetherian induction on $J$. 
Let $\psi'$ be the pullback of the tower $\psi$ to $Z$, which is again a perfectoid tower. 
Let
\[
0 \to M_1 \to M_2 \to M_3 \to 0
\]
be an exact sequence of pseudocoherent $\Gamma$-modules over $\overline{B}$.
By Lemma~\ref{L:localize to projective}, there exists $f \in A$ which does not map to a zero-divisor in $\overline{A}$ or $\overline{B}$ such that $M_{3,f}$ is a finite projective module over $\overline{B}_f$; using the fact that $Z$ and $X$ are seminormal,
we see that $M_{1,f} \to M_{2,f}$ remains injective upon tensoring over 
$\tilde{A}_\psi$ with $\tilde{A}_{\psi'}$.
By the induction hypothesis, $M_1[f^\infty] \to M_2[f^\infty]$
remains injective upon base extension; this proves the claim.
\end{proof}

\begin{remark}
Theorem~\ref{T:closed immersion equivalence} cannot be promoted to the statement that the higher derived functors
$L_i j_{\proet}^*: \calC_{Z \subset X} \to \calC_Z$ vanish for $i>0$; this is a consequence of the fact that for $I$ an ideal of a ring $R$ and $M$ an $(R/I)$-module, we may have
$\Tor_i^R(R/I, M) \neq 0$ for $i>0$. For instance, if $I = (f)$ for $f \in R$ not a zero-divisor, we have $\Tor_1^R(R/I, M) = M$.
\end{remark}

\begin{defn}
By virtue of Theorem~\ref{T:closed immersion equivalence}, 
for closed immersions $Z \to Y \to X$ with $X,Y$ seminormal, the pullback functor
$\calC_{Z \subset X} \to \calC_{Z \subset Y}$ is an exact equivalence.
Consequently, for any $Z$, we obtain a well-defined category
$\calC_{Z \subset X}$ by embedding $Z$ into a seminormal space $X$.
Note that unlike $\calC_Z$, this category depends on $Z$ itself and not just on its seminormalization.

In case $j: Z \to X$ is a closed immersion with $Z$ and $X$ both seminormal,
we may invert the equivalence $j_{\proet}^*: \calC_{Z \subset X} \to \calC_Z$ to obtain
a \emph{virtual pushforward} functor $\calC_Z \to \calC_X$ replacing the literal pushforward (compare Remark~\ref{R:bad pushforward2}).
\end{defn}

\begin{remark}
One application of Theorem~\ref{T:closed immersion equivalence} is to reduce questions
about $\calC_X$ to the case where $X$ is smooth, for which decompleting towers become available. It is possible to use this approach to recover Theorem~\ref{T:pseudocoherent noetherian}, but some effort is required to avoid circular reasoning.
\end{remark}

As a corollary of this discussion, we obtain a statement to the effect that 
for $f: Y \to X$ a morphism of seminormal spaces,
the functors $L_i f_{\proet}^*$ can be computed by pullback along $f$.

\begin{theorem} \label{T:algebraic pullback}
Suppose that $X = \Spa(A,A^+)$ and $Y = \Spa(A', A^{\prime +})$ are affinoid and seminormal,
and fix a morphism $f: Y \to X$. Let
$\psi$ be a finite \'etale perfectoid tower over $X$. Let $\psi'$ be a finite \'etale perfectoid tower over $Y$ which factors through the pullback of $\psi$ to $X$ along $f$.
Let $M$ be a pseudocoherent $\Gamma$-module over $\tilde{A}_\psi$.
Then for $i \geq 0$, the natural morphism
\[
\Tor_i^A(M, A') \otimes_{\tilde{A}_\psi \otimes_A A'} \tilde{A}_{\psi'} \to 
\Tor_i^{\tilde{A}_\psi}(M, \tilde{A}_{\psi'})
\]
is an isomorphism. In particular, if $f$ is flat, then $L_i f_{\proet}^* = 0$ for all $i>0$.
\end{theorem}
\begin{proof}
By Proposition~\ref{P:Robba weak flatness} and Theorem~\ref{T:pseudocoherent noetherian}(c), the claim does not depend on the choice of $\psi$ or $\psi'$.
The case where $f$ is a closed immersion is immediate from Theorem~\ref{T:closed immersion equivalence}. The case where $A' = A\{T^{\pm}\}$ is also straightforward, as there is no loss of generality in taking $\psi'$ to be a relative toric tower with respect to $\psi$, for
which the claim is immediate. The general case then follows by composition.
\end{proof}
 
\subsection{Hodge-Tate and de Rham cohomology}

We mention in passing an application to the Hodge-Tate and de Rham functors, in the spirit of the work of Brinon \cite{brinon} on relative de Rham representations.
This discussion has been amplified by the second author and Xinwen Zhu in 
\cite{liu-zhu}; see Remark~\ref{R:virtual comparison}.

\begin{defn}
We say that the analytic field $K$ is \emph{pro-coherent} if the $K$-vector spaces $H^i(\Spa(K,\gotho_K)_{\proet}, \widehat{\calO})$ are finite-dimensional for all $i \geq 0$ and zero for all sufficiently large $i$. For instance, this holds if $K$ is perfectoid, as then we get $K$ for $i=0$ and $0$ for $i>0$.
It also holds if $K$ is a $p$-adic field, e.g., by 
 Corollary~\ref{C:toric tower coherent}  and Remark~\ref{R:higher vanishing} plus Theorem~\ref{T:cyclotomic decompleting}.
 Note that even if this does not hold, by Proposition~\ref{P:ax-sen-tate} we have
 $H^0(\Spa(K,\gotho_K)_{\proet}, \widehat{\calO}) = K$.
\end{defn}

\begin{theorem} \label{T:torsion coherent cohomology}
Suppose that $K$ is pro-coherent, and choose $\calF \in \calC_X$.
\begin{enumerate}
\item[(a)]
For some $N \geq 0$ depending only on $K$ and $X$, the sheaves $R^i \nu_{\proet *} \calF$ of $\calO_X$-modules are coherent for all $i \geq 0$ and zero for all $i >N$.
(For $i=0$, we may drop the condition on $K$.)
\item[(b)]
For any morphism $f: Y \to X$ of rigid analytic spaces with $X$ and $Y$ seminormal, the natural map
\[
\LL f^* \RR \nu_{\proet *} \calF \to \RR \nu_{\proet *} \LL f_{\proet}^* \calF
\]
is an isomorphism (in the derived category of sheaves on $Y$).
\end{enumerate}
\end{theorem}
\begin{proof}
Using the pro-coherent hypothesis via a Leray spectral sequence,
we reduce to the case where $K$ is perfectoid. Note that if we only consider (a) for $i=0$,
this reduction does not in fact require any initial hypothesis on $K$.

Using Remark~\ref{R:resolution}, Theorem~\ref{T:pseudocoherent noetherian},
and Theorem~\ref{T:vector bundle pullback},
we may reduce from $X$ to a resolution; that is, we may assume that $X$ is smooth.
By Remark~\ref{R:cover smooth by toric}, we may further assume that $X = \Spa(A,A^+)$ is the base of a restricted toric tower $\psi$. 
By Theorem~\ref{T:standard toric decompleting}, 
the tower $\psi$ satisfies the conditions of
Remark~\ref{R:decompleting conditions}.
We may thus apply Corollary~\ref{C:toric tower coherent}  and Remark~\ref{R:higher vanishing} to deduce (a). 

To deduce (b),  we use an argument modeled on \cite[Theorem~4.4.3]{kpx}. In light of (a), we may invoke \cite[Lemma~4.1.5]{kpx} to reduce to the case where $Y$ is a single $K$-rational point. 
At this point, we may assume that $K$ is perfectoid and $X$ is affinoid, and then
appeal to Theorem~\ref{T:algebraic pullback} to conclude.
\end{proof}
\begin{cor}
Suppose that $K$ is pro-coherent. Let $X$ be a proper rigid space over $K$
and choose $\calF \in \calC_X$.
Then the $K$-vector spaces
$H^i(X_{\proet}, \calF)$ are finite-dimensional for all $i \geq 0$.
(For $i=0$, we may drop the condition on $K$.)
\end{cor}
\begin{proof}
Combine Theorem~\ref{T:torsion coherent cohomology} with Kiehl's finiteness theorem for cohomology of coherent sheaves \cite{kiehl-finiteness}, noting that the categories of coherent sheaves on $X$ and $X_{\et}$ coincide
(e.g., by Theorem~\ref{T:refined Kiehl} plus Theorem~\ref{T:refined Kiehl etale}) and have the same cohomology
(by Theorem~\ref{T:pseudocoherent acyclicity} and Theorem~\ref{T:pseudocoherent etale acyclicity},
or see \cite[Proposition~3.2.5]{dejong-vanderput} which also applies to nonreduced rigid spaces).
\end{proof}

\begin{cor} \label{C:generic base extension}
Suppose that $K$ is pro-coherent.
For any $\calF \in \calC_X$,
there exists a Zariski closed, nowhere Zariski dense subspace $Z$ of $X$ such that on $X \setminus Z$, for $i \geq 0$,
the sheaf $R^i \nu_{\proet *} \calF$ of $\calO_X$-modules is locally finite free 
and its formation commutes with arbitrary base extension.
(For $i=0$, we may drop the condition on $K$.)
\end{cor}
\begin{proof}
By Theorem~\ref{T:torsion coherent cohomology}(a), the sheaves $R^i \nu_{\proet *} \calF$
are coherent for all $i$ and zero for $i$ sufficiently large. 
We may choose $Z$ so that $X \setminus Z$ is smooth.
By Theorem~\ref{T:make projective}, we may 
enlarge $Z$ so that on $X \setminus Z$, the sheaf $\calF$ is locally finite free over $\widehat{\calO}_X$; by a similar argument using Fitting ideals
(as in the proof of Lemma~\ref{L:localize to projective}), we may ensure that the sheaves $R^i \nu_{\proet *} \calF$ are locally finite free.
By Theorem~\ref{T:torsion coherent cohomology}(b), their formation then commutes with arbitrary base extension.
\end{proof}

\begin{defn}
As in \cite[Definition~9.3.11]{part1}, let $\bB^+_{\dR,X}$ (resp.\ $\bB_{\dR,X}$) be the sheaf of filtered rings on $X_{\proet}$ whose value on a perfectoid subdomain $Y$ is the $\ker(\theta)$-adic completion of
$\tilde{\bB}^{1}_X(Y)$
(resp.\ the localization at $\ker(\theta)$ of said completion).
Let $\bB^+_{\HT,X}$ (resp. $\bB_{\HT,X}$) be the associated sheaf of graded rings.

Let $\calO\bB^+_{\dR,X}$ be the sheafification of the presheaf
sending $Y = \Spa(A,A^+) = \varprojlim \Spa(A_i, A_i^+)$ to the ring obtained as follows:
let $k$ be the residue field of $K$, let $R$ be the tilt of $A$,
take the $p$-adic completion of $A_i^+ \otimes_{W(k)} W(R^+)$,
invert $p$, complete again with respect to the kernel of the map to $A$ induced by the canonical map
$A_i \to A$ and the map $\theta: W(R^+) \to A^+$, and finally take the direct limit over $i$.
(This definition follows that given in the erratum to \cite{scholze2}, which in turn follows \cite{brinon}; the definition in the text of \cite{scholze2} is missing the $p$-adic completion.)
Put $\calO \bB_{\dR,X} = \calO\bB^+_{\dR,X} \otimes_{\bB^+_{\dR,X}} \bB_{\dR,X}$.
Let $\calO \bB_{\HT,X}$ be the sheaf of graded rings associated to $\calO \bB_{\dR,X}$;
the degree 0 component of $\calO \bB_{\HT,X}$ coincides with the ring $S_\infty$ of \cite{hyodo}.

Suppose now that $K$ is a $p$-adic field.
For $\calF$ a $\varphi^a$-module over $\tilde{\bC}^\infty_X$, put
\begin{align*}
D_{\dR}(\calF) &= \nu_{\proet *} (\calF \otimes_{\tilde{\bC}^\infty_X} \calO\bB_{\dR,X}) \\
D_{\HT}(\calF) &= \nu_{\proet *} (\calF \otimes_{\tilde{\bC}^\infty_X} \calO\bB_{\HT,X}).
\end{align*}
We say $\calF$ is \emph{Hodge-Tate} (resp.\ \emph{de Rham}) if the induced map
\[
\nu_{\proet}^* D_{\HT}(\calF) \otimes_{\calO_X} \calO\bB_{\HT,X}\to \calF \otimes_{\tilde{\bC}^\infty_X} \calO\bB_{\HT,X}
\]
\[
\mbox{(resp.\ 
$\nu_{\proet}^* D_{\dR}(\calF) \otimes_{\calO_X} \calO\bB_{\dR,X}\to \calF \otimes_{\tilde{\bC}^\infty_X} \calO\bB_{\dR,X}$)}
\]
is an isomorphism; note that by Corollary~\ref{C:provector bundle injective}, this map is necessarily injective in the Hodge-Tate case, hence also in the de Rham case
(compare \cite[Proposition~8.2.8]{brinon}).
We define the \emph{Hodge-Tate weights} of $\calF$ as the ranks of the graded components of $D_{\HT}(\calF)$.
\end{defn}

\begin{theorem} \label{T:generic HT DR}
Let $K$ be a $p$-adic field.
Let $\calF$ be a $\varphi^a$-module over $\tilde{\bC}_X$. 
Suppose that there exists a multisubset $\mu$ of $\ZZ$
such that for some Zariski dense subset $T$ of $X$ consisting of classical rigid points,
for each $z \in T$ the restriction of $\calF$ to $\tilde{\bC}_x$ is Hodge-Tate (resp.\ de Rham) with Hodge-Tate weights $\mu$. Then the same holds for all $z$ outside of some Zariski closed, nowhere Zariski dense subset of $X$.
\end{theorem}
\begin{proof}
The de Rham case reduces at once to the Hodge-Tate case.
In the Hodge-Tate case, Corollary~\ref{C:generic base extension} applied in the case $i=0$ (so no condition on $K$ is required) and shows that (after removing a suitable Zariski closed subset from $X$)
\[
\nu_{\proet}^*  \nu_{\proet *} (\calF \otimes_{\tilde{\bC}^\infty_X} \bB_{\HT,X}) \otimes_{\calO_X} \bB_{\HT,X}\to \calF \otimes_{\tilde{\bC}^\infty_X} \bB_{\HT,X}
\]
is an isomorphism. This then implies the same with $\bB_{\HT,X}$ replaced by $\calO \bB_{\HT,X}$, proving the claim.
\end{proof}

\begin{remark} \label{R:virtual comparison}
In the de Rham case, \cite{liu-zhu} includes a much stronger form of Theorem~\ref{T:generic HT DR} for \'etale $\Qp$-local systems, which asserts that the de Rham property at \emph{one} classical point implies the de Rham property at \emph{every} classical point in the same connected component. 

This stronger statement is applied in \cite{liu-zhu} to construct \emph{virtual comparison isomorphisms} in certain settings where one has an \'etale local system which is expected to arise from a family of motives, but this is only known in isolated cases. Such cases occur commonly in the theory of Shimura varieties.

Turning back to the Hodge-Tate setting, we note that Shimizu \cite{shimizu} has proved local constancy of generalized Hodge-Tate weights for a $\QQ_p$-local system; in particular, if these are integers at one point of $X$, they are integers everywhere on the same connected component. However, the Hodge-Tate property amounts to this integrality plus semisimplicity of the Sen operator, and it is unknown whether the latter condition spreads out in a similar fashion.
\end{remark}

\subsection{Fitting ideals in type \texorpdfstring{$\tilde{\bC}$}{C}}
\label{subsec:Fitting C}

We now extend the previous analysis to $\Gamma$-modules over rings of type $\tilde{\bC}$.

\begin{hypothesis} \label{H:Fitting C}
Throughout \S\ref{subsec:Fitting C},
assume that $K$ is perfectoid
and that $X = \Spa(A,A^+)$ is affinoid, 
let $\psi$ be a finite \'etale perfectoid tower over $X$,
and fix $r,s$ with $0 < s \leq r$.
\end{hypothesis}

We begin with a crucial argument to the effect that for $\Gamma$-modules in type $\tilde{\bC}$, the failure of projectivity can in some sense be isolated along horizontal divisors.

\begin{lemma} \label{L:no other torsion}
Suppose that $\psi$ is a restricted toric tower,
and let $M$ be a finitely generated $\bC^{[s,r]}_\psi$-module with $\Gamma$-structure. 
Then the module $M \otimes_{\tilde{\bC}^{[s,r]}_K} \Frac(\tilde{\bC}^{[s,r]}_K)$ over
$\bC^{[s,r]}_\psi \otimes_{\tilde{\bC}^{[s,r]}_K} \Frac(\tilde{\bC}^{[s,r]}_K)$ is finite projective.
\end{lemma}
\begin{proof}
We may assume that $K$ contains all $p$-power roots of unity.
Put $R = \bC^{[s,r]}_\psi$, $R_0 = \tilde{\bC}^{[s,r]}_K$, 
$F = \Frac(R_0)$, $R_F = R \otimes_{R_0} F$,
and $M_F = M \otimes_{R_0} F$.
View $R_F$ as a differential ring equipped with the derivations
$\frac{\partial}{\partial T_1},\dots,\frac{\partial}{\partial T_n}$.
Since $R_F$ is a regular ring by
Theorem~\ref{T:curve noetherian}, it is a \emph{nondegenerate differential ring}
in the sense of \cite[Definition~3.1.2]{kedlaya-goodformal2}. By \cite[Lemma~3.3.3]{kedlaya-goodformal2}, any finitely generated $R_F$-module admitting an $F$-linear connection is finite projective. Since $M_F$ carries such a connection defined by the action of the Lie algebra of $\Gamma$, it follows that $M_F$ is a finite projective $R_F$-module.
\end{proof}

\begin{lemma} \label{L:Gamma-stable ideal C}
Any nonzero $\Gamma$-stable ideal of $\tilde{\bC}^{[s,r]}_\psi$ contains
some nonzero element of $\tilde{\bC}^{[s,r]}_K$.
\end{lemma}
\begin{proof}
We may assume that $X$ is connected.
Suppose first that $\psi$ is a restricted toric tower, so in particular $X$ is smooth.
Let $I$ be a nonzero $\Gamma$-stable ideal of $\tilde{\bC}^{[s,r]}_\psi$. By Lemma~\ref{L:lift to psc},
there exist a pseudocoherent module $M$ with $\Gamma$-structure
over $\breve{\bC}^{[s,r]}_\psi$,
which we can take to be freely generated by one element,
 and a surjection
$M \otimes_{\breve{\bC}^{[s,r]}_\psi} \tilde{\bC}^{[s,r]}_\psi \to \tilde{\bC}^{[s,r]}_\psi/I$ of $\tilde{\bC}^{[s,r]}_\psi$-modules with $\Gamma$-structure.
By Lemma~\ref{L:no other torsion}, 
there exists some nonzero $t \in \tilde{\bC}^{[s,r]}_K$ such that
$M[t^{-1}]$ is a finite projective module over $\breve{\bC}^{[s,r]}_\psi[t^{-1}]$,
necessarily of rank $0$ because $X$ is connected.
It follows that $t \in I$.

To treat the general case, set notation as in Remark~\ref{R:resolution}.
Let $I$ be a $\Gamma$-stable ideal of $\tilde{\bC}^{[s,r]}_\psi$. 
Let $\calF$ be the sheaf of $\tilde{\bC}^{[s,r]}_X$-modules
corresponding to the module $\tilde{\bC}^{[s,r]}_\psi/I$. By the previous paragraph,
there exists some nonzero $t \in \tilde{\bC}^{[s,r]}_K$ annihilating
$f_{\proet}^* \calF$
(calculating the latter as usual using a finite \'etale perfectoid tower factoring through the pullback of $\psi$). Since $f$ is surjective, the fact that $\calF/t\calF$ pulls back to the zero sheaf
via $f_{\proet}$ implies that it is itself zero (by virtue of \cite[Lemma~2.3.12]{part1}
and Lemma~\ref{L:Robba v-covering}),
whence $t \in I$.
\end{proof}

\begin{cor} \label{C:invert projective}
Suppose that $X$ is affinoid. Let $\psi$ be a finite \'etale perfectoid tower over $X$,
and let $M$ be a finitely generated $\tilde{\bC}^{[s,r]}_\psi$-module with $\Gamma$-structure.
Then there exists some nonzero $t \in \tilde{\bC}^{[s,r]}_K$ such that
$M[t^{-1}]$ is a finite projective module over $\tilde{\bC}^{[s,r]}_\psi[t^{-1}]$.
\end{cor}
\begin{proof}
We may assume that $X$ is also connected. By Lemma~\ref{L:Gamma-stable ideal C},
the first nonzero Fitting ideal of $M$ contains a nonzero element $t$ of $\tilde{\bC}^{[s,r]}_K$, which has the desired effect.
\end{proof}

\begin{remark} \label{R:t-torsion-free}
Using Newton polygons, one verifies easily that for $t \in \tilde{\bC}^{[s,r]}_K$ nonzero,
$t$ is a zero-divisor in $\tilde{\bC}_X$
and that $\tilde{\bC}^{[s,r]}_X$ is $t$-adically separated.
\end{remark}

\begin{remark} \label{R:infinite orbit2}
Suppose that $K$ is an extension of an algebraic Galois extension $K_0$ with Galois group $G$
and that $X = X_0 \times_{K_0} K$ where $X_0$ is an affinoid space over $K_0$.
Let $\calF$ be a sheaf of $\tilde{\bC}^{[s,r]}_{X_0}$-modules represented by a finitely generated
module $M$ over $\tilde{\bC}^{[s,r]}_\psi$. 
By Lemma~\ref{L:Gamma-stable ideal C}, the first nonzero Fitting ideal of $M$ intersects
$\tilde{\bC}^{[s,r]}_K$ in a nonzero $G$-stable ideal.

Suppose now that $K_0$ is a $p$-adic field and $K$ is the completion of the $p$-cyclotomic extension.
By Remark~\ref{R:infinite orbit}, the only nonzero $G$-stable ideals are the powers of the ideals $\ker(\theta \circ \varphi^n)$ for those $n \in \ZZ$ for which $p^n \in [s,r]$;
consequently, the value of $t$ in Corollary~\ref{C:invert projective} may be taken to be a product of generators of such ideals.
\end{remark}

\subsection{Deformations of the base field}
\label{subsec:deformations}

In light of the preceding discussion, we discuss the multiplicative structure of $\tilde{\bC}^{[s,r]}_K$,
and the effect of quotienting by a prime ideal of this ring on a ring of the form $\tilde{\bC}^{[s,r]}_\psi$.
Note that rather different results appear if one works in a more ``homotopic'' manner; see Remark~\ref{R:effect of deformation}.

Throughout \S\ref{subsec:deformations}, continue to retain Hypothesis~\ref{H:Fitting C}.

\begin{lemma} \label{L:field factorization}
If $K$ is algebraically closed, then every nonzero element of $\tilde{\bC}^{[s,r]}_K$ 
factors as a product $t_1\cdots t_n$ in which for each $i$,
$K_i = \tilde{\bC}^{[s,r]}_K/(t_i)$ is a perfectoid field such that the images of $K$ and $K_i$ under the perfectoid correspondence are canonically identified (i.e., $K$ and $K_i$ have the same tilt).
\end{lemma}
\begin{proof}
The factorization into irreducibles is a consequence of 
$\tilde{\bC}^{[s,r]}_K$ being a principal ideal domain (Theorem~\ref{T:curve noetherian}(b)). The structure of irreducible elements follows from 
\cite[Theorem~6.8]{kedlaya-witt} (or see \cite{fargues-fontaine}).
\end{proof}

\begin{remark} \label{R:match norm}
Beware that in Lemma~\ref{L:field factorization}, matching the tilts of $K$ and $K_i$ as normed fields may require replacing the norm on $K_i$ with a power thereof. This has no notable effect on the following arguments.
\end{remark}

\begin{remark} \label{R:effect of deformation}
In Lemma~\ref{L:field factorization}, one possible irreducible factor
is a generator $t_\theta$ of the kernel of the theta map $\theta: \tilde{\bC}^{[s,r]}_K \to K$, which can occur whenever $1 \in [s,r]$. In classical $p$-adic Hodge theory, $t_\theta$ is enough to account for all possible torsion of pseudocoherent $(\varphi, \Gamma)$-modules (see Remark~\ref{R:infinite orbit} and Remark~\ref{R:infinite orbit2}).
A similar statement holds if one works exclusively with rigid spaces over a $p$-adic field; in this case, the structure of pseudocoherent $(\varphi, \Gamma)$-modules can be directly controlled in terms of pseudocoherent $\widehat{\calO}$-modules. 

By contrast,
for general $K$, other irreducible factors may occur.
For smooth spaces, there is a weak sense in which these behave like
$t_\theta$ for other choices of the field $K$; for instance, Vezzani \cite{vezzani} uses this 
type of argument to obtain some consequences for motives of rigid analytic spaces. However,
we will see below that for our purposes, other factors contribute essentially nothing to the module theory of
pseudocoherent $(\varphi, \Gamma)$-modules, so the situation is ultimately similar to that observed over a $p$-adic base field.
\end{remark}

\begin{defn}
For $0 < s \leq r$,
an element $x = \sum_{i \in \ZZ[p^{-1}]} x_i T^i \in K\{ (s/T)^{p^{-\infty}}, (T/r)^{p^{-\infty}}\}$ is
\emph{prepared} if there exists some $e \in \ZZ[p^{-1}] \cap [0, +\infty)$ such that:
\begin{gather*}
\left| x_0 \right| = \left| x_e \right| = 1;  \\
\left| x_i \right| \leq 1 \qquad ( i \in (0,e)); \\
x_i = 0 \qquad (i \notin [0,e]).
\end{gather*}
We refer to $e$ as the \emph{width} of $x$. Note that the product of prepared elements of widths $e_1,\dots,e_m$ is prepared of width $e_1 + \cdots + e_m$.
\end{defn}

\begin{lemma} \label{L:perfect Weierstrass prep}
For $0 < s \leq r$,
every nonzero element of $K\{(s/T)^{p^{-\infty}}, (T/r)^{p^{-\infty}}\}$ factors as a unit times a prepared element.
\end{lemma}
\begin{proof}
The proof of the Weierstrass preparation theorem for Tate algebras given in \cite[Proposition~8.3.2]{kedlaya-course} adapts without incident.
\end{proof}

\begin{lemma} \label{L:separate zeroes}
Let $x \in K\{T^{\pm p^{-\infty}}\}$ be prepared. Choose $\mu \in \gothm_K$ nonzero and
$\lambda_1,\dots,\lambda_m \in \gotho_K$ with distinct images in $\kappa_K$. 
\begin{enumerate}
\item[(a)]
There exists a unique factorization of $x$ into a product $x_1 \cdots x_m y$ with the property that for each $\alpha \in \calM(K\{T^{\pm p^{-\infty}}\})$,
\begin{itemize}
\item
if $\alpha(T - 1 + \lambda_i \mu) < \alpha(\mu)$ for some $i \in \{1,\dots,m\}$,
then $\alpha(x_j) \neq 0$ for $j \in \{1,\dots,m\} \setminus \{i\}$ and
$\alpha(y) \neq 0$;
\item
otherwise, $\alpha(x_j) \neq 0$ for $j \in \{1,\dots,m\}$. 
\end{itemize}
\item[(b)]
Suppose that $p^{1/(p-1)} < \left| \mu \right| < 1$.
For any $\nu \in \gothm_K$ and any nonnegative integer $k$, if
$x$ is congruent modulo $\nu$ to an element of $\gotho_K[T^{ p^{-k}}]$, then so are $x_1,\dots,x_m,y$.
\end{enumerate}
\end{lemma}
\begin{proof}
Both parts may be deduced using a suitably strong version of Hensel's lemma, such as
\cite[Theorem~2.2.2]{kedlaya-course}.
\end{proof}
\begin{remark}
In Lemma~\ref{L:separate zeroes}, the lower bound on the norm of $\mu$ is best possible in light of the example $x = T-1$.
\end{remark}

\begin{lemma} \label{L:no gamma kernel}
Let $L$ be the tilt of $K$ and choose $\overline{q} \in 1 + \gothm_L$
such that $q = \theta([\overline{q}])$ is not a root of unity.
Equip $K\{T^{\pm p^{-\infty}}\}$ with the action of $\Gamma = \ZZ_p$ given by
\[
\gamma\left( \sum_i x_i T^i \right) = \sum_i x_i \theta([\overline{q}^{\gamma i}]) T^i.
\]
Let $\alpha \in \calM(K\{T^{\pm p^{-\infty}}\})$ be the seminorm given by quotienting by the ideal
$(T-1, T^{1/p}-1, \dots)$.
Then 
\[
\bigcap_{\gamma \in \Gamma} \ker(\gamma^*(\alpha)) = 0.
\]
\end{lemma}
\begin{proof}
By replacing $q$ with a suitable $p$-power root, we may reduce to the case where $\left| 1-q \right| > p^{1/(p-1)}$.
Suppose to the contrary that $x$ is a nonzero element of the intersection;
by Lemma~\ref{L:perfect Weierstrass prep}, we may assume in addition that $x$ is prepared.
Define the Frobenius lift $\varphi$ on $K\{T^{\pm p^{-\infty}}\}$ by
\[
\varphi\left( \sum_i x_i T^i \right) = \sum_i x_i T^{pi}.
\]
Define a sequence $x = x_0, x_1, \dots$ of elements of the intersection 
of width at most $e$
as follows.
Using Lemma~\ref{L:separate zeroes}, 
factor $x_l$ as a product $x_{l,0} \cdots x_{l,p-1} y_l$
of prepared elements with the property that for $\beta \in \calM(K\{T^{\pm p^{-\infty}}\})$,
\begin{itemize}
\item
if $\beta(T-q^i) < \left| 1-q \right|$ for some (necesssarily unique) $i \in \{0,\dots,p-1\}$, then $\beta(x_{l,j}) > 0$ for $j \in \{0,\dots,p-1\} \setminus \{i\}$
and $\beta(y_l) > 0$;
\item
otherwise, $\beta(x_{l,i}) > 0$ for $i \in \{0,\dots,p-1\}$.
\end{itemize}
In particular, $\alpha(x_{l,0}) = 0$; more generally, for $\gamma \in \Gamma$ corresponding to an element of $\ZZ_p$ congruent to $i$ modulo $p$, we have $\alpha(\gamma^{-1}(x_{l,i})) = 0$.

Since the widths of $x_{l,0},\dots,x_{l,p-1},y$ add up to at most $e$, we can find $i \in \{0,\dots,p-1\}$ such that $x_{l,i}$ has width at most $e/p$. Choose some such $i$, let $\gamma_i \in \Gamma_p$ be the element corresponding to $i \in \ZZ_p$, and put
\[
x_{l+1} = (\varphi \circ \gamma_i^{-1})(x_{l,i});
\]
then again $\alpha(\gamma(x_{l+1})) = 0$ for all $\gamma \in \Gamma$.

By Lemma~\ref{L:separate zeroes}(b), if $x_l$ is congruent modulo some $\nu \in \gothm_K$
to an element of $\gotho_K[T^{p^{-k}}]$ for some positive integer $k$, then $x_{l+1}$ is congruent modulo $\nu$ to an element of $\gotho_K[T^{\pm p^{1-k}}]$. In particular, for any fixed nonzero $\nu \in \gothm_K$, for $l$ sufficiently large, $x_l$ is congruent to an element of $K[T^{\pm}]$. In particular, the width of $l$ is bounded by $m = \lfloor e \rfloor$.

Since $q$ is not a root of unity, we can choose $\nu$ to be nonzero but with norm less than $\left| (1-q)\cdots(1 - q^m) \right|$.
By the previous paragraph, for each sufficiently large $l$, there exists a polynomial $P(T) \in \gotho_K[T]$ of degree at most $m$ such that $\left| P - x_l \right| \leq \mu$. 
Write 
\[
P(T) = \sum_{i=0}^m P_i (T-1)\cdots (T-q^{i-1}) \qquad (P_i \in \gotho_K).
\]
Since $x_l$ belongs to the intersection of the kernels, we have $\left| P(q^i) \right| \leq \nu$ for all $i \in \ZZ_p$, so in particular for $i=0,\dots,m$. However, this implies that
\[
\left| P_i \right| \leq \frac{\nu}{\left| (1-q)\cdots (1-q^i) \right|} \qquad (i=0,\dots,m),
\]
but this contradicts the fact that $x_l$ is prepared and so $\left| x_l \right| = 1$.
This contradiction completes the proof.
\end{proof}

\begin{prop} \label{P:twisted stable submodules}
Suppose that $X$ is connected and
$t \in \tilde{\bC}^{[s,r]}_K$ is irreducible and coprime to $\varphi^m(t_\theta)$ for all $m \in \ZZ$.
Let $I$ be a nonzero ideal of $\tilde{\bC}^{[s,r]}_{\psi}/(t)$
whose two base extensions to $\tilde{\bC}^{[s,r]}_{\psi^{1}}/(t)$ are the same ideal. Then $I$ is the unit ideal.
\end{prop}
\begin{proof}
By  \cite[Corollary~2.3.5]{part1},
it suffices to derive a contradiction under the assumption that $I$ is nonzero and there exists $\alpha \in \calM(\tilde{\bC}^{[s,r]}_{\psi}/(t))$ with $I \subseteq \ker(\alpha)$.
We proceed by induction on $n = \dim(X)$, the case $n=0$ being immediate from
Lemma~\ref{L:field factorization}. 

To treat the case $n=1$, we first reduce to the case where $K$ is algebraically closed,
so that by Lemma~\ref{L:field factorization},
$K' = \tilde{\bC}^{[s,r]}_K/(t)$ is a perfectoid field.
We then replace $X$ by its normalization (which is smooth),
then work locally to reduce to the case where $X$ is 
finite \'etale over $K\{s/T, T/r\}$ for some $0 < s \leq r$.
(Note that \emph{a priori} one must work with a rational localization of $K\{s/T, T/r\}$, but
using Belyi maps in characteristic $p$ one can further reduce to this setting.)
In particular, $X$ is the base of a restricted toric tower,
which we may as well take $\psi$ to factor through. 
Using Remark~\ref{R:perfectoid tower splitting Robba}, we may further reduce to the case where
$\psi$ itself is the restricted toric tower.
In this case, by Theorem~\ref{T:Fontaine perfectoid compatibility},
$\tilde{\bC}^{[s,r]}_{\psi}/(t)$ is finite \'etale over 
$K'\{(s/T)^{p^{-\infty}}, (T/r)^{p^{-\infty}}\}$.
By replacing $K$ with a suitable base extension, we may enlarge $K'$ to include an element $z$ such that $\alpha(T-z) = 0$;
by rescaling $T$, we may reduce to the case $z = 1$.
Since $t$ is coprime to $\varphi^m(t_\theta)$ for all $m \in \ZZ$, the action of a topological generator of $\Gamma \cong \ZZ_p$ on $\tilde{\bC}^{[s,r]}_\psi/(t)$ must take the form $T \mapsto [\overline{q}] T$ for some $\overline{q} \in 1 + \gothm_{L}$ such that $\theta([\overline{q}])$ is not a root of unity.
However, we now obtain the desired contradiction by taking norms to
$K'\{(s/T)^{p^{-\infty}}, (T/r)^{p^{-\infty}}\}$ and applying Lemma~\ref{L:no gamma kernel}.
 
To treat the general case, we first use Remark~\ref{R:resolution} to reduce to the case where
$X$ is smooth. In this case, there exists a smooth morphism $X \to X'$ with $X'$ smooth affinoid of dimension
$n-1$; we may further assume that this map admits a section.
As in the previous paragraph, we may further reduce to the case where $\psi$
is a relative toric tower over a restricted toric tower $\psi'$ with base $X'$.
Project $\alpha$ to $\alpha' \in \calM(\tilde{\bC}^{[s,r]}_{\psi'}/(t))$;
via Theorem~\ref{T:Fontaine perfectoid correspondence}, this point corresponds to some
$\beta \in \calM(\tilde{A}_{\psi'})$.
Let $\psi_\beta$ be the pullback of $\psi$ along $\calH(\beta) \to X'$; by
applying the dimension 1 case, we deduce that the image of $I$ in
$\tilde{\bC}^{[s,r]}_{\psi_\beta}$ must vanish.
We then apply the dimension $n-1$ case to the chosen section of the map $X \to X'$
to obtain the desired contradiction.
\end{proof}

\begin{cor} \label{C:twisted projective torsion}
Suppose that $t \in \tilde{\bC}^{[s,r]}_K$ is irreducible and coprime to $\varphi^m(t_\theta)$ for all $m \in \ZZ$.
Let $M$ be a $\tilde{\bC}^{[s,r]}_\psi/(t)$-module contained in some finitely generated module.
If $M$ admits a $\Gamma$-structure, then $M$ is finite projective.
\end{cor}
\begin{proof}
We may assume $X$ is connected.
Let $I$ be the set of $f \in \tilde{\bC}^{[s,r]}_\psi/(t)$ for which there exists an injective homomorphism $M \to N$ of $\tilde{\bC}^{[s,r]}_\psi/(t)$-modules such that $N$ is finitely generated and $M_f \to N_f$ splits. By Corollary~\ref{C:splitting ideal},
$I$ is an ideal of $\tilde{\bC}^{[s,r]}_\psi/(t)$ not contained in any minimal prime ideal.
Since $I$ is defined intrinsically in terms of $M$, $I$ is a $\Gamma$-submodule of $\tilde{\bC}^{[s,r]}_\psi/(t)$. By Proposition~\ref{P:twisted stable submodules}, $I$ must equal the unit ideal, proving that $M$ is finitely generated.
By Proposition~\ref{P:twisted stable submodules} again,
each Fitting ideal of $M$ is either the zero ideal or the unit ideal; this proves the claim.
\end{proof}

\subsection{Applications to \texorpdfstring{$\Gamma$}{Gamma}-modules}
\label{subsec:applications Gamma}

We continue to derive properties of $\Gamma$-modules over rings of the form $\tilde{\bC}^{[s,r]}_\psi$,
in parallel with the corresponding results over $\tilde{A}_\psi$.
Throughout \S\ref{subsec:applications Gamma}, continue to retain Hypothesis~\ref{H:Fitting C}.

\begin{remark} \label{R:pseudocoherent factor}
Before continuing, let us summarize the results so far.
Let $M$ be a finitely generated $\tilde{\bC}^{[s,r]}_\psi$-module with $\Gamma$-structure.
By Corollary~\ref{C:invert projective}, there exists some nonzero $t \in \tilde{\bC}^{[s,r]}_K$
such that $M_t$ is a finite projective module over $\tilde{\bC}^{[s,r]}_\psi[t^{-1}]$.
Factor $t = t_1 t_2$ where $t_1$ is a product of powers of $\varphi^n(t_\theta)$
for $n \in \ZZ$ (for $t_\theta$ as in Remark~\ref{R:effect of deformation}) and $t_2$
is coprime to $\varphi^n(t_\theta)$ for all $n \in \ZZ$. Then on one hand, 
by Corollary~\ref{C:twisted projective torsion},
the failure of projectivity of $M$ at the support of $t_2$ is limited to $t_2$-power torsion;
on the other hand, the failure of projectivity at the support of $t_1$ can be quite complicated,
but can be understood completely in terms of our results on the category $\calC_X$.
\end{remark}

\begin{prop} \label{P:stable finite2}
Every finitely generated $\tilde{\bC}^{[s,r]}_\psi$-module with $\Gamma$-structure is \'etale-stably pseudocoherent, and hence a pseudocoherent $\Gamma$-module.
\end{prop}
\begin{proof}
Let $M$ be a finitely generated $\tilde{\bC}^{[s,r]}_\psi$-module with $\Gamma$-structure and set
notation as in Remark~\ref{R:pseudocoherent factor}.
Let $T$ be the $t$-power torsion submodule of $M$,
so that $M_t = (M/T)_t$ and $(M/T)[t] = 0$; 
by Lemma~\ref{L:pseudocoherent 2 of 3}, to prove that $M$ is pseudocoherent,
it will suffice to prove that $T$ and $M/T$ are pseudocoherent.
We first observe that 
\[
(M/T)/t(M/T) = (M/T)/t_1(M/T) \oplus (M/T)/t_2(M/T)
\]
is pseudocoherent: namely, Proposition~\ref{P:stable finite} implies the pseudocoherence
of $(M/T)/t_1(M/T)$ while Corollary~\ref{C:twisted projective torsion} does likewise for 
$(M/T)/t_2(M/T)$.
By Lemma~\ref{L:pseudocoherent glueing}(a)
(and Remark~\ref{R:t-torsion-free}), it follows that $M/T$ is pseudocoherent. 
Lemma~\ref{L:pseudocoherent 2 of 3} then
implies that $T$ is finitely generated; in particular, $T$ is annihilated by
$t^m$ for some positive integer $m$. 
Using Lemma~\ref{L:pseudocoherent 2 of 3} and the previous argument, we deduce by induction on $i$ that
$t^{i-1}T/t^i T$ is pseudocoherent for $i=1,\dots,m$. Using Lemma~\ref{L:pseudocoherent 2 of 3} once more,
we conclude that $T$ is pseudocoherent, as then is $M$.
As in the proof of Proposition~\ref{P:stable finite}, we may use Corollary~\ref{C:closure finitely generated} to deduce that $M$ is strictly pseudocoherent, then repeat the argument after a base extension to deduce that $M$ is \'etale-stably pseudocoherent.
\end{proof}

\begin{prop} \label{P:ascending sequence2}
Let $M$ be a pseudocoherent $\Gamma$-module over $\tilde{\bC}^{[s,r]}_\psi$. Then
any ascending sequence of $\Gamma$-submodules of $M$ stabilizes.
\end{prop}
\begin{proof}
It is equivalent to verify that any sequence $M = M_0 \to M_1 \to \cdots$ of surjective morphisms of $\Gamma$-modules stabilizes. 
We may assume that $X$ is connected.
By Corollary~\ref{C:invert projective}, for each $i$,
there exists some nonzero $t \in \tilde{\bC}^{[s,r]}_K$ (\emph{a priori} depending on $i$)
such that $(M_i)_t$ is a finite projective module over $\tilde{\bC}^{[s,r]}_\psi[t^{-1}]$,
necessarily of some constant rank $m_0$ because $X$ is connected.
Since the $m_i$ form a nonincreasing sequence of nonnegative integers, they must stabilize;
by discarding some initial terms of the sequence, we may assume that the $m_i$ are all equal to a single value $m$.

Now set notation as in Remark~\ref{R:pseudocoherent factor}, taking $t_1$ to be 
divisible by $\varphi^n(t_\theta)$ for each $n$ for which $p^{-n} \in [s,r]$.
For any nonzero $t' \in \tilde{\bC}^{[s,r]}_K$ coprime to $t$,
by Corollary~\ref{C:twisted projective torsion},
each module $M_i/t'M_i$ is finite projective over $\tilde{\bC}^{[s,r]}_K/(t')$ of some rank.
The latter rank is necessarily greater than or equal to the generic rank of $M_i$,
with equality for $i=0$ because $M$ is $t'$-torsion-free; since $M/t'M \to M_i/t'M_i$ is surjective
and the generic ranks of $M_i$ are all equal, it follows that the rank of
$M_i/t'M_i$ over $\tilde{\bC}^{[s,r]}_K/(t')$ must always be equal to the generic rank of $M_i$.
In particular, if we apply Corollary~\ref{C:invert projective} to choose $t'$ so that
$(M_i)_{tt'}$ is finite projective over $\tilde{\bC}^{[s,r]}_\psi[(tt')^{-1}]$, we may then deduce that also $(M_i)_t$ is finite projective over $\tilde{\bC}^{[s,r]}_\psi[t^{-1}]$. That is, the choice of $t$
in Remark~\ref{R:pseudocoherent factor} can be made uniformly over the $M_i$.

The sequence $(M_0)_t \to (M_1)_t \to \cdots$ now consists of 
finite projective modules over $\tilde{\bC}^{[s,r]}_\psi[t^{-1}]$ of rank $m$, so 
it consists entirely of isomorphisms.
By Proposition~\ref{P:ascending sequence2}, the sequence $M_0/t_1M_0 \to M_1/t_1M_1 \to \cdots$ stabilizes;
by Corollary~\ref{C:twisted projective torsion}, the sequence $M_0/t_2M_0 \to M_1/t_2 M_1 \to \cdots$ stabilizes.
Putting these two statements together, we see that the sequence $M_0/tM_0 \to M_1/t M_1 \to \cdots$ stabilizes.
By Proposition~\ref{P:stable finite2}, each $M_i$ is pseduocoherent,
as is the quotient of $M_i$ by its $t$-power-torsion submodule;
we may thus apply Lemma~\ref{L:stable sequence} to deduce that the original sequence stabilizes.
\end{proof}

\begin{prop} \label{P:ext tor stability2}
Let $M$ and $N$ be pseudocoherent $\Gamma$-modules over $\tilde{\bC}^{[s,r]}_\psi$. Then
for all $i \geq 0$, $\Ext^i(M, N)$ and
$\Tor_i(M,N)$ (with the ring label $\tilde{\bC}^{[s,r]}_\psi$ suppressed for convenience)
are pseudocoherent $\Gamma$-modules over $\tilde{\bC}^{[s,r]}_\psi$.
\end{prop}
\begin{proof}
By Corollary~\ref{C:invert projective},
there exists some nonzero $t \in \tilde{\bC}^{[s,r]}_K$
such that $M_t, N_t$ are finite projective modules over $\tilde{\bC}^{[s,r]}_\psi[t^{-1}]$,
Note that $\Ext^i(M, N)_t$ and $\Tor_i(M,N)_t$ are projective over $\tilde{\bC}^{[s,r]}_\psi[t^{-1}]$
(and in fact zero for $i>0$).
By Remark~\ref{R:t-torsion-free} and Lemma~\ref{L:pseudocoherent glueing}, it suffices to check that
the kernel and cokernel of multiplication by $t$ on $\Ext^i(M, N),
\Tor_i(M, N)$ are pseudocoherent
(and hence \'etale-stably pseudocoherent by Proposition~\ref{P:stable finite2}). We work out the case of $\Ext$ in detail, the case of $\Tor$ being similar.

By Remark~\ref{R:transfer pseudocoherence to quotient}, 
$M[t]$ and $M/tM$ are pseudocoherent.
Factor $t$ as in Remark~\ref{R:pseudocoherent factor}.
By Proposition~\ref{P:ext tor stability} (to treat the $t_1$-torsion component)
and Corollary~\ref{C:twisted projective torsion} (to treat the $t_2$-torsion component),
the modules
\[
\Ext^i(M[t], N),
\Ext^i(M/tM, N),
\]
are pseudocoherent. 
(More precisely, this is true when the Ext groups are computed in the category of
$\tilde{\bC}^{[s,r]}_\psi/(t)$-modules with the target being $N/tN$; however, since $t$ is not a zero-divisor in $\tilde{\bC}^{[s,r]}_\psi$, this immediately implies the statement as made.)
Now apply derived functors to the exact sequences
\[
0 \to M[t] \to M \to tM \to 0, \qquad 0 \to tM \to M \to M/tM \to 0
\]
to obtain exact sequences
\begin{gather*}
\Ext^{i-1}(M[t], N) \to \Ext^i(tM, N) \to \Ext^i(M, N) \to \Ext^i(M[t], N),
\\
\Ext^i(M/tM, N) \to \Ext^i(M, N) \to \Ext^i(tM, N) \to \Ext^{i+1}(M/tM, N).
\end{gather*}
By Proposition~\ref{P:stable finite2} (and Lemma~\ref{L:pseudocoherent 2 of 3}),
pseudocoherent $\Gamma$-modules over $\tilde{\bC}^{[s,r]}_\psi/(t)$ form a Serre subcategory of $\tilde{\bC}^{[s,r]}_\psi/(t)$-modules with semilinear $\Gamma$-action; in the quotient category, the morphisms $\Ext^i(tM, N) \to \Ext^i(M, N)$,
$\Ext^i(M, N) \to \Ext^i(tM, N)$ become isomorphisms.
Hence multiplication by $t$ on $\Ext^i(M,N)$ becomes an isomorphism in the quotient category, that is, its kernel and cokernel are pseudocoherent.
\end{proof}

\begin{prop} \label{P:stable finite sub2}
Let $M$ be a $\tilde{\bC}^{[s,r]}_\psi$-module satisfying the following conditions.
\begin{enumerate}
\item[(a)]
There exists a $\Gamma$-structure on $M$.
\item[(b)]
For every $t \in \tilde{\bC}^{[s,r]}_K$ which is not a unit, $M/tM$ is finitely generated
(and hence \'etale-stably pseudocoherent by condition (a) and Proposition~\ref{P:stable finite2}).
\item[(c)]
The module $M$ can be realized as a submodule of some finite $\tilde{\bC}^{[s,r]}_\psi$-module
(not necessarily admitting a $\Gamma$-structure).
\item[(d)]
One of the following conditions holds.
\begin{enumerate}
\item[(i)]
In (c), the finite $\tilde{\bC}^{[s,r]}_\psi$-module can be taken to be complete for the natural topology
and $M$ can be taken to be a closed submodule.
\item[(ii)]
For every $t \in \tilde{\bC}^{[s,r]}_K$ which is not a unit, $M[t^\infty]$ is finitely generated.
\end{enumerate}

\end{enumerate}
Then $M$ is a pseudocoherent $\Gamma$-module.
\end{prop}
\begin{proof}
Let $I$ be the set of $f \in \tilde{\bC}^{[s,r]}_\psi$ for which
exists a homomorphism $M \to N$ of $\tilde{\bC}^{[s,r]}_\psi$-modules such that $N$ is finitely generated 
(but not necessarily complete for the natural topology) and $M_f \to N_f$ is a split inclusion. By condition (c) and Corollary~\ref{C:splitting ideal}, we see that $I$ is an ideal not contained in any minimal prime ideal. Since $I$ is defined intrinsically in terms of $M$, it inherits a $\Gamma$-structure by condition (a); by Lemma~\ref{L:Gamma-stable ideal C}, $I$ contains some nonzero $t \in \tilde{\bC}^{[s,r]}_K$.
By condition (b) and Lemma~\ref{L:splitting to finite}, $M/M[t^\infty]$ is finitely generated,
and hence \'etale-stably pseudocoherent by Proposition~\ref{P:stable finite}.
From here, we may continue as in the proof of Proposition~\ref{P:stable finite sub} to conclude.
\end{proof}

\subsection{Applications to \texorpdfstring{$(\varphi,\Gamma)$}{(phi, Gamma)}-modules}
\label{subsec:final applications}

To conclude, we derive some properties of categories of $\Gamma$-modules and $(\varphi, \Gamma)$-modules in type $\tilde{\bC}$.

\begin{defn}
For $0 < s \leq r$, let $\calC^{[s,r]}_X$ denote the category of pseudocoherent
$\tilde{\bC}^{[s,r]}_X$-modules on $X_{\proet}$. As in Definition~\ref{D:canonical section}, we identify $\calC_X$ with a full subcategory of $\calC^{[s,r]}_X$ whenever $1 \in [s,r]$.
For $S \subseteq X$, let $\calC^{[s,r]}_{X,S}$ be the direct 2-limit of $\calC^{[s,r]}_{U}$ over all open subspaces $U$ of $X$ containing $x$; by Proposition~\ref{P:Robba weak flatness}, the transition functors are exact.

We say that an object of $\calC^{[s,r]}_X$ or $\calC^{[s,r]}_{X,S}$ is \emph{$\theta$-local} if locally on $X$ (or equivalently, on every quasicompact open subset of $X$), it is annihilated by some finite product $\prod_i \varphi^{n_i}(t_\theta)$ for some $n_i \in \ZZ$.
\end{defn}

\begin{theorem} \label{T:abelian category gamma}
The category $\calC^{[s,r]}_X$ has the following properties.
\begin{enumerate}
\item[(a)]
It is a full abelian subcategory of the category of sheaves of $\tilde{\bC}^{[s,r]}_X$-modules on $X_{\proet}$. Moreover, the formation of kernels and cokernels in $\calC^{[s,r]}_X$ is compatible with the larger category.
\item[(b)]
If $X$ is quasicompact, then the ascending chain condition holds: given any sequence
$\calF_0 \to \calF_1 \to \cdots$ of epimorphisms in $\calC^{[s,r]}_X$, there exists $i_0 \geq 0$ such that for all $i \geq i_0$, the map $\calF_i \to \calF_{i+1}$ is an isomorphism.
\item[(c)]
The functors $\Tor_i$ take $\calC^{[s,r]}_X \times \calC^{[s,r]}_X$ into $\calC^{[s,r]}_X$. In particular, this is true for $i=0$, so $\calC^{[s,r]}_X$ admits tensor products. Moreover, for $i=1$ (resp.\ $i>1$), any sheaf in the essential image of $\Tor_i$ is 
$\tilde{\bC}^{[s,r]}_K$-torsion (resp.\ $\theta$-local).
\item[(d)]
The functors $\Ext^i$ take $\calC^{[s,r]}_X \times \calC^{[s,r]}_X$ into $\calC^{[s,r]}_X$. In particular, this is true for $i=0$, so $\calC^{[s,r]}_X$ admits internal Homs. Moreover, for $i=1$ (resp.\ $i>1$), any sheaf in the essential image of $\Ext^i$ is 
$\tilde{\bC}^{[s,r]}_K$-torsion (resp.\ $\theta$-local).
\end{enumerate}
\end{theorem}
\begin{proof}
The claim is local on $X$, so we may assume that $X = \Spa(A,A^+)$ is affinoid.
We may then deduce (a) from Proposition~\ref{P:stable finite2},
(b) from Proposition~\ref{P:ascending sequence2},
and (c) and (d) from Proposition~\ref{P:ext tor stability2}
and (for the $\theta$-locality assertions)
Corollary~\ref{C:twisted projective torsion}.
\end{proof}

\begin{theorem} \label{T:vector bundle pullback gamma}
For $f: X' \to X$ a morphism and $i \geq 0$, the functors
$L_i f_{\proet}^*$ take $\calC^{[s,r]}_X$ into $\calC^{[s,r]}_{X'}$.
Moreover, for $i>0$, any sheaf in the essential image of $L_i f_{\proet}^*$ is $\theta$-local.
\end{theorem}
\begin{proof}
Using Proposition~\ref{P:stable finite2} (as used in the proof of
Theorem~\ref{T:abelian category gamma})
in place of Proposition~\ref{P:stable finite},
the proof of Theorem~\ref{T:vector bundle pullback} carries over to prove the first statement.
The second statement follows from the first statement plus Corollary~\ref{C:invert projective}
and Corollary~\ref{C:twisted projective torsion};
note that in contrast with Theorem~\ref{T:abelian category gamma}(c), $L_1 f_{\proet}^{*}$ cannot have any prime-to-$t_\theta$-torsion because the base extension is between rings with no such torsion.
\end{proof}

\begin{theorem} \label{T:local acc gamma}
Let $S$ be a subset of $X$ such that the ideal sheaves of the localization of $X$ at $S$ satisfy the ascending chain condition. (By Proposition~\ref{P:stalk noetherian}, this holds if $S$ is a singleton set.)
Then the ascending chain condition holds for $\calC^{[s,r]}_{X,S}$: given any sequence $\calF_0 \to \calF_1 \to \cdots$
of epimorphisms in $\calC^{[s,r]}_{X,S}$, there exists $i_0 \geq 0$ such that for all $i \geq i_0$, the map $\calF_i \to \calF_{i+1}$ is an isomorphism.
\end{theorem}
\begin{proof}
By following the proof of Proposition~\ref{P:ascending sequence2}, we reduce to 
the case where all of the objects are killed by $t_\theta$, which is Theorem~\ref{T:local acc}.
\end{proof}

\begin{defn}
For $a$ a positive integer, let $\CPhi^a_X$ denote the category of pseudocoherent $\varphi^a$-modules over $\tilde{\bC}_X$. As in Definition~\ref{D:canonical section}, we identify $\calC_X$ with a full subcategory of $\CPhi_X^a$.
For $S \subseteq X$, let $\CPhi^a_{X,S}$ be the direct 2-limit of $\CPhi^a_{U}$ over all open subspaces $U$ of $X$ containing $S$; by Proposition~\ref{P:Robba weak flatness} and Theorem~\ref{T:perfect generalized phi-modules}, the transition functors are exact.

Let $I_\theta$ be the $\varphi$-stable ideal of $\tilde{\bC}_X$ corresponding to the product of the ideals
$(\varphi^n(t_\theta))$ on $\tilde{\bC}^{[s,r]}_X$ over the finitely many values $n \in \ZZ$
for which the ideals are nontrivial.
We say that an object of $\CPhi^a_X$ or $\CPhi^a_{X,S}$ is \emph{$\theta$-local} if locally on $X$ (or equivalently, on every quasicompact open subset of $X$), it is annihilated by some power of $I_\theta$.
(Note that if $a>1$, then $I_\theta$ factors as a product of $a$ different $\varphi^a$-stable ideals.)
\end{defn}

\begin{theorem} \label{T:abelian category phi}
The category $\CPhi^a_X$ has the following properties.
\begin{enumerate}
\item[(a)]
It is a full abelian subcategory of the category of sheaves of $\tilde{\bC}_X$-modules on $X_{\proet}$. Moreover, the formation of kernels and cokernels in $\CPhi^a_X$ is compatible with the larger category.
\item[(b)]
If $X$ is quasicompact, then the ascending chain condition holds: given any sequence
$\calF_0 \to \calF_1 \to \cdots$ of epimorphisms in $\CPhi^a_X$, there exists $i_0 \geq 0$ such that for all $i \geq i_0$, the map $\calF_i \to \calF_{i+1}$ is an isomorphism.
\item[(c)]
The functors $\Tor_i$ take $\CPhi^a_X \times \CPhi^a_X$ into $\CPhi^a_X$. In particular, this is true for $i=0$, so $\CPhi^a_X$ admits tensor products. Moreover, 
for $i=1$ (resp.\ $i>1$), any sheaf in the essential image of $\Tor_i$ is 
locally-on-$X$ $\tilde{\bC}_K$-torsion (resp.\ $\theta$-local).
\item[(d)]
The functors $\Ext^i$ take $\CPhi^a_X \times \CPhi^a_X$ into $\CPhi^a_X$. In particular, this is true for $i=0$, so $\CPhi^a_X$ admits internal Homs.  Moreover, 
for $i=1$ (resp.\ $i>1$), any sheaf in the essential image of $\Ext^i$ is 
locally-on-$X$ $\tilde{\bC}_K$-torsion (resp.\ $\theta$-local).
\end{enumerate}
\end{theorem}
\begin{proof}
Using Theorem~\ref{T:perfect generalized phi-modules},
this reduces at once to Theorem~\ref{T:abelian category gamma}.
\end{proof}

\begin{theorem} \label{T:vector bundle pullback phi}
For $f: X' \to X$ a morphism and $i \geq 0$, the functors
$L_i f_{\proet}^*$ take $\CPhi^a_X$ into $\CPhi^a_{X'}$.
Moreover, for $i>0$, any sheaf in the essential image of $L_i f_{\proet}^*$ is $\theta$-local.
\end{theorem}
\begin{proof}
Using Theorem~\ref{T:perfect generalized phi-modules},
this reduces at once to Theorem~\ref{T:vector bundle pullback gamma}.
\end{proof}

\begin{theorem} \label{T:local acc phi}
Let $S$ be a subset of $X$ such that the ideal sheaves of the localization of $X$ at $S$ satisfy the ascending chain condition.
 (By Proposition~\ref{P:stalk noetherian}, this holds if $S$ is a singleton set.)
Then the ascending chain condition holds for $\CPhi^a_{X,S}$: given any sequence $\calF_0 \to \calF_1 \to \cdots$
of epimorphisms in $\CPhi^a_{X,S}$, there exists $i_0 \geq 0$ such that for all $i \geq i_0$, the map $\calF_i \to \calF_{i+1}$ is an isomorphism.
\end{theorem}
\begin{proof}
Using Theorem~\ref{T:perfect generalized phi-modules},
this reduces at once to Theorem~\ref{T:local acc gamma}.
\end{proof}

\begin{remark}
In the case where $X$ is smooth, one can derive parts (a) of Theorem~\ref{T:abelian category gamma}
and Theorem~\ref{T:abelian category phi} using descent from $\tilde{\bC}^{[s,r]}_\psi$ to $\breve{\bC}^{[s,r]}_\psi$
in the case where $\psi$ is a restricted toric tower (Lemma~\ref{L:Gamma equivalences2}, Theorem~\ref{T:add tilde2}, and Theorem~\ref{T:standard toric decompleting}) and the fact that $\breve{\bC}^{[s,r]}_\psi$ is a coherent ring (Remark~\ref{R:strongly noetherian imperfect}). However, it is less straightforward to obtain
Theorem~\ref{T:abelian category gamma}(b) using this method.
\end{remark}

\begin{remark}
Using the ascending chain condition (Theorem~\ref{T:abelian category phi}(b)),
one may formally imitate the proof of the Artin-Rees lemma for noetherian rings and its corollaries \cite[Chapter~10]{atiyah-macdonald}, yielding the following statements (among others).
\begin{itemize}
\item
For $X$ quasicompact, the category of graded sheaves on $X_{\proet}$ in which each summand belongs to $\calC_X^{[s,r]}$ (resp.\ $\CPhi^a_X$) satisfies the ascending chain condition. (Imitate the standard proof of the Hilbert basis theorem.)
\item
For $X$ affinoid and $\calI$ an ideal subsheaf of $\tilde{\bC}^{[s,r]}_X$ (resp.\ a $\varphi^a$-stable ideal subsheaf of $\tilde{\bC}_X$), completion with respect to $\calI$ defines an exact functor on $\calC_X^{[s,r]}$ (resp.\ $\CPhi_X^a$) which coincides with base extension to the $\calI$-adic  completion of the base ring.
\item
Krull's theorem: for $X$ affinoid, $\calI$ an ideal subsheaf of $\tilde{\bC}^{[s,r]}_X$ (resp.\ a $\varphi^a$-stable ideal subsheaf of $\tilde{\bC}_X$), and $\calF$ an object of $\calC_X^{[s,r]}$ (resp.\ $\CPhi_X^a$), the kernel of the map from $\calF$ to its $\calI$-adic completion consists of the union of the subsheaves annihilated by sections of $1+\calI$.
\end{itemize}
We leave details to the interested reader.
\end{remark}

\begin{theorem} \label{T:limit to Zariski closed}
Let $S$ be a Zariski-closed subset of $X$. Then the functor $\CPhi_{X,S}^a \to \CPhi_S^a$ induces an equivalence on the subcategories of prime-to-$I_\theta$-torsion subobjects.
\end{theorem}
\begin{proof}
By Theorem~\ref{T:perfect Robba Kiehl}
 and Corollary~\ref{C:twisted projective torsion}, the objects in both subcategories are sheaves for the v-topology, so we may check the claim after resolving singularities as in Remark~\ref{R:resolution2}; we may thus assume that $X$ is smooth and $S$ is a normal crossings divisor in $X$. By considering the components of $S$ individually, we may further reduce to case where $S$ is a smooth divisor. By working locally, we may 
assume that $X$ is the base of a restricted toric tower $\psi$ and that $S$ is the zero locus of $T_n-1$,
so in particular there exists a projection $X \to S$ of which the inclusion $S \to X$ is a section.
Let $\psi'$ be the induced tower over $S$.

We may assume that $K$ is perfectoid with tilt $L$
and consider objects over $\tilde{\bC}^{[s,r]}_\psi$ killed by some irreducible $t \in \tilde{\bC}^{[s,r]}_K$ coprime to $\varphi^m(t_\theta)$ for all $m \in \ZZ$.
Let $\tilde{M}$ be such an object; by Corollary~\ref{C:twisted projective torsion}, it is a finite projective
$\tilde{\bC}^{[s,r]}_\psi$-module. 
By Theorem~\ref{T:add tilde2} and Theorem~\ref{T:standard toric decompleting}, $\tilde{M}$ descends to
a pseudocoherent $(\varphi, \Gamma)$-module over $\bC^{[s,r]}_\psi$.
As in the proof of Proposition~\ref{P:twisted stable submodules}, the ring $\bC^{[s,r]}_\psi/(t)$
is isomorphic to a localization of $\bC^{[s,r]}_{\psi'}\{T_n^{\pm}\}$, and the action of the $n$-th copy of $\ZZ_p$ in $\Gamma$ on this ring is via a substitution of the form $T_n \mapsto [\overline{q}]T_n$ for some
$\overline{q} \in 1 + \gothm_L$; the fact that $t$ is coprime to $\varphi^m(t_\theta)$ for all $m \in \ZZ$ ensures that $[\overline{q}]$ is not a root of unity. Consequently, the Lie algebra of $\ZZ_p$ then acts via a connection with respect to the nonzero derivation $\log([\overline{q}]) \frac{\partial}{\partial T_n}$, so the $p$-adic Cauchy theorem \cite[Proposition~9.3.3]{kedlaya-course} implies that one can trivialize the action over some neighborhood of the zero locus of $T_n-1$. This yields the claim.
\end{proof}

\begin{remark}
There is also an analogue of Theorem~\ref{T:limit to Zariski closed} for pseudocoherent $(\varphi^a, \Gamma)$-modules over $\tilde{\bA}_X$ killed by $p^n$ for some positive integer $n$. Namely, by 
Theorem~\ref{T:pseudocoherent type A} this equates to an analogous statement about torsion \'etale local systems, which follows from the henselian property of adic local rings \cite[Theorem~1.2.8, Lemma~2.4.17]{part1}.
\end{remark}

\begin{remark} \label{R:bad pushforward1}
We will prove in a subsequent paper that for $f: X' \to X$ smooth and proper, 
the functors $R^i f_{\proet *}$ take $\calC_{X'}$ into $\calC_X$,
$\calC^{[s,r]}_{X'}$ into $\calC^{[s,r]}_X$,
and $\CPhi^a_{X'}$ into $\CPhi^a_{X}$.
In light of Remark~\ref{R:bad pushforward2}, the hypothesis of smoothness cannot be entirely removed, but it can probably be somewhat relaxed.
One natural guess would be algebraic flatness of the induced maps of
analytic local rings; as discussed in \cite{ducros}, this property is
stable under base change for a proper morphism of good analytic spaces
(i.e., one obtained by glueing Gelfand spectra using their topology,
e.g., an affinoid space, an analytification of a scheme, or the generic
fiber of an affine or proper formal scheme), though not more generally.
\end{remark}

\begin{remark}
In cases where a finite generation property is not appropriate, one might try to form ``quasicoherent $\varphi^a$-modules'' by taking direct limits of objects in $\calC\Phi^a_X$. We do not know if the resulting category serves any useful purpose, e.g., by making it possible to take direct images in greater generality than in the pseudocoherent case.
\end{remark}
 
 \appendix
\section{Appendix: Errata for \texorpdfstring{\cite{part1}}{part 1}}

We record some errata for \cite{part1}. Thanks to David Hansen, Florian Herzig, Zonglin Jiang, Lucas Mann, Peter Schneider, and Kazuma Shimomoto for contributions included here.

\begin{itemize}
\item
Lemma 1.5.2: change ``the bottom left block'' to ``the top left block''.

\item
Definition 2.4.11: in two places, change ``the generic point'' to ``the unique closed point''.

\item
Remark 2.5.3: change ``[38, Th\'eor\`eme~2.6]'' to ``[38, Th\'eor\`eme 2.13]''.

\item
Definition 2.8.1: the stated equivalence is false; conditions (a)--(c) are equivalent and should be taken as the definition of a uniform Banach ring. Condition (d) follows from (a)--(c), and implies that the spectral norm defines the same norm topology as the original norm, but does not imply (a)--(c): for any nonarchimedean field $K$ with norm $\left| \bullet \right|$, the function 
\[
x \mapsto \begin{cases} \left| x \right| & \mbox{if $\left| x \right| \leq 1$} \\
\left| x \right|(1 + \log \left| x \right|) & \mbox{if $\left| x \right| > 1$}
\end{cases}
\]
is a submultiplicative norm defining the same topology as its spectral norm $\left| \bullet \right|$, but not equivalent to the latter. That said, if $A$ is a Banach ring containing a topologically nilpotent unit $z$
for which $\left| z \right| \left| z^{-1} \right| = 1$, then one may check that (d) does imply (a)--(c).

\item
Lemma 3.2.6: 
In the first paragraph of the proof, the ring $S$ given by the formula is not quite the one needed: it represents the functor taking $T$ to $n$-tuples of invariant elements of $M \otimes_R T$. This
functor has a closed-open subfunctor taking $T$ to invariant bases of $M \otimes_R T$, and the object representing this is the one that is needed. (With respect to a reference invariant basis, the determinant of any given invariant $n$-tuple belongs to $\FF_p$, so its nonvanishing is a closed-open condition.) Similar considerations apply in the last paragraph.

\item
Proposition 3.2.7: change $M(T) = W_n(R_n)^G$ to $M(T) = (W_n(R_n)
\otimes_{\ZZ/p^n \ZZ} (\ZZ/p^n \ZZ)^d)^G$ (where the $G$-action is the
diagonal one).

\item
Lemma 3.5.4: To apply Artin's lemma, one must also check that the action of $G$ on $E$ is faithful. To wit, any nontrivial $g \in G$ moves some $\overline{x} \in M$, and for $n$ sufficiently large we may write
\[
\beta\left(\theta([\overline{x}^{p^{-n}}]) - g(\theta([\overline{x}^{p^{-n}}]))
- \theta([(\overline{x}-g(\overline{x}))^{p^{-n}}]) \right) \leq p^{-1} \alpha\left( \overline{x} \right)^{p^{-n}}
\]
to see that $g$ also moves $\theta([\overline{x}^{p^{-n}}])$.

\item 
Starting in section 3.6, the definition of ``primitive of degree 1" for
an element $z$ in $W(R^+)$ should also include the condition that
$\overline{z} \in R^\times$.

\item
Definition 3.6.4: while Lemma 3.4.5 does imply that $R(A)$ is a perfect uniform Banach $\FF_p$-algebra,
an additional argument is required to show that $R^+(A^+)$ is a ring of integral elements in $R(A)$,
and thus to conclude that $(R(A), R^+(A^+))$ is an adic Banach $\FF_p$-algebra. 
\begin{itemize}
\item
To see that $R^+(A^+)$ is open in $R(A)$, note that if $\overline{x} \in R^+(A^+)$ has norm less than $p^{-1}$, then $\theta([\overline{x}])$ has norm less than $p^{-1}$ and hence is divisible by $p$ in $A^+$; by Lemma 3.6.3, we can write $\theta([\overline{x}]) = p \theta(y)$ for some $y \in W(R^+(A^+))$, at which point Lemma 3.6.3 implies that $[\overline{x}]-py$ is divisible by $z$ in $W(R^+(A^+))$ and so $\overline{x}$ is divisible by $\overline{z}$ in $R^+(A^+)$.
\item
To see that $R^+(A^+)$ is integrally closed in $R(A)$, in light of the previous point we need only check that $R^+(A^+)/\mathfrak{m}_{R(A)}$ is integrally closed in $\mathfrak{o}_{R(A)}/\mathfrak{m}_{R(A)}$.
However, these rings are respectively isomorphic to $A^+/\mathfrak{m}_{A}$ and $\mathfrak{o}_{A}/\mathfrak{m}_A$, so the fact that $A^+$ is integrally closed in $\mathfrak{o}_A$ implies the claim.
\end{itemize}

\item
Proposition 3.6.9: in the statement of (a), change ``and surjective if $\psi$ is'' to ``and surjective if $\overline{\psi}$ is''. Added the next one. Six lines from the bottom of the proof, change $p^{-1} \overline{\beta}(\overline{y})$ to $p^{-1/2} \overline{\beta}(\overline{y})$.

\item
Proposition 3.6.11: this proof is incomplete. For a correction, see Theorem~\ref{T:tensor product} herein.

\item
Theorem 3.6.14: in the unique displayed equation in the proof, replace $U \cap \calM(A)$ with $\calM(A)$.

\item
Theorem 3.6.15: change ``proprety'' to ``property''.

\item
Theorem 3.6.17: four lines from the bottom of the proof, change $\gamma(\overline{x} - \overline{y})$ to $\overline{\beta}(\overline{x} - \overline{y})$.

\item
Proposition 3.6.19: this proof is incomplete, as it does not show that $B$ is uniform.
For a correction, see Theorem~\ref{T:Fontaine perfectoid compatibility}(v) herein.

\item
Proposition 3.6.22: the proof
cites ``[122, Tag 07S5]''. This does imply the statement we
gave, but the exact statement is itself available for citation: it is
\emph{loc.\ cit.}, Tag 03BM.

\item
Remark 3.6.27: the reference to ``[114, Conjecture 2.16]'' points to the
wrong paper. The correct paper is the one herein labeled \cite{scholze-cdm}.
\item
Corollary 5.2.12: in the proof, the second displayed equation should read
\[
\limsup_{r \to +\infty} \lambda(\alpha^r)(x)^{1/r}
\leq
\sup_{r \in [s,qs]} \{ \limsup_{m \to \infty} p^{-mn/(rq^m)} \lambda(\alpha^r)(x)^{1/(rq^m)} \}\leq 1.
\]

\item
Proposition 6.2.4: the matrix $1+X$ should be $[\overline{\pi}^s] + X$.

\item
Lemma 5.5.5: in the last three lines of the proof, change three occurrences of $\overline{y}_0$ to $\overline{x}_0$ and two occurrences of $\overline{\alpha}$
to $\alpha$.

\item
Remark 7.4.12: the equality $H^1_{\varphi^a}(M(-n)) = 0$ should read 
$H^1_{\varphi^a}(M(n)) = 0$.

\item
Definition 8.1.6: 
Delete the second sentence (``Equivalently, ... open subset.'')

\item
Remark 8.1.7: The last sentence should read ``Any quasicompact morphism of schemes is spectral.''

\item
Definition 8.2.11: change ``spaceLet'' to ``space. Let''.

\item
Proposition 8.2.20: change ``$Z_{ij}$ is finite \'etale and $Y_{ij}$ is an open immersion'' to
``$Z_{ij} \to Y$ is finite \'etale and $Y_{ij} \to Z_{ij}$ is an open immersion''.

\item
Theorem 8.6.4: delete ``applying Theorem 6.2.9 and'' (there is no such theorem).

\item
Proposition 9.2.6: change ``subodmains'' to ``subdomains''.

\end{itemize}

\end{document}